\documentclass[a4paper,10pt]{amsart}
\usepackage{times}
\usepackage{amsmath,amsthm,amssymb,enumerate}
\usepackage{epsfig}
\usepackage{esint}
\usepackage{amssymb}
\usepackage{amsmath}
\usepackage{amssymb}
\usepackage{amsmath,amsthm}
\usepackage[latin1]{inputenc}
\usepackage{bbm}
\usepackage{amsfonts}
\usepackage{amsxtra}
\usepackage{euscript,mathrsfs}
\usepackage{color}
\usepackage[left=3.8cm,right=3.8cm,top=3.25cm,bottom=2.5cm]{geometry}
\usepackage[colorlinks=true, linktocpage=true, linkcolor=red!70!black, citecolor=green!50!black]{hyperref}
\allowdisplaybreaks

\usepackage{enumitem}
\setenumerate{label={\rm (\alph{*})}}

\usepackage{amsfonts}
\usepackage{amsxtra}

\numberwithin{equation}{section}

\theoremstyle{plain}

\newtheorem*{theorem*}{Theorem}
\newtheorem{theorem}{Theorem}[section]
\newtheorem{lemma}[theorem]{Lemma}
\newtheorem{proposition}[theorem]{Proposition}

\newtheorem{corollary}[theorem]{Corollary}

\newtheorem{definition}[theorem]{Definition}

\theoremstyle{remark}

\newtheorem{remark}[theorem]{Remark}
\newtheorem{example}[theorem]{Example}

\allowdisplaybreaks
\usepackage{tikz-cd}
\usetikzlibrary{decorations.pathreplacing}

\newcommand{\R}{\mathbb R}
\newcommand{\N}{\mathbb N}

\newcommand{\E}{\mathbb E}

\newcommand{\dd}{\mathrm{d}}
\newcommand{\dx}{\,\mathrm{d}x}

\newcommand{\dista}{\mathrm{dist}}

\DeclareMathOperator{\BV}{BV}

\newcommand{\dif}{\mathrm{d}}

\renewcommand{\leq}{\leqslant}

\renewcommand{\dif}{\operatorname{d}\!}
\newcommand{\lebe}{\operatorname{L}}
\newcommand{\sobo}{\operatorname{W}}

\newcommand{\dashint}{\fint}

\newcommand{\locc}{\operatorname{loc}}

\newcommand{\hold}{\operatorname{C}}

\newcommand{\bv}{\operatorname{BV}}

\newcommand{\ball}{\operatorname{B}}
\newcommand{\di}{\operatorname{div}}
\newcommand{\A}{\mathbb{A}}
\renewcommand{\locc}{\mathrm{loc}}

\newcommand{\spt}{\operatorname{spt}}

\newcommand{\trace}{\mathrm{tr}}

\newcommand{\qc}{\mathrm{qc}}
\allowdisplaybreaks

\newcommand{\mres}{\mathbin{\vrule height 1.6ex depth 0pt width
0.13ex\vrule height 0.13ex depth 0pt width 1.3ex}}

\numberwithin{equation}{section}

\begin{document}
\keywords{Variational integral, partial regularity, non-standard growth conditions, $(p,q)$-type growth conditions, linear growth conditions, quasiconvexity, signed integrands, relaxed minimizers, coercivity, lower semicontinuity, Fubini-type theorems, functions of bounded variation.}

\title[Quasiconvexity and Relaxed minimizers]{Quasiconvex functionals of $(p,q)$-growth and \\ the partial regularity of relaxed minimizers}
\author[F.~Gmeineder]{Franz Gmeineder}
\address{University of Konstanz, Department of Mathematics and Statistics, Universit\"{a}tsstra\ss e 10, 78457 Konstanz, Germany.}
\email{franz.gmeineder@uni-konstanz.de} 
\author[J.~Kristensen]{Jan Kristensen}
\address{University of Oxford, Mathematical Institute, Radcliffe Observatory Quarter, OX2 6HG, Oxford, United Kingdom.}
\email{jan.kristensen@maths.ox.ac.uk} 
\date{\today}

\maketitle
\date\today

\begin{abstract}
We establish $\hold^{\infty}$-partial regularity results for relaxed minimizers of strongly quasiconvex functionals 
\begin{align*}
\mathscr{F}[u;\Omega]:=\int_{\Omega}F(\nabla u)\dif x,\qquad u\colon\Omega\to\R^{N},
\end{align*}
subject to a $q$-growth condition $|F(z)|\leq c(1+|z|^{q})$, $z\in\R^{N\times n}$, and natural $p$-mean coercivity conditions on $F\in\hold^{\infty}(\R^{N\times n})$ for the basically optimal exponent range $1\leq p\leq q<\min\{\frac{np}{n-1},p+1\}$. With the $p$-mean coercivity condition being stated in terms of a strong quasiconvexity condition on $F$, our results include pointwise $(p,q)$-growth conditions as special cases. Moreover, we directly allow for signed integrands which is natural in view of coercivity considerations and hence the direct method, but is novel in the study of relaxed problems. 

In the particular case of classical pointwise $(p,q)$-growth conditions, our results extend  the previously known exponent range from \textsc{Schmidt}'s foundational work~\cite{SchmidtPR} for non-negative integrands to the maximal range for which relaxations are meaningful, moreover allowing for $p=1$. As further key novelties, our results apply to the canonical class of signed integrands and do not rely in any way on measure representations \`{a} la \textsc{Fonseca \& Mal\'{y}}~\cite{FM}.  
\end{abstract}

\tableofcontents 

\section{Introduction}
\subsection{Aim and scope}\label{sec:aimandscope}
A key problem in the vectorial Calculus of Variations that has defined the development of the field in large parts, is the study of variational principles 
\begin{align}\label{eq:varprin}
\text{to minimize}\;\;\;\mathscr{F}[u;\Omega]:=\int_{\Omega}F(\nabla u)\dif x
\end{align}
subject to certain constraints on the admissible competitors $u\colon\Omega\to\R^{N}$. Here, $\Omega\subset\R^{n}$ is an open and bounded set with Lipschitz boundary $\partial\Omega$, and $F\colon\R^{N\times n}\to\R$ is a continuous integrand. When we specify suitable Dirichlet classes as side constraints for \eqref{eq:varprin}, the existence of minima in spaces of weakly differentiable functions necessitates both growth bounds and suitable lower semicontinuity properties of $\mathscr{F}[-;\Omega]$, the latter being reflected by semiconvexity assumptions on $F$. 

Routine assumptions in this direction are given by the \emph{$p$-growth integrands}, $1\leq p <\infty$, for which there exist $c_{1},c_{2},c_{3}>0$ such that 
\begin{align}\label{eq:prowth}
c_{1}|z|^{p}-c_{2}\leq F(z)\leq c_{3}(1+|z|^{p})\qquad\text{for all}\;z\in\R^{N\times n}.
\end{align}
Such integrands comprise the classical Dirichlet integrand $z\mapsto \frac{1}{2}|z|^{2}$ for $p=2$ or the area-integrand $z\mapsto\sqrt{1+|z|^{2}}$ for $p=1$ as known from the non-parametric minimal surface problem. The lower semicontinuity of $\mathscr{F}[-;\Omega]$ on spaces of weakly differentiable functions in turn is strongly intertwined with \textsc{Morrey}'s \emph{quasiconvexity}~\cite{Morrey1}, a condition that reduces to convexity for $N=1$ or $n=1$, but is a strict generalisation of convexity for all other dimensional constellations. We recall that $F\in\hold(\R^{N\times n})$ is called \emph{quasiconvex at} $z_{0}\in\R^{N\times n}$ provided 
\begin{align}
F(z_{0})\leq \dashint_{\ball_{1}(0)}F(z_{0}+\nabla\varphi)\dif x \qquad\text{for all}\;\varphi\in\hold_{c}^{\infty}(\ball_{1}(0);\R^{N}),
\end{align}
and simply \emph{quasiconvex} provided it is quasiconvex at every $z_{0}\in\R^{N\times n}$. By its link to lower semicontinuity -- see \cite{AcerbiFusco,Marcellini0,Meyers,Morrey1} --  and hence from the perspective of the direct method, it is thus particularly important to understand the regularity properties of minimizers for quasiconvex variational integrals.

Since the variational principles \eqref{eq:varprin} are genuinely vectorial problems, even in presence of suitable ellipticity and smoothness assumptions on $F$, full $\hold^{1,\alpha}$-H\"{o}lder regularity of minima \`{a} la \textsc{De Giorgi, Nash} and \textsc{Moser} \cite{DeGiorgi0,Moser,Nash} cannot be expected, see e.g.~\cite{DeGiorgi,GiustiMiranda,HaoLN,Necas1,Necas2}. The suitable substitute here is ($\hold^{1,\alpha}$-)\emph{partial regularity}, meaning that the minimizers are of class $\hold_{\locc}^{1,\alpha}$ for any $0<\alpha<1$ on a relatively open set of full Lebesgue measure. Inspired by developments in geometric measure theory following the works of \textsc{Almgren} \cite{Almgren} and \textsc{Allard} \cite{Allard}, the systematic study of partial regularity for quasiconvex problems was initiated in the 1980s by the seminal works of \textsc{Evans} \cite{Ev1}, \textsc{Fusco \& Hutchinson}~\cite{FuHo1}, \textsc{Acerbi \& Fusco}~\cite{AcFu1} and \textsc{Giaquinta \& Modica}~\cite{GiaquintaModica}. Since then, partial regularity for quasiconvex variational problems has witnessed a vast number of contributions, see \cite{CFM,CPDN,DLSV,Duzaar1,Duzaar1A,Duzaar2,Gm20,GK1,KT,KuusiMingione} for a non-exhaustive list and \textsc{Mingione}~\cite{Mingione1,Mingione2}, \textsc{Giusti}~\cite{Giusti} for surveys. 

If the growth assumption \eqref{eq:prowth} is assumed, then the functional $\mathscr{F}[-;\Omega]$ is both well-defined and coercive on Dirichlet subclasses of $\sobo^{1,p}(\Omega;\R^{N})$. This, in turn, is not the case for \emph{$(p,q)$-growth} functionals, $1\leq p<q<\infty$, where \eqref{eq:prowth} is typically replaced by 
\begin{align}\label{eq:pqgrowth}
c_{1}|z|^{p}-c_{2}\leq F(z)\leq c_{3}(1+|z|^{q})\qquad\text{for all}\;z\in\R^{N\times n}.
\end{align}
Following the foundational work of~\textsc{Marcellini} \cite{Marcellini1,Marcellini2,Marcellini3}, the study of variational integrals obeying $(p,q)$-growth conditions has faced a vast number of contributions with considerable recent interest; see \cite{BellaSchaeffner1,BellaSchaeffner2,BildFuchs1,Bild,BildFuchs2,Boegelein,BFM,CPK,DF1,DF2,EspositoLeonettiMingione,EspositoMingione1,
EspositoLeonettiMingione1,FM,FMParma,FonMar,Kristensen97,Kristensen98,Kristensen99,Mingione1,PasSie,SchmidtPR0,SchmidtPR1,SchmidtPR} for an incomplete list that only scratches the wealth of results available so far.

When assuming \eqref{eq:pqgrowth} in the context of Dirichlet variational principles~\eqref{eq:varprin}, the natural domain of definition $\sobo_{v}^{1,q}(\Omega;\R^{N}):=v+\sobo_{0}^{1,q}(\Omega;\R^{N})$ for $v\in\sobo^{1,q}(\Omega;\R^{N})$ is strictly smaller than the space $\sobo_{v}^{1,p}(\Omega;\R^{N})$ in which minimising sequences are bounded and thus (weak) compactness can be achieved. In the following, we shall then informally refer to $\sobo_{v}^{1,p}(\Omega;\R^{N})$ as a \emph{compactness space} for $\mathscr{F}[-;\Omega]$. Adopting this terminology,~\eqref{eq:pqgrowth} expresses that the natural domain of definition differs from the compactness space $\sobo_{v}^{1,p}(\Omega;\R^{N})$ for $\mathscr{F}[-;\Omega]$. This circumstance is typical for variational problems arising in elastic cavitation theories (cf.~\textsc{Ball}~\cite{Ball2}), and requires $\mathscr{F}$ to be suitably extended or \emph{relaxed}. Deferring the precise set-up to Section~\ref{sec:relaxintro} below, it is clear that for such relaxations to be meaningful, some dimensional balance between $p$ and $q$ is required. In the wake of the foundational work of \textsc{Marcellini} et al. \cite{FonMar,Marcellini1} and, in particular, \textsc{Fonseca \& Mal\'{y}} et al. \cite{BFM,FM,FMParma} (also see \cite{AcerbiDalMaso,CDM,EspositoMingione1,Maly1,Maly2}), the natural threshold condition in this context is given by
\begin{align}\label{eq:threshold1}
1\leq p\leq q <q_{p}:=\frac{np}{n-1}.
\end{align}
Due to \textsc{Mal\'{y}}~\cite{Maly1}, who established the potential triviality of the natural relaxations for $q>q_{p}$, or \textsc{Acerbi \& Dal Maso} \cite{AcerbiDalMaso}, who demonstrated the failure of subadditivity for $q=q_{p}$, condition \eqref{eq:threshold1} is the canonical exponent range for which both the existence and regularity of minima should be addressed. First conditions that link $p$ and $q$ to the regularity of minima have been identified in the convex context by \textsc{Marcellini} \cite{Marcellini2,Marcellini3}, and since then have been expanded in various directions; see e.g.  \cite{AcFu2,Breit,Boegelein,CPK,CPK1,DF1,DF2,EspositoLeonettiMingione,FuHo2,FuHo3,Koch,SchmidtPR,SchmidtPR1,SchmidtPR4}. 

Most of these regularity results are confined to the convex situation and do not apply to genuinely \emph{quasiconvex} functionals\footnote{For the intermediate polyconvexity, see~\textsc{Fusco \& Hutchinson}~\cite{FuHo2,FuHo3}.}. At present, there are only very few contributions that actually deal with the partial regularity of minimizers for the relaxations of quasiconvex variational principles~\eqref{eq:varprin} subject to the $(p,q)$-growth condition~\eqref{eq:pqgrowth}. Most notably, by the pioneering work of \textsc{Schmidt}~\cite{SchemmSchmidt,SchmidtPR0,SchmidtPR} and especially his seminal paper~\cite{SchmidtPR}, partial regularity of such \emph{relaxed} minimizers is known to hold in the exponent range
\begin{align}\label{eq:SchmidtRange1}
1<p\leq q<p+\frac{\min\{2,p\}}{2n}.
\end{align}
Based on an intricate combination of measure representations and the construction of low layer energy competitors \`{a} la \textsc{Fonseca \& Mal\'{y}}~\cite{FM}, this result still displays the landmark for partial regularity in the setting described above. 

While the exponent range~\eqref{eq:SchmidtRange1} was slightly amplified in the subquadratic case $p\leq 2$ by \textsc{Schemm \& Schmidt}~\cite{SchemmSchmidt} to $1<p\leq q<\frac{2n-1}{2n-2}p$, there is still a crucial exponent gap in view of the canonical range~\eqref{eq:threshold1}; recall that, subject to~\eqref{eq:pqgrowth}, ~\eqref{eq:threshold1} is precisely the range for which $\mathscr{F}$ can be meaningfully relaxed to $\sobo^{1,p}$.

As is by now well-known, a wealth of quasiconvex integrands can very well be unbounded below yet leading to coercive variational integrals. More precisely, since $F$ is not assumed convex but only quasiconvex, the pointwise condition~\eqref{eq:pqgrowth} implies, but is \emph{far from} being necessary for the coercivity of $\mathscr{F}[-;\Omega]$ on Dirichlet subclasses of $\sobo^{1,p}(\Omega;\R^{N})$. For such genuinely signed integrands, while being generic in the quasiconvex situation, no partial regularity results are known at all. On an even more fundamental level, there is no systematic theory for relaxations in the quasiconvex, signed context as for integrands satisfying the pointwise bounds~\eqref{eq:pqgrowth}. Most crucially, this is so because the signed case comes with substantially weaker lower semicontinuity results than those which are available in the unsigned context.

By the interplay of compactness and regularity it is clear that the borderline case $p=1$,  still being included in the natural range~\eqref{eq:threshold1}, comes with a variety of other obstructions that are invisible in the superlinear growth scenario. In fact, even for $p=q=1$, potential concentration effects here already require a suitable relaxation to $\bv(\Omega;\R^{N})$, the space of functions of bounded variation (see Section~\ref{sec:functionspaces} for more detail). In this situation, the partial regularity of minimizers has been established only recently by the authors in~\cite{GK1}. Since in the growth regime~\eqref{eq:pqgrowth} for $p=1$ concentration effects are anticipated to enter in an even more delicate manner, the underlying proofs must be sufficiently robust to work in this borderline case as well. 

Striving for the essential hypotheses on $F$ that let the direct method apply to variational principles~\eqref{eq:varprin} \emph{and} let one prove the partial regularity of minimizers, it is thus natural to include quasiconvex integrands $F$ that are potentially unbounded from below, yet yielding the requisite coerciveness of $\mathscr{F}[-;\Omega]$. As established by~\textsc{Chen} and the second author~\cite{CK} (also see the discussion in Section~\ref{sec:props} below), this requirement is equivalent to the $p$\emph{-strong quasiconvexity} of $F$ at some $z_{0}\in\R^{N\times n}$, meaning that there exists $\ell>0$ such that
\begin{align}\label{eq:pstrongQCintro}
z\mapsto F(z)-\ell V_{p}(z) := F(z)-\ell\big((1+|z|^{2})^{\frac{p}{2}}-1 \big)\;\;\text{is quasiconvex at $z_{0}$}.
\end{align}
This assumption is automatically implied by quasiconvex integrands obeying~\eqref{eq:pqgrowth} but not vice versa. As such, the natural generalisation for the pointwise $(p,q)$-growth condition~\eqref{eq:pqgrowth} is
\begin{align*}
\begin{cases} \text{(a)\; a pointwise $q$-growth condition}\;|F(\cdot)|\leq L(1+|\cdot|^{q})\;\emph{\text{and}} \\  \text{(b)\; a $p$-strong quasiconvexity condition on $F$}.\end{cases}
\end{align*}
Following the foundational works of \textsc{Evans}~\cite{Ev1} or \textsc{Acerbi \& Fusco}~\cite{AcFu1}, signed integrands that verify a slight strengthening of~\eqref{eq:pstrongQCintro} -- also see hypotheses~\ref{item:H2} and~\ref{item:H2p} below -- display the natural framework within which the key partial regularity results in the standard growth case are stated. In view of this discussion, the aim of the present paper thus is
\begin{itemize}
\item[(i)] to \emph{extend the exponent range} for partial regularity of relaxed minimizers to the natural growth range given by~\eqref{eq:threshold1}, 
\item[(ii)] especially including the slightly more involved \emph{borderline case} $p=1$, where compactness can only be achieved in spaces of functions of bounded variation, and to 
\item[(iii)] simultaneously allow for the natural class of \emph{quasiconvex, signed} integrands.
\end{itemize}
This necessitates a functional set-up which is slightly different from that considered in \cite{BFM,FM,FMParma,FonMar,Marcellini1,SchmidtPR0,SchmidtPR}. Before we embark on the description of our main results, we thus pause and explain the definition of the underlying relaxed functionals first.

\subsection{Relaxed functionals and notions of minimality}\label{sec:relaxintro} 
In view of the direct method of the Calculus of Variations, we aim to extend (or relax) $\mathscr{F}[-;\Omega]$ by lower semicontinuity. Here, lower semicontinuity must keep track of the convergence that yields compactness. Specifically, if $\sup_{j\in\mathbb{N}}\mathscr{F}[u_{j};\Omega]<\infty$ implies $\sup_{j\in\mathbb{N}}\|u_{j}\|_{\sobo^{1,p}(\Omega)}<\infty$ for a given sequence $(u_{j})\subset\sobo^{1,q}(\Omega;\R^{N})$, the Banach-Alaoglu theorem yields weak compactness of $(u_{j})$ in $\sobo^{1,p}(\Omega;\R^{N})$ provided $1<p<\infty$. For instance, this is the case if~\eqref{eq:pqgrowth} is in action and $(u_{j})\subset\sobo_{v}^{1,q}(\Omega;\R^{N})$ satisfies $\sup_{j\in\mathbb{N}}\mathscr{F}[u_{j};\Omega]<\infty$. If, however, $p=1$, potential concentration effects only allow to conclude weak*-compactness in the space $\bv(\Omega;\R^{N})$ of functions of bounded variation. 

Consequently, depending on whether $1<p<\infty$ or $p=1$, the corresponding relaxations must be taken for weak convergence in $\sobo^{1,p}(\Omega;\R^{N})$ or weak*-convergence in $\bv(\Omega;\R^{N})$. As discussed above, it is natural to include quasiconvex integrands that are not necessarily bounded below. In this situation, it is well-known that functionals of the form~\eqref{eq:varprin} need not even prove sequentially weak(*)-lower semicontinuous on $\sobo^{1,p}(\Omega;\R^{N})$ if $F$ satisfies the standard growth bound $|F(z)|\leq L(1+|z|^{p})$; also see Lemma~\ref{lem:LSC}\emph{ff.}. This obstruction, however, vanishes if one considers weak(*)-convergence in \emph{fixed Dirichlet classes} and hereafter weak(*)-convergence of sequences with fixed boundary values. In consequence, to achieve that the relaxed functionals  \emph{extend} the original functionals indeed, we are bound to incorporate Dirichlet classes into the corresponding relaxations.

To motivate our definition of relaxed functionals, note that if $1<p\leq q<\infty$, $v\in\sobo^{1,q}(\Omega;\R^{N})$, $u\in\sobo^{1,p}(\Omega;\R^{N})$ and $(u_{j})\subset\sobo_{v}^{1,q}(\Omega;\R^{N})$ are such that $u_{j}\rightharpoonup u$ weakly in $\sobo^{1,p}(\Omega;\R^{N})$, then necessarily $u\in\sobo_{v}^{1,p}(\Omega;\R^{N})$. This is so because the boundary trace operator $\trace_{\partial\Omega}$ is continuous for weak convergence in $\sobo^{1,p}(\Omega;\R^{N})$. On the contrary, continuity of the boundary trace operator on $\bv(\Omega;\R^{N})$ does \emph{not} persist when $\bv(\Omega;\R^{N})$ is endowed with weak*-convergence. As such, weak*-limits of sequences $(u_{j})\subset\sobo_{v}^{1,q}(\Omega;\R^{N})$ might have a boundary trace different from that of $v$. Still, in order to grasp the boundary behaviour of such weak*-limits and as is customary in the context of linear growth functionals (cf.~\textsc{Giaquinta, Modica \& Sou\v{c}ek}~\cite{GiaModSouc}), some control of the traces of weak*-limits can be obtained by employing  \emph{solid boundary values}. Based on this discussion, we will henceforth focus on the slightly more involved case $p=1$ and alongside discuss the superlinear growth scenario of $1<p<\infty$, for which several of the above issues do not arise at all.

Let $\Omega\Subset\Omega'\subset\R^{n}$ be two open and bounded sets with Lipschitz boundaries. Given a map $u_{0}\in\sobo^{1,q}(\Omega';\R^{N})$, we define the associated \emph{solid boundary value class} via
\begin{align}\label{eq:SolidClass}
\mathscr{A}_{u_{0}}^{q}(\Omega,\Omega'):=\{v\in\sobo^{1,q}(\Omega';\R^{N})\colon\;v=u_{0}\;\text{in}\;\Omega'\setminus\overline{\Omega}\}. 
\end{align}
We then define for $\mathbf{u}\in\bv(\Omega';\R^{N})$ the \emph{weak*-relaxed functional} or \textsc{Lebesgue-Serrin-Marcellini} \emph{extension} for solid boundary values $u_{0}$ by
\begin{align}\label{eq:bdryrelaxed1}
\overline{\mathscr{F}}_{u_{0}}^{*}[\mathbf{u};\Omega,\Omega']:=\inf\left\{\liminf_{j\to\infty}\int_{\Omega'}F(\nabla u_{j})\dif x\colon\;\begin{array}{c} (u_{j})\subset \mathscr{A}_{u_{0}}^{q}(\Omega,\Omega'), \\ u_{j}\stackrel{*}{\rightharpoonup}\mathbf{u}\;\text{in}\;\bv(\Omega';\R^{N})\end{array} \right\},
\end{align}
where we here and in the sequel adopt the convention $\inf\emptyset =\infty$. Any sequence $(u_{j})\subset\mathscr{A}_{u_{0}}^{q}(\Omega,\Omega')$ with $u_{j}\stackrel{*}{\rightharpoonup}\mathbf{u}$ and $\overline{\mathscr{F}}_{u_{0}}^{*}[\mathbf{u};\Omega,\Omega']=\lim_{j\to\infty}\mathscr{F}[u_{j};\Omega']$ then is referred to as a \emph{generating} or \emph{recovery sequence} for $\overline{\mathscr{F}}_{u_{0}}^{*}[\mathbf{u};\Omega,\Omega']$.

 To obtain a functional solely defined on maps $u\in\bv(\Omega;\R^{N})$ and thereby remove the artificial dependence on the larger domain $\Omega'$, we put for $v\in\sobo^{1,q}(\Omega;\R^{N})$
\begin{align}\label{eq:bdryrelaxed2}
\overline{\mathscr{F}}_{v}^{*}[u;\Omega]:= \overline{\mathscr{F}}_{u_{0}}^{*}[\mathbf{u};\Omega,\Omega']-\mathscr{F}[u_{0};\Omega'\setminus\overline{\Omega}]
\end{align}
where $u_{0}\in\sobo_{0}^{1,q}(\Omega';\R^{N})$ is an extension of $v$ to $\Omega'$ and $\mathbf{u}$ is the extension of $u$ to $\Omega'$ by $u_{0}$. By a routine gluing principle in $\bv$, one then has $\mathbf{u}\in\bv(\Omega';\R^{N})$ indeed and competitors as required for~\eqref{eq:bdryrelaxed1} exist, cf.~Lemma~\ref{lem:SolidLemma}. As shall be established in Section~\ref{sec:props}, this definition does not depend on the specific extension $u_{0}$ and can be canonically generalised to maps $v\in\bv(\Omega;\R^{N})$ that attain boundary traces in $\sobo^{1-1/q,q}(\partial\Omega;\R^{N})$ (with the convention of $\sobo^{0,1}:=\lebe^{1}$); see Lemma~\ref{lem:welldefinedbdryvalues} and Corollary~\ref{cor:welldefinedbdryvalues1}. Specifically, $\overline{\mathscr{F}}_{v}^{*}[u;\Omega]$ is well-defined and lower semicontinuous for weak*-convergence in $\bv(\Omega;\R^{N})$. 

It is then customary to equally define for $1<p\leq q <\infty$ and maps $v\in\sobo^{1,q}(\Omega;\R^{N})$,  $u_{0}\in\sobo^{1,q}(\Omega';\R^{N})$, $u\in\sobo^{1,p}(\Omega;\R^{N})$ such that $\trace_{\partial\Omega}(u_{0})=\trace_{\partial\Omega}(v)$ $\mathscr{H}^{n-1}$-a.e. on $\partial\Omega$
\begin{align}\label{eq:bdryrelaxed3}
\begin{split}
\overline{\mathscr{F}}_{u_{0}}[\mathbf{u};\Omega,\Omega'] &:=\inf\left\{\liminf_{j\to\infty}\int_{\Omega'}F(\nabla u_{j})\dif x\colon\;\begin{array}{c} (u_{j})\subset \mathscr{A}_{u_{0}}^{q}(\Omega,\Omega'), \\ u_{j}\rightharpoonup \mathbf{u}\;\text{in}\;\sobo^{1,p}(\Omega';\R^{N})\end{array} \right\},\\
\overline{\mathscr{F}}_{v}[u;\Omega] &:= \overline{\mathscr{F}}_{u_{0}}[\mathbf{u};\Omega,\Omega']-\mathscr{F}[u_{0};\Omega'\setminus\overline{\Omega}], 
\end{split}
\end{align}
where again $\mathbf{u}$ is the extension of $u$ to $\Omega'$ by $u_{0}$. Note that, different from~\eqref{eq:bdryrelaxed2}, the continuity of the trace operator for weak convergence in $\sobo^{1,p}(\Omega;\R^{N})$ and our convention $\inf\emptyset=\infty$  immediately yield $\overline{\mathscr{F}}_{v}[u;\Omega]=\infty$ provided $\mathscr{H}^{n-1}(\{x\in\partial\Omega\colon\;\trace_{\partial\Omega}(u)(x)\neq \trace_{\partial\Omega}(v)(x)\})>0$. Similarly as above, definition~\eqref{eq:bdryrelaxed3} then can be extended to $v\in\sobo^{1,p}(\Omega;\R^{N})$ with $\trace_{\partial\Omega}(v)\in\sobo^{1-1/q,q}(\partial\Omega;\R^{N})$. 

As key features of the relaxed functionals from above, we point out that in presence of quasiconvexity of $F$ and the bound $|F(\cdot)|\leq c(1+|\cdot|^{q})$, in the growth regime~\eqref{eq:threshold1} both coincide with $\mathscr{F}[-;\Omega]$ when restricted to Dirichlet classes $\sobo_{v}^{1,q}(\Omega;\R^{N})$ for maps $v\in\sobo^{1,q}(\Omega;\R^{N})$, cf.~Lemma~\ref{lem:consistency}. Therefore, they may be regarded as \emph{extensions} of $\mathscr{F}[-;\Omega]$ by lower semicontinuity  indeed. 

We may now proceed to the notion of minimality that will underlie our main results, Theorems~\ref{thm:main1} and \ref{thm:main2} below:
\begin{definition}[Minimality] \label{def:minimality}
Let $\Omega\subset\R^{n}$ be an open and bounded domain with Lipschitz boundary $\partial\Omega$. 
\begin{enumerate}
\item\label{item:defM1} Given $1=p\leq q < \infty$ and $v\in\bv(\Omega;\R^{N})$ such that $\trace_{\partial\Omega}(v)\in\sobo^{1-1/q,q}(\partial\Omega;\R^{N})$, we say that $u\in\bv(\Omega;\R^{N})$ is a \emph{$\bv$-minimizer of $\overline{\mathscr{F}}_{v}^{*}[-;\Omega]$  (subject to the Dirichlet constraint $v$)} if and only if $\overline{\mathscr{F}}_{v}^{*}[u;\Omega]<\infty$ and
\begin{align}\label{eq:MAINminimality1}
\overline{\mathscr{F}}_{v}^{*}[u;\Omega]\leq \overline{\mathscr{F}}_{v}^{*}[w;\Omega]\qquad\text{for all}\;w\in\bv(\Omega;\R^{N}). 
\end{align}
If the previous inequality holds for all $w\in\bv(\Omega;\R^{N})$ such that $u-w$ is compactly supported in $\Omega$, then we call $u$ a \emph{$\bv$-minimizer of $\overline{\mathscr{F}}_{v}^{*}[-;\Omega]$ for compactly supported variations}. 
\item\label{item:defM2} Given $1<p\leq q < \infty$ and $v\in\sobo^{1,p}(\Omega;\R^{N})$ such that $\trace_{\partial\Omega}(v)\in\sobo^{1-1/q,q}(\partial\Omega;\R^{N})$, we say that $u\in\sobo^{1,p}(\Omega;\R^{N})$ is a \emph{minimizer for $\overline{\mathscr{F}}_{v}[-;\Omega]$ (subject to the Dirichlet constraint $v$)} if and only if $\overline{\mathscr{F}}_{v}^{*}[u;\Omega]<\infty$ and 
\begin{align}\label{eq:MAINminimality2}
\overline{\mathscr{F}}_{v}[u;\Omega]\leq \overline{\mathscr{F}}_{v}[w;\Omega]\qquad\text{for all}\;w\in\sobo^{1,p}(\Omega;\R^{N}). 
\end{align}
If the previous inequality holds for all $w\in\sobo^{1,p}(\Omega;\R^{N})$ such that $u-w$ is compactly supported in $\Omega$, then we call $u$ a \emph{minimizer of $\overline{\mathscr{F}}_{v}[-;\Omega]$ for compactly supported variations}. 
\end{enumerate}
\end{definition}
Subject to natural growth, semiconvexity and coercivity conditions, the existence of ($\bv$-)minimizers will be established in Section~\ref{sec:props}. We wish to emphasize that, in the setting of Definition~\ref{def:minimality}~\ref{item:defM2}, inequality~\eqref{eq:MAINminimality2} could equivalently be replaced by requiring $\overline{\mathscr{F}}_{v}[u;\Omega]\leq \overline{\mathscr{F}}_{v}[w;\Omega]$ to hold for all $w\in\sobo_{v}^{1,p}(\Omega;\R^{N})$ because of $\overline{\mathscr{F}}_{v}[w;\Omega]=\infty$ if $\mathscr{H}^{n-1}(\{x\in\partial\Omega\colon\;\trace_{\partial\Omega}(w)(x)\neq\trace_{\partial\Omega}(v)(x)\})>0$. 

If $p=q=1$, the compactness space of functionals~\eqref{eq:varprin} is $\bv(\Omega;\R^{N})$ and hence already a proper superspace of the natural domain of definition $\sobo^{1,1}(\Omega;\R^{N})$. This is why this case already requires a proper relaxation for weak*-convergence. Assuming quasiconvexity of $F$, we then are in the classical setting of \textsc{Ambrosio \& Dal Maso} \cite{AD} (also see~\textsc{Fonseca \& M\"{u}ller}~\cite{FonMul0,FonMul}) for non-negative or \textsc{Rindler} and the second named author \cite{KR} for signed linear growth integrands. This provides us with an integral representation for $\overline{\mathscr{F}}_{v}^{*}[-;\Omega]$, whereby the $\bv$-minimality of $u\in\bv(\Omega;\R^{N})$ for $\overline{\mathscr{F}}_{v}^{*}[-;\Omega]$ then amounts to minimality of $u$ for the integral functional 
\begin{align}\label{eq:lingrowth}
\begin{split}
\overline{\mathscr{F}}_{v}^{*}[w;\Omega] = \int_{\Omega}F(\nabla w)\dif x & + \int_{\Omega}F^{\infty}\Big(\frac{\dif D^{s}w}{\dif|D^{s}w|}\Big)\dif|D^{s}w| \\ &  + \int_{\partial\Omega}F^{\infty}(\trace_{\partial\Omega}(w-v)\otimes\nu_{\partial\Omega})\dif\mathscr{H}^{n-1}
\end{split}
\end{align} 
over all $w\in\bv(\Omega;\R^{N})$. Here, $F^{\infty}(z)=\lim_{t\searrow 0}tF(\frac{z}{t})$ is the recession function associated with $F$ and $Dw=\nabla w\mathscr{L}^{n}\mres\Omega+\frac{\dif D^{s}w}{\dif|D^{s}w|}|D^{s}w|$ is the Lebesgue-Radon-Nikod\'{y}m decomposition of $Dw$; see Section~\ref{sec:functionspaces} for more detail. In this sense, the boundary integral in~\eqref{eq:lingrowth} penalises the deviation of $\trace_{\partial\Omega}(w)$ from the prescribed Dirichlet data. Albeit for $1=p<q<\infty$ the functional $\overline{\mathscr{F}}_{v}^{*}[u;\Omega]$ cannot be represented as a variational integral as in~\eqref{eq:lingrowth} (also see Section~\ref{sec:CompaIntro}), a similar penalisation effect is inherent in Definition~\ref{def:minimality}~\ref{item:defM1}. 

Let us note that almost by definition, the functionals $\overline{\mathscr{F}}^{*}$ or $\overline{\mathscr{F}}$ do not feature so-called Lavrentiev gaps in the following sense (also see \textsc{Buttazzo \& Mizel}~\cite{ButtazzoMizel} for a discussion of this matter and \textsc{Esposito, Leonetti \& Mingione}~\cite{EspositoLeonettiMingione1} in view of its impact on regularity): If, e.g. $1<p\leq q<\frac{np}{n-1}$ and $v\in\sobo^{1,q}(\Omega;\R^{N})$, then we have 
\begin{align}\label{eq:lavrentiev}
\begin{split}
\min_{u\in\sobo_{v}^{1,p}(\Omega;\R^{N})}\overline{\mathscr{F}}_{v}[u;\Omega] & \leq \inf_{u\in\sobo_{v}^{1,q}(\Omega;\R^{N})}\overline{\mathscr{F}}_{v}[u;\Omega] \\ & \!\!\!\!\!\stackrel{\text{Lem.~\ref{lem:consistency}}}{=} \inf_{u\in\sobo_{v}^{1,q}(\Omega;\R^{N})}\mathscr{F}[u;\Omega] \leq \min_{u\in\sobo_{v}^{1,p}(\Omega;\R^{N})}\overline{\mathscr{F}}_{v}[u;\Omega], 
\end{split}
\end{align}
where the ultimate inequality directly follows from the definition of $\overline{\mathscr{F}}_{v}[u;\Omega]$. 

Yet, the previous definition does not allow to conclude that a ($\bv$-)minimizer is local ($\bv$-)minimizer in the obvious sense; indeed, even in the unsigned case, the relaxed functionals only prove additive on \emph{certain} subsets of $\Omega$ (see Section~\ref{sec:CompaIntro}). Similarly as in \cite{SchmidtPR}, we shall therefore work with the following definition of \emph{local minimality}:
\begin{definition}[Local minimality for compactly supported variations]\label{def:locmin}
Let $\Omega\subset\R^{n}$ be open and bounded. 
\begin{enumerate}
\item\label{item:WLM2} Given $1=p\leq q <\infty$, we say that $u\in\bv_{\locc}(\Omega;\R^{N})$ is a \emph{local $\bv$-minimizer of $\overline{\mathscr{F}}^{*}$ for compactly supported variations} provided every $x_{0}\in\Omega$ has a neighbourhood $\omega\Subset\Omega$ with Lipschitz boundary $\partial\omega$ such that $u|_{\omega}$ is a BV-minimizer of $\overline{\mathscr{F}}_{u}^{*}[-;\omega]$ for compactly supported variations.
\item\label{item:WLM1} Given $1<p\leq q <\infty$, we say that $u\in\sobo_{\locc}^{1,p}(\Omega;\R^{N})$ is a \emph{local minimizer of $\overline{\mathscr{F}}$ for compactly supported variations} provided every $x_{0}\in\Omega$ has a neighbourhood $\omega\Subset\Omega$ with Lipschitz boundary $\partial\omega$ such that $u|_{\omega}$ is a minimizer of $\overline{\mathscr{F}}_{u}[-;\omega]$ for compactly supported variations.
\end{enumerate}
\end{definition}
By the definition of $\overline{\mathscr{F}}_{u}^{*}[-;\omega]$ and $\overline{\mathscr{F}}_{u}[-;\omega]$, cf.~\eqref{eq:bdryrelaxed1} and  \eqref{eq:bdryrelaxed3}, the inequality underlying the local ($\bv$-)minimality for compactly supported variations has the following implication: If, for some suitable open $\omega\ni x_{0}$ with Lipschitz boundary $\partial\omega$, $u$ is a ($\bv$-)minimizer of $\overline{\mathscr{F}}_{u}^{*}[-;\omega]$ or $\overline{\mathscr{F}}_{u}[-;\omega]$ for compactly supported variations, we necessarily have $\trace_{\partial\omega}(u)\in\sobo^{1-1/q,q}(\partial\omega;\R^{N})$ \emph{by the very definition} of $\overline{\mathscr{F}}_{u}^{*}[u;\omega]$ or $\overline{\mathscr{F}}_{u}[u;\omega]$. By a standard Fubini-type theorem for $\bv$- or Sobolev spaces (see e.g.~\cite[Lem.~2.3]{GK1}), for each $x_{0}\in\Omega$ we always find $r>0$ such that $\trace_{\partial\!\ball_{r}(x_{0})}(u)\in\bv(\partial\!\ball_{r}(x_{0});\R^{N})$ if $u\in\bv_{\locc}(\Omega;\R^{N})$ or $\trace_{\partial\!\ball_{r}(x_{0})}(u)\in\sobo^{1,p}(\partial\!\ball_{r}(x_{0});\R^{N})$ if $u\in\sobo_{\locc}^{1,p}(\Omega;\R^{N})$. Here, the weak differentiability is understood with respect to the distributional tangential derivatives, and $\mathscr{L}^{1}$-a.e. $0<r<\dista(x_{0},\partial\Omega)$ will do. Since e.g. the embedding $\sobo^{1,p}(\partial\!\ball_{r}(x_{0});\R^{N})\hookrightarrow\sobo^{1-1/q,q}(\partial\!\ball_{r}(x_{0});\R^{N})$ holds if and only if $q\leq\frac{np}{n-1}$, the local minimality condition in both cases \ref{item:WLM2} and \ref{item:WLM1} is genuinely non-vacuous for the exponent range $1\leq p \leq q \leq \frac{np}{n-1}$. In particular, the latter is satisfied by~\eqref{eq:threshold1}.
\section{Main results}\label{sec:mainresults}
We now proceed to display the main results of the present paper. To avoid overburdening technicalities and to emphasize that our focus is on the quasiconvexity rather than the minimal smoothness of the integrands, we assume $F\in\hold^{\infty}(\R^{N\times n})$ in the sequel. Moreover, we define for $1\leq p<\infty$ the auxiliary integrand $V_{p}\colon\R^{N\times n}\to\R$ via $V_{p}(z):=(1+|z|^{2})^{\frac{p}{2}}-1$ and set $V:=V_{1}$. 

Given $1\leq q <\infty$, we suppose that $F\colon\R^{N\times n}\to\R$ satisfies the following hypotheses: 
\begin{enumerate}[label={(H\arabic{*})},start=1]
\item\label{item:H1} There exists $L>0$ such that $|F(z)|\leq L(1+|z|^{q})$ holds for all $z\in\R^{N\times n}$. 
\item\label{item:H2} For any $m>0$ there exists $\ell_{m}>0$ such that, for all $z_{0}\in\R^{N\times n}$ with $|z_{0}|\leq m$, $F-\ell_{m} V$ is quasiconvex at $z_{0}$.  
\item\label{item:H3} $F\in\hold^{\infty}(\R^{N\times n})$. 
\end{enumerate}
As alluded to above, hypothesis~\ref{item:H2} is related to the coercivity of the functional $\mathscr{F}[-;\Omega]$, yet is substantially weaker than the pointwise bound~\eqref{eq:pqgrowth}; see the discussion in Section~\ref{sec:mcex}. Our first main result, being concerned with the case $p=1$, establishes the partial $\hold^{\infty}$-regularity of local $\bv$-minimizers for compactly supported variations: 
\begin{theorem}[$\varepsilon$-regularity, $p=1$]\label{thm:main1}
Let $\Omega\subset\R^{n}$ be open and bounded, and suppose that the variational integrand $F\colon\R^{N\times n}\to\R$ satisfies~\emph{\ref{item:H1}--\ref{item:H3}} with $1\leq q<\frac{n}{n-1}$. Then for any local $\bv$-minimizer $u$ of $\overline{\mathscr{F}}^{*}$ for compactly supported variations the following holds: Given $M>0$ and $x_{0}\in\Omega$, there exist $\varepsilon_{M}=\varepsilon_{M}(n,N,\ell_{M},L,M)>0$ and $R_{0}=R_{0}(n,N,F,M,x_{0},u)\in (0,\mathrm{dist}(x_{0},\partial\Omega))$ such that if $0<R<R_{0}$ satisfies
\begin{align}\label{eq:MainSmallnessA}
\left\vert\frac{Du(\ball_{R}(x_{0}))}{\mathscr{L}^{n}(\ball_{R}(x_{0}))}\right\vert<M
\end{align}
and
\begin{align}\label{eq:MainSmallnessB}
\dashint_{\ball_{R}(x_{0})}|\nabla u-(Du)_{\ball_{R}(x_{0})}|\dif x + \frac{|Du^{s}|(\ball_{R}(x_{0}))}{\mathscr{L}^{n}(\ball_{R}(x_{0}))}<\varepsilon_{M}, 
\end{align}
\emph{then $u$ is of class $\hold^{\infty}(\ball_{R/2}(x_{0});\R^{N})$}. In particular, if we let 
\begin{align*}
\mathrm{Reg}(u):=\{x_{0}\in\Omega\colon\;u\;\text{is of class $\hold^{\infty}$ in an open neighbourhood of $x_{0}$}\}, 
\end{align*} 
then we have with $\Sigma_{u}:=\Omega\setminus\mathrm{Reg}(u)$
\begin{align}
\begin{split}
\Sigma_{u} & = \Sigma_{u}^{1}\cup\Sigma_{u}^{2} \\ & := \left\{x_{0}\in\Omega\colon \liminf_{r\searrow 0}\Big(\dashint_{\ball_{r}(x_{0})}|\nabla u-(Du)_{\ball_{R}(x_{0})}|\dif x + \frac{|D^{s}u|(\ball_{r}(x_{0}))}{\mathscr{L}^{n}(\ball_{r}(x_{0}))}\Big)>0\right\} \\ 
& \;\;\cup \left\{x_{0}\in\Omega\colon\;\limsup_{r\searrow 0}\left\vert\frac{Du(\ball_{r}(x_{0}))}{\mathscr{L}^{n}(\ball_{r}(x_{0}))}\right\vert = \infty \right\}, 
\end{split}
\end{align}
whereby $\mathrm{Reg}(u)$ is relatively open in $\Omega$, $\mathscr{L}^{n}(\Sigma_{u})=0$ and thus $u$ is \emph{$\hold^{\infty}$-partially regular}.
\end{theorem}
Theorem~\ref{thm:main1} and Theorem~\ref{thm:main2} below display the core results of the announcement~\cite{GK2}. Deferring the specific conceptual and technical issues to Section~\ref{sec:keynov} below, we first comment on the general matters underlying the previous result. As is common in partial regularity proofs, Theorem~\ref{thm:main1} follows from an excess decay estimate in the neighbourhoods of points $x_{0}\in\mathrm{Reg}(u)$. The latter, in turn, is a consequence of the Caccioppoli inequality of the second kind and an improved distance estimate of the relaxed minimizer $u$ to suitable $\mathbb{A}$-harmonic approximations. Since we are dealing with \emph{relaxed minimizers}, the localisation scheme of \textsc{Evans}~\cite{Ev1} to arrive at the Caccioppoli inequality of the second kind requires the construction of suitable competitor maps. In order to get access to the requisite estimates, we are only allowed to use the $1$-strong quasiconvexity embodied by~\ref{item:H2}. By virtue of the growth bound~\ref{item:H1}, this entails the $\sobo^{1,q}$-quasiconvexity of $G_{m}:=F-\ell_{m}V$ at every $z_{0}$ with $|z_{0}|\leq m$, meaning that 
\begin{align}\label{eq:W1qQC}
G_{m}(z_{0})\leq \dashint_{\ball_{1}(0)}G_{m}(z_{0}+\nabla\varphi)\dif x\qquad\text{for all}\;\varphi\in\sobo_{0}^{1,q}(\ball_{1}(0);\R^{N}).
\end{align}
In general, however, ~\eqref{eq:W1qQC} does \emph{not} persist for competitor maps $\varphi\in\sobo_{0}^{1,1}(\ball_{1}(0);\R^{N})$. Thus, a particularly careful construction of $\sobo^{1,q}$- or low energy competitors is required, giving us access to~\eqref{eq:W1qQC} while still maintaining the optimal exponent range $1\leq q<\frac{n}{n-1}$, see~\eqref{eq:threshold1} for $p=1$. Further elaborating on this matter in Section~\ref{sec:keynov} below, the construction of $\mathbb{A}$-harmonic comparison maps is the second issue that is intricate in the $p=1$-growth regime. This is so because the interior trace space of $\bv$ is $\lebe^{1}$ along \emph{arbitrary spheres}  and second order boundary value problems with $\lebe^{1}$-boundary data are in general ill-posed. Similar issues in the usual quasiconvex, linear growth context have been addressed by the authors~\cite{Gm20,GK1}, but here we need to establish that the available distance estimates fit into the general line of argument; see Section~\ref{sec:proofmain} for the details. Even though we are exclusively interested in the partial regularity for smooth integrands, we wish to point out that our approach can be employed to yield $\hold_{\locc}^{1,\alpha}$-partial regularity of relaxed minimizers provided $F$ is only assumed to be of class $\hold^{2}$ together with ~\ref{item:H1} and~\ref{item:H2}; also compare with~\cite{BGIK} and~\cite{Li} in the context of functionals~\eqref{eq:lingrowth}. 

We now turn to the case $p>1$ and thus to the functionals $\overline{\mathscr{F}}$, where the natural threshold condition on $q$ is given by~\eqref{eq:threshold1}. Alongside~\ref{item:H1} and~\ref{item:H3}, we now consider the following substitute of~\ref{item:H2}:
\begin{enumerate}[label={$\text{(H\arabic{*})}_{p}$},start=2]
\item\label{item:H2p}  For any $m>0$ there exists $\ell_{m}>0$ such that for all $z_{0}\in\R^{N\times n}$ with $|z_{0}|\leq m$, $F-\ell_{m} V_{p}$ is quasiconvex at $z_{0}$.
\end{enumerate}
The corresponding analogue of Theorem~\ref{thm:main2} then is as follows:
\begin{theorem}[$\varepsilon$-regularity, $p>1$]\label{thm:main2}
Let $1<p\leq q < \min\{\frac{np}{n-1},p+1\}$. Let $\Omega\subset\R^{n}$ be open and bounded, and suppose that the variational integrand $F\colon\R^{N\times n}\to\R$ satisfies~\emph{\ref{item:H1}, \ref{item:H2p}} and \emph{\ref{item:H3}}. Then for any local minimizer $u\in\sobo_{\locc}^{1,p}(\Omega;\R^{N})$ of $\overline{\mathscr{F}}$ for compactly supported variations the following holds: Given $M>0$ and $x_{0}\in\Omega$, there exist $\varepsilon_{M}=\varepsilon_{M}(p,q,n,N,\ell_{M},L,M)>0$ and $R_{0}=R_{0}(n,N,F,M,x_{0},u)\in(0,\mathrm{dist}(x_{0},\partial\Omega)$ such that if $0<R<R_{0}$ satisfies
\begin{align}
\left\vert\dashint_{\ball_{R}(x_{0})}\nabla u\dif y\right\vert<M
\end{align}
and
\begin{align}
\dashint_{\ball_{R}(x_{0})}|\nabla u-(\nabla u)_{\ball_{R}(x_{0})}|^{p}\dif x <\varepsilon_{M}, 
\end{align}
\emph{then $u$ is of class $\hold^{\infty}(\ball_{R/2}(x_{0});\R^{N})$}. Especially, $u$ is \emph{$\hold^{\infty}$-partially regular}. 
\end{theorem} 
In establishing the previous theorem, the difficulties are similar to those encountered for Theorem~\ref{thm:main2}. This particularly concerns the Caccioppoli inequality of the second kind and the validity of the Euler-Lagrange system, from where Theorem~\ref{thm:main2} directly follows from the by now classical $\mathbb{A}$-harmonic approximation technique, cf.~\textsc{Duzaar} et al.~\cite{Duzaar1,Duzaar1A,Duzaar2} and its implementation in the $(p,q)$-growth context by \textsc{Schmidt}~\cite{SchmidtPR0,SchmidtPR1}; also see Remark~\ref{rem:plarger1excess}. 

The growth assumptions displayed in Theorem~\ref{thm:main1} and \ref{thm:main2} are close to optimal. On the one hand, the bound $q<\frac{np}{n-1}$ is natural in view of~\eqref{eq:threshold1} and the discussion afterwards. On the other hand, the requirement $q\leq p+1$ as visible in Theorem~\ref{thm:main2} (which is automatically satisfied in the context of Theorem~\ref{thm:main1}) comes up naturally by validity of the Euler-Lagrange system: Indeed, if $\nabla u\in\lebe^{p}(\Omega;\R^{N\times n})$ and $F$ is quasiconvex with \ref{item:H1}, then $|F'(z)|\leq c(1+|z|^{q-1})$ and so $F'(\nabla u)\in\lebe^{1}(\Omega;\R^{N})$ can be only asserted provided $q-1\leq p$. For the exponent range of Theorem~\ref{thm:main2}, this is automatically satisfied in the regime $1\leq p <n-1$, and then the exponent range of Theorem~\ref{thm:main2} precisely reduces to~\eqref{eq:threshold1}. Instead, if $p>n-1$ and so $p+1<\frac{np}{n-1}$, it is clear from the proof that we can allow for the range $n-1<p\leq q \leq p+1$. 

Let us note that, in the case of strongly convex functionals, the exponent condition $p\leq q<\min\{\frac{np}{n-1},p+1\}$ for partial regularity appears in \textsc{Passarelli di Napoli \& Siepe}~\cite{PasSie} for $p\geq 2$ first, and was improved by \textsc{Bildhauer \& Fuchs}~\cite{BildFuchs1} to the range $q<\frac{n+2}{n}p$ by employing the higher gradient integrability of minimizers; also see~\cite{Breit,CPK1,EspositoLeonettiMingione,Koch}. Such higher integrability results are particularly delicate if $p=1$, where even in the standard growth case a typical $(p,q)$-behaviour is encountered on the level of second derivatives; see \textsc{Bildhauer}~\cite{Bild2002a,Bild2002,Bild}, \textsc{Beck} et al.~\cite{BeckSchmidt,BeckSchmidt1,BeckBulicekMaringova,BeckBulicekGmeineder} and the authors~\cite{Gm20a,GK1A} for results in this direction. Clearly, convex techniques are ruled out in the context of Theorems~\ref{thm:main1} and~\ref{thm:main2}. In order to display the chief difficulties and main points of the present paper in the  \emph{quasiconvex case}, we thus briefly pause and compare our relaxed functionals with previously studied relaxations first.
\subsection{Comparison with other relaxations and main points}\label{sec:CompaIntro}
The reader will notice that the functionals $\overline{\mathscr{F}}$ and $\overline{\mathscr{F}}^{*}$ as introduced in Section~\ref{sec:relaxintro} slightly differ from those considered in previous contributions \cite{FM,SchmidtPR0,SchmidtPR}. More precisely, if $1<p<\infty$, \textsc{Fonseca \& Mal\'{y}}~\cite{FM} or \textsc{Schmidt} \cite{SchmidtPR} consider the locally relaxed functionals 
\begin{align}\label{eq:FonsecaMalySchmidtRelax}
\mathscr{F}_{\locc}[u;\omega]:=\inf\left\{\liminf_{j\to\infty}\int_{\omega}F(\nabla u_{j})\dif x\colon\;\begin{array}{c} (u_{j})\subset (\sobo_{\locc}^{1,q}\cap\sobo^{1,p})(\omega;\R^{N}),\\ u_{j}\rightharpoonup u\;\;\text{in}\;\sobo^{1,p}(\omega;\R^{N})\end{array}\right\}, 
\end{align}
where $\omega\Subset\Omega$ and the quasiconvex integrand $F\in\hold(\R^{N\times n})$ satisfies~\eqref{eq:pqgrowth} with~\eqref{eq:threshold1}. As established in~\cite{FM}, this implies that $\mathscr{F}_{\locc}[u;-]$ has a measure representation: If $u\in\sobo^{1,p}(\Omega;\R^{N})$ is such that $\mathscr{F}_{\locc}[u;\Omega]<\infty$, then there exists a uniquely determined finite (outer) Radon measure $\mu_{u}$ on $\Omega$ such that 
\begin{align}\label{eq:measrepIntro}
\mathscr{F}_{\locc}[u;\omega]=\mu_{u}(\omega)\qquad\text{for all open subsets}\;\omega\subset\Omega.
\end{align}
If $p=1$, where one makes the modification $u_{j}\stackrel{*}{\rightharpoonup}u$ in $\bv(\Omega;\R^{N})$ to arrive at the functionals $\mathscr{F}_{\locc}^{*}[-;\omega]$,  results in this direction are due to \textsc{Soneji}~\cite{Soneji1,Soneji2}. In either case, introducing the notions of \emph{weak or local minimality} for $\mathscr{F}_{\locc}$ as in \cite[Defs.~6.1,~6.2]{SchmidtPR}, the measure representation~\eqref{eq:measrepIntro} immediately yields that any minimizer is a local minimizer. Moreover, the identification of $\frac{\dif\mu_{u}}{\dif\mathscr{L}^{n}}$ (i.e., the density of the absolutely continuous part of $\mu_{u}$ for $\mathscr{L}^{n}$), as $F(\nabla u)$ (cf.~\cite{BFM,FM}) then is instrumental for the derivation of the Euler-Lagrange system satisfied by minima, see~\cite[Lems.~7.1--7.3]{SchmidtPR} and Remark~\ref{rem:measuredensityEuler} below.

One of the key reasons to employ the functionals $\overline{\mathscr{F}}$ or $\overline{\mathscr{F}}^{*}$ instead of those given by~\eqref{eq:FonsecaMalySchmidtRelax} is that the latter do \emph{only extend $\mathscr{F}$ if $F$ is bounded from below}. In fact, if $F$ is quasiconvex and unbounded from below, it might happen that 
\begin{align*}
\mathscr{F}_{\locc}[u;\Omega]\neq \mathscr{F}[u;\Omega] 
\end{align*}
even for $u\in\sobo^{1,q}(\Omega;\R^{N})$ with $1<p\leq q<\frac{np}{n-1}$; see Section~\ref{sec:LSCsimple}. This, in turn, is not the case for the functionals $\overline{\mathscr{F}}_{v}[u;\Omega]$ but happens at the cost that some higher regularity of the prescribed traces is required. On the other hand, the relaxed functionals as introduced in Section~\ref{sec:relaxintro} do not feature Lavrentiev gaps, see~\eqref{eq:lavrentiev}, which is only known to hold for the functionals from~\eqref{eq:FonsecaMalySchmidtRelax} on special sets (cf.~\cite[Sec.~5]{SchmidtPR}). 

In the setting of signed quasiconvex integrands and hereafter the functionals $\overline{\mathscr{F}}$ or $\overline{\mathscr{F}}^{*}$, however, measure representations as for~\eqref{eq:FonsecaMalySchmidtRelax} are not available at present; for now, if $p=1$, one even lacks the identification of $\frac{\dif\mu_{u}}{\dif\mathscr{L}^{n}}$ in the  \emph{unsigned} context. Let us note that, in the language of Definition~\ref{def:locmin}, the conclusion of (global) minimizers being local minimizers for compactly supported variations persists, but needs to be established by independent means. Interestingly, the difficulties underlying the proofs of the basic features of the relaxed functionals shift almost completely (see Remark~\ref{rem:diffshift}). Most importantly, although $\mathscr{F}_{\locc}$ and $\overline{\mathscr{F}}$ are not equal in general, the set of local minimizers for compactly supported variations is in fact the same provided $F$ is quasiconvex and bounded below (see Proposition~\ref{prop:Equivalenceminimizers} and Remark~\ref{rem:weaklocal}). Thus, the partial regularity results stated for $\overline{\mathscr{F}}^{*}$ and $\overline{\mathscr{F}}$ in the previous paragraph in fact carry over to the so-called \emph{weak local minimizers} as considered in~\cite{SchmidtPR}; also compare with Remark~\ref{rem:weaklocal}. 

\subsection{Main points and novelties}\label{sec:keynov}
Based on the discussion in the preceding paragraphs, we now briefly address the main novelties and chief difficulties, both from a conceptual and technical perspective, of the present paper. Since several of these matters mutually depend on each other in the signed case, we start by displaying the part which is also new in the unsigned context.

\subsubsection{$(p,q)$-exponent range.}\label{sec:CaccIntro} The first advancement of Theorems~\ref{thm:main1} and~\ref{thm:main2}, even for non-negative integrands $F$, is the extension of the partial regularity of relaxed minimizers from \textsc{Schmidt}'s range $1<p\leq q<p+\frac{\min\{2,p\}}{2n}$ to $1\leq p\leq q <\min\{\frac{np}{n-1},p+1\}$. As discussed in~\cite{SchmidtPR}, in the unsigned, superlinear growth case $p>1$, the crucial point is to establish the Caccioppoli inequality of the second kind for the latter range of exponents. We here approach the requisite inequalities by exclusively using minimality and additivity properties of the relaxed functionals, systematically avoiding any sort of intermediate bounds that force exponent restrictions beyond $ q <\min\{\frac{np}{n-1},p+1\}$. This necessitates a variation of \textsc{Evans}' original proof of the Caccioppoli inequality of the second kind~\cite{Ev1} and its modification in the $(p,q)$-context by \textsc{Schmidt}~\cite{SchmidtPR0,SchmidtPR}. To run \textsc{Evans}' localisation scheme based on \textsc{Widman}'s hole-filling trick~\cite{Widman}, it is clear that accessing the ($p$-)strong quasiconvexity requires the construction of competitor maps with controllable $ \sobo^{1,q}$-energy on certain annuli. Inspired by~\cite{SchmidtPR}, this is achieved by suitably modifying certain recovery sequences on special annuli by trace-preserving operators to force fixed traces along certain spheres. Deferring the discussion of the underlying \emph{good generation theorem} in the signed case, we stress that, while the  slightly different overall set-up of the Caccioppoli inequality in Section~\ref{sec:Cacc} appears as a conceptual point, it goes hand in hand with its precise technical implementation. Here, a key device is the direct derivation of the requisite layer bounds by use of the \textsc{Whitney}-type trace-preserving operator~\cite{Whitney} as introduced in the $(p,q)$-context by~\textsc{Fonseca \& Mal\'{y}}~\cite{FMParma}. This operator, to be discussed in detail in the $\bv$-setting in Section~\ref{sec:Fubini}, allows for a flexible local handling of the underlying estimates by examining the $V$-function type energies of $\bv$- or $\sobo^{1,p}$-maps and their gradients on the respective Whitney balls. By construction, this operator admits the globalisation of local estimates of $V$-function-type energies by routine embeddings for $\ell^{p}$-sequence spaces in conjunction with suitable finiteness conditions on spherical maximal functions. Based on this approach, it then will be clear from the proof why we have to require $q<\frac{np}{n-1}$. Yet, the key point is that we solely work subject to \emph{quasiconvexity} but not $\sobo^{1,p}$-quasiconvexity conditions, in which case improved results are available (cf.~\textsc{Carozza, Passarelli di Napoli} and the second author \cite{CPK}); however, note that quasiconvexity is far apart from $\sobo^{1,p}$-quasiconvexity in general. 

\subsubsection{Signed integrands.} Other than \textsc{Schmidt} \cite{SchmidtPR} or \textsc{Fonseca} et al. \cite{BFM,FM,FonMar} and as explained in Section~\ref{sec:aimandscope} (see~\eqref{eq:pstrongQCintro}\emph{ff.}), we assume the integrands $F$ to be \emph{signed}. In order to elaborate on why the quasiconvex signed case departs from integrands being bounded below, we note that an overarching task in the study of $(p,q)$-growth problems is to control certain $\sobo^{1,r}$-energies, for some $r>p$, by the a priori available $\sobo^{1,p}$-bounds on minimizing sequences. Different from the convex case (see the discussion in Section~\ref{sec:mainresults}), in the quasiconvex case higher gradient integrability results seem out of reach at present (also see Section~\ref{sec:Mazdiscuss}(b) below) and so we may only rely on the $\lebe^{p}$-boundedness of gradients and the quasiconvexity itself.

\emph{(a) Exclusion of concentration effects.} As a key device to control the $\lebe^{q}$-concentration of gradients on certain small annular layers, \textsc{Schmidt}~\cite[\S 7.7]{SchmidtPR0,SchmidtPR} introduced a boundary regularity criterion to be able with specific recovery sequences with fixed interior traces. When constructing such sequences $(u_{j})$   from \emph{given} recovery sequences $(v_{j})$, we must ensure that no additional mass is created during this modification process. It is here where the signed case comes with aggravated concentration issues: Indeed, since $F$ is signed, it might in principle happen that on small nested annular regions $\mathcal{A}_{j}\Subset\Omega$ with $\mathscr{L}^{n}(\mathcal{A}_{j})\searrow 0$ the gradients $\nabla v_{j}$ blow up in a way such that $\mathscr{F}[v_{j};\mathcal{A}_{j}]\searrow -\infty$ whilst $\sup_{j\in\mathbb{N}}\mathscr{F}[v_{j};\Omega]<\infty$. Passing to recovery sequences $(u_{j})$ with a controllable $\sobo^{1,q}$-energy on $\mathcal{A}_{j}$ that coincide with $v_{j}$ on $\Omega\setminus\mathcal{A}_{j}$, one generically expects $\mathscr{F}[u_{j};\Omega]$ to create additional mass in the limit. It is easy to see that such a behaviour is impossible when $F$ is convex and $p\geq 1$ or $F$ is quasiconvex and minorisable by affine-linear maps\footnote{Note that if $F$ is convex with $p>1$, then $F$ is automatically bounded from below. More generally, by separation, every convex integrand with $p\geq 1$ can be minorised by an affine-linear map.}. As we will display in Section~\ref{sec:props}, there is in fact a wealth of signed, coercive, quasiconvex integrands that cannot be minorised by affine-linear maps; thus it is indeed the \emph{nonconvexity} that lets such obstructions emerge in the signed context at all. The resolution of this matter and therewith the generalisation of~\cite[Lem.~7.7]{SchmidtPR}, here referred to as \emph{good generation theorem}, is given in Section~\ref{sec:goodgeneration}. With this being exemplary, related concentration effects of this sort arise and must be ruled out throughout the course of the paper. 

\emph{(b) Fitting the concentration control to the general scheme of proof.} Potential concentration effects as displayed in the preceding item can be avoided if one imposes certain restrictions on the $(p,q)$-exponent range. However, such restrictions do not allow to cover the range $q<\min\{\frac{np}{n-1},p+1\}$ as appearing in Theorems~\ref{thm:main1} and~\ref{thm:main2}. Since the proof of the Caccioppoli inequality equally hinges on modifications of recovery sequences on small annular layers, we have to limit the $\lebe^{q}$-gradient concentration in the signed case even more carefully here\footnote{Since we directly work with linearised integrands, signed integrands naturally arise in the proof of the Caccioppoli inequality even for unsigned $1$-strongly quasiconvex integrands $F$. This, however, is only of technical nature since the linearisation terms  then can be handled by the available weak*-convergence; this is not so if $F$ is a priori signed.} while simultaneously reaching the optimal exponent range. In parallel, similar effects have to be excluded in the derivation of the Euler-Lagrange system, see Section~\ref{sec:Mazdiscuss} below.

\emph{(c) Non-availability of measure representations.} In the signed case, a substantial part of the general machinery of measure representations used e.g. in \cite{SchmidtPR} is unknown. Such measure representations enter \textsc{Schmidt}'s partial regularity proof in several ways. Generalising results of \textsc{Ball \& Murat}~\cite{BM}, an important  observation of~\textsc{Schmidt} is the $\sobo^{1,p}$-quasiconvexity of the relaxed functionals $\mathscr{F}_{\locc}$ for $p>1$ (see~\cite[Lem.~7.6]{SchmidtPR}). The proof of this result hinges on measure representations and hence so does the proof Caccioppoli inequality, in turn relying on the $\sobo^{1,p}$-quasiconvexity of $\mathscr{F}_{\locc}$. This issue is  circumvented here by directly working on the level of recovery sequences for the linearised relaxed functionals. However, even more conceptually and as discussed in Section~\ref{sec:CompaIntro}, the use of measure representations leads to validity of the Euler-Lagrange system \cite[Lem.~7.3]{SchmidtPR}. While some of the available measure representations and energy density identifications are conceivable to extend to the signed case or $p=1$, another key aspect that they are \emph{not really required}. This is the content of the following paragraph: 

\subsubsection{Mazur's lemma, the Euler-Lagrange system and independence of measure representations.}\label{sec:Mazdiscuss} Theorems~\ref{thm:main1} and \ref{thm:main2} crucially hinge on the Euler-Lagrange system satisfied by local ($\bv$-)minima. As we establish in Section~\ref{sec:MazurEuler}, there \emph{always exist} recovery sequences such that the (approximate) gradients converge in $\mathscr{L}^{n}$-measure to the approximate gradients of ($\bv$-)minimizers $u$ for compactly supported variations. In analogy with other weak-to-strong convergence boosts, we shall refer to this as \emph{Mazur's lemma}. This result allows a direct proof of the validity of the Euler-Lagrange system, see Corollaries~\ref{cor:EulerLagrange} and~\ref{cor:ELMAIN},  \emph{without} relying on measure representations. We highlight two points associated with this result: 

\emph{(a) Purely linear growth case.} If $p=q=1$ (whereby, in the terminology of Section~\ref{sec:aimandscope}, the compactness space is $\bv(\Omega;\R^{N})$ and so the functional~\eqref{eq:varprin} already needs to be relaxed from $\sobo^{1,1}(\Omega;\R^{N})$), Theorem~\ref{thm:main1} appears as a generalisation of the partial regularity result due to \textsc{Anzellotti \& Giaquinta} \cite{AnGi} in the (strongly) convex and of the precursor~\cite{GK1} by the authors in the $1$-strongly quasiconvex case. As a novelty even for this case, the preceding discussion shows that integral (and hence strong forms of measure)  representations \`{a} la \textsc{Reshetnyak}~\cite{Reshetnyak} and \textsc{Goffman \& Serrin}~\cite{GoffmanSerrin} in the convex case or \textsc{Ambrosio \& Dal Maso}~\cite{AD}, \textsc{Fonseca \& M\"{u}ller}~\cite{FonMul} and~\textsc{Rindler} and the second named author~\cite{KR} in the quasiconvex case are \emph{not required} for the aforementioned partial regularity results.

\emph{(b) Oscillation control versus higher integrability.} The convergence in $\mathscr{L}^{n}$-measure rules out a strongly oscillatory behaviour of certain recovery sequences, whereas it does not exclude concentrations. In the standard $p$-growth case with $p>1$, oscillation control is usually provided by the Caccioppoli inequality of the second kind. \textsc{Gehring}'s lemma then leads to uniform local higher gradient integrability as a somewhat quantified version of gradient concentration control. In the situation of relaxed quasiconvex $(p,q)$-growth functionals, the very structure of the Caccioppoli inequality (cf.~\cite[Lem.~7.13]{SchmidtPR} and Theorem~\ref{thm:Cacc} below) only allows for a very weak oscillation control. While still sufficient for the partial regularity proof, its structure does not allow to deduce the reverse H\"{o}lder inequalities with increasing supports as required for Gehring's lemma. If $p=1$, higher integrability results based on Gehring's lemma are in general not expected to hold; this is due to the lack of suitable sublinear Sobolev inequalities (see~\textsc{Buckley \& Koskela}~\cite{BuckleyKoskela} and the discussion in~\cite{Gm20,GK1}). Since there are linear growth scenarios indeed which the Mazur-type lemma from Section~\ref{sec:MazurEuler} applies to, but minimising sequences generically concentrate despite the availability of Caccioppoli type inequalities, the sole exclusion of certain oscillations but not concentrations seems close to optimal\footnote{At present, even in the purely linear growth, signed, quasiconvex context, the only concentration control that has been extracted from the Caccioppoli inequality of the second kind is very weak and works on the level of $\bv$-gradient Young measures; see~\cite[Rem.~4.5]{GK1}.} in the present general setting.

\subsubsection{Maximal conditions as a unifying principle and fully direct comparison method.} From a technical perspective, the paper furnishes maximal conditions (i.e., finiteness conditions in terms of radial maximal operators, cf.~Section~\ref{sec:HLW}) as an overarching principle that is visible in the proofs of all key ingredients for Theorems~\ref{thm:main1} and \ref{thm:main2}. Similarly as in \textsc{Schmidt} \cite{SchmidtPR}, such conditions also enter in the present version of the good generation theorem (cf.~Section~\ref{sec:goodgeneration}) and thus implicitely in the derivation of Mazur's lemma and the Euler-Lagrange system. However, they equally allow for a direct proof of the requisite form of Fubini-type theorems (cf.~Section~\ref{sec:Fubini}) in Sobolev-Slobodecki\u{\i} spaces on spheres. This gives a conceptually easy approach to such estimates, previously employed by the authors e.g. by using sharp embeddings for $\bv$- into Sobolev-Slobodecki\u{\i} spaces \cite{Gm20,GK1}. Specifically, such Fubini-type theorems enable us to solve certain linear comparison systems on \emph{good} balls. This proves particularly relevant for $1=p\leq q<\frac{n}{n-1}$, where the generic interior trace space of $\bv$ along spheres is $\lebe^{1}$ and solving linear elliptic system with $\lebe^{1}$-boundary is not possible in general. As will be visible from the proof, the excess decay estimate is then fully reduced to estimates on balls defined in terms of such maximal conditions. Hence, not only for the lower semicontinuity or even the definition of local minimizers for compactly supported variations but also at all stages of the partial regularity proof, maximal functions are identified as a unifying tool especially in the borderline case $p=1$.

\subsubsection{Orlicz range, differential conditions.} Theorem~\ref{thm:main1} equally includes  integrands of degenerate Orlicz-growth behaviour that have been omitted so far in the literature. This includes, for instance, integrands of $\Delta_{2}$-, non-$\nabla_{2}$-growth such as $\mathrm{L}\log^{\alpha}\lebe$-growth behaviour with $0<\alpha\leq 1$; see Section~\ref{sec:Orlicz} for the underlying terminology. On the one hand, such integrands fall outside the scope of the $\Delta_{2}\cap\nabla_{2}$-assumptions as considered in \textsc{Diening} et al. \cite{DLSV}, yet cannot be handled as a special case of \textsc{Schmidt} \cite{SchmidtPR}; recall that even unsigned integrands with this growth behaviour cannot be bounded from below by some power $|z|^{p}$ for $p>1$. On the other hand, they fall outside the realm of quasiconvex, purely linear growth integrands studied in the precursor \cite{GK1} of the present paper and it is Theorem~\ref{thm:main1} that closes this gap; see Section~\ref{sec:Orlicz} for more detail. 

Finally, motivated by problems from continuum or fluid mechanics (see e.g.~\textsc{Fuchs \& Seregin}~\cite{FuchsSeregin}) the reduction strategy introduced in~\cite{Gm20,CG} allows to  inexpensively formulate a variant of Theorem~\ref{thm:main2} in the context of general first order differential operators; see Section~\ref{sec:diffcond}.

\subsubsection{Possible extensions.} Both Theorems~\ref{thm:main1} and~\ref{thm:main2} display regularity results that deal with the relaxation of the core functionals~\eqref{eq:varprin}. Extensions to functionals with forcing terms \`{a} la \textsc{De Filippis} et al.~\cite{DF1,DF2} or degenerate scenarios \`{a} la \textsc{Schmidt}~\cite{SchmidtPR2a,SchmidtPR3} based on \textsc{Duzaar \& Mingione}'s $p$-harmonic approximation technique~\cite{DuzaarMingione0,DuzaarMingione1,DuzaarMingione2} 
are conceivable and of equal interest. Moreover, since we do not rely on measure representations, we believe that the techniques introduced in the present paper should also allow to treat the closely related functionals  
\begin{align}
\overline{\mathscr{F}}_{\mathrm{FM}}[u;\Omega]:=\inf\left\{\liminf_{j\to\infty}\int_{\Omega}F(\nabla u_{j})\dif x\colon\; \begin{array}{c} (u_{j})\subset\sobo^{1,q}(\Omega;\R^{N})\\ u_{j}\rightharpoonup u\;\text{in}\;\sobo^{1,p}(\Omega;\R^{N})\end{array}\right\}
\end{align}
considered by \textsc{Fonseca \& Mal\'{y}}~\cite{FM} (also see the discussion by \textsc{Schmidt}~\cite[Sec.~1]{SchmidtPR}). By the overall method, the proofs moreover  should be sufficiently robust to also apply to higher order scenarios or functionals depending on differential operators, cf.~Remark~\ref{rem:diffcondsp1}. Still, by the scope of the present paper, this and related questions shall be deferred to be pursued elsewhere. 
\subsection{Structure of the paper}
Apart from these introductory sections, the paper is organised as follows: Section~\ref{sec:prelims} fixes notation and gathers auxiliary material. Section~\ref{sec:Fubini} discusses a particular trace-preserving operator and Fubini-type properties which might be of independent interest. Section~\ref{sec:goodgeneration} then establishes an existence result on \emph{good} recovery sequences. Section~\ref{sec:props} is devoted to examples of integrands, the existence of minimizers and instrumental properties of the relaxed functionals. It is here where we also connect our set-up with that of \textsc{Schmidt}~\cite{SchmidtPR}. After gathering auxiliary facts on linearisations in Section~\ref{sec:linearisation}, both the aforementioned lemma of Mazur-type and the derivation of the Euler-Lagrange system are  addressed in Section~\ref{sec:MazurEuler}. As one of the main ingredients in the proofs of Theorems~\ref{thm:main1} and~\ref{thm:main2}, Section~\ref{sec:Cacc} provides the requisite form of the Caccioppoli inequality of the second kind. The main part of the paper is then concluded by the proof of Theorems~\ref{thm:main1} and~\ref{thm:main2} in Section~\ref{sec:proofmain}, where also implications for the partial regularity of functionals with Orlicz growth or depending on differential operators are discussed. Finally, the appendix in Section~\ref{sec:appendix}  provides the proofs of some minor auxiliary results utilised in the main part of the paper; specifically, this comprises a lower semicontinuity theorem for variational integrals with signed Orlicz integrands.\\

\noindent
{\small \textbf{Acknowledgment.} The authors are thankful to \textsc{Thomas Schmidt} for useful conversations on the theme of the paper. F.G. gratefully acknowledges financial support by the Hector Foundation (Project No. FP 626/21) and the stimulating atmosphere of the 2022 Oberwolfach workshop on the Calculus of Variations, during which this paper was finished.}

\section{Preliminaries}\label{sec:prelims}
\subsection{General notation}
Unless stated otherwise, $\omega$ and $\Omega$ are non-empty, open and bounded subsets of $\R^{n}$. For $x_{0}\in\R^{n}$ and $r>0$, the (euclidean) open ball of radius $r>0$ and centered at $x_{0}$ is denoted $\ball_{r}(x_{0}):=\{x\in\R^{n}\colon\;|x-x_{0}|<r\}$; specifically, we write $\mathbb{S}^{n-1}:=\partial\!\ball_{1}(0)$ for the $(n-1)$-dimensional unit sphere. To distinguish from balls in matrix space, we further denote, for $z\in\R^{N\times n}$, $\mathbb{B}_{r}(z):=\{y\in\R^{N\times n}\colon\;|y-z|<r\}$ where $|\cdot|$ then is the Hilbert-Schmidt norm induced by the Hilbert-Schmidt inner product $\langle\cdot,\cdot\rangle$. We also use $\langle\cdot,\cdot\rangle$ for the usual euclidean inner product on spaces $\R^{n}$ or $\R^{N}$, but no ambiguities will arise from this.

The $n$-dimensional Lebesgue and $(n-1)$-dimensional Hausdorff measures are denoted $\mathscr{L}^{n}$ or $\mathscr{H}^{n-1}$, respectively, and we set $\omega_{n}:=\mathscr{L}^{n}(\ball_{1}(0))$. To abbreviate notation, we shall often write $\dif^{n-1}=\dif\mathscr{H}^{n-1}$. Given a finite-dimensional inner product space $X$, we denote $\mathrm{RM}(\Omega;X)$ and $\mathrm{RM}_{\mathrm{fin}}(\Omega;X)$ the (finite) $X$-valued Radon measures on $\Omega$; for $\mu\in\mathrm{RM}(\Omega;X)$, $|\mu|$ denotes its total variation measure and hereafter $|\mu|(\Omega)$ its total variation. Similarly, if we write $\mu\in\mathrm{RM}(\Omega)$ or $\mu\in\mathrm{RM}_{\mathrm{fin}}(\Omega)$, we understand that $\mu$ is a non-negative (finite) Radon measure on $\Omega$. In each of the cases $\mu\in\mathrm{RM}_{(\mathrm{fin})}(\Omega)$ or $\mu\in\mathrm{RM}_{(\mathrm{fin})}(\Omega;X)$, if $A$ belongs to the Borel $\sigma$-algebra $\mathscr{B}(\Omega)$, we use $\mu\mres A := \mu(\cdot\cap A)$ to denote the restriction of $\mu$ to $A$. The Lebesgue-Radon-Nikod\'{y}m decomposition of $\mu\in\mathrm{RM}(\Omega;X)$ into its absolutely continuous and singular parts $\mu^{a}$, $\mu^{s}$ for $\mathscr{L}^{n}$ then is given by 
\begin{align}\label{eq:LRNdecomposition}
\mu=\mu^{a}+\mu^{s} = \frac{\dif\mu^{a}}{\dif\mathscr{L}^{n}}\mathscr{L}^{n}\mres\Omega + \frac{\dif\mu^{s}}{\dif|\mu^{s}|}|\mu^{s}|, 
\end{align}
where $\frac{\dif\mu^{a}}{\dif\mathscr{L}^{n}}$ and $\frac{\dif\mu^{s}}{\dif|\mu^{s}|}$ are the corresponding densities for $\mathscr{L}^{n}$ or $|\mu^{s}|$, respectively. Whenever $A\in\mathscr{B}(\Omega)$ satisfies $\mathscr{L}^{n}(A)>0$, we put 
\begin{align}\label{eq:meanvaluemeasure}
\dashint_{A}\mu := \frac{\mu(A)}{\mathscr{L}^{n}(A)}.
\end{align}
To avoid overburderning notation, if $A=\partial\!\ball_{r}(x_{0})$ and $f\colon \partial\!\ball_{r}(x_{0})\to\R^{N}$ is $\mathscr{H}^{n-1}$-measurable, we also use the convention
\begin{align*}
(f)_{\partial\!\ball_{r}(x_{0})}:=\dashint_{\partial\!\ball_{r}(x_{0})}f\dif\mathscr{H}^{n-1} := \frac{1}{\mathscr{H}^{n-1}(\partial\!\ball_{r}(x_{0}))}\int_{\partial\!\ball_{r}(x_{0})}f\dif\mathscr{H}^{n-1},
\end{align*}
and it will be clear from the context in which sense the dashed integrals are understood. As usual, $c,C>0$ denote generic constants that might change from line to line and are only specified if their precise values are required. Finally, we write $a\lesssim b$ for two quantities provided $a\leq cb$ for some constant $c>0$ essentially independent of $a$ and $b$. To emphasize important dependences, we also use $a\lesssim_{d}b$ provided the constant $c$ crucially depends on a quantity $d$. In the same vein, we write $a\sim b$ provided $a\lesssim b $ and $b\lesssim a$, and $a\sim_{d}b$ provided $a\lesssim_{d}b$ and $b\lesssim_{d}a$. 
\subsection{Function spaces}\label{sec:functionspaces}
In this section we gather the definitions and several  properties of function spaces that play an important role in the main part of the paper; we refer the reader to \cite{AFP,EvGa,Mazya} for more background information. Throughout, let $\Omega\subset\R^{n}$ be open and bounded. 

\subsubsection{The space $\bv$.} We recall that a measurable map $u\colon\Omega\to\R^{N}$ is said to be of \emph{bounded variation} and then denoted $u\in\bv(\Omega;\R^{N})$ provided $u\in\lebe^{1}(\Omega;\R^{N})$ and its \emph{total variation} 
\begin{align*}
|Du|(\Omega):=\sup\left\{\int_{\Omega}\langle u,\di(\varphi)\rangle\dif x\colon\;\varphi\in\hold_{c}^{\infty}(\Omega;\overline{\mathbb{B}}_{1}(0))\right\}
\end{align*}
is finite, where $\di$ is the row-wise divergence. The space $\bv_{\locc}(\Omega;\R^{N})$ then is defined in the obvious manner.

The space $\bv(\Omega;\R^{N})$ is a normed space when endowed with $\|u\|_{\bv(\Omega)}:=\|u\|_{\lebe^{1}(\Omega)}+|Du|(\Omega)$, but this norm is too restrictive for most applications. Instead, two notions of convergence prove more relevant: Given $u,u_{1},u_{2},...\in\bv(\Omega;\R^{N})$, we say that $(u_{j})$ converges in the \emph{weak*-sense} to $u$ provided $u_{j}\to u$ in $\lebe^{1}(\Omega;\R^{N})$ and $Du_{j}\stackrel{*}{\rightharpoonup}Du$ in $\mathrm{RM}_{\mathrm{fin}}(\Omega;\R^{N\times n})\cong \hold_{0}(\Omega;\R^{N\times n})'$. This convergence then yields a compactness theorem on $\bv$: If $\Omega\subset\R^{n}$ has Lipschitz boundary $\partial\Omega$, any norm-bounded sequence has a weak*-convergent subsequence. As a strengthening of weak*-convergence, we moreover say that $(u_{j})$ converges to $u$ in the \emph{strict sense} provided $u_{j}\to u$ in $\lebe^{1}(\Omega;\R^{N})$ and $|Du_{j}|(\Omega)\to|Du|(\Omega)$ as $j\to\infty$. 

Now let $u\in\bv_{\locc}(\Omega;\R^{N})$. For $\mathscr{L}^{n}$-a.e. $x\in\Omega$, $x$ is a Lebesgue point of $u$ with corresponding Lebesgue value $u(x)$ and there exists $\nabla u(x)\in\R^{N\times n}$ such that 
\begin{align}\label{eq:approxdiff}
\limsup_{r\searrow 0} \frac{1}{r}\dashint_{\ball_{r}(x)}|u(y)-u(x)-\nabla u(x)\cdot(y-x)|\dif y = 0.   
\end{align}
The matrix $\nabla u(x)$ is well-defined and referred to as the \emph{approximate gradient} of $u$ at $x$. Given $u\in\bv_{\locc}(\Omega;\R^{N})$, the Lebesgue-Radon-Nikod\'{y}m decomposition~\eqref{eq:LRNdecomposition} of $Du$ then reads as 
\begin{align}\label{eq:LRNdecomposition1}
Du = D^{a}u + D^{s}u = \nabla u\mathscr{L}^{n}\mres\Omega + \frac{\dif D^{s}u}{\dif|D^{s}u|}|D^{s}u|, 
\end{align}
where $\nabla u\in\lebe_{\locc}^{1}(\Omega;\R^{N\times n})$ is the approximate gradient of $u$ as in~\eqref{eq:approxdiff}. Note that, if $u\in\sobo_{\locc}^{1,1}(\Omega;\R^{N})$, then $\nabla u$ is just the weak gradient. We next recall that $x_{0}\in\Omega$ is a \emph{jump point} of $u$ and then write $x_{0}\in J_{u}$ provided there exist $a,b\in\R^{N}$ with $a\neq b$ and $\nu\in\mathbb{S}^{n-1}$ such that 
\begin{align*}
\lim_{r\searrow 0}\dashint_{\ball_{r}(x_{0})\cap\{x\colon\,\langle x-x_{0},\nu\rangle > 0\}}|u-a|\dif y=\lim_{r\searrow 0}\dashint_{\ball_{r}(x_{0})\cap\{x\colon\,\langle x-x_{0},\nu\rangle < 0\}}|u-b|\dif y=0.
\end{align*}
Finally, we recall that if $\Omega$ has Lipschitz boundary, there exists a surjective, bounded and linear (boundary) trace operator $\trace_{\partial\Omega}\colon\bv(\Omega;\R^{N})\to\lebe^{1}(\partial\Omega;\R^{N})$; here, the latter space is understood with respect to $\mathscr{H}^{n-1}\mres\partial\Omega$. Specifically, this operator is continuous for strict convergence on $\bv(\Omega;\R^{N})$ and can be realised for $\mathscr{H}^{n-1}$-a.e. $x_{0}\in\partial\Omega$ by integral means via  
\begin{align}\label{eq:tracedescription}
\lim_{r\searrow 0}\dashint_{\ball_{r}(x_{0})\cap\Omega}|u(y)-\trace_{\partial\Omega}(u)(x_{0})|\dif y = 0. 
\end{align}
We denote $\bv_{0}(\Omega;\R^{N})$ the nullspace of $\trace_{\partial\Omega}$. This space can be characterised as the strict closure of $\hold_{c}^{\infty}(\Omega;\R^{N})$ in $\bv(\Omega;\R^{N})$, which follows from the strict continuity of $\trace_{\partial\Omega}$ and a slightly more general area-strict approximation result, Lemma~\ref{lem:SolidLemma} below.

In the main part, we moreover require interior trace operators along spheres. Given $x_{0}\in\Omega$ and $r>0$ with $\ball_{r}(x_{0})\Subset\Omega$, we denote for $u\in\bv(\Omega;\R^{N})$
\begin{align*}
\trace_{\partial\!\ball_{r}(x_{0})}^{+}(u):=\trace_{\partial\!\ball_{r}(x_{0})}(\mathbbm{1}_{\ball_{r}(x_{0})}u)\;\text{and}\;\trace_{\partial\!\ball_{r}(x_{0})}^{-}(u):=\trace_{\partial\!\ball_{r}(x_{0})\cup\partial\Omega}(\mathbbm{1}_{\Omega\setminus\overline{\ball}_{r}(x_{0})}u)|_{\partial\!\ball_{r}(x_{0})},  
\end{align*}
where $\trace_{\partial\!\ball_{r}(x_{0})},\trace_{\partial\!\ball_{r}(x_{0})\cup\partial\Omega}$ display the boundary trace operators on $\bv(\ball_{r}(x_{0});\R^{N})$ and $\bv(\Omega\setminus\overline{\ball}_{r}(x_{0});\R^{N})$, the \emph{inner} and \emph{outer trace operators} along $\partial\!\ball_{r}(x_{0})$. We then have 
\begin{align}\label{eq:interiortraces}
|Du|(\partial\!\ball_{r}(x_{0}))=\int_{\partial\!\ball_{r}(x_{0})}|\trace_{\partial\!\ball_{r}(x_{0})}^{+}(u)-\trace_{\partial\!\ball_{r}(x_{0})}^{-}(u)|\dif\mathscr{H}^{n-1}
\end{align}
and, recalling the jump set $J_{u}$ from above, 
\begin{align}\label{eq:interiortraces1}
\mathscr{H}^{n-1}((\partial\!\ball_{r}(x_{0})\cap J_{u})\setminus\{x\in\partial\!\ball_{r}(x_{0})\colon \trace_{\partial\!\ball_{r}(x_{0})}^{+}(u)(x)\neq \trace_{\partial\!\ball_{r}(x_{0})}^{-}(u)(x)\})=0.
\end{align}
The interior trace of $u$ along $\partial\!\ball_{r}(x_{0})$ then is constructed as the arithmetic mean of the inner and outer traces of $u$ along $\partial\!\ball_{r}(x_{0})$. However, to alleviate notation, we write $\trace_{\partial\!\ball_{r}(x_{0})}(u)$ for the interior trace along $\partial\!\ball_{r}(x_{0})$ only if $\trace_{\partial\!\ball_{r}(x_{0})}^{+}(u)=\trace_{\partial\!\ball_{r}(x_{0})}^{-}(u)$ $\mathscr{H}^{n-1}$-a.e. on $\partial\!\ball_{r}(x_{0})$. 
\subsubsection{Functions of measures}
For our future purposes, we now revisit the application of functions to measures as originally due to \textsc{Goffman \& Serrin}~\cite{GoffmanSerrin} and \textsc{Reshetnyak}~\cite{Reshetnyak}; in our setting, the underlying measures will typically be gradients. Given $1\leq p<\infty$ and $m\in\mathbb{N}$, let $h\colon\R^{m}\to\R$ be a convex function satisfying the bounds
\begin{align}\label{eq:growthconvexH}
c_{1}|z|^{p}-c_{2}\leq h(z)\leq c_{3}(1+|z|^{p})\qquad\text{for all $z\in\R^{m}$}
\end{align}
and constants $c_{1},c_{2},c_{3}>0$. If $p=1$ and $\mu\in\mathrm{RM}(\Omega;\R^{m})$, we let $\mu=\frac{\dif\mu}{\dif\mathscr{L}^{n}}\mathscr{L}^{n}\mres\Omega +\frac{\dif\mu}{\dif|\mu^{s}|}|\mu^{s}|$ be its Lebesgue-Radon-Nikod\'{y}m decomposition~\eqref{eq:LRNdecomposition}, and then define the measure $h(\mu)$ by 
\begin{align}\label{eq:functionsofmeasures}
h(\mu)(A):=\int_{A}h(\mu):=\int_{A}h\Big(\frac{\dif\mu^{a}}{\dif\mathscr{L}^{n}}\Big)\dif x + \int_{A}h^{\infty}\Big(\frac{\dif \mu^{s}}{\dif|\mu^{s}|} \Big)\dif|\mu^{s}|,\;\; A\in\mathscr{B}(\Omega). 
\end{align}
Here, $h^{\infty}(z):=\lim_{t\searrow 0} th(\frac{z}{t})$ is the \emph{recession function} of $h$. Because of $p=1$ and convexity of $h$, $h^{\infty}(z)$ exists and is finite, whereby $h(\mu)$ is a well-defined measure. If $z_{0}\in\R^{m}$, we put 
\begin{align*}
h(\mu-z_{0}):=h(\mu-z_{0}\mathscr{L}^{n}\mres\Omega)
\end{align*}
and, using the convention~\eqref{eq:meanvaluemeasure}, then record from \cite[Prop.~2.4]{AnGi} Jensen's inequality in the form 
\begin{align*}
h\Big(\dashint_{\omega}\mu-z_{0} \Big) \leq \dashint_{\omega}h(\mu-z_{0})\qquad\text{for $\omega\in\mathscr{B}(\Omega)$ with $\mathscr{L}^{n}(\omega)>0$.}
\end{align*}
Definition~\eqref{eq:functionsofmeasures} immediately applies to $\mu=Du$ for $u\in\bv_{\locc}(\Omega;\R^{N})$ by virtue of~\eqref{eq:LRNdecomposition1}. As we shall concentrate on the case $p=1$ in the main part of the paper and only point out the modifications for $p>1$, it is useful to note that~\eqref{eq:functionsofmeasures} is consistent with this case too. More precisely, if $h$ satisfies \eqref{eq:growthconvexH} with $1<p<\infty$ and $u\in\sobo_{\locc}^{1,p}(\Omega;\R^{N})$, then $D^{s}u=0$ and hence
\begin{align*}
\int_{A}h(Du)=\int_{A}h(\nabla u)\dif x\qquad\text{for all}\;A\in\mathscr{B}(\Omega)
\end{align*}
despite of $h^{\infty}(z)=\infty$ for all $z\neq 0$. We proceed by recording a variant of the \textsc{Reshetnyak} lower semicontinuity theorem~\cite{Reshetnyak} in the version of \cite[Props.~2.20 and~2.21]{Bild}: 
\begin{lemma}\label{lem:reshetnyak}
Let $\Omega\subset\R^{n}$ be open and bounded, and let $u,u_{1},u_{2},...\in\bv(\Omega;\R^{N})$ be such that $u_{j}\to u$ in $\lebe_{\locc}^{1}(\Omega;\R^{N})$. Then for any convex function $h\colon\R^{N\times n}\to\R$ that satisfies \eqref{eq:growthconvexH} with $p=1$ we have 
\begin{align}\label{eq:resh}
h(Du)(\Omega)\leq \liminf_{j\to\infty}h(Du_{j})(\Omega). 
\end{align}
\end{lemma}
Finally, we display an approximation result that particularly ensures the existence of sequences as required in the definition of the relaxed functionals~\eqref{eq:bdryrelaxed1}:\begin{lemma}[{\cite[Lem.~B.2]{Bild}}]\label{lem:SolidLemma}
Let $\Omega\Subset\Omega'$ be two open and bounded sets with Lipschitz boundaries and let $u_{0}\in\sobo^{1,1}(\Omega';\R^{N})$. Denoting for $v\in\bv(\Omega;\R^{N})$ by $\mathbf{v}$ its extension to $\Omega'$ by $u_{0}$, for any $u\in\bv(\Omega;\R^{N})$ there exists $(u_{j})\subset u_{0}+\hold_{c}^{\infty}(\Omega;\R^{N})$ such that $(\mathbf{u}_{j})$ converges to $\mathbf{u}$ \emph{area-strictly} in $\bv(\Omega';\R^{N})$, i.e.,  
\begin{align*}
\mathbf{u}_{j}\to \mathbf{u}\;\text{in}\;\lebe^{1}(\Omega';\R^{N})\;\;\text{and}\;\;\sqrt{1+|D\mathbf{u}_{j}|^{2}}(\Omega')\to\sqrt{1+|D\mathbf{u}|^{2}}(\Omega'). 
\end{align*}
Especially, we have $\mathbf{u}_{j}\to\mathbf{u}$ strictly and hence $\mathbf{u}_{j}\stackrel{*}{\rightharpoonup}\mathbf{u}$ in $\bv(\Omega';\R^{N})$. 
\end{lemma}

\subsubsection{The spaces $\sobo^{s,p}$.} Let $0<s<1$ and $1\leq p <\infty$. For a Lipschitz hypersurface $\Sigma\subset\R^{n}$, the space $\sobo^{s,p}(\Sigma;\R^{N})$ is the linear space of all $\mathscr{H}^{n-1}$-measurable maps $v\colon\Sigma\to\R^{N}$ for which 
\begin{align*}
\|v\|_{\sobo^{s,p}(\Sigma)} & := (\|v\|_{\lebe^{p}(\Sigma)}^{p} + [v]_{\sobo^{s,p}(\Sigma)}^{p})^{\frac{1}{p}} \\ & :=\Big(\int_{\Sigma}|v(x)|^{p}\dif^{n-1}x + \iint_{\Sigma\times\Sigma}\frac{|v(x)-v(y)|^{p}}{|x-y|^{n-1+sp}}\dif^{n-1}x\dif^{n-1}y\Big)^{\frac{1}{p}}
\end{align*}
is finite. As a straightforward consequence of Jensen's inequality, there exists a constant $c=c(n,s,p)>0$ such that for all $x_{0}\in\R^{n}$ and $r>0$ the \emph{fractional Poincar\'{e}-type inequality}
\begin{align}\label{eq:fracPoinc}
\begin{split}
\Big(\dashint_{\partial\!\ball_{r}(x_{0})}|u(x)& -(u)_{\partial\!\ball_{r}(x_{0})}|^{p}\dif^{n-1}x\Big)^{\frac{1}{p}}\\ & \leq c\,r^{s}\Big(\dashint_{\partial\!\ball_{r}(x_{0})}\int_{\partial\!\ball_{r}(x_{0})}\frac{|u(x)-u(y)|^{p}}{|x-y|^{n-1+sp}}\dif^{n-1}x\dif^{n-1}y \Big)^{\frac{1}{p}}
\end{split}
\end{align}
holds for all $u\in\sobo^{s,p}(\partial\!\ball_{r}(x_{0});\R^{N})$; clearly, in \eqref{eq:fracPoinc}, the dash is understood with respect to $\mathscr{H}^{n-1}\mres\partial\!\ball_{r}(x_{0})$. It is well-known that for any open and bounded set $\Omega\subset\R^{n}$ with Lipschitz boundary $\partial\Omega$ and $1<p<\infty$, there exists a bounded, linear and
\begin{align}
\text{surjective trace operator $\trace_{\partial\Omega}\colon\sobo^{1,p}(\Omega;\R^{N})\to\sobo^{1-\frac{1}{p},p}(\partial\Omega;\R^{N})$}.
\end{align}
For future reference, we record the following scaling behaviour of the underlying inequality: If $\Omega=\ball_{r}(x_{0})$ is an open ball and $\mathrm{tr}=\mathrm{tr}_{\partial\!\ball_{r}(x_{0})}\colon\sobo^{1,p}(\ball_{r}(x_{0});\R^{N})\to\sobo^{1-\frac{1}{p},p}(\partial\!\ball_{r}(x_{0});\R^{N})$ the corresponding trace operator, then there exists $c=c(n,N,p)>0$ such that for all Sobolev maps $u\in\sobo^{1,p}(\ball_{r}(x_{0});\R^{N})$ there holds 
\begin{align}\label{eq:tracescale}
\begin{split}
\Big(\dashint_{\partial\!\ball_{r}(x_{0})}&|\trace (u)(x)|^{p}\dif^{n-1}x\Big)^{\frac{1}{p}}\\ &  +r^{1-\frac{1}{p}}\Big(\dashint_{\partial\!\ball_{r}(x_{0})}\int_{\partial\!\ball_{r}(x_{0})}\frac{|\trace (u)(x)-\trace(u)(y)|^{p}}{|x-y|^{n-2+p}}\dif^{n-1}x\dif^{n-1}y\Big)^{\frac{1}{p}}\\ & \leq c\Big(\Big(\dashint_{\ball_{r}(x_{0})}|u(x)|^{p}\dif x\Big)^{\frac{1}{p}} + r\Big(\dashint_{\ball_{r}(x_{0})}|\nabla u(x)|^{p}\dif x \Big)^{\frac{1}{p}}\Big).
\end{split}
\end{align}
We conclude this section with the following simple observation:
\begin{lemma}\label{lem:DirClasses}
Let $1\leq p\leq q<\infty$ and let $\Omega\subset\R^{n}$ be an open and bounded domain with Lipschitz boundary. If $u\in\sobo^{1,p}(\Omega;\R^{N})$ is such that $\trace_{\partial\Omega}(u)\in\sobo^{1-1/q,q}(\partial\Omega;\R^{N})$ and $v\in\sobo^{1,q}(\Omega;\R^{N})$ satisfies $\trace_{\partial\Omega}(u)=\trace_{\partial\Omega}(v)$ $\mathscr{H}^{n-1}$-a.e. on $\partial\Omega$, then we have 
\begin{align}\label{eq:trivialidentity}
(\sobo^{1,q}\cap\sobo_{u}^{1,p})(\Omega;\R^{N})=\sobo_{v}^{1,q}(\Omega;\R^{N}). 
\end{align}
\end{lemma}
\begin{proof}
If $w$ belongs to the right-hand side of~\eqref{eq:trivialidentity}, $w\in\sobo^{1,q}(\Omega;\R^{N})$ and we have $w-u\in(v-u)+\sobo_{0}^{1,q}(\Omega;\R^{N})\subset \sobo_{0}^{1,p}(\Omega;\R^{N})$. Hence $w$ also belongs to the left-hand side. Next note that $u+\sobo_{0}^{1,p}(\Omega;\R^{N})=v+\sobo_{0}^{1,p}(\Omega;\R^{N})$ as $u-v\in\sobo_{0}^{1,p}(\Omega;\R^{N})$. Thus, if $w$ belongs to the left-hand side of~\eqref{eq:trivialidentity}, we may write $w=v+\varphi\in\sobo^{1,q}(\Omega;\R^{N})$ with $\varphi\in(\sobo_{0}^{1,p}\cap\sobo^{1,q})(\Omega;\R^{N})$. But $(\sobo_{0}^{1,p}\cap\sobo^{1,q})(\Omega;\R^{N})=\sobo_{0}^{1,q}(\Omega;\R^{N})$, which can be seen directly by~\eqref{eq:tracedescription}, and so~\eqref{eq:trivialidentity} follows. The proof is complete. 
\end{proof} 
 
\subsection{Averaged Taylor polynomials}
For our applications in Sections~\ref{sec:Fubini} and~\ref{sec:MazurEuler}, we require some background facts on averaged Taylor polynomials of degree $1$; see~\textsc{Maz'ya}~\cite[\S 1.1.10]{Mazya} for more detail. Let $\ball=\ball_{r}(x)$ be an open ball of radius $r>0$ and let $w_{\ball}\in\hold_{c}^{\infty}(\ball;[0,1])$ be a radially symmetric weight function with $\int_{\ball}w_{\ball}(x)\dif x =1$. This weight function can be obtained from a corresponding weight function $w=w_{\ball_{1}(0)}$ by $w_{\ball}(x):=\frac{1}{r^{n}}w(\tfrac{x-x_{0}}{r})$. For $u\in\lebe_{\locc}^{1}(\ball;\R^{N})$, we then define for $l\in\{0,1\}$ the \emph{averaged Taylor polynomial} of degree $l$ by 
\begin{align}\label{eq:ATP0}
\Pi_{\ball}^{l}u(x):=\sum_{|\alpha|\leq l}\frac{(-1)^{|\alpha|}}{\alpha!}\int_{\ball}u(y)\partial_{y}^{\alpha}(w_{\ball}(y)(x-y)^{\alpha})\dif y.
\end{align}
We record the following lemma, which directly follows from the definition of $\Pi_{\ball}^{l}u$, integration by parts and the scaling properties of $w_{\ball}$:
\begin{lemma}
The map $\Pi_{\ball}^{1}$ is the identity on the affine-linear maps. Moreover, there exists a constant $c=c(n,N,w)>0$ such that the following hold for all $u\in\bv_{\locc}(\R^{n};\R^{N})$, all open balls $\ball=\ball_{r}(x_{0})$ and $l\in\{0,1\}$:
\begin{align}\label{eq:ATP}
\begin{split}
&\frac{1}{c}\|\Pi_{\ball}^{l}u\|_{\lebe^{\infty}(\ball)}\leq \dashint_{\ball}|\Pi_{\ball}^{l}u|\dif x\leq c\dashint_{\ball}|u|\dif x, \\ 
&\|\nabla\Pi_{\ball}^{l}u\|_{\lebe^{\infty}(\ball)}\leq c \frac{|Du|(\ball)}{r^{n}}, \\ 
&\frac{1}{c}\|\Pi_{\ball}^{1}u-\Pi_{\ball}^{0}u\|_{\lebe^{\infty}(\ball)}\leq \dashint_{\ball}|\Pi_{\ball}^{1}u-\Pi_{\ball}^{0}u|\dif x \leq c\frac{|Du|(\ball)}{r^{n-1}}.
\end{split}
\end{align}
Moreover, we have \emph{Poincar\'{e}'s inequality}
\begin{align}\label{eq:ATPPoinc}
\dashint_{\ball}|u-\Pi_{\ball}^{0}u|\dif x \leq c \frac{|Du|(\ball)}{r^{n-1}}. 
\end{align}
\end{lemma}
\subsection{Spherical maximal functions and a refined selection lemma}\label{sec:HLW}
In the main part of the paper we will frequently employ constructions that hinge on \emph{good radii}, a notion that will be made precise in Section~\ref{sec:goodgeneration}. To ensure the existence of sufficiently many such good radii, we now introduce the main tools for their construction. Given an open set $\Omega\subset\R^{n}$ and a finite dimensional inner product space $X$, let $\mu\in\mathrm{RM}(\Omega;X)$ and $x_{0}\in\Omega$. We define for $0<t<d:=\mathrm{dist}(x_{0},\partial\Omega)$ the (spherical) maximal operator
\begin{align}\label{eq:maxoperator}
\mathcal{M}\mu(x_{0},t):=\sup\left\{\frac{|\mu|(\ball_{t+\varepsilon}(x_{0})\setminus\overline{\ball}_{t-\varepsilon}(x_{0}))}{2\varepsilon}\colon\;0<\varepsilon<\min\{t,d-t\}\right\}. 
\end{align}
By routine means, one obtains that $\mathcal{M}\mu(x_{0},t)<\infty$ for $\mathscr{L}^{1}$-a.e. $t\in (0,d)$. This is a consequence of the weak-$(1,1)$-bound 
\begin{align*}
\mathscr{L}^{1}(\{t\in (0,R)\colon\;\mathcal{M}\mu(x_{0},t)>\lambda\})\leq \frac{c}{\lambda}|\mu|(\ball_{R}(x_{0})),\qquad \lambda>0,\;0<R\leq d,
\end{align*}
where $c>0$ is a constant independent of $x_{0},R$ and $\mu$. The other ingredient that we require is a lemma of Hardy-Littlewood type in the spirit of~\cite[Lem.~2.3]{FM},~\cite[Lem.~4.6]{SchmidtPR}, now allowing for double-sided control of difference quotients and exceptional sets of non-zero measure.
\begin{lemma}\label{lem:HLW}
Let $0<r<s<\infty$, $\theta\in [0,\tfrac{1}{8})$ and let $f\colon [r,s]\to\R$ is a non-decreasing and right-continuous function. Then for any measurable subset $E\subset[r,s]$ with $\mathscr{L}^{1}(E)<\theta(s-r)$ there exist $r<\widetilde{r}< \widetilde{s}<s$, with $\widetilde{r},\widetilde{s}\notin E$ and the following properties: We have 
\begin{align}\label{eq:constructssr}
\frac{f(a)-f(\tau)}{a-\tau}\leq \frac{800}{1-8\theta} \frac{f(s)-f(r)}{s-r}&\;\text{for all}\;\tau<a, \\ 
\frac{f(\tau)-f(a)}{\tau-a}\leq \frac{800}{1-8\theta} \frac{f(s)-f(r)}{s-r}&\;\text{for all}\;\tau>a
\end{align}
for $a\in\{\widetilde{r},\widetilde{s}\}$ and 
\begin{align}\label{eq:comparessr}
(\widetilde{s}-\widetilde{r})\leq (s-r) \leq 8(\widetilde{s}-\widetilde{r}). 
\end{align}
\end{lemma}
The elementary proof of this lemma is provided in the appendix, Section~\ref{sec:appendix}. The condition of right-continuity only enters by the use of the Lebesgue-Stieltjes measure in the proof of Lemma~\ref{lem:HLW}. 
For future reference, we note that whenever $\mu\in\mathrm{RM}(\Omega;X)$, $x_{0}\in\Omega$ and $0<r<s<\mathrm{dist}(x_{0},\partial\Omega)$, then 
\begin{align*}
f(t):=|\mu|(\overline{\ball_{t}(x_{0})}),\qquad t\in[r,s], 
\end{align*}
is non-decreasing and right-continuous: Since $|\mu|$ is a positive Radon measure, we have for all $t\in[r,s)$ and  $t<t_{j}<t'_{j}<t_{j}+2^{-j}$ with $t_{j}\searrow t$ and $t_{j}+2^{-j}<s$
\begin{align*}
f(t) & \leq \lim_{j\to\infty}f(t_{j}) \leq |\mu|(\overline{\ball_{t}(x_{0})}) + \lim_{j\to\infty} |\mu|(\ball_{t'_{j}}(x_{0})\setminus\overline{\ball_{t}(x_{0})}) =|\mu|(\overline{\ball_{t}(x_{0})}) = f(t). 
\end{align*}
Moreover, if $u\in\bv_{\locc}(\Omega;\R^{N})$ and $x_{0}\in\Omega$, $r>0$ are such that $\mathcal{M}Du(x_{0},r)<\infty$, then 
\begin{align}\label{eq:traceequal}
\trace_{\partial\!\ball_{r}(x_{0})}^{+}(u)=\trace_{\partial\!\ball_{r}(x_{0})}^{-}(u)\qquad\text{$\mathscr{H}^{n-1}$-a.e. on $\partial\!\ball_{r}(x_{0})$},
\end{align}
as can be seen from the estimate
\begin{align*}
\int_{\partial\!\ball_{r}(x_{0})}|\trace_{\partial\!\ball_{r}(x_{0})}^{+}(u)-\trace_{\partial\!\ball_{r}(x_{0})}^{-}(u)|\dif\mathscr{H}^{n-1} &  \stackrel{\eqref{eq:interiortraces}}{=} |Du|(\partial\!\ball_{r}(x_{0})) \\ & \;\leq 2\varepsilon \Big(\frac{1}{2\varepsilon}|Du|(\ball_{r+\varepsilon}(x_{0})\setminus\overline{\ball}_{r-\varepsilon}(x_{0}))\Big) \\ & \;\leq 2\varepsilon\mathcal{M}Du(x_{0},r)
\end{align*}
for all $\varepsilon>0$ sufficiently small and then sending $\varepsilon\searrow 0$. 

\subsection{Reference integrand estimates} 
We now record some estimates on the reference integrands $V_{p}$, which we recall to be defined for $1\leq p<\infty$  by 
\begin{align*}
V_{p}(z):=\langle z\rangle^{p}-1:=(1+|z|^{2})^{\frac{p}{2}}-1,\qquad z\in X,
\end{align*}
for any finite dimensional inner product space $X$; for brevity, we put $V:=V_{1}$. The following lemma is routine and gathers estimates from \cite[Prop.~2.5]{AnGi} and \cite[Sec.~2.5,~Lem.~4.1]{GK1}:
\begin{lemma}\label{lem:Efunction}
Let $X$ be a finite dimensional real inner product space. Then for all $z,w\in X$ and $\lambda\geq 1$ the following hold:
\begin{enumerate}
\item \label{item:EfctA} $(\sqrt{2}-1)\min\{|z|,|z|^{2}\}\leq V(z)\leq \min\{|z|,|z|^{2}\}$. In general, for any $1\leq p<\infty$ there exist $c_{p},C_{p}>0$ such that $\frac{1}{c_{p}}|z|^{2}\leq V_{p}(z) \leq c_{p}|z|^{2}$ for $|z|\leq 1$ and $\frac{1}{C_{p}}|z|^{p}\leq V_{p}(z)\leq C_{p}|z|^{p}$ for $|z|\geq 1$. 
\item \label{item:VpCompa2} For any $\mathbf{m}>0$ there exists a constant $c=c(\mathbf{m})>0$ such that $\frac{1}{c}|z|^{2}\leq V(z) \leq c|z|^{2}$ for all $z\in X$ with $|z|\leq \mathbf{m}$. 
\item \label{item:EfctB} $V(\lambda z) \leq \lambda^{2}V(z)$. 
\item \label{item:EfctC} $V(z+w) \leq 2(V(z)+V(w))$.
\item \label{item:EfctE} $|V(z)-V(w)|\leq |z-w|$.
\item \label{item:VpCompa1} $V(z)/(16 (1+|w|^{2})^{\frac{3}{2}})\leq V(z+w)-V(w)-\langle V'(w),z\rangle$.
\end{enumerate}
\end{lemma}
We conclude this section by recording an estimate that will be required in Section~\ref{sec:props}:
\begin{lemma}[{\cite[Lem.~2.4]{DFLM}}]\label{lem:compaDal}
For any $1<p<\infty$ there exists a constant $c=c(p)>0$ such that 
\begin{align*}
\frac{1}{c}(1+|z|^{2}+|z'|^{2})^{\frac{p-2}{2}}|z'|^{2} \leq V_{p}(z+z')-V_{p}(z)-\langle V'_{p}(z),z'\rangle \leq c(1+|z|^{2}+|z'|^{2})^{\frac{p-2}{2}}|z'|^{2}
\end{align*}
holds for all $z,z'\in\R^{N\times n}$.
\end{lemma}
\subsection{Semiconvexity and Legendre-Hadamard elliptic systems} 
We conclude this preliminary section by recording auxiliary facts on semiconvex problems. First, the \emph{quasiconvex envelope} $F^{\qc}$ of an integrand $F\in\hold(\R^{N\times n})$ is the largest quasiconvex function below $F$: 
\begin{align}\label{eq:defQCenvelope}
F^{\qc}(z):=\sup\{G(z)\colon\;G\;\text{is quasiconvex with}\;G\leq F\},\qquad z\in\R^{N\times n}. 
\end{align}
Since $F$ is assumed real-valued (and not extended real-valued), we have either $F^{\qc}\equiv-\infty$ or $F^{\qc}>-\infty$ everywhere (see \textsc{Ball} et al.~\cite[\S 2]{BKK}). By \textsc{Dacorogna}'s formula~\cite{Dac82} (also see \textsc{Kinderlehrer \& Pedregal}~\cite[\S 8]{KinderlehrerPedregal} in the present context of integrands that are potentially unbounded below), we then have the representation formula
\begin{align}\label{eq:QCenvelope}
F^{\qc}(z)=\inf\left\{\dashint_{\ball_{1}(0)}F(z+\nabla\varphi)\dif x\colon\;\varphi\in\sobo_{0}^{1,\infty}(\ball_{1}(0);\R^{N})\right\},\qquad z\in\R^{N\times n}.
\end{align}
By a routine scaling and covering argument, we note that one might equally replace the class of competitors in~\eqref{eq:QCenvelope} by those $\varphi\in\sobo_{0}^{1,\infty}(\ball_{1}(0);\R^{N})$ that satisfy $\|\varphi\|_{\lebe^{\infty}(\ball_{1}(0))}\leq 1$. 

Now let $F\in\hold(\R^{N\times n})$ be a quasiconvex integrand. Then $F$ in particular is rank-one convex, that is, the function $t\mapsto F(z+ta\otimes b)$ is convex for all $z\in\R^{N\times n}$, $a\in\R^{N}$ and $b\in\R^{n}$. If $F$ moreover satisfies the growth bound
\begin{align}\label{eq:Fgrowthbound}
|F(z)|\leq L(\lambda+|z|^{q})\qquad\text{for all}\;z\in\R^{N\times n}
\end{align}
 for some $1\leq q<\infty$, $L\geq 0$ and $\lambda\geq 0$, the rank-one convexity and~\eqref{eq:Fgrowthbound} (which in particular is satisfied for $\lambda=1$ by~\ref{item:H1}) moreover combine in a standard way to the estimate
\begin{align}\label{eq:lipschitz}
  \bigl| F(z)-F(w) \bigr| \leq c\bigl(\lambda+|w|^{q-1}+|z-w|^{q-1} \bigr) |z-w| \quad \text{for all}\, z, \, w \in \R^{N\times n},
\end{align}
where $c=c(q,n,N,L)>0$ is a constant. If, moreover, $F\in\hold^{1}(\R^{N\times n})$, then \eqref{eq:lipschitz} implies 
\begin{align}\label{eq:DERIVBOUND}
|F'(z)| \leq c(\lambda+|z|^{q-1})\qquad\text{for all}\;z\in\R^{N\times n}. 
\end{align}
This inequality extends to $\mathscr{L}^{Nn}$-almost all $z\in\R^{N\times n}$ provided $F$ is only assumed Lipschitz. In view of verifying growth bounds of the type~\eqref{eq:Fgrowthbound}, we point out that that for a rank-one-convex function $F\colon\R^{N\times n}\to\R$ and $1\leq q<\infty$ one has the implication (see~\cite[Lem.~2.5]{Kristensen99})
\begin{align}\label{eq:KristImplication}
\limsup_{|z|\to\infty}\frac{F(z)}{|z|^{q}}<\infty \Longrightarrow \limsup_{|z|\to\infty}\frac{|F(z)|}{|z|^{q}}<\infty. 
\end{align}
Next recall that a symmetric bilinear form $\mathbb{A}$ on $\R^{N\times n}$, to be tacitly identified with its matrix representative $\mathbb{A}\in\R^{Nn\times Nn}$, is said to be \emph{strongly} or \emph{Legendre-Hadamard elliptic} if there exist $0<\lambda\leq \Lambda<\infty$ such that 
\begin{align}\label{eq:LHE}
\lambda |a|^{2}|b|^{2} \leq \mathbb{A}[a\otimes b,a\otimes b] \leq \Lambda |a|^{2}|b|^{2}\qquad\text{for all}\;a\in\R^{N},b\in\R^{n}. 
\end{align}
By routine means, one finds that if $F\in\hold^{2}(\R^{N\times n})$ satisfies~\ref{item:H1} and~\ref{item:H2} or~\ref{item:H2p}, respectively, then $\mathbb{A}=F''(w)$ satisfies~\eqref{eq:LHE} for all $w\in\R^{N\times n}$ with $|w|\leq m$, where $\lambda,\Lambda>0$ both solely depend on $p,q,m,\ell_{m},L,n,$ and $N$. We conclude this preliminary section by recording an auxiliary theorem on the solvability and regularity for Legendre-Hadamard elliptic systems:
\begin{lemma}[{\cite[Lem.~15.2.1]{MazSha},~\cite[Prop.~2.10]{CFM}}]\label{lem:linearsystems}
Let $1<p<\infty$, $x_{0}\in\R^{N}$ and $R>0$. Furthermore, assume that $v\in\sobo^{1-\frac{1}{p},p}(\partial\!\ball_{R}(x_{0});\R^{N})$ and $T\in\lebe^{p}(\ball_{R}(x_{0});\R^{N})$. Given a bilinear form $\mathbb{A}$ on $\R^{N\times n}$ satisfying~\eqref{eq:LHE} for some $0<\lambda\leq\Lambda<\infty$, the following hold:
\begin{enumerate}
\item\label{item:linsys1} There exists a unique weak solution $h\in(\sobo_{v}^{1,p}\cap\hold^{\infty})(\ball_{R}(x_{0});\R^{N})$ of the system 
\begin{align*}
\begin{cases}
-\di(\A\nabla h) = 0&\;\text{in}\;\ball_{R}(x_{0}),\\
h= v&\;\text{on}\;\partial\!\ball_{R}(x_{0}), 
\end{cases}
\end{align*}
and there exist $c=c(p,\lambda,\Lambda,n,N)>0$ and $C=C(p,\lambda,\Lambda,n,N)>0$ such that we have 
\begin{align*}
\sum_{0\leq i \leq 2}R^{i}\Big(\dashint_{\ball_{R}(x_{0})}|\nabla^{i}h|^{p}& \dif x\Big)^{\frac{1}{p}}\leq c \Big(\Big(\dashint_{\partial\!\ball_{R}(x_{0})}|v|^{p}\dif^{n-1} x\Big)^{\frac{1}{p}}\Big. \\ & \Big.+ R^{1-\frac{1}{p}}\Big(\dashint_{\partial\!\ball_{R}(x_{0})}\int_{\partial\!\ball_{R}(x_{0})}\frac{|v(x)-v(y)|^{p}}{|x-y|^{n+p-2}}\dif^{n-1}x\dif^{n-1}y \Big)^{\frac{1}{p}}\Big)
\end{align*}
and, for all $\ball_{r}(x)\Subset\ball_{R}(x_{0})$,
\begin{align}\label{eq:inverseest}
\sup_{\ball_{r/2}(x)}|Dh| + r\sup_{\ball_{r/2}(x)}|D^{2}h| \leq C \dashint_{\ball_{3r/4}(x)}|Dh|\dif y.
\end{align}
\item\label{item:linsys2} There exists a unique weak solution $h\in(\sobo_{0}^{1,p}\cap\sobo^{2,p})(\ball_{R}(x_{0});\R^{N})$ of the system 
\begin{align*}
\begin{cases}
-\di(\A\nabla h) = T&\;\text{in}\;\ball_{R}(x_{0}),\\
h= 0&\;\text{on}\;\partial\!\ball_{R}(x_{0}),
\end{cases}
\end{align*}
and there exists $c=c(p,\lambda,\Lambda,n,N)>0$ such that we have
\begin{align*}
\sum_{0\leq i \leq 2}R^{i-2}\Big(\dashint_{\ball_{R}(x_{0})}|\nabla^{i}h|^{p}\dif x\Big)^{\frac{1}{p}}\leq c \Big(\dashint_{\ball_{R}(x_{0})}|T|^{p}\dif x\Big)^{\frac{1}{p}}.
\end{align*}
\end{enumerate}
\end{lemma}
\section{Trace-preserving operators and Fubini-type theorems}\label{sec:Fubini}
For our future purposes we now provide a construction of a trace-preserving operator on $\bv$. This appears as a slight modification of the operator introduced by \textsc{Fonseca \& Mal\'{y}} \cite{FMParma} in the higher order Sobolev case and is reminiscent of the usual \textsc{Whitney}~ smoothing procedure \cite{Whitney}. Since this operator and modifications thereof take a central role in the paper, both in the quest of good minimising sequences, Fubini-type properties, the Euler-Lagrange system through the Mazur-type lemma and in the proof of the Caccioppoli inequality, we give here the detailled construction. As a technical novelty, our approach mostly only uses properties of finite dimensional function spaces, which should also prove useful in  more general settings.

 Let $\mathfrak{U}\subset\R^{n}$ be open and let $(\ball^{i})=(\ball_{r_{i}}(x_{i}))$ be a Whitney covering of $\mathfrak{U}$. By this we understand that there exist constants $c>0$, $\Lambda>0$ and $\mathtt{N}\in\mathbb{N}$ such that 
\begin{enumerate}[label={$(\mathrm{W}\arabic{*})$},start=1] 
\item\label{item:Whitney1} $\bigcup_{i\in\mathbb{N}}\ball^{i}=\mathfrak{U}$, 
\item\label{item:Whitney2} for each $i_{0}\in\mathbb{N}$ there are at most $\mathtt{N}$ indices $i\in\mathbb{N}$ such that $\ball^{i_{0}}\cap\,\ball^{i}\neq\emptyset$ and $\Lambda\!\ball^{i_{0}}\cap\,\Lambda\!\ball^{i}\neq\emptyset$. Moreover, denoting $\mathcal{N}(i_{0}):=\{i\in\mathbb{N}\colon\ball^{i}\cap\ball^{i_{0}}\neq\emptyset\}$, we have $\bigcup_{i\in\mathcal{N}(i_{0})}\ball^{i}\subset\Lambda\!\ball^{i_{0}}\Subset\mathfrak{U}$. 
\item\label{item:Whitney2a} whenever $i,j\in\mathbb{N}$ are such that $\ball^{i}\cap\ball^{j}\neq\emptyset$, then we have 
\begin{align*}
\mathscr{L}^{n}(\ball^{i}\cap\ball^{j})\geq \frac{1}{c}\max\{\mathscr{L}^{n}(\ball^{i}),\mathscr{L}^{n}(\ball^{j})\}. 
\end{align*} 
In particular, if $\ball^{i}\cap\ball^{j}\neq\emptyset$, then $\frac{1}{c}r_{i}\leq r_{j}\leq cr_{i}$. 
\item\label{item:Whitney3} $\frac{1}{c}r_{i}\leq\dista(\Lambda\!\ball^{i},\mathfrak{U}^{\complement})<cr_{i}$ for all $i\in\mathbb{N}$. 
\end{enumerate}
Consequently, we may choose a partition of unity $(\rho_{i})\subset\hold_{c}^{\infty}(\mathfrak{U};[0,1])$ with
\begin{enumerate}[label={$(\mathrm{W}\arabic{*})$},start=6] 
\item\label{item:Whitney4} $\spt(\rho_{i})\subset\ball^{i}$ for all $i\in\mathbb{N}$, 
\item\label{item:Whitney5} $\sum_{i\in\mathbb{N}}\rho_{i}\equiv 1$ on $\mathfrak{U}$.  
\item\label{item:Whitney6} $|\nabla\rho_{i}|\leq \frac{c}{r_{i}}$ for all $i\in\mathbb{N}$. 
\end{enumerate}
Given $u\in\lebe_{\locc}^{1}(\R^{n};\R^{N})$, we lastly define $\mathbb{E}_{\mathfrak{U}}u\colon \mathfrak{U}\to\R^{N}$ and $\widetilde{\mathbb{E}}_{\mathfrak{U}}u\colon \mathfrak{U}\to\R^{N}$ by 
\begin{align}\label{eq:extoperator}
\begin{split}
&\mathbb{E}_{\mathfrak{U}}u := \sum_{i\in\mathbb{N}}\rho_{i}u_{i} := \sum_{i\in\mathbb{N}}\rho_{i}\dashint_{\ball^{i}}u\dif x, \\
&\widetilde{\mathbb{E}}_{\mathfrak{U}}u := \sum_{i\in\mathbb{N}}\rho_{i}\Pi_{\ball^{i}}^{1}u
\end{split}
\end{align}
with the first order averaged Taylor polynomial $\Pi_{\Lambda\!\ball^{i}}^{1}u$, cf.~\eqref{eq:ATP0}. The most instrumental features of these operators are summarised in the following lemma: 
\begin{lemma}[Boundedness properties of $\mathbb{E}$ and $\widetilde{\mathbb{E}}$]\label{lem:extensionoperator}
Let $\Omega\subset\R^{n}$ be open and $\mathfrak{U}\subset\Omega$ be an open ball or annulus, respectively, and let $u\in\bv_{\locc}(\Omega;\R^{N})$.  For each $1\leq q <\infty$ there exists a constant $c=c(n,N,q)>0$ such that the following hold for $\mathbb{G}\in\{\mathbb{E}_{\mathfrak{U}},\widetilde{\mathbb{E}}_{\mathfrak{U}}\}$:
\begin{enumerate}
\item\label{item:extension1} For all $i_{0}\in\mathbb{N}$, $x\in\ball^{i_{0}}$ and $j\in\{0,1\}$ we have 
\begin{align}\label{eq:extensionpointwise}
\begin{split}
&|\nabla^{j}\mathbb{G}u(x)|\leq c\,\dashint_{\Lambda\!\ball^{i_{0}}}|D^{j}u|,\\ 
&\left(\dashint_{\ball^{i_{0}}}|\nabla^{j}\mathbb{G}u|^{q}\dif y \right)^{\frac{1}{q}} \leq c\,\dashint_{\Lambda\!\ball^{i_{0}}}|D^{j}u|. 
\end{split}
\end{align}
\item\label{item:extension2} 
\emph{$\lebe^{q}$-(gradient) stability:} We have 
\begin{align}\label{eq:gradientstab}
\int_{\mathfrak{U}}  |\nabla^{j}\mathbb{G}u|^{q}\dif x \leq c \int_{\mathfrak{U}}  |D^{j}u|^{q}, 
\end{align}
where the right-hand side might be infinite. In consequence, for each $1\leq q <\infty$, $\mathbb{G}$ is a bounded linear operator
\begin{align*}
&\mathbb{G}\colon\bv(\mathfrak{U};\R^{N}) \to \bv(\mathfrak{U};\R^{N}),\\ 
&\mathbb{G}\colon\sobo^{1,q}(\mathfrak{U};\R^{N}) \to \sobo^{1,q}(\mathfrak{U};\R^{N})
\end{align*}
and is \emph{trace-preserving} in the sense that 
 $\trace_{\partial\mathfrak{U}}(u-\mathbb{G}u)=0$ holds $\mathscr{H}^{n-1}$-a.e. on $\partial\mathfrak{U}$. 
\item\label{item:extension3a} In case $\mathbb{G}=\widetilde{\mathbb{E}}_{\mathfrak{U}}$, $\mathbb{G}$ is the identity on the affine-linear maps. 
\item\label{item:extension4} Let $x_{0}\in\Omega$ and $0<r<s<\dista(x_{0},\partial\Omega)$. Given  $1\leq q < \frac{n}{n-1}$, there exist constants $c=c(n,N,q)>0$ with $c(n,N,q)\nearrow\infty$ as $q\nearrow\frac{n}{n-1}$ and $\lambda=\lambda(n)>1$ such that the following hold: For all $0<\varepsilon<s-r$ we have with $\delta_{0}:=\delta_{0}(\varepsilon):=\min\{s-r,\lambda\varepsilon\}$ for $j\in\{0,1\}$
\begin{align}\label{eq:fundamental}
\begin{split}
 \int_{(\ball_{s}(x_{0})\setminus\overline{\ball}_{r}(x_{0}))_{\varepsilon}^{\complement}}  |\nabla^{j}\mathbb{G}u|\dif x & \leq c\,(s-r)\sup_{0<\delta<\delta_{0}}\Big(\frac{1}{\delta} \int_{(\ball_{s}(x_{0})\setminus\overline{\ball}_{r}(x_{0}))_{\delta}^{\complement}} |D^{j}u|      \Big), \\ 
 \Big(\int_{(\ball_{s}(x_{0})\setminus\overline{\ball}_{r}(x_{0}))_{\varepsilon}^{\complement}}  |\nabla^{j}\mathbb{G}u|^{q}\dif x\Big)^{\frac{1}{q}} & \leq c\,(s-r)^{\frac{n}{q}-n+1}\times \\ & \times \sup_{0<\delta<\delta_{0}}\Big(\frac{1}{\delta} \int_{(\ball_{s}(x_{0})\setminus\overline{\ball}_{r}(x_{0}))_{\delta}^{\complement}} |D^{j}u|      \Big).
\end{split}
\end{align} 
\item\label{item:extension5} Let $x_{0}\in\Omega$ and $0<r<\dista(x_{0},\partial\Omega)$. Given  $1\leq q < \frac{n}{n-1}$, there exist constants  $c=c(n,N,q)>0$ with $c(n,N,q)\nearrow\infty$ as $q\nearrow\frac{n}{n-1}$ and $0<\varepsilon<\frac{r}{2}$ we have for $j\in\{0,1\}$
\begin{align}\label{eq:fundamental1}
\begin{split}
 & \int_{\ball_{r}(x_{0})}  |\nabla^{j}\mathbb{G}u|\dif x \leq c\frac{r}{\varepsilon}\int_{\ball_{r}(x_{0})}|D^{j}u| \\ & \;\;\;\;\;\;\;\;\;\;\;\;\;\;\;\;\;\;\;\;\;\;\;\;\;\;\;\;\;\;\;\;\;\;\;\;+c\,r \sup_{0<\delta<\varepsilon}\Big(\frac{1}{\delta} \int_{(\ball_{r}(x_{0}))_{\delta}^{\complement}} |D^{j}u|      \Big), \\ 
 & \Big(\int_{\ball_{r}(x_{0})}  |\nabla^{j}\mathbb{G}u|^{q}\dif x\Big)^{\frac{1}{q}} \leq c\frac{r^{n\frac{1-q}{q}+1}}{\varepsilon}\Big(\int_{\ball_{r}(x_{0})}|D^{j}u|\Big)\\ & \;\;\;\;\;\;\;\;\;\;\;\;\;\;\;\;\;\;\;\;\;\;\;\;\;\;\;\;\;\;\;\;\;\;\;\;+c\,r^{\frac{n}{q}-n+1} \sup_{0<\delta<\varepsilon}\Big(\frac{1}{\delta} \int_{(\ball_{r}(x_{0}))_{\delta}^{\complement}} |D^{j}u|      \Big).
\end{split}
\end{align} 
\end{enumerate}
Here we have used the convention $U_{\delta}:=\{x\in U\colon\;\mathrm{dist}(x,\partial U)>\delta\}$ and $U_{\delta}^{\complement}:=U\setminus U_{\delta}$ for any set $U\subset\R^{n}$.
\end{lemma}
\begin{proof}
We focus on $\mathbb{G}=\widetilde{\mathbb{E}}_{\mathfrak{U}}$, the estimates for $\mathbb{E}_{\mathfrak{U}}$ being similar but easier. Ad~\ref{item:extension1}. Let $x\in\ball^{i_{0}}$. By~\ref{item:Whitney2a}, if $i\in\mathcal{N}(i_{0})$, then $r_{i}$ and $r_{i_{0}}$ are uniformly proportional, $r_{i}\sim r_{0}$. Then we have
\begin{align*}
|\widetilde{\mathbb{E}}_{\mathfrak{U}}u(x)| & \leq \sum_{i\in\mathcal{N}(i_{0})}\rho_{i}(x)\|\Pi_{\ball^{i}}^{1}u\|_{\lebe^{\infty}(\ball^{i})} \stackrel{\eqref{eq:ATP}_{1}}{\leq} c \sum_{i\in\mathcal{N}(i_{0})}\rho_{i}(x)\dashint_{\ball^{i}}|u|\dif y \\ 
& \leq c\sum_{i\in\mathcal{N}(i_{0})}\rho_{i}(x)\dashint_{\Lambda\!\ball^{i_{0}}}|u|\dif y = c\dashint_{\Lambda\!\ball^{i_{0}}}|u|\dif y,
\end{align*}
since the radii of two balls $\ball^{i}$ and $\ball^{i_{0}}$ with $i\in\mathcal{N}(i_{0})$ are uniformly comparable by~\ref{item:Whitney2a}, and $|\mathcal{N}(i_{0})|\leq\mathtt{N}$, $\bigcup_{i\in\mathcal{N}(i_{0})}\ball^{i}\subset\Lambda\!\ball^{i_{0}}$ by \ref{item:Whitney2}. This settles~$\eqref{eq:extensionpointwise}_{1}$ for $j=0$;~$\eqref{eq:extensionpointwise}_{2}$ then follows by integrating~$\eqref{eq:extensionpointwise}_{1}$ over $\ball^{i_{0}}$ and Jensen's inequality. For the case $j=1$, observe that since $\sum_{i}\rho_{i}=1$ in $\mathfrak{U}$, we have by the finite overlap of the Whitney balls that $\nabla(\sum_{i}\rho_{i}\Pi_{\ball^{i_{0}}}^{0}u)=0$ in $\mathfrak{U}$. In consequence, 
\begin{align*}
|\nabla\widetilde{\mathbb{E}}_{\mathfrak{U}}u(x)| & \leq \left\vert\nabla\Big(\sum_{i\in\mathcal{N}(i_{0})}\rho_{i}(x)(\Pi_{\ball^{i_{0}}}^{0}u(x)-\Pi_{\ball^{i}}^{1}u(x)) \Big)\right\vert \\ 
& \!\!\!\!\!\!\!\!\!\!\!\!\!\!\!\!\!\!\!\!\stackrel{\text{\ref{item:Whitney6}}}{\leq}c\sum_{i\in\mathcal{N}(i_{0})}\Big(\frac{\|\Pi_{\ball^{i_{0}}}^{0}u-\Pi_{\ball^{i}}^{1}u\|_{\lebe^{\infty}(\ball^{i_{0}}\cap\ball^{i})}}{r_{i}}+\|\nabla\Pi_{\ball^{i}}^{1}u\|_{\lebe^{\infty}(\ball^{i_{0}}\cap\ball^{i})}\Big) \\ 
&\!\!\!\!\!\!\!\!\!\!\!\!\!\!\!\! =: c\sum_{i\in\mathcal{N}(i_{0})}(\mathrm{I}_{i}^{(1)}+\mathrm{I}_{i}^{(2)}). 
\end{align*}
For the estimation of $\mathrm{I}_{i}^{(1)}$, first note that by~\ref{item:Whitney2a} and the equivalence of all norms on finite dimensional spaces, we have  
\begin{align}\label{eq:InverseTracePreserve}
\|\pi\|_{\lebe^{\infty}(\ball^{i_{0}}\cap\ball^{i})}\leq c\dashint_{\ball^{i_{0}}\cap\ball^{i}}|\pi|\dif y
\end{align}
for all polynomials $\pi\colon\R^{n}\to\R^{N}$ of degree one, with $c>0$ solely depending on $n$ and $N$. Therefore, whenever $i\in\mathcal{N}(i_{0})$, 
\begin{align*}
\frac{\|\Pi_{\ball^{i_{0}}}^{0}u-\Pi_{\ball^{i}}^{1}u\|_{\lebe^{\infty}(\ball^{i_{0}}\cap\ball^{i})}}{r_{i}} & \leq \frac{\|\Pi_{\ball^{i_{0}}}^{0}u-\Pi_{\ball^{i}}^{0}u\|_{\lebe^{\infty}(\ball^{i_{0}}\cap\ball^{i})}}{r_{i}}+\frac{\|\Pi_{\ball^{i}}^{1}u-\Pi_{\ball^{i}}^{0}u\|_{\lebe^{\infty}(\ball^{i})}}{r_{i}} \\ 
& \!\!\!\!\!\!\!\!\!\stackrel{\eqref{eq:InverseTracePreserve},\,\eqref{eq:ATP}_{3}}{\leq} \frac{c}{r_{i}}\dashint_{\ball^{i_{0}}\cap\ball^{i}}|\Pi_{\ball^{i_{0}}}^{0}u-\Pi_{\ball^{i}}^{0}u|\dif y + c\frac{|Du|(\ball^{i})}{r_{i}^{n}}\\ 
& \leq \frac{c}{r_{i}}\dashint_{\ball^{i_{0}}\cap\ball^{i}}|\Pi_{\ball^{i_{0}}}^{0}u-u|\dif y + \frac{c}{r_{i}}\dashint_{\ball^{i_{0}}\cap\ball^{i}}|\Pi_{\ball^{i}}^{0}u-u|\dif y\\ & \;\;\;\;\;\;\;\;\;\;\;\;\;\;\;\;+ c\frac{|Du|(\ball^{i})}{r_{i}^{n}}\\
& \!\!\!\stackrel{\text{\ref{item:Whitney2a}}}{\leq} \frac{c}{r_{i}}\dashint_{\ball^{i_{0}}}|\Pi_{\ball^{i_{0}}}^{0}u-u|\dif y + \frac{c}{r_{i}}\dashint_{\ball^{i}}|\Pi_{\ball^{i}}^{0}u-u|\dif y\\ & \;\;\;\;\;\;\;\;\;\;\;\;\;\;\;\;+ c\frac{|Du|(\ball^{i})}{r_{i}^{n}}\\ 
& \!\!\!\!\!\!\!\!\stackrel{\text{\ref{item:Whitney2a},\,\eqref{eq:ATPPoinc}}}{\leq}c\dashint_{\Lambda\!\ball^{i_{0}}}|Du|, 
\end{align*}
whereby~\ref{item:Whitney2} implies 
\begin{align}\label{eq:sumup1}
\sum_{i\in\mathcal{N}(i_{0})}\mathrm{I}_{i}^{(1)} \leq c\dashint_{\Lambda\!\ball^{i_{0}}}|Du|.
\end{align}
For $\mathrm{I}_{i}^{(2)}$, we use~$\eqref{eq:ATP}_{2}$, \ref{item:Whitney2} and~\ref{item:Whitney3} to infer 
\begin{align}\label{eq:sumup2}
\sum_{i\in\mathcal{N}(i_{0})}\mathrm{I}_{i}^{(2)} \stackrel{\eqref{eq:ATP}_{2}}{\leq} c\sum_{i\in\mathcal{N}(i_{0})}\dashint_{\ball^{i}}|Du| \stackrel{\text{\ref{item:Whitney2},\,\ref{item:Whitney3}}}{\leq}c\dashint_{\Lambda\!\ball^{i_{0}}}|Du|, 
\end{align}
and combining~\eqref{eq:sumup1} and~\eqref{eq:sumup2} yields $\eqref{eq:extensionpointwise}_{1}$ for $j=1$, and then $\eqref{eq:extensionpointwise}_{2}$ for $j=1$ follows by integrating over $\ball^{i_{0}}$ and Jensen's inequality. Ad~\ref{item:extension2}. The boundedness follows at once from $\eqref{eq:extensionpointwise}_{1}$ by integrating over $\mathfrak{U}$ and using the finite overlap of the Whitney balls $\ball^{i}$, cf.~\ref{item:Whitney2}.  On the other hand, for any finite index set $\mathcal{I}\subset\mathbb{N}$, $x\mapsto \sum_{i\in\mathcal{I}}\rho_{i}\Pi_{\ball^{i}}^{1}u$ belongs to $\hold_{c}^{\infty}(\mathfrak{U};\R^{N})$, and~\ref{item:Whitney3} implies that $i\to\infty$ is equivalent to $\dista(\Lambda\!\ball^{i},\mathfrak{U}^{\complement})\to 0$. For each $m_{0}\in\mathbb{N}$, we thus find $\varepsilon>0$ such that $\bigcup_{m=m_{0}}^{\infty}\Lambda\!\ball^{m}\subset\mathfrak{U}_{\varepsilon}^{\complement}$, and $\varepsilon\to 0$ as $m_{0}\to\infty$. In consequence, for an arbitrary $m_{0}\in\mathbb{N}$, we find by $u=\sum_{i\in\mathbb{N}}\rho_{i}u$ that
\begin{align*}
\int_{\bigcup_{m=m_{0}}^{\infty}\ball^{m}}|u-\widetilde{\mathbb{E}}_{\mathfrak{U}}u|\dif x &\leq c\sum_{m=m_{0}}^{\infty}\sum_{i\in\mathbb{N}}\int_{\ball^{m}}|\rho_{i}(u-\Pi_{\ball^{i}}^{1}u)|\dif x \\ 
& \!\!\!\!\!\!\!\!\!\!\stackrel{\text{\ref{item:Whitney2},~\ref{item:Whitney4}}}{\leq} c\sum_{m=m_{0}}^{\infty}\sum_{i\in\mathcal{N}(m)}\int_{\ball^{i}}|\rho_{i}(u-\Pi_{\ball^{i}}^{1}u)|\dif x \\ 
& \!\!\!\!\!\!\!\!\!\!\!\!\!\!\!\!\!\stackrel{\eqref{eq:ATP}_{1},~\text{\ref{item:Whitney2},\,\ref{item:Whitney4}}}{\leq} c\sum_{m=m_{0}}^{\infty}\|u\|_{\lebe^{1}(\Lambda\!\ball^{m})} \leq c\int_{\mathfrak{U}_{\varepsilon}^{\complement}}|u|\dif x \to 0
\end{align*}
as $m_{0}\to\infty$. Moreover, using a similar argument as in the estimation of $\mathrm{I}^{(i)}$ from above,  
\begin{align*}
|&D(u-\widetilde{\mathbb{E}}_{\mathfrak{U}}u)|\Big(\bigcup_{m=m_{0}}^{\infty}\ball^{m} \Big) \leq c\sum_{m=m_{0}}^{\infty}\sum_{i\in\mathbb{N}}\int_{\ball^{m}}|D(\rho_{i}(u-\Pi_{\ball^{i}}^{1}u))| \\ 
& \!\!\!\!\!\!\!\!\!\!\stackrel{\text{\ref{item:Whitney2},~\ref{item:Whitney4}}}{\leq} c \sum_{m=m_{0}}^{\infty}\sum_{i\in\mathcal{N}(m)} \Big(\int_{\ball^{i}} \left\vert\frac{u-\Pi_{\ball^{i}}^{1}u}{r_{i}}\right\vert \dif x + \int_{\ball^{i}}|Du| + r_{i}^{n}\|\nabla\Pi_{\ball^{i}}^{1}u\|_{\lebe^{\infty}(\ball^{i})}\Big) \\ 
& \!\!\!\!\stackrel{\eqref{eq:ATP}_{2}}{\leq} c \sum_{m=m_{0}}^{\infty}\sum_{i\in\mathcal{N}(m)} \Big(\int_{\ball^{i}} \left\vert\frac{u-\Pi_{\ball^{i}}^{0}u}{r_{i}}\right\vert \dif x+\int_{\ball^{i}} \left\vert\frac{\Pi_{\ball^{i}}^{0}u-\Pi_{\ball^{i}}^{1}u}{r_{i}}\right\vert \dif x + \int_{\ball^{i}}|Du|\Big) \\ 
& \!\!\!\!\!\!\!\!\!\!\stackrel{\text{\eqref{eq:ATPPoinc}},~\eqref{eq:ATP}_{3}}{\leq} c \sum_{m=m_{0}}^{\infty}\sum_{i\in\mathcal{N}(m)}\int_{\ball^{i}}|Du|\\
& \!\!\!\stackrel{\text{\ref{item:Whitney2}}}{\leq} c \sum_{m=m_{0}}^{\infty}\int_{\Lambda\!\ball^{m}}|Du|\leq c \int_{\mathfrak{U}_{\varepsilon}^{\complement}}|Du| \to 0
\end{align*}
for some $\varepsilon=\varepsilon(m_{0})>0$ with $\varepsilon\to 0$ as $m_{0}\to\infty$. By the continuity properties of the boundary trace operator on $\bv$, this implies that $\trace_{\partial\mathfrak{U}}(u-\widetilde{\mathbb{E}}_{\mathfrak{U}}u)=0$ $\mathscr{H}^{n-1}$-a.e. on $\partial\mathfrak{U}$. This establishes \ref{item:extension2}. Ad~\ref{item:extension3a}. This follows from $\Pi_{\ball}^{1}a=a$ for any affine-linear map $a$, which is a direct consequence of the definition of the averaged Taylor polynomials. Ad~\ref{item:extension4}. Let $\mathfrak{U}:=\ball_{s}(x_{0})\setminus\overline{\ball}_{r}(x_{0})$. Our focus is on $\eqref{eq:fundamental}_{2}$, $\eqref{eq:fundamental}_{1}$ following by analogous means. We define for $m\in\mathbb{N}_{0}$
\begin{align}\label{eq:Imdef}
\mathcal{I}^{m}:=\big\{i\in\mathbb{N}\colon\;2^{-m-1}(s-r)\leq r_{i}<2^{-m}(s-r) \big\}, 
\end{align}
and note that by~\ref{item:Whitney2} and \ref{item:Whitney3}, there exists $c'=c'(n)>1$ such that $
\bigcup_{i\in\mathcal{I}^{m}}\Lambda\!\ball^{i}\subset S^{m}$, where the annulus
\begin{align*}
S^{m}:=\{x\in\mathfrak{U}\colon\;\frac{1}{c'}2^{-m}(s-r)<\dista(x,\partial\mathfrak{U})\leq c'2^{-m}(s-r)\}
\end{align*}
has width uniformly proportional to its distance to $\partial\mathfrak{U}$. Given $0<\varepsilon<s-r$, we choose the maximal $m_{0}\in\mathbb{N}_{0}$ such that $\mathfrak{U}_{\varepsilon}^{\complement}\subset\bigcup_{m=m_{0}}^{\infty}\bigcup_{i\in\mathcal{I}^{m}}\ball^{i}$. By construction, we thus find $\lambda=\lambda(n)>1$ such that $S	 ^{m}\subset \mathfrak{U}_{\lambda\varepsilon}^{\complement}$ for all $m\geq m_{0}$. By the uniformly finite overlap of the Whitney balls and the width of $S^{m}$ being uniformly proportional to its distance to $\partial\mathfrak{U}$, we have with $\delta_{0}:=\min\{s-r,\lambda\varepsilon\}$
\begin{align*}
\int_{(\ball_{s}(x_{0})\setminus\overline{\ball}_{r}(x_{0}))_{\varepsilon}^{\complement}}|\nabla&\widetilde{\mathbb{E}}_{\ball_{s}(x_{0})\setminus\overline{\ball}_{r}(x_{0})}u|^{q} \dif x  \stackrel{\eqref{eq:extensionpointwise}_{2}}{\leq} 
c\,\sum_{m=m_{0}}^{\infty}\sum_{i\in\mathcal{I}^{m}}r_{i}^{n(1-q)}\Big(\int_{\Lambda\!\ball^{i}}|Du|\Big)^{q} \\ & \!\!\stackrel{\ell^{1}\hookrightarrow\ell^{q}}{\leq} c\Big(\sum_{m=m_{0}}^{\infty}\sum_{i\in\mathcal{I}^{m}}r_{i}^{n(\frac{1}{q}-1)}\int_{\Lambda\!\ball^{i}}|Du|\Big)^{q} \\ 
& \;\,\leq c\Big(\sum_{m=m_{0}}^{\infty}\Big(\,\frac{s-r}{2^{m}}\,\Big)^{n(\frac{1}{q}-1+\frac{1}{n})}\frac{2^{m}}{s-r}\sum_{i\in\mathcal{I}^{m}}\int_{\Lambda\!\ball^{i}}|Du|\Big)^{q}\\
& \;\,\leq c\Big(\sum_{m=m_{0}}^{\infty}\Big(\,\frac{s-r}{2^{m}}\,\Big)^{n(\frac{1}{q}-1+\frac{1}{n})}\frac{2^{m}}{s-r}\int_{S^{m}}|Du|\Big)^{q}\\
& \;\,\leq c (s-r)^{n(1-q+\frac{q}{n})} \Big( \sup_{0<\delta<\delta_{0}}\frac{1}{\delta}\int_{\ball_{s}(x_{0})\setminus\overline{\ball}_{s-\delta}(x_{0})}|Du|+\Big. \\ & \;\, \Big. \;\;\;\;\;\;\;\;\;\;\;\;\;\;\;\;\;\;\;\;\;\;\;\;\;\;\;\;\;+ \sup_{0<\delta<\delta_{0}}\frac{1}{\delta}\int_{\ball_{r+\delta}(x_{0})\setminus\overline{\ball}_{r}(x_{0})}|Du|\Big)^{q}, 
\end{align*}
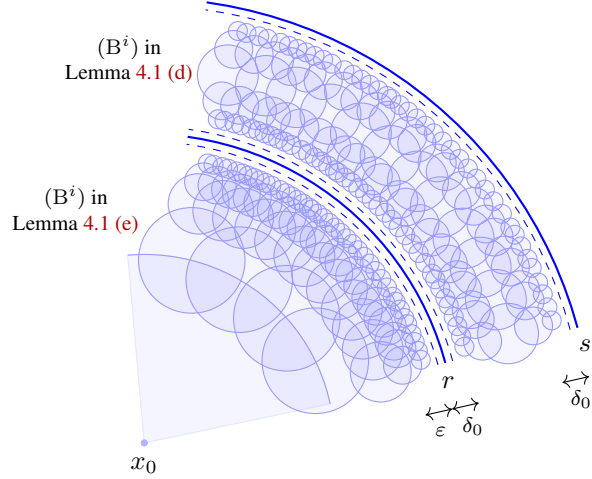
\begin{figure}
\begin{tikzpicture}[scale=1]
 \foreach \x in {0.22,0.32,0.4,0.49,0.57,0.648,0.72,0.78,0.84,0.885,0.925,0.957}
\draw[fill=blue!60!white, blue!40!white, fill opacity=0.1] (5*\x,{5*sqrt(1-\x*\x)}) circle (0.4cm);
\draw[blue,thick,domain=15:82] plot ({4.1*cos(\x)},{4.1*sin(\x)});
\draw[blue,thick,domain=15:82] plot ({5.9*cos(\x)},{5.9*sin(\x)});
\draw[blue,dashed,domain=15:82] plot ({4.2*cos(\x)},{4.2*sin(\x)});
\draw[blue,dashed,domain=15:82] plot ({5.8*cos(\x)},{5.8*sin(\x)});
\draw (4,0.8) node {$r$};
\draw (5.8,1.3) node {$s$};
  \foreach \x in {0.22,0.27,0.32,0.36,0.4,0.44,0.49,0.53,0.57,0.6,0.648,0.67,0.72,0.74,0.78,0.81, 0.84,0.86,0.885,0.9,0.925,0.936,0.957}
\draw[fill=blue!40!white, blue!40!white, fill opacity=0.1] (5.35*\x,{5.35*sqrt(1-\x*\x)}) circle (0.25cm);
 \foreach \x in {0.22,0.27,0.32,0.36,0.4,0.44,0.49,0.53,0.57,0.6,0.648,0.67,0.72,0.74,0.78,0.81, 0.84,0.86,0.885,0.9,0.925,0.936,0.957}
\draw[fill=blue!40!white, blue!40!white, fill opacity=0.1] (4.65*\x,{4.65*sqrt(1-\x*\x)}) circle (0.25cm);
\foreach \x in {0.22,0.24,0.27,0.29,0.32,0.34,0.36,0.38,0.4,0.42,0.44,0.47,0.49,0.51, 0.53,0.55,0.57,0.58, 0.60,0.62,0.64,0.66,0.68,0.7,0.72,0.74,0.76,0.78,0.8,0.82,0.84,0.86,0.87,0.885,0.9,0.91,0.925,0.936,0.947,0.957}
\draw[fill=blue!40!white, blue!40!white, fill opacity=0.1] (5.6*\x,{5.6*sqrt(1-\x*\x)}) circle (0.125cm);
\foreach \x in {0.22,0.24,0.27,0.29,0.32,0.34,0.36,0.38,0.4,0.42,0.44,0.47,0.49,0.51, 0.53,0.55,0.57,0.58, 0.60,0.62,0.64,0.66,0.68,0.7,0.72,0.74,0.76,0.78,0.8,0.82,0.84,0.86,0.87,0.885,0.9,0.91,0.925,0.936,0.947,0.957}
\draw[fill=blue!40!white, blue!40!white, fill opacity=0.1] (4.4*\x,{4.4*sqrt(1-\x*\x)}) circle (0.125cm);
\draw [blue!10!white, fill=blue!40!white, fill opacity=0.1](95:2.5)--(0,0)--(12:2.5) arc (12:95:2.5)--cycle;
\draw[blue!40!white,domain=12:95] plot ({2.5*cos(\x)},{2.5*sin(\x)});
 \foreach \x in {0.25,0.5,0.75,0.9}
 \draw[fill=blue!40!white, blue!40!white, fill opacity=0.1] (2.5*\x,{2.5*sqrt(1-\x*\x)}) circle (0.7cm);
 \foreach \x in {0.22,0.32,0.4,0.49,0.57,0.648,0.72,0.78,0.84,0.885,0.925,0.957}
\draw[fill=blue!60!white, blue!40!white, fill opacity=0.1] (3.25*\x,{3.25*sqrt(1-\x*\x)}) circle (0.4cm);
\foreach \x in {0.22,0.27,0.32,0.36,0.4,0.44,0.49,0.53,0.57,0.6,0.648,0.67,0.72,0.74,0.78,0.81, 0.84,0.86,0.885,0.9,0.925,0.936,0.957}
\draw[fill=blue!40!white, blue!40!white, fill opacity=0.1] (3.55*\x,{3.55*sqrt(1-\x*\x)}) circle (0.25cm);
\foreach \x in {0.22,0.24,0.27,0.29,0.32,0.34,0.36,0.38,0.4,0.42,0.44,0.47,0.49,0.51, 0.53,0.55,0.57,0.58, 0.60,0.62,0.64,0.66,0.68,0.7,0.72,0.74,0.76,0.78,0.8,0.82,0.84,0.86,0.87,0.885,0.9,0.91,0.925,0.936,0.947,0.957}
\draw[fill=blue!40!white, blue!40!white, fill opacity=0.1] (3.8*\x,{3.8*sqrt(1-\x*\x)}) circle (0.125cm);
\draw[blue,dashed,domain=15:82] plot ({4*cos(\x)},{4*sin(\x)});
\draw[fill=blue!30!white,blue!30!white] (0,0) circle (0.04cm);
\draw (0,-0.3) node {$x_{0}$};
\draw[<->] (3.7,0.37) -- (4.05,0.45);
\draw[<->] (4.05,0.45) -- (4.4,0.53);
\draw[<->] (5.5,0.8) -- (5.85,0.88);
\node at (3.9,0.15) {{\footnotesize{$\varepsilon$}}};
\node at (4.325,0.22) {{\footnotesize{$\delta_{0}$}}};
\node at (5.75,0.58) {{\footnotesize{$\delta_{0}$}}};
\node at (-0.2,5.25) {{\footnotesize{$(\ball^{i})$ \text{in}}}};
\node at (-0.2,4.9) {{\footnotesize{\text{Lemma~\ref{lem:extensionoperator}~\ref{item:extension4}}}}};
\node at (-0.9,3.25) {{\footnotesize{$(\ball^{i})$ \text{in}}}};
\node at (-0.9,2.9) {{\footnotesize{\text{Lemma~\ref{lem:extensionoperator}~\ref{item:extension5}}}}};
\end{tikzpicture}
\caption{Schematic display of the Whitney coverings underlying the $\lebe^{1}$- or $\lebe^{q}$-(gradient) bounds of Lemma~\ref{lem:extensionoperator}~\ref{item:extension4} (annular region) and~\ref{item:extension5} (inner ball around $x_{0}$). Balls close to $\partial\!\ball_{r}(x_{0})$ (at distance $\delta_{0}$ or $\varepsilon$, respectively), are controlled by the spherical maximal condition on $|Du|$ whereas balls far away from the boundary are controlled by stability estimates.} 
\end{figure}
where $c(n,N,\mathtt{N},q,\Lambda)>0$ is a constant. Note that it is precisely at the ultimate chain of inequalities where we require the assumption $1\leq q <\frac{n}{n-1}$ as otherwise the geometric series in the penultimate line does not converge. Passing to the $q$-th root on both sides, $\eqref{eq:fundamental}_{2}$ follows. Ad~\ref{item:extension5}. Here, $\mathfrak{U}=\ball_{r}(x_{0})$. Similarly as for~\ref{item:extension4}, we define for $m\in\mathbb{N}_{0}$ 
\begin{align*}
\mathcal{I}_{m}:=\{i\in\mathbb{N}\colon\;2^{-m-1}r\leq r_{i}< 2^{-m}r\}
\end{align*} 
so that by~\ref{item:Whitney2} and~\ref{item:Whitney3} there exists $c''=c''(n)>1$ such that with 
\begin{align*}
S_{m}:=\{x\in\mathfrak{U}\colon\;\frac{1}{c''}2^{-m}r<\dista(x,\partial\mathfrak{U})\leq c''2^{-m}r\}
\end{align*}
we have $\bigcup_{i\in\mathcal{I}_{m}}\Lambda\!\ball^{i}\subset S_{m}$ for all $m\in\mathbb{N}_{0}$. We choose the minimal $m_{0}\in\mathbb{N}$ such that $\bigcup_{m=m_{0}+1}^{\infty}\bigcup_{i\in\mathcal{I}_{m}}\Lambda\!\ball^{i}\subset\mathfrak{U}_{\varepsilon}^{\complement}$, so that $2^{-m_{0}}r$ is uniformly proportional to $\varepsilon$, so $\frac{1}{c}\varepsilon\leq 2^{-m_{0}}r\leq c\varepsilon$ for some $c=c(n)>1$. We then split 
\begin{align*}
\int_{\ball_{r}(x_{0})}|\nabla\widetilde{\mathbb{E}}_{\mathfrak{U}}u|^{q}\dif x \leq c\Big(\sum_{m=0}^{m_{0}}\sum_{i\in\mathcal{I}_{m}}+\sum_{m=m_{0}+1}^{\infty}\sum_{i\in\mathcal{I}_{m}}\Big)\int_{\ball^{i}}|\nabla\widetilde{\mathbb{E}}_{\mathfrak{U}}u|^{q}\dif x =: \mathrm{I}+\mathrm{II}. 
\end{align*}
For both terms we may adapt the estimation in the proof of~\ref{item:extension4}, and we only indicate the modifications. For term $\mathrm{I}$, we obtain similarly as above 
\begin{align*}
\mathrm{I} & \leq c\Big(\sum_{m=0}^{m_{0}}(2^{-m}r)^{n\frac{1-q}{q}}(2^{-m}r)\Big(\frac{2^{m}}{r}\sum_{i\in\mathcal{I}_{m}}\int_{\Lambda\!\ball^{i}}|Du| \Big) \Big)^{q} \\ 
& \!\!\!\!\!\!\! \!\!\!\!\!\!\!\!\!\stackrel{m\leq m_{0},\,2^{-m_{0}}r\sim\varepsilon}{\leq}c\Big(\sum_{m=0}^{m_{0}}(2^{-m})^{n\frac{1-q}{q}+1}\Big)^{q}\Big(\int_{\ball_{r}(x_{0})}|Du|\Big)^{q}\frac{r^{n(1-q)+q}}{\varepsilon^{q}}, 
\end{align*}
and since the final sum converges, we obtain the first term on the right hand side of~\eqref{eq:fundamental1}. For the second term, we may follow the same argument as in~\ref{item:extension4}, now invoking $S_{m}$ instead of $S^{m}$. Adding both estimates consequently yields~\ref{item:extension5}. The proof is complete. 
\end{proof} 
To set up the comparison argument for the partial regularity proof later on, we require a Fubini-type property of $\bv$-maps. In \cite{GK1} we approached this matter by the fact that for $u\in\bv_{\locc}(\R^{n};\R^{N})$ the restrictions $u|_{\partial\!\ball_{r}(x_{0})}$ belong to $\bv(\partial\!\ball_{r}(x_{0});\R^{N})$ for $\mathscr{L}^{1}$-a.e. $r>0$; if $n=2$, the failure of the sharp fractional Sobolev embedding of $\bv(\partial\!\ball_{r}(x_{0});\R^{N})$ into $\sobo^{s,p}(\partial\!\ball_{r}(x_{0});\R^{N})$ for $sp=1$ then necessitates an argument via intermediate Nikolski\u{\i} estimates and re-embedding the Nikolski\u{\i}- into fractional Sobolev spaces. We thus pause and give a self-contained proof of the following scaled Fubini-type theorem, which solely hinges on the trace-preserving extension operator from Lemma~\ref{lem:extensionoperator} and is applicable to all $n\geq 2$:
\begin{corollary}[of Fubini-type]\label{cor:Fubini}
Let $n\geq 2$ and $x_{0}\in\R^{n}$. Then the following hold: 
\begin{enumerate}
\item\label{item:Fubini1} Let $1<\vartheta<\frac{n}{n-1}$. There exists a constant $C=C(n,N,\vartheta)>0$ such that the following holds for all $u\in\bv_{\locc}(\R^{n};\R^{N})$: For every $R_{0}>0$ and every $N\subset(\frac{9}{10}R_{0},R_{0})$ with $\mathscr{L}^{1}(N)=0$ there exists $R\in (\frac{9}{10}R_{0},R_{0})\setminus N$ such that $\mathcal{M}Du(x_{0},R)<\infty$, $\trace_{\partial\!\ball_{R}(x_{0})}u\in\sobo^{1-1/\vartheta,\vartheta}(\partial\!\ball_{R}(x_{0});\R^{N})$, and we have  
\begin{align}\label{eq:scalebound2}
\begin{split}
\Big(\dashint_{\partial\!\ball_{R}(x_{0})}\int_{\partial\!\ball_{R}(x_{0})}\frac{|\trace_{\partial\!\ball_{R}(x_{0})}(u)(x)-\trace_{\partial\!\ball_{R}(x_{0})}(u)(y)|^{\vartheta}}{|x-y|^{n-2+\vartheta}}\dif^{n-1}x\dif^{n-1}y\Big)^{\frac{1}{\vartheta}} \\
\;\;\;\;\;\;\;\;\;\;\;\;\;\;\;\;\;\;\;\;\;\;\;\;\;\;\;\;\;\;\;\;\;\;\;\;\;\;\;\;\;\;\;\;\;\;\;\;\;\;\;\;\;\;\;\;\;\;\;\;\;\;\;\;\;\leq CR_{0}^{\frac{1}{\vartheta}}\dashint_{\overline{\ball}_{R_{0}}(x_{0})}|Du|.
\end{split}
\end{align}
\item\label{item:Fubini2} Let $1<p<\infty$ and $p<\vartheta_{p}<\frac{np}{n-1}$. There exists a constant $C=C(n,N,p,\vartheta_{p})>0$ such that the following hold for all $u\in\sobo_{\locc}^{1,p}(\R^{n};\R^{N})$: For every $R_{0}>0$ and every $N\subset(\frac{9}{10}R_{0},R_{0})$ with $\mathscr{L}^{1}(N)=0$ there exists $R\in (\frac{9}{10}R_{0},R_{0})\setminus N$ such that $\mathcal{M}(|\nabla u|^{p}\mathscr{L}^{n})(x_{0},R)<\infty$, $\trace_{\partial\!\ball_{R}(x_{0})}u\in\sobo^{1-1/\vartheta_{p},\vartheta_{p}}(\partial\!\ball_{R}(x_{0});\R^{N})$, and we have  
\begin{align}\label{eq:scalebound2A}
\begin{split}
\Big(\dashint_{\partial\!\ball_{R}(x_{0})}\int_{\partial\!\ball_{R}(x_{0})}\frac{|\trace_{\partial\!\ball_{R}(x_{0})}(u)(x)-\trace_{\partial\!\ball_{R}(x_{0})}(u)(y)|^{\vartheta_{p}}}{|x-y|^{n-2+\vartheta_{p}}}\dif^{n-1}x\dif^{n-1}y\Big)^{\frac{1}{\vartheta_{p}}} \\
\;\;\;\;\;\;\;\;\;\;\;\;\;\;\;\;\;\;\;\;\;\;\;\;\;\;\;\;\;\;\;\;\;\;\;\;\;\;\;\;\;\;\;\;\;\;\;\;\;\;\;\;\;\;\;\;\;\;\;\;\;\;\;\;\;\leq CR_{0}^{\frac{1}{\vartheta_{p}}}\Big(\dashint_{\ball_{R_{0}}(x_{0})}|\nabla u|^{p}\dif x\Big)^{\frac{1}{p}}.
\end{split}
\end{align}
\end{enumerate}
Specifically, whenever we have $\mathcal{M}Du(x_{0},r)<\infty$ in the setting of~\ref{item:Fubini1} or $\mathcal{M}(|\nabla u|^{p}\mathscr{L}^{n})(x_{0},r)<\infty$ in the setting of~\ref{item:Fubini2}, then $\trace_{\partial\!\ball_{r}(x_{0})}(u)\in\sobo^{1-1/q,q}(\partial\!\ball_{r}(x_{0});\R^{N})$. 
\end{corollary}
\begin{proof}
Ad~\ref{item:Fubini1}. Let $x_{0}\in\R^{n}$. We write $\ball_{r}:=\ball_{r}(x_{0})$ in the sequel and put $w:=u-(u)_{\ball_{R_{0}}}$. Fix an exponent $1<\vartheta<\frac{n}{n-1}$. Because of $Dw\in\mathrm{RM}(\R^{n};\R^{N\times n})$, the set $E:=\{t\in (\frac{9}{10}R_{0},R_{0})\colon \mathcal{M}Dw(x_{0},t)=\infty\}\cup N$ is negligible for $\mathscr{L}^{1}$. For any $t\in (\frac{9}{10}R_{0},R_{0})\setminus E$,~\eqref{eq:traceequal} and~\eqref{eq:interiortraces1} imply that $\mathscr{H}^{n-1}$-a.e. $x\in\partial\!\ball_{t}$ is an $\mathscr{L}^{n}$-Lebesgue point for $w$ and the single-sided traces $w^{+}:=\trace_{\partial\!\ball_{t}}^{+}(w)$ and $w^{-}:=\trace_{\partial\!\ball_{t}}^{-}(w)$ along $\partial\!\ball_{t}$ coincide and equal $u^{+}-(u)_{\ball_{R_{0}}}$; thus it suffices to establish estimate \eqref{eq:scalebound2} for $w^{+}$ along $\partial\!\ball_{R}$ for some suitable $R$. 

We apply Lemma~\ref{lem:HLW} to $r=\frac{9}{10}R_{0}$, $s=R_{0}$ and $E$ as above. We then find $\frac{9}{10}R_{0}<\widetilde{r}<\widetilde{s}<R_{0}$ such that $\widetilde{r},\widetilde{s}\notin E$ and, with $\Theta(t):=\tfrac{1}{t}\|w\|_{\lebe^{1}(\ball_{t})}+|Dw|(\overline{\ball}_{t})$ for $t>0$, 
\begin{align}\label{eq:Fubinigoodradius}
\begin{split}
&\frac{\Theta(\widetilde{s})-\Theta(\tau)}{\widetilde{s}-\tau} \leq 8000 \frac{\Theta(R_{0})-\Theta(\tfrac{9}{10}R_{0})}{R_{0}}\qquad\text{for all}\; \tfrac{9}{10}R_{0}<\tau<\widetilde{s},\\
&\frac{\Theta(\tau)-\Theta(\widetilde{r})}{\tau-\widetilde{r}} \leq 8000 \frac{\Theta(R_{0})-\Theta(\tfrac{9}{10}R_{0})}{R_{0}}\qquad\text{for all}\; \widetilde{r}<\tau<R_{0},\\
&\frac{1}{10}R_{0}=(s-r)\leq 8(\widetilde{s}-\widetilde{r}),\;\;\;\text{so}\;\;\;\frac{R_{0}}{80}\leq (\widetilde{s}-\widetilde{r}). 
\end{split}
\end{align}
Next note that $\eqref{eq:Fubinigoodradius}_{3}$ yields
\begin{align*}
\left\vert\Big(\widetilde{s}-\frac{R_{0}}{1000}\Big)-\Big(\widetilde{r}+\frac{R_{0}}{1000}\Big)\right\vert\geq \frac{21}{2000}R_{0}\geq \frac{R_{0}}{1000}, 
\end{align*}
and so we may choose $\rho\in\hold_{c}^{\infty}(\ball_{R_{0}};[0,1])$ with 
\begin{align}\label{eq:fubinicutoff}
\mathbbm{1}_{\ball_{\widetilde{r}+\frac{R_{0}}{1000}}}\leq \rho \leq \mathbbm{1}_{\ball_{\widetilde{s}-\frac{R_{0}}{1000}}}\;\;\;\text{and}\;\;\;|\nabla\rho|\leq \frac{2000}{R_{0}}.
\end{align}
We then put $v:=(1-\rho)\mathbb{E}w$, defined on $\ball_{\widetilde{s}}$, where $\mathbb{E}:=\mathbb{E}_{\ball_{\widetilde{s}}\setminus\overline{\ball}_{\widetilde{r}}}$ in the notation of Lemma~\ref{lem:extensionoperator}. By our above discussion and since $\E$ is trace-preserving, the interior trace of $w$ along $\partial\!\ball_{\widetilde{s}}$  then coincides with the boundary trace of $v$ on $\partial\!\ball_{\widetilde{s}}$. Next note that for any $0<\delta<\widetilde{s}-\widetilde{r}$ we have 
\begin{align}\label{eq:pinacolada}
\begin{split}
\frac{1}{\delta}\|w\|_{\lebe^{1}(\ball_{\widetilde{s}}\setminus\ball_{\widetilde{s}-\delta})} & = \frac{\widetilde{s}}{\delta}\Big(\frac{\|w\|_{\lebe^{1}(\ball_{\widetilde{s}})}}{\widetilde{s}}-\frac{\|w\|_{\lebe^{1}(\ball_{\widetilde{s}-\delta})}}{\widetilde{s}-\delta}\Big)+\frac{\|w\|_{\lebe^{1}(\ball_{\widetilde{s}-\delta})}}{\widetilde{s}-\delta}\\ 
& \!\!\!\!\!\!\!\!\!\!\stackrel{\tfrac{9}{10}R_{0}<\widetilde{s}-\delta}{\leq} \widetilde{s}\frac{\Theta(\widetilde{s})-\Theta(\widetilde{s}-\delta)}{\delta}+\frac{10}{9}\frac{\|w\|_{\lebe^{1}(\ball_{R_{0}})}}{R_{0}} \\ 
& \!\!\!\!\!\!\!\!\!\!\!\stackrel{\widetilde{s}\leq R_{0},~\eqref{eq:Fubinigoodradius}_{1}}{\leq} 8000 R_{0}\frac{\Theta(R_{0})-\Theta(\tfrac{9}{10}R_{0})}{R_{0}}+\frac{10}{9}\frac{\|w\|_{\lebe^{1}(\ball_{R_{0}})}}{R_{0}}\\
& \leq 9000\,\Theta(R_{0})
\end{split}
\end{align}
and similarly 
\begin{align}\label{eq:pinacolada1}
\frac{1}{\delta}\|w\|_{\lebe^{1}(\ball_{\widetilde{r}+\delta}\setminus\ball_{\widetilde{r}})} \leq 9000\,\Theta(R_{0}).
\end{align}
Then, by definition of $v$ and $\frac{9}{10}R_{0}<\widetilde{s}<R_{0}$,  
\begin{align*}
\Big(&\dashint_{\ball_{\widetilde{s}}}|v|^{\vartheta}\dif x\Big)^{\frac{1}{\vartheta}} \leq c(n,\vartheta)R_{0}^{-\frac{n}{\vartheta}}\Big(\int_{\ball_{\widetilde{s}}\setminus\ball_{\widetilde{r}}}|\E w|^{\vartheta}\dif x\Big)^{\frac{1}{\vartheta}}\\ 
& \!\!\!\!\!\!\!\!\!\stackrel{\text{Lem.}~\ref{lem:extensionoperator}\ref{item:extension4}}{\leq}c(n,\vartheta)R_{0}^{-\frac{n}{\vartheta}}R_{0}^{\frac{n}{\vartheta}-n+1}\Big(\sup_{0<\delta<\widetilde{s}-\widetilde{r}}\frac{1}{\delta}\int_{\ball_{\widetilde{s}}\setminus\ball_{\widetilde{s}-\delta}}|w|\dif x + \frac{1}{\delta}\int_{\ball_{\widetilde{r}+\delta}\setminus\ball_{\widetilde{r}}}|w|\dif x \Big)\\
& \!\!\!\!\!\!\!\!\!\stackrel{\eqref{eq:pinacolada},\,\eqref{eq:pinacolada1}}{\leq}c(n,\vartheta)R_{0}^{-n+1}\times\Theta(R_{0})\leq c(n,\vartheta)\Big(\dashint_{\ball_{R_{0}}}|w|\dif x + R_{0}\dashint_{\overline{\ball}_{R_{0}}}|Dw|\Big). 
\end{align*}
Similarly, now invoking Lemma~\ref{lem:extensionoperator}~\ref{item:extension4} on the gradient level and recalling \eqref{eq:fubinicutoff}, $\frac{9}{10}R_{0}<\widetilde{s}<R_{0}$ gives
\begin{align*}
\widetilde{s}\Big(\dashint_{\ball_{\widetilde{s}}}|\nabla v|^{\vartheta}\dif x\Big)^{\frac{1}{\vartheta}} \leq c(n,\vartheta)\Big(\dashint_{\ball_{R_{0}}}|w|\dif x + R_{0}\dashint_{\overline{\ball}_{R_{0}}}|Dw|\Big). 
\end{align*}
Put $R:=\widetilde{s}$. Then the previous two inequalities combine to $v\in\sobo^{1,\vartheta}(\ball_{\widetilde{s}};\R^{N})$, and since $v$ attains the same traces along $\partial\!\ball_{\widetilde{s}}(x_{0})$ as $w=u-(u)_{\ball_{R_{0}}}$, $u^{+}\in\sobo^{1-\frac{1}{\vartheta},\vartheta}(\partial\!\ball_{R};\R^{N})$. Moreover, the scaled trace inequality \eqref{eq:tracescale} and the preceding two inequalities yield 
\begin{align*}
R_{0}^{1-\frac{1}{\vartheta}}\Big(\dashint_{\partial\!\ball_{R}}\int_{\partial\!\ball_{R}}\frac{|u^{+}(x)-u^{+}(y)|^{\vartheta}}{|x-y|^{n-2+\vartheta}}\dif^{n-1}x\dif^{n-1}y\Big)^{\frac{1}{\vartheta}}\leq c(n,\vartheta)R_{0}\dashint_{\overline{\ball}_{R_{0}}}|Du|
\end{align*}
by virtue of $\frac{9}{10}R_{0}<R<R_{0}$ and Poincar\'{e}'s inequality on $\bv(\ball_{R_{0}};\R^{N})$. This is \eqref{eq:scalebound2}, and the proof of the first part is complete. Lastly, if one only aims at the ultimate, purely qualitative statement of the corollary in the setting of~\ref{item:Fubini1}, one chooses any $0<r'<r$ with $\mathcal{M}Du(x_{0},r')<\infty$ and considers $\mathbb{E}_{\ball_{r}\setminus\overline{\ball}_{r'}}u$. As above, we conclude that $\mathbb{E}_{\ball_{r}\setminus\overline{\ball}_{r'}}u\in\sobo^{1,q}(\ball_{r}\setminus\overline{\ball}_{r'};\R^{N})$ and since $\mathbb{E}$ is trace-preserving, $\trace_{\partial\!\ball_{r}}(u)\in\sobo^{1-1/q,q}(\partial\!\ball_{r};\R^{N})$ follows at once. Ad~\ref{item:Fubini2}. For the proof in the $\sobo^{1,p}$-setting, one may argue analogously, now using the function $\Theta(t):=\tfrac{1}{t^{p}}\|w\|_{\lebe^{p}(\ball_{t})}^{p}+\|\nabla w\|_{\lebe^{p}(\ball_{t})}^{p}$. This completes the proof. 
\end{proof}
\section{The good generation theorem}\label{sec:goodgeneration}
In this section, we address the existence of suitable generating sequences that have fixed traces along certain spheres. As one of the main differences to \cite{SchmidtPR}, we directly deal with the signed situation. This particularly forces us to rule out certain concentration effects that are a priori excluded in the unsigned case, also see Remark~\ref{rem:signed}. 
\subsection{A lower semicontinuity result}\label{sec:LSCsimple} We begin by recording the following lower semicontinuity result which appears as a special case of \cite[Thm.~5.1]{CK} by \textsc{Chen} and the second author: 
\begin{lemma}\label{lem:LSC}
Let $1\leq p \leq q <\frac{np}{n-1}$ and suppose that $F\in\hold(\R^{N\times n})$ is \emph{quasiconvex} and satisfies \emph{\ref{item:H1}}. Given any open and bounded subset $\Omega$ of $\R^{n}$ with Lipschitz boundary and $u_{0}\in\sobo^{1,q}(\Omega;\R^{N})$, the following holds: If $u,u_{1},u_{2},...\in\sobo_{u_{0}}^{1,q}(\Omega;\R^{N})$ satisfy $u_{j}\rightharpoonup u$ in $\sobo^{1,p}(\Omega;\R^{N})$ if $p>1$ or $u_{j}\stackrel{*}{\rightharpoonup} u$ in $\BV(\Omega;\R^{N})$ if $p=1$ (or $u_{j}\to u$ strongly in $\lebe^{p}(\Omega;\R^{N})$ and $\sup_{j\in\mathbb{N}}\|\nabla u_{j}\|_{\lebe^{p}(\Omega)}<\infty$), then 
\begin{align*}
\int_{\Omega}F(\nabla u)\dif x \leq \liminf_{j\to\infty}\int_{\Omega}F(\nabla u_{j})\dif x.
\end{align*}
\end{lemma}
As can be seen from the classical example~$F(z)=\det(z)$ for $z\in\R^{2\times 2}$, which satisfies the hypotheses of the preceding lemma with  $p=q=2$ (cf. \textsc{Ball \& Murat}~\cite{BM} and \textsc{Dacorogna}~\cite[Rem.~8.5~(c),~Ex.~8.6]{Dacorogna}), the assumption $u,u_{1},u_{2},...\in\sobo_{u_{0}}^{1,q}(\Omega;\R^{N})$ is essential and \emph{cannot} be improved to $u,u_{1},u_{2},...\in\sobo^{1,q}(\Omega;\R^{N})$ as in the unsigned case. If $p=1$, a corresponding example can be obtained by easier means:
\begin{example}[Concentration and semicontinuity]\label{ex:introduction}
Let $1\leq q<\infty$ and define the convex $(1,q)$-growth integrand $F=F_{q}\colon\R\to\R$ by 
\begin{align}
F_{q}(t):=\begin{cases}
\frac{1}{q}(t+1)^{q}-\frac{1}{q}&\;\;\;\text{for}\;t\geq 0,\\ 
t&\;\;\;\;\text{for}\;t<0. 
\end{cases}
\end{align}
For $\Omega=(0,1)$, the sequence $(v_{j})\subset\sobo^{1,q}(\Omega)$ which is defined by $v_{j}(x):=-jx+1$ for $0<x<\frac{1}{j}$ and $v_{j}(x)=0$ otherwise, satisfies $v_{j}\stackrel{*}{\rightharpoonup} u \equiv 0\in\sobo^{1,q}(\Omega)$ in $\bv(\Omega)$ but 
\begin{align*}
\liminf_{j\to\infty}\int_{0}^{1}F(\nabla v_{j})\dif x = -1 < 0 = \int_{0}^{1}F(\nabla u)\dif x.
\end{align*}
\end{example}
\subsection{The good generation theorem}\label{sec:goodgeneration1}
For our subsequent arguments it is necessary to be able to work with \emph{good} generating sequences. By good sequences we here understand that each member attains fixed interior traces along the boundaries of suitable balls. Our approach is strongly inspired by that of \cite{SchmidtPR}, however, now facing the difficulty that we allow for signed integrands and thereby entailing a different overall set-up of the proof (also see Remark~\ref{rem:signed} below). The following proposition ensures the existence of such sequences:
\begin{proposition}\label{prop:goodminseqs} Let $1\leq q <\frac{n}{n-1}$ and let $F\in\hold(\R^{N\times n})$ be quasiconvex with \ref{item:H1}. Let $\omega\Subset\omega'\subset\R^{n}$ be two open and bounded sets with Lipschitz boundaries and let $u_{0}\in\sobo^{1,q}(\omega';\R^{N})$ be given. Finally, let $u\in\bv(\omega';\R^{N})$ satisfy $\overline{\mathscr{F}}_{u_{0}}^{*}[u;\omega,\omega']<\infty$.

Whenever $(v_{j})\subset\mathscr{A}_{u_{0}}^{q}(\omega,\omega')$ is a generating sequence for $\overline{\mathscr{F}}_{u_{0}}^{*}[u;\omega,\omega']$ and $x_{0}\in\omega$, then for $\mathscr{L}^{1}$-a.e. $r\in (0,\mathrm{dist}(x_{0},\partial\omega))$ satisfying 
\begin{align}\label{eq:maxcond}
\mathcal{M}Du(x_{0},r)<\infty
\end{align}
with the maximal operator $\mathcal{M}$ from \eqref{eq:maxoperator}, there exists a subsequence $(v_{j_{k}})$ and another \emph{generating sequence} $(u_{j_{k}})\subset\mathscr{A}_{u_{0}}^{q}(\omega,\omega')$ such that the following hold: 
\begin{enumerate}
\item\label{item:FixRadius1} $\mathscr{H}^{n-1}(J_{u}\cap\partial\!\ball_{r}(x_{0}))=0$. 
\item\label{item:FixRadius2} For all $k\in\mathbb{N}$ we have $\trace_{\partial\!\ball_{r}(x_{0})}(u_{j_{k}})=\trace_{\partial\!\ball_{r}(x_{0})}(u)$ $\mathscr{H}^{n-1}$-a.e. on $\partial\!\ball_{r}(x_{0})$.
\item\label{item:FixRadius3} $\nabla v_{j_{k}} - \nabla u_{j_{k}}\to 0$ in $\lebe^{1}(\omega;\R^{N\times n})$ as $k\to\infty$. 
\end{enumerate}
\end{proposition}
We shall refer to the previous proposition as \emph{good generation theorem}. As will become visible from the proof, whenever $(v_{j})\subset\mathscr{A}_{u_{0}}^{q}(\omega,\omega')$ is a generating sequence such that $(|Dv_{j}|)$ converges in the weak*-sense to some $\lambda\in\mathrm{RM}_{\mathrm{fin}}(\omega')$, the previous proposition is applicable to any $x_{0}\in\omega$ and $r>0$ such that $0<r<\mathrm{dist}(x_{0},\partial\omega)$ and $\mathcal{M}\lambda(x_{0},r)+\mathcal{M}Du(x_{0},r)<\infty$. Any such $r$ shall be referred to as \emph{good radius (of $Du$ for $\overline{\mathscr{F}}_{u_{0}}^{*}[u;\omega,\omega']$) at $x_{0}$}. 
\begin{proof}
We split the proof into five steps. 

\emph{Step 1. Set-up.} Since $(|Dv_{j}|(\omega'))$ is bounded, weak*-compactness in $\mathrm{RM}_{\mathrm{fin}}(\omega')$ implies the existence of a (here non-relabeled) subsequence and $\lambda\in\mathrm{RM}_{\mathrm{fin}}(\omega')$ such that $|Dv_{j}|\stackrel{*}{\rightharpoonup}\lambda$. Moreover, it is no loss of generality to assume that 
\begin{align}\label{eq:liminfislim}
\lim_{j\to\infty}\int_{\omega'}F(\nabla v_{j})\dif x = \overline{\mathscr{F}}_{u_{0}}^{*}[u;\omega,\omega']. 
\end{align}
For $\mathscr{L}^{1}$-a.e. $0<r<\mathrm{dist}(x_{0},\partial\omega)$, we have
\begin{align}\label{eq:maxcondlambda1}
\mathcal{M}\lambda(x_{0},r)<\infty
\end{align}
by the results from Section~\ref{sec:HLW}, and so $\mathscr{L}^{1}$-a.e. such radius satisfies $\mathcal{M}Du(x_{0},r)<\infty$ by lower semicontinuity of the total variation with respect to $\lebe_{\locc}^{1}$-convergence (recall that $u_{j}\to u$ strongly in $\lebe^{1}(\omega';\R^{N})$).  

Fix such a radius $r>0$; for notational brevity, we write $\ball_{r}:=\ball_{r}(x_{0})$ in the following. We fix $\ell>0$ so small such that the annulus $\ball_{r+\ell}\setminus\overline{\ball}_{r-\ell}$ is compactly contained in $\omega$ and 
\begin{align}\label{eq:maxcondo1}
\mathcal{M}Du(x_{0},r+\ell),\mathcal{M}Du(x_{0},r-\ell)<\infty,
\end{align}
which is possible by the results from Section~\ref{sec:HLW}. We then choose $j_{0}\in\mathbb{N}$ so large such that 
\begin{align*}
\ball_{r+2^{-j_{0}}}\setminus\overline{\ball}_{r-2^{-j_{0}}}\Subset \ball_{r+\ell}\setminus\overline{\ball}_{r-\ell}.
\end{align*}
We have $\lim_{j\to\infty}|Dv_{j}|(K)=\lambda(K)$ for each compact subset $K$ of $\omega$. By the Rellich-Kondrachov theorem, we thus find by iteration that for each $k\in\mathbb{N}_{\geq 2}$ there exists $j_{k}\in\mathbb{N}$ such that 
\begin{align}\label{eq:RecoveryClosedness}
\begin{split}
& \left\vert\int_{\overline{\ball}_{r+2^{-j_{0}-k+1}}\setminus{\ball}_{r-2^{-j_{0}-k+1}}}|\nabla v_{j_{k}}|\dif x  - \int_{\overline{\ball}_{r+2^{-j_{0}-k+1}}\setminus{\ball}_{r-2^{-j_{0}-k+1}}}\dif\lambda\right\vert \\ &  \;\;\;\;\;\;\;\;\;\;\;\;\;\;\;\;\;\;\;\;\;\;\;\;\;\;\;\;\;\;\;\;\;\;\;\;+ \int_{\overline{\ball}_{r+2^{-j_{0}-k+1}}\setminus{\ball}_{r-2^{-j_{0}-k+1}}}|v_{j_{k}}-u|\dif x < 2^{-2k},
\end{split}
\end{align}
and clearly we may assume that $(j_{k})$ is increasing. Therefore, we obtain for all $k\in\mathbb{N}$: 
\begin{align}\label{eq:MaximalBoundA}
2^{j_{0}+k-2}\int_{\overline{\ball}_{r+2^{-j_{0}-k+1}}\setminus\ball_{r-2^{-j_{0}-k+1}}}|\nabla v_{j_{k}}|\dif x \leq 2\mathcal{M}\lambda(x_{0},r) + 2^{j_{0}-k-2}.
\end{align}
Since $(v_{j_{k}})$ is generating as well, for the statement of the proposition it suffices to consider this subsequence.

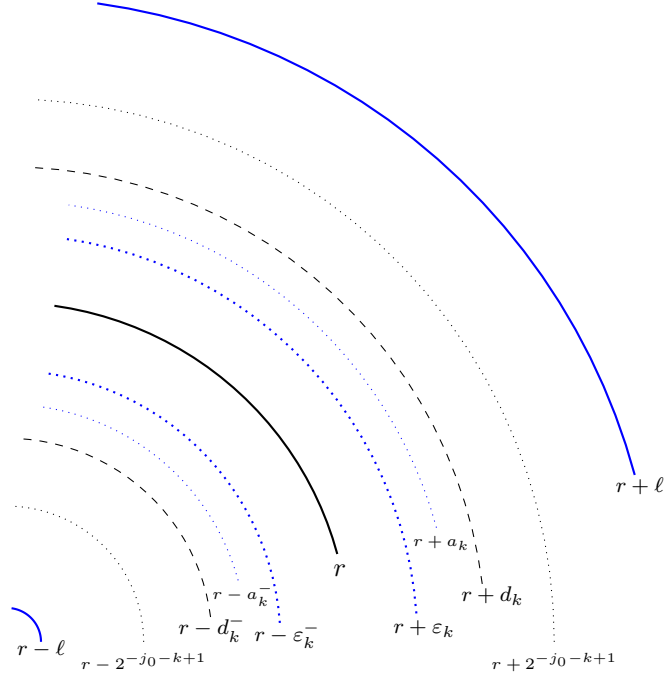
\begin{figure}
\begin{center}
\begin{tikzpicture}[scale=0.9]
\draw[black,dashed,domain=7:87] plot ({3*cos(\x)},{3*sin(\x)});
\draw[black,dotted,domain=0:87] plot ({2*cos(\x)},{2*sin(\x)});
\draw[black,dashed,domain=7:87] plot ({7*cos(\x)},{7*sin(\x)});
\draw[black,dotted,domain=0:87] plot ({8*cos(\x)},{8*sin(\x)});
\draw[black,thick,domain=15:82] plot ({5*cos(\x)},{5*sin(\x)});
\draw[blue,thick,domain=-0:82] plot ({0.5*cos(\x)},{0.5*sin(\x)});
\draw[blue,dotted,domain=15:82] plot ({3.5*cos(\x)},{3.5*sin(\x)});
\draw[blue,thick,dotted,domain=4:82] plot ({4*cos(\x)},{4*sin(\x)});
\draw[blue,thick,dotted,domain=4:82] plot ({6*cos(\x)},{6*sin(\x)});
\draw[blue,dotted,domain=15:82] plot ({6.5*cos(\x)},{6.5*sin(\x)});
\draw[blue,thick,domain=15:82] plot ({9.5*cos(\x)},{9.5*sin(\x)});
\draw (3,0.2) node {\footnotesize $r-d_{k}^{-}$};
\draw (2,-0.3) node {\tiny $r-2^{-j_{0}-k+1}$};
\draw (4.1,0.1) node {\footnotesize $r-\varepsilon_{k}^{-}$};
\draw (6.1,0.2) node {\footnotesize $r+\varepsilon_{k}$};
\draw (3.45,0.7) node {{\tiny $r-a_{k}^{-}$}};
\draw (6.35,1.46) node {\tiny $r+a_{k}$};
\draw (7.1,0.68) node {\footnotesize $r+d_{k}$};
\draw (8,-0.3) node {\tiny $r+2^{-j_{0}-k+1}$};
\draw (9.25,2.3) node {\footnotesize $r+\ell$};
\draw (4.87,1.05) node {$r$};
\draw (0.5,-0.1) node {\footnotesize $r-\ell$};
\end{tikzpicture}
\caption{Radius notation and locations in the proof of Proposition~\ref{prop:goodminseqs}.}\label{fig:radiilocations1}
\end{center}
\end{figure}

By \eqref{eq:maxcondo1} and since $\mathcal{M}Du(x_{0},r)<\infty$, we first conclude by Lemma~\ref{lem:extensionoperator} that $\mathbb{E}_{r-\ell,r}u\in\sobo^{1,q}(\ball_{r}\setminus\overline{\ball}_{r-\ell};\R^{N})$ and $\mathbb{E}_{r,r+\ell}u\in\sobo^{1,q}(\ball_{r+\ell}\setminus\overline{\ball}_{r};\R^{N})$. Moreover, by Lemma~\ref{lem:extensionoperator}, the traces of $\mathbb{E}_{r-\ell,r}u$ and $\mathbb{E}_{r,r+\ell}u$ along $\partial\!\ball_{r}(x_{0})$ coincide with $\trace_{\partial\!\ball_{r}}^{+}(u)$ or $\trace_{\partial\!\ball_{r}}^{-}(u)$, respectively, $\mathscr{H}^{n-1}$-a.e. on $\partial\!\ball_{r}(x_{0})$. Because of~$\mathcal{M}Du(x_{0},r)<\infty$, however,~\eqref{eq:traceequal} implies that $\trace_{\partial\!\ball_{r}}^{+}(u)$ and $\trace_{\partial\!\ball_{r}}^{-}(u)$ coincide $\mathscr{H}^{n-1}$-a.e. on $\partial\!\ball_{r}(x_{0})$, directly yielding~\ref{item:FixRadius1} by~\eqref{eq:interiortraces1}, and this implies that the map defined by
\begin{align}\label{eq:mathbfudef}
\mathbf{u} :=\begin{cases} 
\mathbb{E}_{r-\ell,r}u&\;\text{in}\;\ball_{r}\setminus\overline{\ball}_{r-\ell}, \\ 
\mathbb{E}_{r,r+\ell}u&\;\text{in}\;\ball_{r+\ell}\setminus\overline{\ball}_{r}
\end{cases}
\end{align}
belongs to $\sobo^{1,q}(\ball_{r+\ell}\setminus\overline{\ball}_{r-\ell};\R^{N})$.

 We recall that the set 
\begin{align*}
N:=\{s>0\colon\;\ball_{s}\subset\omega,\;\mathcal{M}Du(x_{0},s)=\infty\}
\end{align*}
satisfies $\mathscr{L}^{1}(N)=0$. In the sequel, we will employ the auxiliary functions
\begin{align}\label{eq:ThetajdefChapter30}
\widetilde{\Theta}_{k}(t) := \int_{\overline{\ball}_{t}}|Du| + \int_{\overline{\ball}_{t}}\left\vert\frac{v_{j_{k}}-\mathbf{u}}{2^{-k}}\right\vert\dif x + \int_{\overline{\ball}_{t}}|\nabla (v_{j_{k}}-\mathbf{u})|\dif x. 
\end{align}
Within the framework of Lemma~\ref{lem:HLW}, which we apply to the interval $\mathcal{I}_{k}=(r,r+2^{-j_{0}-k+1})$ and $E=(N\cap\mathcal{I}_{k})\cup (r,r+2^{-j_{0}-k-4})\cup (r+31\times 2^{-j_{0}-k-4},r+2^{-j_{0}-k+1})$ (so that $\mathscr{L}^{1}(E)=\frac{1}{16}\mathscr{L}^{1}(\mathcal{I}_{k})<\theta\mathscr{L}^{1}(\mathcal{I}_{k})$ with $\theta=\frac{3}{32}$), we find $d_{k}>0$ such that $r+d_{k}\in \mathcal{I}_{k}\setminus E$, 
\begin{align}
\label{eq:unifchooseD}
&\;\;\;\;\;\;\;\;\;\;\;\;\;\;\;\;\;\;\;\;\;\;\;\;\;\;\,\;\,d_{k}\sim_{j_{0}} 2^{-k}\;\text{and}\\ \begin{split} \label{eq:Choose1D}
&\frac{\widetilde{\Theta}_{k}(r+d_{k})-\widetilde{\Theta}_{k}(\tau)}{r+d_{k}-\tau} \leq 3200\frac{\widetilde{\Theta}_{k}(r+2^{-j_{0}-k+1})-\widetilde{\Theta}_{k}(r)}{2^{-j_{0}-k+1}} \\ &\;\;\;\;\;\;\;\;\;\;\;\;\;\;\;\;\;\;\;\;\;\;\;\;\;\;\;\;\;\;\;\;\;\;\;\;\;\;\;\;\;\;\;\;\;\;\;\;\text{for all}\; \tau\in (r,r+d_{k}).
\end{split}
\end{align}
Now put for $t\in(r,r+d_{k})$
\begin{align}\label{eq:ThetajdefChapter3}
\begin{split}
\Theta_{k}(t) & :=\int_{\overline{\ball}_{t}}|\nabla v_{j_{k}}|\dif x + \int_{\overline{\ball}_{t}}\left\vert\frac{v_{j_{k}}-\mathbf{u}}{2^{-k}}\right\vert\dif x + \int_{\overline{\ball}_{t}}|\nabla (v_{j_{k}}-\mathbf{u})|\dif x\\ & + \int_{\overline{\ball}_{t}}\left\vert \frac{\mathbb{E}_{r,r+d_{k}}u-v_{j_{k}}}{2^{-k}}\right\vert\dif x + \int_{\overline{\ball}_{t}}|\nabla(\mathbb{E}_{r,r+d_{k}}u-v_{j_{k}})|\dif x .
\end{split}
\end{align}
By Lemma~\ref{lem:HLW}, which we apply to the interval $\mathcal{J}_{k}=(r,r+d_{k})$ and $E=(N\cap\mathcal{J}_{k})\cup(r,r+\frac{1}{32}d_{k})\cup(r+\frac{31}{32}d_{k},r)$, we find $0<\varepsilon_{k}<a_{k}<d_{k}$ such that $r+\varepsilon_{k},r+a_{k}\in\mathcal{J}_{k}\setminus E$ and 
\begin{align}
\begin{split}\label{eq:Choose1A}
&\frac{\Theta_{k}(r+a_{k})-\Theta_{k}(\tau)}{r+a_{k}-\tau}\leq 3200 \frac{\Theta_{k}(r+d_{k})-\Theta_{k}(r)}{d_{k}}\\ &\;\;\;\;\;\;\;\;\;\;\;\;\;\;\;\;\;\;\;\;\;\;\;\;\;\;\;\;\;\;\;\;\;\;\;\;\;\;\;\;\;\;\;\;\;\;\;\;\text{for all}\; \tau \in (r,r+a_{k}),\end{split}\\ \begin{split}\label{eq:Choose1B}
&\frac{\Theta_{k}(\tau)-\Theta_{k}(r+a_{k})}{\tau-r-a_{k}}\leq 3200 \frac{\Theta_{k}(r+d_{k})-\Theta_{k}(r)}{d_{k}}\\ &\;\;\;\;\;\;\;\;\;\;\;\;\;\;\;\;\;\;\;\;\;\;\;\;\;\;\;\;\;\;\;\;\;\;\;\;\;\;\;\;\;\;\;\;\;\;\;\;\text{for all}\; \tau \in (r+a_{k},r+d_{k}),\end{split}\\ \begin{split}\label{eq:Choose1C}
&\frac{\Theta_{k}(\tau)-\Theta_{k}(r+\varepsilon_{k})}{\tau-r-\varepsilon_{k}} \leq 3200 \frac{\Theta_{k}(r+d_{k})-\Theta_{k}(r)}{d_{k}} \\ &\;\;\;\;\;\;\;\;\;\;\;\;\;\;\;\;\;\;\;\;\;\;\;\;\;\;\;\;\;\;\;\;\;\;\;\;\;\;\;\;\;\;\;\;\;\;\;\;\text{for all}\; \tau\in (r+\varepsilon_{k},r+d_{k}),
\end{split}
\end{align}
and, uniformly in $k$,  
\begin{align}\label{eq:unichoose}
\varepsilon_{k}\sim_{j_{0}}2^{-k},\;\;\;(a_{k}-\varepsilon_{k})\sim_{j_{0}}2^{-k},\;\;\;(d_{k}-a_{k})\sim_{j_{0}}2^{-k}.
\end{align}
See Figure~\ref{fig:radiilocations1} for the locations of these radii. For future reference, we note that 
\begin{align}\label{eq:ThetaboundChap3}
\begin{split}
&\frac{\widetilde{\Theta}_{k}(r+2^{-j_{0}-k+1})-\widetilde{\Theta}_{k}(r)}{2^{-j_{0}-k+1}}\leq c(u,j_{0},x_{0},r)<\infty,\\
&\frac{\Theta_{k}(r+d_{k})-\Theta_{k}(r)}{d_{k}}\leq c(u,j_{0},x_{0},r)<\infty, 
\end{split}
\end{align}
as can be seen as follows: For~$\eqref{eq:ThetaboundChap3}_{1}$, the first term in the definition of $\widetilde{\Theta}_{k}$ is controlled by~\eqref{eq:maxcond}. For the second term in the definition of $\widetilde{\Theta}_{k}$, we estimate 
\begin{align*}
2^{k}\int_{\ball_{r+2^{-j_{0}-k-1}}\setminus\overline{\ball}_{r}}\left\vert\frac{v_{j_{k}}-\mathbf{u}}{2^{-k}}\right\vert\dif x & \leq 2^{k}\int_{\ball_{r+2^{-j_{0}-k-1}}\setminus\overline{\ball}_{r}}\left\vert \frac{v_{j_{k}}-u}{2^{-k}}\right\vert\dif x \\ & + 2^{k}\int_{\ball_{r+2^{-j_{0}-k-1}}\setminus\overline{\ball}_{r}}\left\vert\frac{u-\mathbf{u}}{2^{-k}}\right\vert\dif x =: \mathrm{J}_{k}+\mathrm{JJ}_{k},  
\end{align*}
so that $\mathrm{J}_{k}$ is uniformly bounded by~\eqref{eq:RecoveryClosedness}. To bound $\mathrm{JJ}_{k}$, we directly adopt the notation employed in the construction of the trace-preserving operator $\E_{r,r+\ell}$ with the Whitney ball cover $(\ball^{i})$ of $\ball_{r+\ell}\setminus\overline{\ball}_{r}$ and the corresponding partition of unity $(\eta_{i})$ subject to $(\ball^{i})$. Then, by Poincar\'{e}'s inequality on $\bv$, the uniformly finite overlap of the Whitney balls~\ref{item:Whitney2} and recalling that, for $k$ sufficiently large, the radius of any $\ball^{i}$ with $\ball^{i}\cap(\ball_{r+2^{-j_{0}-k+1}}\setminus\overline{\ball}_{r})\neq\emptyset$ is bounded by a constant multiple of $2^{-j_{0}-k}$ by~\ref{item:Whitney3}, we conclude
\begin{align}\label{eq:butterchicken}
\begin{split}
\int_{\ball_{r+2^{-j_{0}-k+1}}\setminus\overline{\ball}_{r}}\left\vert \frac{u-\mathbf{u}}{2^{-k}}\right\vert\dif x & \leq c\sum_{\substack{i\colon\\ \ball^{i}\cap(\ball_{r+2^{-j_{0}-k+1}}\setminus\overline{\ball}_{r})\neq\emptyset}}\int_{\ball^{i}}\left\vert \frac{\eta_{i}(u-(u)_{\ball^{i}})}{2^{-k}}\right\vert\dif x \\ 
& \leq c \int_{\ball_{r+2^{-j_{0}-k+1}}\setminus\overline{\ball}_{r}}|Du| \\ 
& \leq c2^{-k}\mathcal{M}Du(x_{0},r)  
\end{split}
\end{align}
for all sufficiently large $k$, and so $\sup_{k\in\mathbb{N}}\mathrm{JJ}_{k}<\infty$. Finally, the third term in the definition of $\widetilde{\Theta}_{k}$ is controlled by splitting 
\begin{align*}
2^{k}\int_{\ball_{r+2^{-j_{0}-k+1}}\setminus\overline{\ball}_{r}}|\nabla(v_{j_{k}}-\mathbf{u})|\dif x & \leq 2^{k}\int_{\ball_{r+2^{-j_{0}-k+1}}\setminus\overline{\ball}_{r}}|\nabla v_{j_{k}}|\dif x \\ & + 2^{k}\int_{\ball_{r+2^{-j_{0}-k+1}}\setminus\overline{\ball}_{r}}|\nabla\mathbf{u}|\dif x =: \mathrm{J}'_{k}+\mathrm{JJ}'_{k}
\end{align*}
and using the maximal condition $\mathcal{M}\lambda(x_{0},r)<\infty$ in conjunction with~\eqref{eq:MaximalBoundA} for $\mathrm{J}'_{k}$, whereas $\mathrm{JJ}'_{k}$ is bounded by using~$\eqref{eq:extensionpointwise}_{1}$: 
\begin{align*}
\mathrm{JJ}'_{k} & \stackrel{\eqref{eq:extensionpointwise}_{1}}{\leq} c2^{k}\sum_{i\colon\,\ball^{i}\cap(\ball_{r+2^{-j_{0}-k+1}}\setminus\overline{\ball}_{r})\neq\emptyset}\int_{\Lambda\!\ball^{i}}|Du| \leq c\,\mathcal{M}Du(x_{0},r)<\infty. 
\end{align*} 
Towards the justification of~$\eqref{eq:ThetaboundChap3}_{2}$, the first term in the definition of $\Theta_{k}$ is controlled by $\mathcal{M}\lambda(x_{0},r)<\infty$ in conjunction with~\eqref{eq:MaximalBoundA} and~\eqref{eq:unifchooseD}. The bounds on the second and third term in the definition of $\Theta_{k}$ are achieved as for $\widetilde{\Theta}_{k}$, now realising that $d_{k}\sim_{j_{0}}2^{-k}$ by~\eqref{eq:unifchooseD}. To bound the fourth term in the definition of $\Theta_{k}$, we imitate~\eqref{eq:butterchicken} and use~\eqref{eq:RecoveryClosedness} as follows:
\begin{align*}
\frac{1}{d_{k}}\int_{\ball_{r+d_{k}}\setminus\overline{\ball}_{r}}\left\vert\frac{\mathbb{E}_{r,r+d_{k}}u-v_{j_{k}}}{2^{-k}}\right\vert\dif x & \leq \frac{1}{d_{k}}\int_{\ball_{r+d_{k}}\setminus\overline{\ball}_{r}}\left\vert\frac{\mathbb{E}_{r,r+d_{k}}u-u}{2^{-k}}\right\vert\dif x \\ 
& + \frac{1}{d_{k}}\int_{\ball_{r+d_{k}}\setminus\overline{\ball}_{r}}\left\vert\frac{u-v_{j_{k}}}{2^{-k}}\right\vert\dif x \\ 
& \leq c(\mathcal{M}Du(x_{0},r)+1)<\infty,
\end{align*}
where we used that $d_{k}\sim_{j_{0}}2^{-k}$. Lastly, the fifth term in the definition of $\Theta_{k}$ is bounded by the $\lebe^{1}$-stability of the trace-preserving operators via 
\begin{align*}
\frac{1}{d_{k}}\int_{\ball_{r+d_{k}}\setminus\overline{\ball}_{r}}|\nabla\mathbb{E}_{r,r+d_{k}}u-v_{j_{k}}|\dif x & \leq \frac{1}{d_{k}}\int_{\ball_{r+d_{k}}\setminus\overline{\ball}_{r}}|\nabla\mathbb{E}_{r,r+d_{k}}u|\dif x \\ & + \frac{1}{d_{k}}\int_{\ball_{r+d_{k}}\setminus\overline{\ball}_{r}}|\nabla v_{j_{k}}|\dif x \\ 
& \leq \frac{c}{d_{k}}\int_{\ball_{r+d_{k}}\setminus\overline{\ball}_{r}}|Du| \\ & + \frac{c}{2^{-k}}\int_{\ball_{r+2^{-j_{0}-k+1}}\setminus\overline{\ball}_{r}}|\nabla v_{j_{k}}|\dif x
\end{align*}
and recalling $d_{k}\sim_{j_{0}}2^{-k}$ as well as the maximal conditions $\mathcal{M}Du(x_{0},r),\mathcal{M}\lambda(x_{0},r)<\infty$ together with~\eqref{eq:MaximalBoundA}. This establishes~$\eqref{eq:ThetaboundChap3}_{2}$. On $\ball_{r}\setminus\overline{\ball}_{r-\ell}$ we similarly find $0<\varepsilon_{k}^{-}<a_{k}^{-}<d_{k}^{-}<2^{-j_{0}-k+1}$ with the analogous properties, tacitly assumed to be picked in the sequel. 

\emph{Step 2. Construction of the sequence $(u_{j_{k}})$.} To construct $u_{j}$, we initially modify $v_{j}$ by applying the operator $\E$ from Section~\ref{sec:Fubini} as follows:
\begin{align}\label{eq:vDEF}
\widetilde{v}_{j_{k}}:= \begin{cases} 
v_{j_{k}}&\;\text{in}\;\ball_{r-a_{k}^{-}},\\
\mathbb{E}_{r-a_{k}^{-},r-\varepsilon_{k}^{-}}v_{j_{k}}&\;\text{in}\;\ball_{r-\varepsilon_{k}^{-}}\setminus\overline{\ball}_{r-a_{k}^{-}},\\
v_{j_{k}} &\;\text{in}\;\ball_{r+\varepsilon_{k}}\setminus\overline{\ball}_{r-\varepsilon_{k}^{-}}, \\
\mathbb{E}_{r+\varepsilon_{k},r+a_{k}}v_{j_{k}}&\;\text{in}\;\ball_{r+a_{k}}\setminus\overline{\ball}_{r+\varepsilon_{k}},\\
v_{j_{k}}&\;\text{in}\;\omega\setminus\overline{\ball}_{r+a_{k}},
\end{cases} 
\end{align}
and moreover define
\begin{align}\label{eq:wDEF}
\widetilde{w}_{j_{k}}:=\begin{cases} 
u & \;\text{in}\;\ball_{r-d_{k}^{-}},\\
\E_{r-a_{k}^{-},r-\varepsilon_{k}^{-}}(\E_{r-d_{k}^{-},r}u) &\;\text{in}\;\ball_{r-\varepsilon_{k}^{-}}\setminus\overline{\ball}_{r-a_{k}^{-}} \\ 
\E_{r-d_{k}^{-},r}u&\;\text{in}\;(\ball_{r}\setminus\overline{\ball}_{r-d_{k}^{-}})\setminus(\ball_{r-\varepsilon_{k}^{-}}\setminus\overline{\ball}_{r-a_{k}^{-}}), \\ 
\mathbb{E}_{r,r+d_{k}}u&\;\text{in}\;(\ball_{r+d_{k}}\setminus\overline{\ball}_{r})\setminus (\ball_{r+a_{k}}\setminus\overline{\ball}_{r+\varepsilon_{k}}),\\ 
\mathbb{E}_{r+\varepsilon_{k},r+a_{k}}(\mathbb{E}_{r,r+d_{k}}u)&\;\text{in}\;\ball_{r+a_{k}}\setminus\overline{\ball}_{r+\varepsilon_{k}},\\
u&\;\text{in}\;\omega\setminus\overline{\ball}_{r+d_{k}}.
\end{cases} 
\end{align}
Next choose a cut-off function $\rho_{k}\in\hold_{c}^{\infty}(\omega;[0,1])$ such that 
\begin{align}\label{eq:cutoffprop1}
\begin{split}
&\mathbbm{1}_{\ball_{r+\varepsilon_{k}}\setminus\overline{\ball}_{r}}\leq\rho_{k}|_{\ball_{r+\ell}\setminus\overline{\ball}_{r}}\leq \mathbbm{1}_{\ball_{r+a_{k}}\setminus\overline{\ball}_{r}}\;\;\;\;\,\,\text{and}\;\;\;|\nabla\rho_{k}|\leq \frac{4}{a_{k}-\varepsilon_{k}}\;\;\;\text{on}\;\ball_{r+\ell}\setminus\overline{\ball}_{r}, \\  
&\mathbbm{1}_{\ball_{r}\setminus\ball_{r-\varepsilon_{k}^{-}}}\leq\rho_{k}|_{\ball_{r}\setminus\overline{\ball}_{r-\ell}}\leq \mathbbm{1}_{\ball_{r}\setminus\ball_{r-a_{k}^{-}}}\;\;\;\text{and}\;\;\;|\nabla\rho_{k}|\leq \frac{4}{a_{k}^{-}-\varepsilon_{k}^{-}}\;\;\;\text{on}\;\ball_{r}\setminus\overline{\ball}_{r-\ell}.
\end{split}
\end{align}
The requisite map $u_{j_{k}}\in\mathscr{A}_{u_{0}}^{q}(\omega;\omega')$ as claimed in the lemma then is defined by 
\begin{align}\label{eq:ujkdefine}
u_{j_{k}}:=\rho_{k}\widetilde{w}_{j_{k}} + (1-\rho_{k})\widetilde{v}_{j_{k}}. 
\end{align}
The claimed underlying regularity of $u_{j_{k}}$ immediately follows from the definition of the trace-preserving operators upon realising that, similar to the argument preceding~\eqref{eq:mathbfudef}, $\widetilde{w}_{j_{k}}$ has no jumps along $\partial\!\ball_{r}$. By Lemma~\ref{lem:extensionoperator} and the definition of $\widetilde{w}_{j_{k}}$, the traces of $u_{j_{k}}$ and $u$ along $\partial\!\ball_{r}$ coincide $\mathscr{H}^{n-1}$-a.e. on $\partial\!\ball_{r}$ for each $k\in\mathbb{N}$. Hence \ref{item:FixRadius2} holds, and so it remains to establish validity of  \ref{item:FixRadius3} and that $(u_{j_{k}})$ is generating for $\overline{\mathscr{F}}_{u_{0}}^{*}[u;\omega,\omega']$ indeed. We split the remaining proof into steps 3--5; for steps 3 and 4, it will be sufficient to exclusively consider the upper annuli $\ball_{r+\cdot}\setminus\overline{\ball}_{r}$ since the arguments are the same for the lower annuli $\ball_{r}\setminus\ball_{r-\cdot}$. It is only at step 5 where we need to consider the full annulus $\ball_{r+\ell}\setminus\overline{\ball}_{r-\ell}$. 

\emph{Step 3. $\nabla u_{j_{k}}-\nabla v_{j_{k}}\to 0$ in $\lebe^{1}$.}
In view of \ref{item:FixRadius3} and since $u_{j_{k}}(x)=v_{j_{k}}(x)$ for $x\in\omega\setminus\ball_{r}(x_{0})$ with $|x-x_{0}|\geq r+a_{k}$, the support properties of $\rho_{k}$ imply
\begin{align*}
\|\nabla(u_{j_{k}}-v_{j_{k}})\|_{\lebe^{1}(\omega\setminus\overline{\ball}_{r})} & \leq \|\rho_{k}\nabla (\widetilde{w}_{j_{k}}-v_{j_{k}})\|_{\lebe^{1}(\ball_{r+a_{k}}\setminus\overline{\ball}_{r})} \\ & + \|(\widetilde{w}_{j_{k}}-v_{j_{k}})\otimes\nabla\rho_{k}\|_{\lebe^{1}(\ball_{r+a_{k}}\setminus\overline{\ball}_{r+\varepsilon_{k}})} \\ 
& + \|(\nabla v_{j_{k}}-\nabla \widetilde{v}_{j_{k}})\|_{\lebe^{1}(\ball_{r+a_{k}}\setminus\overline{\ball}_{r+\varepsilon_{k}})}\\
& + \|(v_{j_{k}}-\widetilde{v}_{j_{k}})\otimes\nabla\rho_{k}\|_{\lebe^{1}(\ball_{r+a_{k}}\setminus\overline{\ball}_{r+\varepsilon_{k}})} \\ & =: \mathrm{I}_{k}+...+\mathrm{IV}_{k}.
\end{align*} 
Ad~$\mathrm{I}_{k}$. By the $\lebe^{1}$-stability of the trace-preserving operator $\mathbb{E}$, we obtain:
\begin{align*}
\mathrm{I}_{k} & \leq \|\nabla\mathbb{E}_{r,r+d_{k}}u\|_{\lebe^{1}(\ball_{r+d_{k}}\setminus\overline{\ball}_{r})} + \|\nabla\mathbb{E}_{r+\varepsilon_{k},r+a_{k}}\mathbb{E}_{r,r+d_{k}}u\|_{\lebe^{1}(\ball_{r+a_{k}}\setminus\overline{\ball}_{r+\varepsilon_{k}})}  \\ 
& + \|\nabla v_{j_{k}}\|_{\lebe^{1}(\ball_{r+a_{k}}\setminus\overline{\ball}_{r})} \\
& \leq c\|\nabla\mathbb{E}_{r,r+d_{k}}u\|_{\lebe^{1}(\ball_{r+d_{k}}\setminus\overline{\ball}_{r})} + \|\nabla v_{j_{k}}\|_{\lebe^{1}(\ball_{r+a_{k}}\setminus\overline{\ball}_{r})}\\ 
& \leq c|Du|(\ball_{r+d_{k}}\setminus\overline{\ball}_{r}) + \|\nabla v_{j_{k}}\|_{\lebe^{1}(\ball_{r+a_{k}}\setminus\overline{\ball}_{r})}\\ 
& \leq c2^{-k}\mathcal{M}Du(x_{0},r) + c2^{-k}\mathcal{M}\lambda(x_{0},r) + c2^{-2k} \to 0\;\;\;(\text{by \eqref{eq:maxcond},~\eqref{eq:maxcondlambda1}\;\text{and}\;\eqref{eq:MaximalBoundA}}). 
\end{align*}
Ad~$\mathrm{II}_{k}$. We now directly adopt the notation employed in the construction for the trace-preserving operator $\mathbb{E}_{r+\varepsilon_{k},r+a_{k}}$, with a corresponding Whitney ball cover $(\ball^{i})$ of the annulus $\ball_{r+a_{k}}\setminus\overline{\ball}_{r+\varepsilon_{k}}$ and the partition of unity $(\eta_{i})$ subject to $(\ball^{i})$. In consequence,
\begin{align*}
\mathrm{II}_{k} & \!\!\!\stackrel{\eqref{eq:cutoffprop1}}{\leq} c\int_{\ball_{r+a_{k}}\setminus\overline{\ball}_{r+\varepsilon_{k}}}\frac{|\widetilde{w}_{j_{k}}-v_{j_{k}}|}{a_{k}-\varepsilon_{k}}\dif x \\ 
& = c\int_{\ball_{r+a_{k}}\setminus\overline{\ball}_{r+\varepsilon_{k}}}\frac{|\mathbb{E}_{r+\varepsilon_{k},r+a_{k}}(\mathbb{E}_{r,r+d_{k}}u)-v_{j_{k}}|}{a_{k}-\varepsilon_{k}}\dif x \\
& =  \frac{c}{a_{k}-\varepsilon_{k}}\sum_{i\in\mathbb{N}}\int_{\ball_{r+a_{k}}\setminus\overline{\ball}_{r+\varepsilon_{k}}}|\eta_{i}(\mathbb{E}_{r,r+d_{k}}u)_{\ball^{i}}-\eta_{i}v_{j_{k}}|\dif x \\ 
& \leq \frac{c}{a_{k}-\varepsilon_{k}}\sum_{i\in\mathbb{N}}\int_{\ball^{i}}|(\mathbb{E}_{r,r+d_{k}}u)_{\ball^{i}}-v_{j_{k}}|\dif x \\
& \leq \frac{c}{a_{k}-\varepsilon_{k}}\sum_{i\in\mathbb{N}}\int_{\ball^{i}}|(\mathbb{E}_{r,r+d_{k}}u)_{\ball^{i}}-\mathbb{E}_{r,r+d_{k}}u| + |\mathbb{E}_{r,r+d_{k}}u - u| + |u -v_{j_{k}}|\dif x \\ & =: \mathrm{II}_{k}^{(1)}.
\end{align*}
By Poincar\'{e}'s inequality, $r(\ball^{i})\leq (a_{k}-\varepsilon_{k})$ for all $i\in\mathbb{N}$  and the uniformly finite overlap of the Whitney balls $\ball^{i}$, we then obtain by use of the $\lebe^{1}$-gradient stability of the trace-preserving operator $\mathbb{E}$: 
\begin{align*}
\mathrm{II}_{k}^{(1)} & \leq c\int_{\ball_{r+d_{k}}\setminus\overline{\ball}_{r}}|\nabla\mathbb{E}_{r,r+d_{k}}u|\dif x \\ & + \frac{c}{a_{k}-\varepsilon_{k}}\int_{\ball_{r+d_{k}}\setminus\overline{\ball}_{r}}|\mathbb{E}_{r,r+d_{k}}u - u|\dif x + \frac{c}{a_{k}-\varepsilon_{k}}\int_{\ball_{r+d_{k}}\setminus\overline{\ball}_{r}}|u-v_{j_{k}}|\dif x \\
& \stackrel{(*)}{\leq} c|Du|(\ball_{r+d_{k}}\setminus\overline{\ball}_{r}) + c2^{k}\int_{\ball_{r+d_{k}}\setminus\overline{\ball}_{r}}|u-v_{j_{k}}|\dif x \to 0, 
\end{align*}
as $k\to\infty$; also see~\eqref{eq:RecoveryClosedness}, \eqref{eq:MaximalBoundA} and \eqref{eq:unichoose}. Here, $(*)$ can be seen similarly as~\eqref{eq:butterchicken}.  

Ad~$\mathrm{III}_{k}$. Here we have by the $\lebe^{1}$-gradient stability of $\mathbb{E}$
\begin{align*}
\mathrm{III}_{k} & \leq \int_{\ball_{r+a_{k}}\setminus\overline{\ball}_{r+\varepsilon_{k}}}|\nabla (v_{j_{k}}-\mathbb{E}_{r+\varepsilon_{k},r+a_{k}}v_{j_{k}})|\dif x \\ 
& \leq c \int_{\ball_{r+a_{k}}\setminus\overline{\ball}_{r+\varepsilon_{k}}}|\nabla v_{j_{k}}|\dif x \leq c \int_{\ball_{r+d_{k}}\setminus\overline{\ball}_{r}}|\nabla v_{j_{k}}|\dif x \to 0
\end{align*}
by \eqref{eq:maxcondlambda1} and \eqref{eq:MaximalBoundA}. 

Ad~$\mathrm{IV}_{k}$. We directly employ the Whitney balls $\ball^{i}$ underlying the trace-preserving operator $\mathbb{E}_{r+\varepsilon_{k},r+a_{k}}$ and the corresponding partition of unity $(\eta_{i})$ to find by use of \eqref{eq:cutoffprop1}
\begin{align*}
\mathrm{IV}_{k} & \leq \frac{c}{a_{k}-\varepsilon_{k}}\sum_{i\in\mathbb{N}}\int_{\ball^{i}}|\eta_{i}(v_{j_{k}}-(v_{j_{k}})_{\ball^{i}})|\dif x \leq c\int_{\ball_{r+a_{k}}\setminus\overline{\ball}_{r}}|\nabla v_{j_{k}}|\dif x \to 0,\;\;\;k\to\infty, 
\end{align*}
again by \eqref{eq:maxcondlambda1} and \eqref{eq:MaximalBoundA}.

The same argument applies to the lower annuli $\ball_{r}\setminus\overline{\ball}_{r-a_{k}^{-}}$. In consequence, gathering estimates, $\nabla (u_{j_{k}}-v_{j_{k}})\to 0$ in $\lebe^{1}(\omega';\R^{N})$ as $k\to\infty$ and the proof of \ref{item:FixRadius3} is complete. 

Imitating the estimate for $\mathrm{II}_{k}$ and $\mathrm{IV}_{k}$, one similarly establishes that $u_{j_{k}}\to u$ in $\lebe^{1}(\omega';\R^{N})$. In consequence, by~\ref{item:FixRadius3}, $\nabla v_{j_{k}}\mathscr{L}^{n}\stackrel{*}{\rightharpoonup} Du$ in $\mathrm{RM}_{\mathrm{fin}}(\omega';\R^{N\times n})$ and $u_{j_{k}}\to u$ in $\lebe^{1}(\omega';\R^{N})$, $u_{j_{k}}\stackrel{*}{\rightharpoonup}u$ in $\bv(\omega';\R^{N})$ too. Because of $(u_{j_{k}})\subset\mathscr{A}_{u_{0}}^{q}(\omega,\omega')$ and $u_{j_{k}}\stackrel{*}{\rightharpoonup}u$ in $\bv(\omega';\R^{N})$, the inequality $\overline{\mathscr{F}}_{u_{0}}^{*}[u;\omega,\omega']\leq \liminf_{k\to\infty}\mathscr{F}[u_{j_{k}};\omega']$ holds trivially. We may therefore conclude the proof by showing 
\begin{align}\label{eq:CentralGeneration}
\liminf_{k\to\infty}\int_{\ball_{r+a_{k}}\setminus\overline{\ball}_{r-a_{k}^{-}}}F(\nabla u_{j_{k}})-F(\nabla v_{j_{k}})\dif x \leq 0
\end{align}
as we may then infer (recall that $u_{j_{k}}=v_{j_{k}}$ outside $\ball_{r+a_{k}}\setminus\overline{\ball}_{r-a_{k}^{-}}$)
\begin{align*}
\liminf_{k\to\infty}\mathscr{F}[& u_{j_{k}};\omega'] \leq \liminf_{k\to\infty}\Big(\int_{\omega'}F(\nabla v_{j_{k}})\dif x + \int_{\omega'}F(\nabla u_{j_{k}})-F(\nabla v_{j_{k}})\dif x\Big) \\
& = \liminf_{k\to\infty}\Big(\int_{\omega'}F(\nabla v_{j_{k}})\dif x + \int_{\ball_{r+a_{k}}\setminus\overline{\ball}_{r-a_{k}^{-}}}F(\nabla u_{j_{k}})-F(\nabla v_{j_{k}})\dif x\Big)\\
& \!\!\!\!\!\!\!\!\stackrel{\eqref{eq:liminfislim},\,\eqref{eq:CentralGeneration}}{\leq} \lim_{k\to\infty}\int_{\omega'}F(\nabla v_{j_{k}})\dif x = \overline{\mathscr{F}}_{u_{0}}^{*}[u;\omega,\omega'].  
\end{align*}
Our plan is to establish 
\begin{align}\label{eq:generation1}
&\lim_{k\to\infty}\int_{\ball_{r+a_{k}}\setminus\overline{\ball}_{r-a_{k}^{-}}}|F(\nabla u_{j_{k}})|\dif x = 0,\\\label{eq:generation2} & \liminf_{k\to\infty}\int_{\ball_{r+a_{k}}\setminus\overline{\ball}_{r-a_{k}^{-}}}F(\nabla v_{j_{k}})\dif x\geq 0, 
\end{align}
from where \eqref{eq:CentralGeneration} follows at once. 

\emph{Step 4. Generation: Inequality~\eqref{eq:generation1}.} For \eqref{eq:generation1}, it suffices to argue on the upper annuli; the argument for the lower annuli is analogous. By the growth bound \ref{item:H1} on $F$ and the definition of $u_{j_{k}}$, cf.~\eqref{eq:ujkdefine}, we have 
\begin{align*}
\int_{\ball_{r+a_{k}}\setminus\overline{\ball}_{r}}|& F(\nabla u_{j_{k}})|\dif x \leq L\int_{\ball_{r+a_{k}}\setminus\overline{\ball}_{r}}(1+|\nabla u_{j_{k}}|^{q})\dif x \\ 
& \leq c\Big( \mathscr{L}^{n}(\ball_{r+a_{k}}\setminus\overline{\ball}_{r}) + \int_{\ball_{r+\varepsilon_{k}}\setminus\overline{\ball}_{r}}|\nabla \mathbb{E}_{r,r+d_{k}}u|^{q}\dif x\Big. \\ 
& \Big. + \int_{\ball_{r+a_{k}}\setminus\overline{\ball}_{r+\varepsilon_{k}}}|\nabla (\rho_{k}(\mathbb{E}_{r+\varepsilon_{k},r+a_{k}}(\mathbb{E}_{r,r+d_{k}}u)-\mathbb{E}_{r+\varepsilon_{k},r+a_{k}}v_{j_{k}}))|^{q}\dif x \Big. \\ & \Big. + \int_{\ball_{r+a_{k}}\setminus\overline{\ball}_{r+\varepsilon_{k}}}|\nabla\mathbb{E}_{r+\varepsilon_{k},r+a_{k}}v_{j_{k}}|^{q}\dif x \Big) \\ 
& =: \mathrm{V}_{k} + ... + \mathrm{VIII}_{k}. 
\end{align*}
Trivially, $\mathrm{V}_{k}\to 0$. 

Ad~$\mathrm{VI}_{k}$. Combining Lemma~\ref{lem:extensionoperator}~\ref{item:extension4},~\eqref{eq:maxcond}, \eqref{eq:Choose1D} and \eqref{eq:ThetaboundChap3},  
\begin{align*}
\mathrm{VI}_{k} & \stackrel{\varepsilon_{k}<d_{k}}{\leq} \int_{\ball_{r+d_{k}}\setminus\overline{\ball}_{r}}|\nabla \mathbb{E}_{r,r+d_{k}}u|^{q}\dif x \\ & \!\!\!\!\!\!\!\!\!\!\stackrel{\text{Lem.~\ref{lem:extensionoperator}\ref{item:extension4}}}{\leq} cd_{k}^{n-nq+q}\Big(\sup_{0<\delta\ll d_{k}}\frac{1}{\delta}\int_{\ball_{r+d_{k}}\setminus\overline{\ball}_{r+d_{k}-\delta}}|Du| + \sup_{0<\delta\ll d_{k}}\frac{1}{\delta}\int_{\ball_{r+\delta}\setminus\overline{\ball}_{r}}|Du| \Big)^{q} \\ & \!\!\!\!\!\!\stackrel{\eqref{eq:Choose1D}}{\leq} cd_{k}^{n-nq+q}\Big(\frac{\widetilde{\Theta}_{k}(r+2^{-j_{0}-k+1})-\widetilde{\Theta}_{k}(r)}{2^{-j_{0}-k+1}} +\mathcal{M}Du(x_{0},r)\Big)^{q}\\ & \!\!\!\!\!\!\!\stackrel{\eqref{eq:ThetaboundChap3}_{1}}{\leq} c d_{k}^{n-nq+q}
\end{align*}
which tends to zero as $k\to\infty$ by $d_{k}\sim_{j_{0}}2^{-k}$ (cf.~\eqref{eq:unifchooseD}) and $q<\frac{n}{n-1}$.

Ad~$\mathrm{VII}_{k}$. Recalling \eqref{eq:cutoffprop1}, we have as in the estimation of $\mathrm{VI}_{k}$:
\begin{align*}
\mathrm{VII}_{k} & \leq c\int_{\ball_{r+a_{k}}\setminus\overline{\ball}_{r+\varepsilon_{k}}}|\nabla (\mathbb{E}_{r+\varepsilon_{k},r+a_{k}}(\mathbb{E}_{r,r+d_{k}}u-v_{j_{k}}))|^{q}\dif x \\ 
& + c\int_{\ball_{r+a_{k}}\setminus\overline{\ball}_{r+\varepsilon_{k}}}\left\vert \frac{\mathbb{E}_{r+\varepsilon_{k},r+a_{k}}(\mathbb{E}_{r,r+d_{k}}u-v_{j_{k}})}{a_{k}-\varepsilon_{k}}\right\vert^{q}\dif x \\ 
& \!\!\!\!\!\!\!\!\!\!\!\!\stackrel{\text{Lem.}~\ref{lem:extensionoperator}~\ref{item:extension4}}{\leq} c (a_{k}-\varepsilon_{k})^{n-nq+q}\Big(\sup_{0<\delta\ll(a_{k}-\varepsilon_{k})}\frac{1}{\delta}\int_{\ball_{r+\varepsilon_{k}+\delta}\setminus\overline{\ball}_{r+\varepsilon_{k}}} |\nabla(\mathbb{E}_{r,r+d_{k}}u-v_{j_{k}})|\dif x\Big)^{q} \\ 
& +  c (a_{k}-\varepsilon_{k})^{n-nq+q}\Big(\sup_{0<\delta\ll(a_{k}-\varepsilon_{k})}\frac{1}{\delta}\int_{\ball_{r+a_{k}}\setminus\overline{\ball}_{r+a_{k}-\delta}} |\nabla(\mathbb{E}_{r,r+d_{k}}u-v_{j_{k}})|\dif x\Big)^{q}\\
& + c (a_{k}-\varepsilon_{k})^{n-nq+q}\Big(\sup_{0<\delta\ll(a_{k}-\varepsilon_{k})}\frac{1}{\delta}\int_{\ball_{r+\varepsilon_{k}+\delta}\setminus\overline{\ball}_{r+\varepsilon_{k}}} \left\vert \frac{\mathbb{E}_{r,r+d_{k}}u-v_{j_{k}}}{a_{k}-\varepsilon_{k}}\right\vert 
\dif x\Big)^{q} \\ 
& + c (a_{k}-\varepsilon_{k})^{n-nq+q}\Big(\sup_{0<\delta\ll(a_{k}-\varepsilon_{k})}\frac{1}{\delta}\int_{\ball_{r+a_{k}}\setminus\overline{\ball}_{r+a_{k}-\delta}} \left\vert \frac{\mathbb{E}_{r,r+d_{k}}u-v_{j_{k}}}{a_{k}-\varepsilon_{k}}\right\vert 
\dif x\Big)^{q}\\
& \leq c (a_{k}-\varepsilon_{k})^{n-nq+q}\Big(\frac{\Theta_{k}(r+d_{k})-\Theta_{k}(r)}{d_{k}} \Big)^{q}\\
& \leq c2^{-k(n-nq+q)}\to 0
\end{align*}
by \eqref{eq:ThetajdefChapter3}, \eqref{eq:Choose1A}, \eqref{eq:Choose1C}, \eqref{eq:unichoose} and \eqref{eq:ThetaboundChap3}.

Ad $\mathrm{VIII}_{k}$. Equally by \eqref{eq:ThetajdefChapter3}, \eqref{eq:Choose1A}, \eqref{eq:Choose1C}, \eqref{eq:unichoose} and \eqref{eq:ThetaboundChap3}, 
\begin{align*}
\mathrm{VIII}_{k} & \leq c (a_{k}-\varepsilon_{k})^{n-nq+q}\Big(\sup_{0<\delta\ll(a_{k}-\varepsilon_{k})}\frac{1}{\delta}\int_{\ball_{r+\varepsilon_{k}+\delta}\setminus\overline{\ball}_{r+\varepsilon_{k}}} |\nabla v_{j_{k}}|\dif x\Big)^{q} \\ 
& +  c (a_{k}-\varepsilon_{k})^{n-nq+q}\Big(\sup_{0<\delta\ll(a_{k}-\varepsilon_{k})}\frac{1}{\delta}\int_{\ball_{r+a_{k}}\setminus\overline{\ball}_{r+a_{k}-\delta}} |\nabla v_{j_{k}}|\dif x\Big)^{q} \\ & \leq c(a_{k}-\varepsilon_{k})^{n-nq+q} \Big(\frac{\Theta_{k}(r+d_{k})-\Theta_{k}(r)}{d_{k}} \Big)^{q}\\ 
& \leq c2^{-k(n-nq+q)}\to 0.
\end{align*}
The estimates for $\mathrm{V}_{k},...,\mathrm{VIII}_{k}$ and their analogues for the lower annuli now combine to \eqref{eq:generation1}.

\begin{figure}
\begin{tikzpicture}[scale=1]
\draw[<->] (-0.125,1.75) -- (-0.125,2.25);
\node at (-0.75,2) {$\sim 2^{-k}$};
\filldraw[blue!20!white, fill=blue!20!white] (1,0) arc [radius=1, start angle=0, delta angle=90]                  -- (0,3) arc [radius=3, start angle=90, delta angle=-90]
                  -- cycle;
\node at (2,-0.2) {{\small $r$}}; 
\node at (1,-0.2) {{\small $r-\ell$}};  
\node at (3,-0.2) {{\small $r+\ell$}};     
\filldraw[blue!20!white, opacity=0.9, fill=blue!20!white] (1,0) arc [radius=1, start angle=0, delta angle=90]                  -- (0,3) arc [radius=3, start angle=90, delta angle=-90]
                  -- cycle;                  
                  \filldraw[opacity=0.9,blue!10!white, fill=blue!10!white] (1.75,0) arc [radius=1.75, start angle=0, delta angle=90]
                  -- (0,2.25) arc [radius=2.25, start angle=90, delta angle=-90]
                  -- cycle;      
                  \draw[blue,dotted,thick] (2,0) arc [radius=2, start angle=0, delta angle=90]; 
                  \draw[black,dotted] (3,0) arc [radius=3, start angle=0, delta angle=90];
                  \draw[black,dotted] (1,0) arc [radius=1, start angle=0, delta angle=90];       
                  \node[blue!80!black] at (2.25,1.4) {$v_{j_{k}}$};
                  \node[blue!80!black] at (1.25,0.7) {$v_{j_{k}}$};
\end{tikzpicture}
\begin{tikzpicture}[scale=1]
\draw[<->] (-0.125,1.65) -- (-0.125,2.35);
\node at (-0.75,2) {$\sim 2^{-k}$};
\filldraw[blue!20!white, fill=blue!20!white] (1,0) arc [radius=1, start angle=0, delta angle=90]                  -- (0,3) arc [radius=3, start angle=90, delta angle=-90]
                  -- cycle;
\node at (2,-0.2) {{\small $r$}}; 
\node at (1,-0.2) {{\small $r-\ell$}};  
\node at (3,-0.2) {{\small $r+\ell$}};     
\filldraw[blue!20!white, opacity=0.9, fill=blue!20!white] (1,0) arc [radius=1, start angle=0, delta angle=90]                  -- (0,3) arc [radius=3, start angle=90, delta angle=-90]
                  -- cycle;                  
                  \filldraw[opacity=0.9,blue!10!white, fill=blue!10!white] (1.65,0) arc [radius=1.65, start angle=0, delta angle=90]
                  -- (0,2.35) arc [radius=2.35, start angle=90, delta angle=-90]
                  -- cycle;      
                  \draw[blue,dotted,thick] (2,0) arc [radius=2, start angle=0, delta angle=90]; 
                  \draw[black,dotted] (3,0) arc [radius=3, start angle=0, delta angle=90];
                  \draw[black,dotted] (1,0) arc [radius=1, start angle=0, delta angle=90];       
                  \node[blue!80!black] at (2.3,1.4) {$\mathbf{u}$};
                  \node[blue!80!black] at (1.1,0.7) {$\mathbf{u}$};
\end{tikzpicture}
\caption{The conceptual difference between step 4 and 5. Whereas in step 4 the sequence $(v_{j_{k}})$ for the proof of~\eqref{eq:generation1} is modified to coincide with $u$ along $\partial\!\ball_{r}$, in step 5 the key sequence for the proof of~\eqref{eq:generation2} is constructed to coincide with $\mathbf{u}$ \emph{away} from $\partial\!\ball_{r}$ in order to have access to the lower semicontinuity result of Lemma~\ref{lem:LSC}; for this, fixed boundary values are required.} \label{fig:conceptual}
\end{figure}
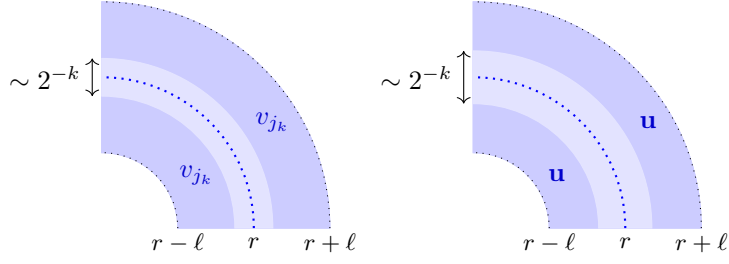

\emph{Step 5. Generation: Inequality~\eqref{eq:generation2}.} We now come to the final part of the proof. With $d_{k}$ as in \eqref{eq:Choose1D} (and analogously $d_{k}^{-}$), choose a cut-off function $\varrho_{k}\in\hold_{c}^{\infty}(\ball_{r+\ell}\setminus\overline{\ball}_{r-\ell};[0,1])$ such that $\varrho_{k}=1$ on $\ball_{r+a_{k}}\setminus\overline{\ball}_{r-a_{k}^{-}}$, $\rho_{k}=0$ outside $\ball_{r+d_{k}}\setminus\overline{\ball}_{r-d_{k}^{-}}$ and 
\begin{align}\label{eq:gradientboundetta}
|\nabla\varrho_{k}|\leq 4\max\left\{\frac{1}{d_{k}-a_{k}},\frac{1}{d_{k}^{-}-a_{k}^{-}}\right\}\sim_{j_{0}}2^{k}, 
\end{align}
which is possible by \eqref{eq:unichoose} and its corresponding analogue for $a_{k}^{-},d_{k}^{-}$. We then define
\begin{align}\label{eq:fatvdefine}
\mathbf{v}_{j_{k}}:=\begin{cases} \mathbb{E}_{r,r+\ell}u\;\;\;\;\;\;\;\;\;\;\;\;\;\text{in}\; \ball_{r+\ell}\setminus\overline{\ball}_{r+d_{k}} \\ 
\varrho_{k}\mathbb{E}_{r+a_{k},r+d_{k}}v_{j_{k}} + (1-\varrho_{j})\mathbb{E}_{r+a_{k},r+d_{k}}\mathbb{E}_{r,r+\ell}u&\; \\ \,\;\;\;\;\;\;\;\;\;\;\;\;\;\;\;\;\;\;\;\;\;\;\;\;\text{in}\;\ball_{r+d_{k}}\setminus\overline{\ball}_{r+a_{k}} \\
v_{j_{k}}\;\;\;\;\;\;\;\;\;\;\;\;\;\;\;\;\;\;\;\;\text{in}\;\ball_{r+a_{k}}\setminus\overline{\ball}_{r-a_{k}^{-}}, \\  
\varrho_{k}\mathbb{E}_{r-d_{k}^{-},r-a_{k}^{-}}v_{j_{k}} + (1-\varrho_{k})\mathbb{E}_{r-d_{k}^{-},r-a_{k}^{-}}\mathbb{E}_{r-\ell,r}u&\; \\ \,\;\;\;\;\;\;\;\;\;\;\;\;\;\;\;\;\;\;\;\;\;\;\;\;\text{in}\;\ball_{r-a_{k}^{-}}\setminus\overline{\ball}_{r-d_{k}^{-}}, \\
\mathbb{E}_{r-\ell,r}u\;\;\;\;\;\;\;\;\;\;\;\;\;\text{in}\; \ball_{r-d_{k}^{-}}\setminus\overline{\ball}_{r-\ell}.
\end{cases} 
\end{align}
Our first claim is that, with $\mathbf{u}$ being defined in \eqref{eq:mathbfudef}, 
\begin{align}\label{eq:L1conv1}
\mathbf{v}_{j_{k}}\to \mathbf{u}\qquad\text{in}\;\lebe^{1}(\ball_{r+\ell}\setminus\overline{\ball}_{r-\ell};\R^{N})\;\;\text{as}\;\;k\to\infty. 
\end{align}
To see \eqref{eq:L1conv1}, we consider the single terms in the definition of $\mathbf{v}_{j_{k}}$ separately. We have 
\begin{align*}
\int_{\ball_{r+a_{k}}\setminus\overline{\ball}_{r-a_{k}^{-}}}|\mathbf{v}_{j_{k}}-\mathbf{u}|\dif x & \leq \int_{\ball_{r+a_{k}}\setminus\overline{\ball}_{r-a_{k}}}|u-v_{j_{k}}|\dif x + \int_{\ball_{r+a_{k}}\setminus\overline{\ball}_{r-a_{k}^{-}}}|u|\dif x \\ & + \int_{\ball_{r+a_{k}}\setminus\overline{\ball}_{r}}|\mathbb{E}_{r,r+\ell}u|\dif x + \int_{\ball_{r}\setminus\overline{\ball}_{r-a_{k}^{-}}}|\mathbb{E}_{r-\ell,r}u|\dif x \to 0
\end{align*}
as $k\to\infty$ because of \eqref{eq:RecoveryClosedness}, $u\in\lebe^{1}(\ball_{r+\ell}\setminus\ball_{r-\ell};\R^{N})$, $\E_{r,r+\ell}u\in\sobo^{1,q}(\ball_{r+\ell}\setminus\overline{\ball}_{r};\R^{N})$, $\E_{r-\ell,r}u\in\sobo^{1,q}(\ball_{r}\setminus\overline{\ball}_{r-\ell};\R^{N})$ and $\mathscr{L}^{1}(\ball_{r+a_{k}}\setminus\overline{\ball}_{r-a_{k}})\to 0$. We only treat the layer term in the upper annulus, the corresponding lower one being analogous. Successively applying the $\lebe^{1}$-stability of $\E$, 
\begin{align*}
\int_{\ball_{r+d_{k}}\setminus\overline{\ball}_{r+a_{k}}}|\mathbf{v}_{j_{k}}-\mathbf{u}|\dif x \leq c\int_{\ball_{r+d_{k}}\setminus\overline{\ball}_{r+a_{k}}}|v_{j_{k}}-u|+|u|+|\E_{r,r+\ell}u|\dif x \to 0,
\end{align*}
again by \eqref{eq:RecoveryClosedness}, $u,\E_{r,r+\ell}u\in\lebe^{1}(\ball_{r+\ell}\setminus\overline{\ball}_{r};\R^{N})$ and $\mathscr{L}^{n}(\ball_{r+d_{k}}\setminus\overline{\ball}_{r+a_{k}})\to 0$ as $k\to\infty$. In view of the definition of $\mathbf{u}$, this concludes the proof of \eqref{eq:L1conv1}. In order to apply Lemma~\ref{lem:LSC}, our second claim is that 
\begin{align}\label{eq:theredone1}
\sup_{k\in\mathbb{N}}\|\nabla\mathbf{v}_{j_{k}}\|_{\lebe^{1}(\ball_{r+\ell}\setminus\overline{\ball}_{r-\ell})}<\infty. 
\end{align}
To see this, it still suffices to argue for the upper annulus and once again split, using~\eqref{eq:gradientboundetta},
\begin{align*}
\int_{\ball_{r+\ell}\setminus\overline{\ball}_{r}}|\nabla\mathbf{v}_{j_{k}}|\dif x & \leq \int_{\ball_{r+\ell}\setminus\overline{\ball}_{r+d_{k}}}|\nabla\mathbb{E}_{r,r+\ell}u|\dif x \\ & \!\!\!\! \!\!\!\!+ c\int_{\ball_{r+d_{k}}\setminus\overline{\ball}_{r+a_{k}}}\left\vert \nabla(\varrho_{k}\mathbb{E}_{r+a_{k},r+d_{k}}v_{j_{k}}+(1-\varrho_{k})\mathbb{E}_{r+a_{k},r+d_{k}}\mathbb{E}_{r,r+\ell}u )\right\vert\dif x \\ 
& \!\!\!\!\!\!\!\!+ \int_{\ball_{r+a_{k}}\setminus\overline{\ball}_{r}}|\nabla v_{j_{k}}|\dif x =: \mathrm{IX}_{k}^{(1)}+\mathrm{IX}_{k}^{(2)}+\mathrm{IX}_{k}^{(3)}. 
\end{align*}
Since $\mathrm{E}_{r,r+\ell}u\in\sobo^{1,q}(\ball_{r+\ell}\setminus\overline{\ball}_{r};\R^{N})$, $\mathrm{IX}_{k}^{(1)}$ stays bounded in $k$, and by~\eqref{eq:MaximalBoundA} we see that $\mathrm{IX}_{k}^{(3)}$ stays bounded in $k$ (in fact, tends to zero as $k\to\infty$) too. Successively applying the $\lebe^{1}$-gradient stability of the trace preserving operator $\mathbb{E}$ in combination with~\eqref{eq:MaximalBoundA}, one can also see the boundedness of $\mathrm{IX}_{k}^{(2)}$, but for future reference we record an estimate that is stronger and entails the requisite estimate for~$\mathrm{IX}_{k}^{(2)}$. 
Namely, we now estimate 
\begin{align*}
\mathrm{X}_{k}:=\int_{\ball_{r+d_{k}}\setminus\overline{\ball}_{r+a_{k}}}&|\nabla (\varrho_{k}\mathbb{E}_{r+a_{k},r+d_{k}}v_{j_{k}} + (1-\varrho_{k})\mathbb{E}_{r+a_{k},r+d_{k}}\mathbb{E}_{r,r+\ell}u) |^{q}\dif x \\ 
& \leq c\int_{\ball_{r+d_{k}}\setminus\overline{\ball}_{r+a_{k}}}|\nabla (\varrho_{k}(\mathbb{E}_{r+a_{k},r+d_{k}}v_{j_{k}} -\mathbb{E}_{r+a_{k},r+d_{k}}\mathbb{E}_{r,r+\ell}u)) |^{q}\dif x \\ 
& + c\int_{\ball_{r+d_{k}}\setminus\overline{\ball}_{r+a_{k}}}|\nabla\mathbb{E}_{r+a_{k},r+d_{k}}\mathbb{E}_{r,r+\ell}u|^{q}\dif x =: \mathrm{X}_{k}^{(1)} + \mathrm{X}_{k}^{(2)}.
\end{align*}
The term $\mathrm{X}_{k}^{(2)}$ vanishes in the limit by the $\lebe^{q}$-gradient stability of $\mathbb{E}$ and the fact that $\mathbf{u}\in\sobo^{1,q}(\ball_{r+\ell}\setminus\overline{\ball}_{r-\ell};\R^{N})$. Toward the estimation of the term $\mathrm{X}_{k}^{(1)}$, we note that by linearity of the trace-preserving operator $\mathbb{E}_{r+a_{k},r+d_{k}}$ and~\eqref{eq:gradientboundetta}, 
\begin{align*}
\mathrm{X}_{k}^{(1)} & \leq c\int\limits_{\ball_{r+d_{k}}\setminus\overline{\ball}_{r+a_{k}}}\left\vert \frac{\mathbb{E}_{r+a_{k},r+d_{k}}(v_{j_{k}} -\mathbf{u})}{d_{k}-a_{k}}\right\vert^{q}\dif x \\ & + c\int\limits_{\ball_{r+d_{k}}\setminus\overline{\ball}_{r+a_{k}}}\left\vert \nabla(\mathbb{E}_{r+a_{k},r+d_{k}}(v_{j_{k}} -\mathbf{u}))\right\vert^{q}\dif x =: \mathrm{XI}_{k}^{(1)} + \mathrm{XI}_{k}^{(2)}.
\end{align*}
By linearity and the properties of the trace-preserving operators and our choice of $q$, the uniform comparability $(d_{k}-a_{k})\sim_{j_{0}}2^{-k}$ (cf.~\eqref{eq:unichoose}) implies 
\begin{align*}
\mathrm{XI}_{k}^{(1)} & \stackrel{\text{Lem.~\ref{lem:extensionoperator}~\ref{item:extension4}}}{\leq} c(d_{k}-a_{k})^{n-nq+q}\Big(\sup_{0<\delta\ll (d_{k}-a_{k})}\frac{1}{\delta}\int_{\ball_{r+a_{k}+\delta}\setminus\overline{\ball}_{r+a_{k}}}\left\vert \frac{v_{j_{k}} -\mathbf{u}}{d_{k}-a_{k}}\right\vert \Big)^{q}\\ 
& \;\;\;\;\;+ c(d_{k}-a_{k})^{n-nq+q}\Big(\sup_{0<\delta\ll (d_{k}-a_{k})}\frac{1}{\delta}\int_{\ball_{r+d_{k}}\setminus\overline{\ball}_{r+d_{k}-\delta}}\left\vert \frac{v_{j_{k}} -\mathbf{u}}{d_{k}-a_{k}}\right\vert \Big)^{q} \\ 
& \!\stackrel{\eqref{eq:Choose1D},\eqref{eq:Choose1B}}{\leq} c(d_{k}-a_{k})^{n-nq+q}\Big(\frac{\Theta_{k}(r+d_{k})-\Theta_{k}(r)}{d_{k}}\Big) \\
& \;\;\;\;\;+ c(d_{k}-a_{k})^{n-nq+q}\Big(\frac{\widetilde{\Theta}_{k}(r+2^{-j_{0}-k+1})-\widetilde{\Theta}_{k}(r)}{2^{-k}} \Big)\\
& \stackrel{\eqref{eq:ThetaboundChap3}}{\leq} c(d_{k}-a_{k})^{n-nq+q}\to 0
\end{align*}
as $k\to\infty$. Similarly, 
\begin{align*}
\mathrm{XI}_{k}^{(2)} & \stackrel{\text{Lem.~\ref{lem:extensionoperator}~\ref{item:extension4}}}{\leq} c(d_{k}-a_{k})^{n-nq+q}\Big(\sup_{0<\delta\ll (d_{k}-a_{k})}\frac{1}{\delta}\int_{\ball_{r+a_{k}+\delta}\setminus\overline{\ball}_{r+a_{k}}}\left\vert \nabla (v_{j_{k}} -\mathbf{u})\right\vert \dif x\Big)^{q}\\ 
& + c(d_{k}-a_{k})^{n-nq+q}\Big(\sup_{0<\delta\ll (d_{k}-a_{k})}\frac{1}{\delta}\int_{\ball_{r+d_{k}}\setminus\overline{\ball}_{r+d_{k}-\delta}}\left\vert \nabla (v_{j_{k}} -\mathbf{u})\right\vert \Big)^{q} \\ 
& \!\!\!\!\!\!\!\!\!\stackrel{\eqref{eq:Choose1D},\eqref{eq:Choose1B}}{\leq} c(d_{k}-a_{k})^{n-nq+q}\Big(\frac{\Theta_{k}(r+d_{k})-\Theta_{k}(r)}{d_{k}}\Big) \\
& \;\;\;\;\;+ c(d_{k}-a_{k})^{n-nq+q}\Big(\frac{\widetilde{\Theta}_{k}(r+2^{-j_{0}-k+1})-\widetilde{\Theta}_{k}(r)}{2^{-k}} \Big)\\
& \!\!\!\stackrel{\eqref{eq:ThetaboundChap3}}{\leq} c(d_{k}-a_{k})^{n-nq+q}\to 0
\end{align*}
as $k\to\infty$. Since an analogous argument applies to the lower annuli, we have 
\begin{align}\label{eq:jeffbuckley}
\limsup_{k\to\infty}\mathrm{XII}_{k}:= \limsup_{k\to\infty} \int_{(\ball_{r+d_{k}}\setminus\overline{\ball}_{r+a_{k}})\cup(\ball_{r-a_{k}^{-}}\setminus\overline{\ball}_{r-d_{k}^{-}})}|\nabla\mathbf{v}_{j_{k}}|\dif x = 0,
\end{align}
and this directly entails that $\lim_{k\to\infty}\mathrm{IX}_{k}^{(2)}=0$, whereby in particular the proof of ~\eqref{eq:L1conv1} and~\eqref{eq:theredone1} is complete. 

It is at this stage that quasiconvexity of $F$ enters  via lower semicontinuity: Since $\mathbf{u}\in\sobo^{1,q}(\ball_{r+\ell}\setminus\overline{\ball}_{r-\ell};\R^{N})$,~\eqref{eq:fatvdefine} implies that $\mathbf{v}_{j_{k}}\in\sobo_{\mathbf{u}}^{1,q}(\ball_{r+\ell}\setminus\overline{\ball}_{r-\ell};\R^{N})$. Therefore,~\eqref{eq:L1conv1} and~\eqref{eq:theredone1} imply by virtue of Lemma~\ref{lem:LSC}
\begin{align}\label{eq:LscL1}
\mathscr{F}[\mathbf{u};\ball_{r+\ell}\setminus\overline{\ball}_{r-\ell}] \leq \liminf_{k\to\infty}\mathscr{F}[\mathbf{v}_{j_{k}};\ball_{r+\ell}\setminus\overline{\ball}_{r-\ell}].
\end{align}
It is precisely here where we need to work the full annulus $\ball_{r+\ell}\setminus\overline{\ball}_{r-\ell}$; indeed, as the argument below necessitates to leave $v_{j}$ unchanged close to $\partial\!\ball_{r}$, the use of the lower semicontinuity theorem (and thus the need of fixed boundary values in the signed situation, see Figure~\ref{fig:conceptual}) forces us to work with a two-sided approximation by the very geometry of the annuli. 

We now conclude \eqref{eq:generation2}. By~\eqref{eq:LscL1} and the growth bound~\ref{item:H1}, we estimate with $\mathrm{XII}_{k}$ as in~\eqref{eq:jeffbuckley} and $\mathrm{XIII}_{k}:=\mathscr{L}^{n}((\ball_{r+d_{k}}\setminus\overline{\ball}_{r+a_{k}})\cup(\ball_{r-a_{k}^{-}}\setminus\overline{\ball}_{r-d_{k}^{-}}))$
\begin{align}\label{eq:almostdonefinal}
\begin{split}
\int_{\ball_{r+\ell}\setminus\overline{\ball}_{r-\ell}}F(\nabla\mathbf{u})\dif x &  \leq \liminf_{k\to\infty}\Big(\int_{\ball_{r+a_{k}}\setminus\overline{\ball}_{r-a_{k}}}F(\nabla v_{j_{k}})\dif x  \Big. \\ 
& \!\!\!\!\!\!\!\!\Big. + L(\mathrm{XII}_{k}+\mathrm{XIII}_{k})\Big. \\ & \!\!\!\!\!\!\!\!\Big. + \int_{(\ball_{r+\ell}\setminus\overline{\ball}_{r-\ell})\setminus (\ball_{r+d_{k}}\setminus\overline{\ball}_{r-d_{k}})} F(\nabla\mathbf{u})\dif x\Big) \\ 
& \!\!\!\!\!\!\!\!\leq \liminf_{k\to\infty}\Big(\int_{\ball_{r+a_{k}}\setminus\overline{\ball}_{r-a_{k}}}F(\nabla v_{j_{k}})\dif x\Big) \\ 
& \!\!\!\!\!\!\!\!+ \limsup_{k\to\infty}\Big(L(\mathrm{XII}_{k}+\mathrm{XIII}_{k}) - \int_{\ball_{r+d_{k}}\setminus\overline{\ball}_{r-d_{k}}}F(\nabla\mathbf{u})\dif x\Big) \\ 
& \!\!\!\!\!\!\!\!+ \int_{\ball_{r+\ell}\setminus\overline{\ball}_{r-\ell}}F(\nabla\mathbf{u})\dif x =: \mathrm{XIV}+\mathrm{XV}+\mathrm{XVI}. 
\end{split}
\end{align}
By~\eqref{eq:jeffbuckley}, $\limsup_{k\to\infty}\mathrm{XII}_{k}+\mathrm{XIII}_{k}=0$, and by the growth bound~\ref{item:H1} on $F$, we use $\mathbf{u}\in\sobo^{1,q}(\ball_{r+\ell}\setminus\overline{\ball}_{r-\ell};\R^{N})$ to infer
\begin{align*}
\limsup_{k\to\infty}\int_{\ball_{r+d_{k}}\setminus\overline{\ball}_{r-d_{k}}}|F(\nabla\mathbf{u})|\dif x  \leq \lim_{k\to\infty}L\int_{\ball_{r+d_{k}}\setminus\overline{\ball}_{r-d_{k}}}1+|\nabla\mathbf{u}|^{q}\dif x = 0.  
\end{align*}
Hence, $\mathrm{XV}=0$, and so~\eqref{eq:almostdonefinal} gives 
\begin{align*}
\liminf_{k\to\infty}\int_{\ball_{r+a_{k}}\setminus\ball_{r-a_{k}}}F(\nabla v_{j_{k}})\dif x \geq 0. 
\end{align*}
This is \eqref{eq:generation2}, and the proof is complete. 
\end{proof}
The following elementary observation, obtained by iterating the construction from Proposition~\ref{prop:goodminseqs}, will be required in Section~\ref{sec:MazurEuler}: 
\begin{corollary}\label{cor:allalong}
In the situation of Proposition~\ref{prop:goodminseqs}, let $\lambda\in\mathrm{RM}_{\mathrm{fin}}(\omega')$ be a weak*-limit of a suitable subsequence of $(|Dv_{j_{k}}|)$. Moreover, let $(\overline{\ball}_{r_{i}}(x_{i}))$ be a sequence of mutually disjoint closed balls in $\omega$ such that both 
\begin{align*}
\mathcal{M}Du(x_{i},r_{i})<\infty\;\;\;\text{and}\;\;\;\mathcal{M}\lambda(x_{i},r_{i})<\infty
\end{align*} hold for all $i\in\mathbb{N}$. Then there exists another generating sequence $(u_{l})\subset\mathscr{A}_{u_{0}}^{q}(\omega;\omega')$ for $\overline{\mathscr{F}}_{u_{0}}^{*}[u;\omega,\omega']$ such that for any $i\in\mathbb{N}$  
\begin{align}\label{eq:finallytraceincluded}
\text{there 
exists $l_{i}\in\mathbb{N}$ such that $\trace_{\partial\!\ball_{r_{i}}(x_{i})}(u_{l})=\trace_{\partial\!\ball_{r_{i}}(x_{i})}(u)$ holds for all $l\geq l_{i}$. }
\end{align}
\end{corollary} 
\begin{remark}[$1<p<\infty$]
Proposition~\ref{prop:goodminseqs} extends mutatis mutandis to the relaxed functionals $\overline{\mathscr{F}}_{u_{0}}[-;\omega,\omega']$ and maps $u\in\sobo^{1,p}(\omega';\R^{N})$ with $1<p\leq q<\frac{np}{n-1}$ such that $\overline{\mathscr{F}}_{u_{0}}[u;\omega,\omega']<\infty$ and $\mathcal{M}(|\nabla u|^{p}\mathscr{L}^{n}\mres\omega')(x_{0},r)<\infty$. Similarly as in~\cite{SchmidtPR}, $\lambda$ then needs to be taken as a weak*-limit of a suitable subsequence of $(|\nabla v_{j}|^{p}\mathscr{L}^{n}\mres\omega')$,  and in this situation it is customary to call $r>0$ a \emph{good radius of $\nabla u$ for $\overline{\mathscr{F}}_{u_{0}}[u;\omega,\omega']$} provided $\mathcal{M}(|\nabla u|^{p}\mathscr{L}^{n}\mres\omega')(x_{0},r)+\mathcal{M}\lambda(x_{0},r)<\infty$. Proposition~\ref{prop:goodminseqs}~\ref{item:FixRadius3} then translates to $\nabla u_{j_{k}}-\nabla v_{j_{k}}\to 0$ strongly in $\lebe^{p}(\omega';\R^{N\times n})$. Moreover, note that here as in Proposition~\ref{prop:goodminseqs}, the corresponding maximal conditions on $Du$ (or $|\nabla u|^{p}\mathscr{L}^{n}\mres\omega'$ if $p>1$) without $\mathcal{M}\lambda(x_{0},r)<\infty$ are in principle sufficient to conclude, but we do not need this improvement in the sequel.
\end{remark}
We conclude this section by discussing two scenarios in which the above proof simplifies. This, in particular, applies to the unsigned setting considered in \cite{SchmidtPR}:
\begin{remark}\label{rem:signed}
In the situation of integrands which are bounded from below, so satisfy 
\begin{align}\label{eq:signedsimplify}
c_{1}|z|^{p}-c_{2}\leq F(z)\leq c_{3}(1+|z|^{q})\;\;\text{for some $c_{1},c_{2},c_{3}>0$ and all $z\in\R^{N\times n}$},
\end{align}
the fifth step of the above proof trivialises and thereby entails an overall simplification; note that considering the map $\mathbf{u}$ then already becomes irrelevant in the first step. Indeed, in this situation, \eqref{eq:generation2} is a direct consequence of the lower bound of \eqref{eq:signedsimplify}. As such, the fifth part of the above proof solely serves to exclude potential concentration effects of $(F(\nabla v_{j_{k}}))$ along $\partial\!\ball_{r}(x_{0})$; note that, if we could not rule out such effects, such concentrations  might imply that the left-hand side of \eqref{eq:generation2} is strictly smaller than zero as $F$ is signed. This would imply that, by modifying a recovery sequence by fixing its boundary values along $\partial\!\ball_{r}(x_{0})$, additional mass is created in the limit and the modified sequence would fail to qualify as a recovery sequence.
\end{remark}
\begin{remark}\label{rem:convex}
The proof equally shortens if we do not assume \eqref{eq:signedsimplify} but suppose that $F$ is  convex. In this case, we find an affine linear function $h(z):=\langle \ell,z\rangle + b$ for some $\ell\in\R^{N\times n}$ and $b\in\R$ such that $h(z)\leq F(z)$ holds for all $z\in\R^{N\times n}$. In consequence, \eqref{eq:generation2} then immediately follows from 
\begin{align*}
\int_{\ball_{r+a_{k}}\setminus\ball_{r-a_{k}^{-}}}|\langle \ell,\nabla v_{j_{k}}\rangle +b|\dif x \leq  c \int_{\ball_{r+a_{k}}\setminus\ball_{r-a_{k}^{-}}}|\nabla v_{j_{k}}|\dif x + c \mathscr{L}^{n}(\ball_{r+a_{k}}\setminus\ball_{r-a_{k}^{-}}), 
\end{align*}
and the right-hand side vanishes in the limit $k\to\infty$ by virtue of~\eqref{eq:MaximalBoundA}. Subject to the assumptions of Proposition~\ref{prop:goodminseqs}, however, minorising $F$ by an affine-linear map is impossible; see Example~\ref{ex:ALminorisation}. 
\end{remark}
\section{Mean coercivity and properties of the relaxed functional}\label{sec:props}
This section serves to demonstrate the naturality of the $p$-strong quasiconvexity assumption underlying Theorems~\ref{thm:main1} and \ref{thm:main2}, the wealth of integrands satisfying this hypothesis and to record various properties of the relaxed functional used in the subsequent sections.
\subsection{Mean coercivity, existence of minimizers and examples}\label{sec:mcex}
In this subsection we address the definition of relaxed functionals for solid boundary values and the existence of minimizers. Before we embark on this matter, we briefly pause and discuss both the meaning of hypotheses \ref{item:H1}--\ref{item:H3} (or~\ref{item:H1},~\ref{item:H2p},~\ref{item:H3} if $p>1$, respectively) and several examples that also deal with signed integrands of $(p,q)$-growth. The following result is part of \cite[Thm.~1]{CK} that we, for the convenience of the reader, state in the notation used in the present paper. Note that it clarifies the meaning of our strong quasiconvexity hypotheses~\ref{item:H2} and~\ref{item:H2p}. 
\begin{proposition}\label{qc-coerciv}
Let $F \colon \R^{N \times n} \to \R$ be a continuous integrand satisfying the growth condition $| f(z)| \leq L \bigl(1+ |z|^{q} \bigr)$ for all $z \in \R^{N \times n}$,
where $L>0$ is a constant and $q \in [1,\infty )$ a fixed exponent. Then for a bounded open non-empty subset $\Omega$ of $\R^n$, $g \in \sobo^{1,q}( \R^n ; \R^N )$
and $p \in [1,q]$ the following are equivalent:
\begin{itemize}
\item[\emph{(i)}] There exists an increasing function $\theta \colon [0, \infty ) \to \R$ with $\theta (t) \to +\infty$ as $t \to +\infty$ such that we have 
\begin{align*}
\int_{\Omega} \! F( \nabla u) \, \dx \geq \theta \left( \int_{\Omega} \! | \nabla u|^{p} \, \dx \right)\qquad\text{for all}\;u \in \sobo_{g}^{1,q}( \Omega; \R^N ).
\end{align*}
\item[\emph{(ii)}] There exist constants $c_{1}>0$, $c_2 \in \R$ such that we have
\begin{align*}
\int_{\Omega} \! F( \nabla u) \, \dx \geq c_{1}\int_{\Omega} \! | \nabla u|^{p} \, \dx + c_{2}\qquad\text{for all}\;\sobo_{g}^{1,q}(\Omega;\R^{N}).
\end{align*}
\item[\emph{(iii)}] There exist $\ell > 0$ and $z_{0} \in \R^{N \times n}$ such that the integrand $F-\ell V_{p}$ is quasiconvex at $z_0$.
\end{itemize}
\end{proposition}
Let us say that an integrand $F$ satisfying one of the equivalent conditions (i)--(iii) is \emph{$\sobo^{1,p}$-coercive}.
In view of our definition of the relaxed functionals $\overline{\mathscr{F}}^{*}$ and $\overline{\mathscr{F}}$ below, it will be clear that their coercivity on Dirichlet classes $\sobo^{1,p}_{g}$ is
equivalent to $\sobo^{1,p}$-coercivity of the integrand $F$. The corresponding result in the case $p=1$
goes through the formulation of the Dirichlet condition in terms of \textit{solid boundary values} and requires (for an easy treatment)
the set $\Omega$ to be a Lipschitz domain.

The hypotheses \ref{item:H1}--\ref{item:H3} for the $p=1$-case (and \ref{item:H1},~\ref{item:H2p},~\ref{item:H3} for the $p>1$-case) based on quasiconvexity rather than the pointwise condition~\eqref{eq:pqgrowth} allow for many examples of nonconvex variational problems. It is certainly the case that the examples are harder to write down explicitly, but
the quasiconvexity based hypotheses are more natural and allow for a much richer structure. We refer the reader to \cite[Proposition 2.14]{GK1}
and \cite{CKK} for results that illustrate just how rich the structure of integrands satisfying our hypotheses really is, and now proceed to two examples that display key properties of integrands to which Theorems~\ref{thm:main1} and~\ref{thm:main2} apply. To this end, we require a preparatory lemma as follows:
\begin{lemma}\label{lem:prephom}
Let $p\geq 1$ and suppose that $F\colon\R^{N\times n}\to\R$ is a $p$-homogeneous function. If there exists a function $\theta\colon [0,\infty)\to \R$ such that $\theta(t_{0})>0$ for some $t_{0}\in (0,\infty)$ and we have 
\begin{align}\label{eq:phom1}
\int_{\ball_{1}(0)}F(\nabla\varphi)\dif x \geq \theta\Big(\int_{\ball_{1}(0)}|\nabla\varphi|^{p}\dif x  \Big)\qquad\text{for all}\;\varphi\in\sobo_{0}^{1,p}(\ball_{1}(0);\R^{N}),
\end{align}
then there exists $a>0$ such that 
\begin{align}\label{eq:phom2}
\int_{\ball_{1}(0)}F(\nabla\varphi)\dif x \geq a\int_{\ball_{1}(0)}|\nabla\varphi|^{p}\dif x \qquad\text{for all}\;\varphi\in\sobo_{0}^{1,p}(\ball_{1}(0);\R^{N}). 
\end{align}
\end{lemma}
\begin{proof} 
We set $a:=\frac{\theta(t_{0})}{t_{0}}$ so that $at_{0}=\theta(t_{0})$. Let $\varphi\in\sobo_{0}^{1,p}(\ball_{1}(0);\R^{N})\setminus\{0\}$. Choose 
\begin{align}\label{eq:lambdahom}
\lambda := \Big(\frac{t_{0}}{\|\nabla\varphi\|_{\lebe^{p}(\ball_{1}(0))}^{p}}\Big)^{\frac{1}{p}}\in (0,\infty).
\end{align}
As a consequence, using the $p$-homogeneity in the first step,
\begin{align*}
\lambda^{p}\int_{\ball_{1}(0)}F(\nabla\varphi)\dif x = \int_{\ball_{1}(0)}F(\nabla(\lambda\varphi))\dif x \stackrel{\eqref{eq:phom1}}{\geq} \theta(t_{0})=at_{0}\stackrel{\eqref{eq:lambdahom}}{=} a\lambda^{p}\|\nabla\varphi\|_{\lebe^{p}(\ball_{1}(0))}^{p}. 
\end{align*}
This is~\eqref{eq:phom2} for $\varphi\in\sobo_{0}^{1,p}(\ball_{1}(0);\R^{N})$, and since $F(0)=0$ by the $p$-homogeneity of $F$, this inequality also holds for $\varphi=0$. The proof is complete.  
\end{proof}  

\begin{example}\label{ex:ALminorisation}
For dimensions $n \geq 2$ and each pair of exponents $(p,q)$ satisfying $1 \leq p < q < \min \bigl\{ p+1, \tfrac{np}{n-1} \bigr\}$ we give
examples of integrands $F = F_{n,p,q} \colon \R^{n \times n} \to \R$ satisfying the hypotheses~\ref{item:H1}--\ref{item:H3} when $p=1$ and~\ref{item:H1},~\ref{item:H2p} and~\ref{item:H3} when $p>1$ such that
\begin{itemize}
\item[(1)] $F \mbox{ is not } \sobo^{1,q} \mbox{-coercive, that is, (i)--(iii) of Proposition \ref{qc-coerciv} all fail for } p=q$;
\item[(2)] $F$ is genuinely signed, that is,
  $$
  \liminf_{|z| \to \infty} \frac{F(z)}{|z|^{q}} < 0.
  $$
\end{itemize}

It follows in particular from (2) that $F$ cannot be minorised by an affine integrand on $\R^{n \times n}$ and it follows from (1) that
the treatment of the variational problem with integrand $F$ requires the formulation via the relaxed functional (or a stronger quasiconvexity assumption on $F$, such as $\sobo^{1,p}$-quasiconvexity, see \cite{SchmidtPR1} and \cite{CPK}).

We turn to the construction of $F$ and fix a dimension $n \geq 2$ and an exponent $1<q<\infty$. Then Korn's inequality states that
\begin{equation}\label{korn}
\int_{\ball_{1}(0)} \! | \nabla \varphi |^{q} \, \dx \leq K\int_{\ball_{1}(0)} \! \bigl| \nabla \varphi + (\nabla \varphi )^{\top} \bigr|^{q} \, \dx
\end{equation}  
holds for all $\varphi\in\hold^{\infty}_{c}( \ball_{1}(0);\R^n )$ for some constant $K=K(n,q)>0$ independent of $\varphi$. Now let $K$ denote the least such constant
and put
\begin{align*}
f(z) := K \bigl| z+z^{\top} \bigr|^{q} - |z|^{q}, \quad z \in \R^{n \times n}.
\end{align*}
The Korn inequality expresses quasiconvexity of $f$ at $0$ and it follows that the quasiconvex envelope $f^{\mathrm{qc}}$ of $f$ (cf.~\eqref{eq:defQCenvelope} and the discussion afterwards) is a real-valued and quasiconvex integrand. Because $f$ is $q$-homogeneous it follows easily that also its envelope is $q$-homogeneous. Moreover, since $K$ is the smallest constant
for which \eqref{korn} holds and $f^{\mathrm{qc}} \leq f$, we infer from Lemma~\ref{lem:prephom} that (i)--(iii) with $p=q$ of Proposition \ref{qc-coerciv} with $g\equiv 0$ all must fail for $f^{\mathrm{qc}}$; moreover, we have for any skew-symmetric matrix $z_{0}\in\R_{\mathrm{skew}}^{n\times n}\setminus\{0\}$ 
\begin{align}\label{eq:badminor}
\liminf_{|z|\to\infty}\frac{f^{\qc}(z)}{|z|^{q}}\leq\liminf_{\lambda\to\infty}\frac{f(\lambda z_{0})}{\lambda^{q}|z_{0}|^{q}}=-1.
\end{align}
Because $f^{\mathrm{qc}}$ is not $\hold^\infty$ we need to mollify it and show that we in this process do not reinstate the properties (i)--(iii) from Proposition~\ref{qc-coerciv}.
Let $\Phi_{\varepsilon}$ be a standard non-negative, radial and smooth mollifier on $\R^{n \times n}$. Because $f^{\mathrm{qc}}$ is real-valued a standard result
implies that it is (locally Lipschitz) continuous, and so by the $q$-homogeneity we get that $| f^{\mathrm{qc}}(z)| \leq L_{1}|z|^{q}$ for all $z$. Now the local Lipschitz bound
mentioned before can be quantified: For some constant $c=c(n,q)>0$ we have $| \bigl( f^{\mathrm{qc}}\bigr)^{\prime}(z)| \leq cL_{1}|z|^{q-1}$ holds
for almost all $z$ (cf.~\eqref{eq:DERIVBOUND}\emph{ff}.). Now by routine estimations
$$
f^{\mathrm{qc}}(z) \geq \bigl( \Phi_{\varepsilon} \ast f^{\mathrm{qc}}\bigr) (z) - c_{1}\varepsilon \bigl( |z|^{q-1}+1 \bigr)
$$
holds for all $z \in \R^{n \times n}$, where $c_1 >0$ is a constant. Fix $p \in \bigl( \max \{ q-1,(1-\tfrac{1}{n})q \} , q \bigr) \cap [1,\infty )$ and
define for $\varepsilon$, $\ell  > 0$
$$
F(z) := \bigl( \Phi_{\varepsilon} \ast f^{\mathrm{qc}}\bigr) (z) + \ell V_{p}(z) , \quad z \in \R^{n \times n}.
$$
Then $F$ satisfies \ref{item:H1}--\ref{item:H3} when $p=1$ and \ref{item:H1},~\ref{item:H2p},~\ref{item:H3} for $p>1$. However, because $q-1<p$ we may use Lemma~\ref{lem:Efunction}~\ref{item:EfctA}
to estimate that, for some $c_{2}=c_{2}(p,q)>0$, 
\begin{align}\label{eq:boundesquebelow}
f^{\mathrm{qc}}(z) \geq F(z) - (\ell +c_{2}\varepsilon ) V_{p}(z) - c_{2}\varepsilon
\end{align}
holds for all $z$. But then $\sobo^{1,q}$-coercivity of $F$ would imply the same for $f^{\mathrm{qc}}$, hence $F$ cannot be $\sobo^{1,q}$-coercive for any $\varepsilon$, $\ell > 0$, thereby establishing the property (1). In turn, property (2) follows easily from~\eqref{eq:badminor} and~\eqref{eq:boundesquebelow}. Still, the integrands $F$ satisfy the requirements of Theorems~\ref{thm:main1} and~\ref{thm:main2} and so the partial $\hold^{\infty}$-regularity of relaxed minimizers will follow, while not being accessible by previously available results  \cite{SchmidtPR0,SchmidtPR} in the unsigned case for the exponent range $1<p\leq q<p+\frac{\min\{2,p\}}{2n}$; also see Section~\ref{sec:schmidt} below. 
\end{example}

\begin{example}\label{ex:wealthSQC}
Integrands that are connected with problems from nonlinear elasticity (cf.~\textsc{Ball} et al.~\cite{Ball1,Ball2,BM} and~\textsc{Marcellini}~\cite{Marcellini1}) and in view of partial regularity have been explicitely addressed by~\textsc{Fusco \& Hutchinson}~\cite{FuHo2,FuHo3} and~\textsc{Schmidt}~\cite[Sec.~3]{SchmidtPR} are given by 
\begin{align*}
&F(z):=\big(1+|z|^{2}\big)^{\frac{p}{2}}+f(\det(z)),\qquad z\in\R^{n\times n}, \\ 
&G(z):=\big(1+|z|^{2}\big)^{\frac{p}{2}}+g(\mathrm{Cof}(z)),\qquad z\in\R^{n\times n},
\end{align*}
where $f\in\hold^{\infty}(\R)$, $g\in\hold^{\infty}(\R^{n\times n})$ are non-trivial convex functions with $0\leq f(t)\leq L(1+|t|^{s_{1}})$, $0\leq g(t)\leq L(1+|t|^{s_{2}})$ for all $t\in\R$ and some $s_{1},s_{2}\geq 1$. In this situation, letting $1< p\leq q<\min\{\frac{np}{n-1},p+1\}$ and $s_{1}=\frac{q}{n}$, $s_{2}=\frac{q}{n-1}$ (whereby $q\geq n$ or $q\geq n-1$ by convexity of $f$ or $g$), the integrands satisfy hypotheses~ \ref{item:H1},~\ref{item:H2p},~\ref{item:H3}. Here, $F$ and $G$ are non-negative, and then the advancement of Theorems~\ref{thm:main1} and~\ref{thm:main2} in view of~\cite{SchmidtPR0,SchmidtPR1} is the boost from the exponent regime~$1<p<q<p+\frac{\min\{2,p\}}{2n}$ to $1\leq p \leq q <\min\{\frac{np}{n-1},p+1\}$. This proves particularly relevant for $F$ when $(n-1)<p<n$, where the associated multiple integrals are not $\sobo^{1,p}$-quasiconvex~\cite{BM}. Using the specific structure of the integrands, improved exponent ranges are available~\cite{FuHo2,FuHo3} with $p\geq 2$ in low dimensions, while these methods do not extend to the setting of Theorem~\ref{thm:main2}. For $p=n-1$ and $q=n$, $G$ still proves $\sobo^{1,n-1}$-quasiconvex, and then the partial regularity follows from a recent borderline result~\cite{CPK}. However, this quasiconvexity property is lost when $p$ is lowered; it is then Theorem~\ref{thm:main2} that allows for a new exponent range for partial regularity.  
\end{example}
We turn to the definition of the relaxed functionals with solid boundary values and hence let $q\geq 1$. Here we focus directly on relaxations for weak(*)-convergence, whereas a pure $\lebe^{p}$-approach to relaxations is possible as well; see the appendix, Section~\ref{sec:LpApproach}. Throughout this and the following subsections, let $\Omega\subset\R^{n}$ be open and bounded with Lipschitz boundary. Given $v\in\sobo^{1,q}(\Omega;\R^{N})$, choose an open and bounded set $\Omega'$ with $\Omega\Subset\Omega'$ and pick an arbitrary $u_{0}\in\sobo^{1,q}(\Omega';\R^{N})$ with $\trace_{\partial\Omega}(u_{0})=\trace_{\partial\Omega}(v)$ $\mathscr{H}^{n-1}$-a.e. on $\partial\Omega$. As $\sobo^{1-1/q,q}(\partial\Omega;\R^{N})$ is the interior trace space of $\sobo^{1,q}(\Omega';\R^{N})$ along $\partial\Omega$ (recall our convention $\sobo^{0,1}(\partial\Omega;\R^{N}):=\lebe^{1}(\partial\Omega;\R^{N})$ for $q=1$), such a map $u_{0}$ exists. For $u\in\bv(\Omega;\R^{N})$, we define
\begin{align}\label{eq:bfUdef}
&\mathbf{u}:=\begin{cases} 
u&\;\text{in}\;\Omega, \\ 
u_{0}&\;\text{in}\;\Omega'\setminus\overline{\Omega},
\end{cases}
\end{align}
and put, with $\overline{\mathscr{F}}_{u_{0}}^{*}[\mathbf{u};\Omega,\Omega']$ as in~\eqref{eq:bdryrelaxed1}, 
\begin{align}\label{eq:relaxedviaboundaryvalues}
\overline{\mathscr{F}}_{v}^{*}[u;\Omega]:=\overline{\mathscr{F}}_{u_{0}}^{*}[\mathbf{u};\Omega,\Omega']-\mathscr{F}[u_{0};\Omega'\setminus\overline{\Omega}], 
\end{align}
where the underlying integrand $F\in\hold(\R^{N\times n})$ is supposed to satisfy \ref{item:H1}; note that if $q=1$, all of the following is classical. If $1<p\leq q<\infty$ and $u\in\sobo^{1,p}(\Omega;\R^{N})$, we then define analogously 
\begin{align}\label{eq:relaxoviaboundaryvalues}
\overline{\mathscr{F}}_{v}[u;\Omega]:=\overline{\mathscr{F}}_{u_{0}}[\mathbf{u};\Omega,\Omega']-\mathscr{F}[u_{0};\Omega'\setminus\overline{\Omega}]. 
\end{align}
We proceed to establish that the functionals defined by~\eqref{eq:relaxedviaboundaryvalues} or~\eqref{eq:relaxoviaboundaryvalues} are well-defined. Here we focus exclusively on the functionals $\overline{\mathscr{F}}_{v}^{*}[-;\Omega]$ on $\bv(\Omega;\R^{N})$; the corresponding results for the functionals $\overline{\mathscr{F}}_{v}[-;\Omega]$ on $\sobo^{1,p}(\Omega;\R^{N})$ hold true with the obvious modifications.
\begin{lemma}\label{lem:welldefinedbdryvalues}
Let $\Omega\subset\R^{n}$ be open and bounded with Lipschitz boundary and $\Omega'\subset\R^{n}$ be open and bounded with $\Omega\Subset\Omega'$. Let $1\leq q < \frac{n}{n-1}$ and suppose that $u_{0},v_{0}\in\sobo^{1,q}(\Omega';\R^{N})$ satisfy $\trace_{\partial\Omega}(u_{0})=\trace_{\partial\Omega}(v_{0})$ $\mathscr{H}^{n-1}$-a.e. on $\partial\Omega$. Define, for $u\in\bv(\Omega;\R^{N})$, 
\begin{align*}
\mathbf{u}^{(1)}:=\begin{cases} u&\;\text{in}\;\Omega,\\ u_{0}&\;\text{in}\;\Omega'\setminus\overline{\Omega}
\end{cases}\;\;\;\text{and}\;\;\;\mathbf{u}^{(2)}:=\begin{cases} u&\;\text{in}\;\Omega,\\ v_{0}&\;\text{in}\;\Omega'\setminus\overline{\Omega}.
\end{cases}
\end{align*} 
Then we have for any $F\in\hold(\R^{N\times n})$ with  \emph{\ref{item:H1}}
\begin{align}\label{eq:independence}
\overline{\mathscr{F}}_{u_{0}}^{*}[\mathbf{u}^{(1)};\Omega,\Omega']-\mathscr{F}[u_{0};\Omega'\setminus\overline{\Omega}] = \overline{\mathscr{F}}_{v_{0}}^{*}[\mathbf{u}^{(2)};\Omega,\Omega']-\mathscr{F}[v_{0};\Omega'\setminus\overline{\Omega}]
\end{align}
and so $\overline{\mathscr{F}}_{u_{0}}^{*}[\mathbf{u};\Omega,\Omega']-\mathscr{F}[u_{0};\Omega'\setminus\overline{\Omega}]$ as in \eqref{eq:relaxedviaboundaryvalues} only depends on $u_{0}$ via its traces along $\partial\Omega$.
\end{lemma}
\begin{proof} Choose two generating sequences $(u_{j})\subset\mathscr{A}_{u_{0}}^{q}(\Omega,\Omega')$, $(v_{j})\subset\mathscr{A}_{v_{0}}^{q}(\Omega,\Omega')$ such that $u_{j}\stackrel{*}{\rightharpoonup} \mathbf{u}^{(1)}$ and $v_{j}\stackrel{*}{\rightharpoonup} \mathbf{u}^{(2)}$ in $\bv(\Omega';\R^{N})$ together with
\begin{align*}
\overline{\mathscr{F}}_{u_{0}}^{*}[\mathbf{u}^{(1)};\Omega,\Omega'] = \lim_{j\to\infty}\int_{\Omega'}F(\nabla u_{j})\dif x\;\;\;\text{and}\;\;\;\overline{\mathscr{F}}_{v_{0}}^{*}[\mathbf{u}^{(2)};\Omega,\Omega'] = \lim_{j\to\infty}\int_{\Omega'}F(\nabla v_{j})\dif x.
\end{align*}
Then the sequences defined by 
\begin{align*}
w_{j}:=\begin{cases} u_{j}&\;\text{in}\;\Omega,\\ v_{0}&\;\text{in}\;\Omega'\setminus\overline{\Omega}
\end{cases}\;\;\;\text{and}\;\;\;z_{j}:=\begin{cases} v_{j}&\;\text{in}\;\Omega,\\ u_{0}&\;\text{in}\;\Omega'\setminus\overline{\Omega},
\end{cases}
\end{align*}
satisfy $w_{j}\in\mathscr{A}_{v_{0}}^{q}(\Omega,\Omega')$, $z_{j}\in\mathscr{A}_{u_{0}}^{q}(\Omega,\Omega')$ as well as $w_{j}\stackrel{*}{\rightharpoonup} \mathbf{u}^{(2)}$ and $z_{j}\stackrel{*}{\rightharpoonup} \mathbf{u}^{(1)}$ in $\bv(\Omega';\R^{N})$. By definition of $\overline{\mathscr{F}}_{u_{0}}^{*}[-;\Omega,\Omega']$,
\begin{align*}
\overline{\mathscr{F}}_{u_{0}}^{*}&[\mathbf{u}^{(1)};\Omega,\Omega']   \leq \liminf_{j\to\infty}\int_{\Omega'}F(\nabla z_{j})\dif x \\ 
& = \liminf_{j\to\infty}\Big(\int_{\Omega}F(\nabla v_{j})\dif x +\int_{\Omega'\setminus\overline{\Omega}}F(\nabla v_{0})\dif x\Big)+ \int_{\Omega'\setminus\overline{\Omega}}F(\nabla u_{0})-F(\nabla v_{0})\dif x \\
& = \overline{\mathscr{F}}_{v_{0}}^{*}[\mathbf{u}^{(2)};\Omega,\Omega']+ \int_{\Omega'\setminus\overline{\Omega}}F(\nabla u_{0})-F(\nabla v_{0})\dif x \\ 
& \leq \liminf_{j\to\infty}\int_{\Omega'}F(\nabla w_{j})\dif x + \int_{\Omega'\setminus\overline{\Omega}}F(\nabla u_{0})-F(\nabla v_{0})\dif x \\
& = \overline{\mathscr{F}}_{u_{0}}^{*}[\mathbf{u}^{(1)};\Omega,\Omega']. 
\end{align*} 
This yields \eqref{eq:independence}, and the proof is complete. 
\end{proof}
We record the following extension of~\eqref{eq:relaxedviaboundaryvalues} to maps $v\in\{w\in\bv(\Omega;\R^{N})\colon\;\trace_{\partial\Omega}(w)\in\sobo^{1-1/q,q}(\partial\Omega;\R^{N})\}$ which proves crucial for the localisation approaches later on: 
\begin{corollary}\label{cor:welldefinedbdryvalues1}
Let $1\leq q<\frac{n}{n-1}$, $\Omega\subset\R^{n}$ be open and bounded with Lipschitz boundary and $\Omega'\subset\R^{n}$ be open and bounded with $\Omega\Subset\Omega'$. If $v\in\bv(\Omega;\R^{N})$ is such that $\trace_{\partial\Omega}(v)\in\sobo^{1-1/q,q}(\partial\Omega;\R^{N})$ and $u_{0}\in\sobo^{1,q}(\Omega';\R^{N})$ is such that $\trace_{\partial\Omega}(v)=\trace_{\partial\Omega}(u_{0})$, then 
\begin{align}\label{eq:relaxedviaboundaryvalues1}
\overline{\mathscr{F}}_{v}^{*}[u;\Omega]:=\overline{\mathscr{F}}_{u_{0}}^{*}[\mathbf{u};\Omega,\Omega']-\mathscr{F}[u_{0};\Omega'\setminus\overline{\Omega}] 
\end{align}
is well-defined, where $\mathbf{u}$ for $u\in\bv(\Omega;\R^{N})$ is given as in \eqref{eq:bfUdef}.   In particular, the left-hand side of \eqref{eq:relaxedviaboundaryvalues1} is independent of the specific extension $u_{0}$. 
\end{corollary}
The corollary directly follows from Lemma~\ref{lem:welldefinedbdryvalues} and the existence of such a map $u_{0}$. In the situation of Corollary~\ref{cor:welldefinedbdryvalues1}, any sequence $(u_{j})\subset \sobo_{u_{0}}^{1,q}(\Omega;\R^{N})$ such that $(\mathbf{u}_{j})$ is generating for $\overline{\mathscr{F}}_{u_{0}}^{*}[\mathbf{u};\Omega,\Omega']$ in the sense of~\eqref{eq:bdryrelaxed1}\emph{ff}. will be called \emph{generating} for $\overline{\mathscr{F}}_{v}^{*}[u;\Omega]$. Equally, given $x_{0}\in\Omega$, we call $0<r<\mathrm{dist}(x_{0},\partial\Omega)$ a \emph{good radius of $Du$ for $\overline{\mathscr{F}}_{v}^{*}[u;\Omega]$} provided it is a good radius of $D\mathbf{u}$ for $\overline{\mathscr{F}}_{u_{0}}^{*}[\mathbf{u};\Omega,\Omega']$ in the sense of Proposition~\ref{prop:goodminseqs}\emph{ff.}.
This terminology canonically carries over to functionals $\overline{\mathscr{F}}[-;\Omega]$ on $\sobo^{1,p}(\Omega;\R^{N})$.

Lemma~\ref{lem:LSC} then immediately yields that the relaxed functionals as introduced above are extensions of the original functionals $\mathscr{F}[-;\Omega]$ by lower semicontinuity in the following sense: 
\begin{lemma}\label{lem:consistency}
Let $1\leq p\leq q < \frac{np}{n-1}$ and suppose that $F\in\hold(\R^{N\times n})$ satisfies~\emph{\ref{item:H1}} and
\begin{enumerate}
\item \emph{\ref{item:H2}} if $p=1$, and let $v\in\bv(\Omega;\R^{N})$ satisfy $\trace_{\partial\Omega}(v)\in\sobo^{1-1/q,q}(\partial\Omega;\R^{N})$. Then $\overline{\mathscr{F}}_{v}^{*}[-;\Omega]$ is sequentially weak*-lower semicontinuous on $\bv(\Omega;\R^{N})$. 
\item \emph{\ref{item:H2p}} if $p>1$, and let $v\in\sobo^{1,p}(\Omega;\R^{N})$ satisfy $\trace_{\partial\Omega}(v)\in\sobo^{1-1/q,q}(\partial\Omega;\R^{N})$. Then $\overline{\mathscr{F}}_{v}[-;\Omega]$ is sequentially weak*-lower semicontinuous on $\sobo^{1,p}(\Omega;\R^{N})$. 
\end{enumerate}
Moreover, we have  
\begin{align}\label{eq:coincidenceconsistency}
\begin{split}
\mathscr{F}[u;\Omega]=\overline{\mathscr{F}}_{v}^{*}[u;\Omega]\qquad\text{for all}\;u\in\sobo_{\overline{v}}^{1,q}(\Omega;\R^{N})\;\text{if}\;p=1, \\ 
\mathscr{F}[u;\Omega]=\overline{\mathscr{F}}_{v}[u;\Omega]\qquad\text{for all}\;u\in\sobo_{\overline{v}}^{1,q}(\Omega;\R^{N})\;\text{if}\;p>1
\end{split}
\end{align}
whenever $\overline{v}\in\sobo^{1,q}(\Omega;\R^{N})$ satisfies $\trace_{\partial\Omega}(\overline{v})=\trace_{\partial\Omega}(v)$ $\mathscr{H}^{n-1}$-a.e. on $\partial\Omega$. 
\end{lemma}
We now pass to the existence of minimizers:
\begin{proposition}[Existence of minimizers]\label{prop:existenceminimizers}
Let $\Omega\subset\R^{n}$ be open and bounded with Lipschitz boundary $\partial\Omega$. Given $1\leq p \leq q < \frac{np}{n-1}$, the following hold:
\begin{enumerate}
\item\label{item:existence1} Let $p=1$. If $F\in\hold(\R^{N\times n})$ satisfies \emph{\ref{item:H1}} and \emph{\ref{item:H2}} and $v\in\bv(\Omega;\R^{N})$ satisfies $\trace_{\partial\Omega}(v)\in\sobo^{1-1/q,q}(\partial\Omega;\R^{N})$, then $\overline{\mathscr{F}}_{v}^{*}[-;\Omega]$ possesses a minimizer on $\bv(\Omega;\R^{N})$. 
\item\label{item:existence2} Let $p>1$. If $F\in\hold(\R^{N\times n})$ satisfies \emph{\ref{item:H1}} and \emph{\ref{item:H2p}} and $v\in\sobo^{1,p}(\Omega;\R^{N})$ satisfies $\trace_{\partial\Omega}(v)\in\sobo^{1-1/q,q}(\partial\Omega;\R^{N})$, then $\overline{\mathscr{F}}_{v}[-;\Omega]$ possesses a minimizer on $\sobo_{v}^{1,p}(\Omega;\R^{N})$.
\end{enumerate}
\end{proposition}
\begin{proof}
We confine ourselves to \ref{item:existence1} as \ref{item:existence2} is fully analogous. Let $\Omega'\subset\R^{n}$ be open and bounded with $\Omega\Subset\Omega'$ and $u_{0}\in\sobo_{0}^{1,q}(\Omega';\R^{N})$ be such that $\trace_{\partial\Omega}(u_{0})=\trace_{\partial\Omega}(v)$. Then $\overline{\mathscr{F}}_{u_{0}}^{*}[u_{0};\Omega,\Omega']<\infty$ and so $m:=\inf_{u\in\bv(\Omega;\R^{N})}\overline{\mathscr{F}}_{v}^{*}[u;\Omega]<\infty$. To see that $m>-\infty$, let $w\in\mathscr{A}_{u_{0}}^{q}(\Omega,\Omega')$. Then the strong $1$-quasiconvexity from~\ref{item:H2} implies 
\begin{align}\label{eq:OKahn}
\ell_{1}\int_{\Omega'}V(\nabla w)\dif x & \leq \int_{\Omega'}F(\nabla w)-F(0)\dif x 
\end{align}
so that $-\infty<m<\infty$. Let $(u_{j})\subset\bv(\Omega;\R^{N})$ be such that $\overline{\mathscr{F}}_{v}^{*}[u_{j};\Omega]\to m$ and so, passing to a subsequence and adopting the notation from \eqref{eq:bfUdef}, we may assume
\begin{align}\label{eq:gnabry1}
|\overline{\mathscr{F}}_{u_{0}}^{*}[\mathbf{u}_{j};\Omega,\Omega']-\mathscr{F}[u_{0};\Omega'\setminus\overline{\Omega}]-m|<2^{-j}\qquad \text{for all}\;j\in\mathbb{N}.
\end{align}
For any $j\in\mathbb{N}$ we find a sequence $(\mathbf{u}_{j}^{k})\subset\mathscr{A}_{u_{0}}^{q}(\Omega,\Omega')$ such that $\mathbf{u}_{j}^{k}\stackrel{*}{\rightharpoonup}\mathbf{u}_{j}$ in $\bv(\Omega';\R^{N})$ as $k\to\infty$ and, for all $k\in\mathbb{N}$,
\begin{align}\label{eq:gnabry2}
 \int_{\Omega'}F(\nabla\mathbf{u}_{j}^{k})\dif x \leq \overline{\mathscr{F}}_{u_{0}}^{*}[\mathbf{u}_{j};\Omega,\Omega']+2^{-k}.
\end{align}
By hypotheses~\ref{item:H1} and~\ref{item:H2}, we conclude similarly as in~\eqref{eq:OKahn} 
\begin{align}\label{eq:COERC}
\begin{split}
\ell_{1}\int_{\Omega'}V(\nabla\mathbf{u}_{j}^{k})\dif x  
& \stackrel{\eqref{eq:OKahn},\,\eqref{eq:gnabry2}}{\leq} \overline{\mathscr{F}}_{u_{0}}^{*}[\mathbf{u}_{j};\Omega,\Omega'] + |F(0)|\mathscr{L}^{n}(\Omega') + 2^{-k}\\ & \stackrel{\eqref{eq:gnabry1},\,\text{\ref{item:H1}}}{\leq}  L\mathscr{L}^{n}(\Omega')+\mathscr{F}[u_{0};\Omega'\setminus\overline{\Omega}]+2^{-k}+2^{-j}+m, 
\end{split}
\end{align} 
and so $(\nabla\mathbf{u}_{j}^{k})$ is bounded in $\lebe^{1}(\Omega';\R^{N\times n})$ by Lemma~\ref{lem:Efunction}. Since $\mathbf{u}_{j}^{k}=u_{0}$ on $\Omega'\setminus\overline{\Omega}$ for all $j,k\in\mathbb{N}$, we then conclude that $(\mathbf{u}_{j}^{k})$ is bounded in $\sobo_{0}^{1,1}(\Omega';\R^{N})$. By the usual weak*-compactness in $\bv$, a diagonal sequence $(\mathbf{u}_{j}^{k(j)})$ converges  in the weak*-sense to some $\mathbf{u}\in\bv(\Omega';\R^{N})$. Thus, letting $u:=\mathbf{u}|_{\Omega}$, the definition  of $\overline{\mathscr{F}}_{v}^{*}[-;\Omega]$ yields
\begin{align*}
\overline{\mathscr{F}}_{v}^{*}[u;\Omega] & \leq \liminf_{j\to\infty}\int_{\Omega'}F(\nabla\mathbf{u}_{j}^{k(j)})\dif x -\mathscr{F}[u_{0};\Omega'\setminus\overline{\Omega}]\\
& \!\!\stackrel{\eqref{eq:gnabry2}}{\leq} \liminf_{j\to\infty}(\overline{\mathscr{F}}_{u_{0}}^{*}[\mathbf{u}_{j};\Omega,\Omega'] -\mathscr{F}[u_{0};\Omega'\setminus\overline{\Omega}]+2^{-k(j)})\stackrel{\eqref{eq:gnabry1}}{=}m. 
\end{align*}
The proof is complete. 
\end{proof}
\begin{remark}
In the setting of \ref{item:existence2}, any minimizer has the same boundary values as $v$ along $\partial\Omega$, whereas this needs not be the case in the setting of \ref{item:existence1}; note that the trace operator on $\sobo^{1,p}(\Omega;\R^{N})$ with $p>1$ is continuous for weak convergence in $\sobo^{1,p}(\Omega;\R^{N})$ whereas the trace operator on $\bv(\Omega;\R^{N})$ is not for weak*-convergence in $\bv(\Omega;\R^{N})$. This is seen easiest by noting that, if $p=q=1$, $N=1$ and $F(z)=V(z)+1$ is the usual minimal surface integrand, then $\overline{\mathscr{F}}_{v}^{*}[u;\Omega]$ is given by \eqref{eq:lingrowth} and so the classical counterexamples to the attainment of the boundary values apply (cf.~\textsc{Finn} \cite{Finn} and  \textsc{Santi} \cite{Santi}); for genuinely vectorial examples, see \textsc{Beck \& Schmidt} \cite[Thm.~1.5]{BeckSchmidt}. 
\end{remark}
\subsection{Properties of the relaxed functional}
Throughout this subsection, we suppose that $F\in\hold(\R^{N\times n})$ is quasiconvex and satisfies \ref{item:H1} with $1 \leq p \leq q < \frac{np}{n-1}$. As in the original proof of \textsc{Evans}~\cite{Ev1}, the Caccioppoli inequality of the second kind to be addressed in Section~\ref{sec:Cacc}  requires a localisation procedure. A similar issue has been addressed by \textsc{Schmidt} \cite[Lem.~7.10]{SchmidtPR} in the unsigned case. It is here, however, that our framework of signed integrands and the slightly different relaxations require some more features:
\begin{lemma}[Local finiteness and additivity radii]\label{lem:additivityradii}
Let $\Omega\subset\R^{n}$ be open and bounded with Lipschitz boundary $\partial\Omega$. Let $1=p\leq q<\frac{n}{n-1}$ and suppose that $u\in\bv(\Omega;\R^{N})$ and $v\in\bv(\Omega;\R^{N})$ are such that $\trace_{\partial\Omega}(v)\in\sobo^{1-\frac{1}{q},q}(\partial\Omega;\R^{N})$ and $\overline{\mathscr{F}}_{v}^{*}[u;\Omega]<\infty$. Then we already have $\overline{\mathscr{F}}_{v}^{*}[u;\Omega]\in (-\infty,\infty)$. Moreover, whenever $x_{0}\in\Omega$, then for $\mathscr{L}^{1}$-a.e. radius $r\in(0,\mathrm{dist}(x_{0},\partial\Omega))$ with $\mathcal{M}Du(x_{0},r)<\infty$ the following hold: 
\begin{enumerate}
\item\label{item:fin1} The \emph{relaxed functional with its own boundary values} is finite on $\ball_{r}(x_{0})$: 
\begin{align}\label{eq:twosidedfiniteness}
-\infty<\overline{\mathscr{F}}_{u}^{*}[u;\ball_{r}(x_{0})]<\infty. 
\end{align}
\item\label{item:fin2} The relaxed functional satisfies the following \emph{additivity property}: If $w\in\bv(\Omega\setminus\overline{\ball}_{r}(x_{0});\R^{N})$ satisfies $\trace_{\partial\!\ball_{r}(x_{0})}^{-}(w)=\trace_{\partial\!\ball_{r}(x_{0})}^{+}(u)$ and $\trace_{\partial\Omega}(w)=\trace_{\partial\Omega}(v)$ $\mathscr{H}^{n-1}$-a.e. on $\partial\!\ball_{r}(x_{0})$ or $\partial\Omega$, respectively, then 
\begin{align}\label{eq:additivity}
\overline{\mathscr{F}}_{v}^{*}[u;\Omega] =\overline{\mathscr{F}}_{u}^{*}[u;\ball_{r}(x_{0})] +\overline{\mathscr{F}}_{w}^{*}[u;\Omega\setminus\overline{\ball}_{r}(x_{0})].
\end{align}
Any radius $r$ with \eqref{eq:additivity} will be referred to as an \emph{additivity radius} for $\overline{\mathscr{F}}^{*}[u;-]$.
\item\label{item:fin3} Every $\bv$-minimizer $u\in\bv(\Omega;\R^{N})$ of $\overline{\mathscr{F}}_{v}^{*}[-;\Omega]$ is a local $\bv$-minimizer of $\overline{\mathscr{F}}^{*}$ for compactly supported variations. Specifically, for any $x_{0}\in\Omega$ and $\mathscr{L}^{1}$-a.e. $r\in(0,\mathrm{dist}(x_{0},\partial\Omega))$, $u$ is a $\bv$-minimizer of $\overline{\mathscr{F}}_{u}^{*}[-;\ball_{r}(x_{0})]$ for compactly supported variations.
\end{enumerate}
The same holds true with the obvious modifications for the functionals $\overline{\mathscr{F}}$ if $1<p\leq q<\frac{np}{n-1}$.
\end{lemma}
\begin{proof}
In the following, we assume all balls to be centered at $x_{0}$. Since $r\in(0,\mathrm{dist}(x_{0},\partial\Omega))$ is supposed to satisfy $\mathcal{M}Du(x_{0},r)<\infty$, we have $\trace_{\partial\!\ball_{r}}^{-}(u)=\trace_{\partial\!\ball_{r}}^{+}(u)\in\sobo^{1-1/q,q}(\partial\!\ball_{r};\R^{N})$ by~\eqref{eq:traceequal} and Corollary~\ref{cor:Fubini}, and so there exists an open and bounded domain $\Omega'\subset\R^{n}$, $\Omega\Subset\Omega'$ and $u_{0}\in\sobo_{0}^{1,q}(\Omega';\R^{N})$ such that $\trace_{\partial\!\ball_{r}}(u-u_{0})=0$ $\mathscr{H}^{n-1}$-a.e. on $\partial\!\ball_{r}$ and $\trace_{\partial\Omega}(v-u_{0})=0$ $\mathscr{H}^{n-1}$-a.e. on $\partial\Omega$. Letting $\mathbf{u}$ be as in~\eqref{eq:bfUdef} and letting $\lambda\in\mathrm{RM}_{\mathrm{fin}}(\Omega')$ be a weak*-limit of a suitable subsequence of $(|D\psi_{j}|)$ for a generating sequence $(\psi_{j})$ for $\overline{\mathscr{F}}_{u_{0}}^{*}[\mathbf{u};\Omega,\Omega']$, we may moreover may assume $\mathcal{M}\lambda(x_{0},r)+\mathcal{M}Du(x_{0},r)<\infty$; $\mathscr{L}^{1}$-a.e. $r\in(0,\mathrm{dist}(x_{0},\partial\Omega))$ satisfies this property.

Ad~\ref{item:fin1}. First suppose that $\overline{\mathscr{F}}_{u}^{*}[u;\ball_{r}]=\infty$. Then we trivially have that $-\infty<\overline{\mathscr{F}}_{u}^{*}[u;\ball_{r}(x_{0})]$. If instead $\overline{\mathscr{F}}_{u}^{*}[u;\ball_{r}]<\infty$, we then let  $(u_{j})\subset\mathscr{A}_{u_{0}}^{q}(\ball_{r},\Omega')$ be a generating sequence for $\overline{\mathscr{F}}_{u_{0}}^{*}[\mathbf{u}^{(0)};\ball_{r},\Omega'](<\infty)$, where 
\begin{align*}
\mathbf{u}^{(0)}:=\begin{cases} 
u&\;\text{in}\;\ball_{r},\\
u_{0}&\;\text{in}\;\Omega'\setminus\overline{\ball}_{r}. 
\end{cases}
\end{align*}
We pick $r<R<\mathrm{dist}(x_{0},\partial\Omega)$ and a cut-off function $\rho\in\hold_{c}^{\infty}(\Omega;[0,1])$ with $\mathbbm{1}_{\ball_{r}}\leq \rho\leq\mathbbm{1}_{\ball_{R}}$. Define 
\begin{align*}
v_{j}:=\begin{cases}
u_{j}&\;\text{in}\;\ball_{r},\\
\rho u_{0}&\;\text{in}\;\Omega'\setminus\overline{\ball}_{r}, 
\end{cases}
\end{align*}
so that $v_{j}|_{\Omega}\in\sobo_{0}^{1,q}(\Omega;\R^{N})$. Then by the quasiconvexity of $F$ and~\ref{item:H1}, 
\begin{align}\label{eq:QCimplication}
\begin{split}
F(0)\mathscr{L}^{n}(\ball_{R}) &\leq \int_{\ball_{R}}F(\nabla v_{j})\dif x \leq \int_{\ball_{R}\setminus\overline{\ball}_{r}}|F(\nabla (\rho u_{0}))|\dif x + \int_{\ball_{r}}F(\nabla u_{j})\dif x\\ 
& \!\!\stackrel{\text{\ref{item:H1}}}{\leq} L \Big(\mathscr{L}^{n}(\ball_{R}\setminus\overline{\ball}_{r}) + \int_{\ball_{R}\setminus\overline{\ball}_{r}}|\rho\nabla u_{0} + u_{0}\otimes\nabla\rho|^{q}\dif x\Big),\\ 
& + \int_{\Omega'}F(\nabla u_{j})\dif x-\int_{\Omega'\setminus\overline{\ball}_{r}}F(\nabla u_{0})\dif x
\end{split}
\end{align}
so that the lower bound in \eqref{eq:twosidedfiniteness} follows by virtue of \eqref{eq:QCimplication}, Corollary~\ref{cor:welldefinedbdryvalues1} and since $(u_{j})$ is generating for $\overline{\mathscr{F}}_{u_{0}}^{*}[\mathbf{u}^{(0)};\ball_{r},\Omega']$. Thus, $-\infty<\overline{\mathscr{F}}_{u}^{*}[u;\ball_{r}]$, and a similar argument establishes 
\begin{align}\label{eq:twosidedfiniteness2}
-\infty<\overline{\mathscr{F}}_{w}^{*}[u;\Omega\setminus\overline{\ball}_{r}]
\end{align}
whenever $w\in\bv(\Omega\setminus\overline{\ball}_{r};\R^{N})$ satisfies $\trace_{\partial\Omega}(w-v)=0$ $\mathscr{H}^{n-1}$-a.e. on $\partial\Omega$ and $\trace_{\partial\!\ball_{r}}^{-}(w)=\trace_{\partial\!\ball_{r}}(u)$ $\mathscr{H}^{n-1}$-a.e. on $\partial\!\ball_{r}$. By the same line of argument we then already deduce that $\overline{\mathscr{F}}_{v}^{*}[u;\Omega]\in(-\infty,\infty)$ too since $\overline{\mathscr{F}}_{v}^{*}[u;\Omega]<\infty$ holds by assumption.

We proceed to show that $\overline{\mathscr{F}}_{u}^{*}[u;\ball_{r}]=\infty$ cannot happen. By our assumption on $r$ and since $\overline{\mathscr{F}}_{v}^{*}[u;\Omega]<\infty$, we may invoke Proposition~\ref{prop:goodminseqs} to obtain the existence of a generating sequence $(w_{j})\subset\mathscr{A}_{u_{0}}^{q}(\Omega,\Omega')$ for $\overline{\mathscr{F}}_{u_{0}}^{*}[\mathbf{u};\Omega,\Omega']$  such that $\trace_{\partial\!\ball_{r}}(w_{j})=\trace_{\partial\!\ball_{r}}(u)$ for all $j\in\mathbb{N}$. Suppose that, for some suitable subsequence $(w_{j(i)})\subset (w_{j})$, there holds $\lim_{i\to\infty}\int_{\ball_{r}}F(\nabla w_{j(i)})\dif x=+\infty$. Then, since $(w_{j(i)})$ is still generating for $\overline{\mathscr{F}}_{u_{0}}^{*}[\mathbf{u};\Omega,\Omega']$, 
\begin{align*}
\lim_{i\to\infty}\Big(\int_{\ball_{r}}F(\nabla w_{j(i)})\dif x + \int_{\Omega'\setminus\overline{\ball}_{r}}F(\nabla w_{j(i)})\dif x \Big)<\infty
\end{align*}
implies that necessarily 
\begin{align}\label{eq:contras}
\liminf_{i\to\infty}\int_{\Omega'\setminus\overline{\ball}_{r}}F(\nabla w_{j(i)})\dif x = -\infty.
\end{align}
On the other hand, setting
\begin{align*}
z_{j(i)}:=\begin{cases} 
w_{j(i)}&\;\text{in}\;\Omega\setminus\overline{\ball}_{r},\\
u_{0}&\;\text{in}\;\Omega'\setminus \overline{(\Omega\setminus\overline{\ball}_{r})}\end{cases}\;\;\;\text{and}\;\;\;
\mathbf{z}:=\begin{cases} 
u&\;\text{in}\;\Omega\setminus\overline{\ball}_{r},\\
u_{0}&\;\text{in}\;\Omega'\setminus \overline{(\Omega\setminus\overline{\ball}_{r})},\end{cases}
\end{align*}
we have $z_{j(i)}\stackrel{*}{\rightharpoonup}\mathbf{z}$ in $\bv(\Omega';\R^{N})$, and then~\eqref{eq:contras} leads to a contradiction in view of~\eqref{eq:twosidedfiniteness2}. Hence~\ref{item:fin1} follows. Ad~\ref{item:fin2}. Let $u_{0}$ and $w$ be defined as above, and  let $(u_{j}^{(1)})\subset \mathscr{A}_{u_{0}}^{q}(\ball_{r},\Omega')$ and $(u_{j}^{(2)})\subset \mathscr{A}_{u_{0}}^{q}(\Omega\setminus\overline{\ball}_{r};\Omega')$ be two generating sequences for the functionals $\overline{\mathscr{F}}_{u_{0}}^{*}[\mathbf{u}^{(1)};\ball_{r},\Omega']$ and $\overline{\mathscr{F}}_{u_{0}}^{*}[\mathbf{u}^{(2)};\Omega\setminus\overline{\ball}_{r},\Omega']$, respectively, where 
\begin{align}\label{eq:drkroger}
\mathbf{u}^{(1)}:=\begin{cases} u&\;\text{in}\;\ball_{r},\\ 
u_{0}&\;\text{in}\;\Omega'\setminus\overline{\ball}_{r},
\end{cases}\;\;\;\text{and}\;\;\;\mathbf{u}^{(2)}:=\begin{cases} u&\;\text{in}\;\Omega\setminus\overline{\ball}_{r},\\ 
u_{0}&\;\text{in}\;\Omega'\setminus\overline{(\Omega\setminus\overline{\ball}_{r})}. 
\end{cases}
\end{align}
Defining the glued maps $v_{j}\in\mathscr{A}_{u_{0}}^{q}(\Omega,\Omega')$ by 
\begin{align}\label{eq:gluecreate}
v_{j}:=\begin{cases} 
u_{j}^{(1)}&\;\text{in}\;\ball_{r},\\
u_{j}^{(2)}&\;\text{in}\;\Omega\setminus\overline{\ball}_{r}, \\ 
u_{0}&\;\text{in}\;\Omega'\setminus\overline{\Omega},
\end{cases}
\end{align}
we have $v_{j}\stackrel{*}{\rightharpoonup}\mathbf{u}$ in $\bv(\Omega';\R^{N})$.  Therefore, 
\begin{align}\label{eq:ineqSplitDir1}
\begin{split}
\overline{\mathscr{F}}_{v}^{*}[u;\Omega] &\leq \lim_{j\to\infty} \Big(\int_{\ball_{r}}F(\nabla u_{j}^{(1)})\dif x+\int_{\Omega\setminus\overline{\ball}_{r}}F(\nabla u_{j}^{(2)})\dif x\Big) \\ & = \overline{\mathscr{F}}_{u}^{*}[u;\ball_{r}] + \overline{\mathscr{F}}_{w}^{*}[u;\Omega\setminus\overline{\ball}_{r}].
\end{split}
\end{align}
On the other hand, by assumption we may invoke Proposition~\ref{prop:goodminseqs} to find a generating sequence $(u_{j})\subset\mathscr{A}_{u_{0}}^{q}(\Omega,\Omega')$ for $\overline{\mathscr{F}}_{u_{0}}^{*}[\mathbf{u};\Omega,\Omega']$ such that $\trace_{\partial\!\ball_{r}(x_{0})}(u)=\trace_{\partial\!\ball_{r}(x_{0})}(u_{j})$ for all $j\in\mathbb{N}$. In consequence, 
\begin{align}\label{eq:ineqSplitDir2}
\begin{split}
\overline{\mathscr{F}}_{u}^{*}[u;\ball_{r}]& + \overline{\mathscr{F}}_{w}^{*}[u;\Omega\setminus\overline{\ball}_{r}] \\ & \leq \liminf_{j\to\infty}\int_{\ball_{r}}F(\nabla u_{j})\dif x + \liminf_{j\to\infty}\int_{\Omega\setminus\overline{\ball}_{r}}F(\nabla u_{j})\dif x\\
& \leq \liminf_{j\to\infty}\int_{\Omega}F(\nabla u_{j})\dif x = \overline{\mathscr{F}}_{v}^{*}[u;\Omega], 
\end{split}
\end{align} 
which is~\ref{item:fin2}. Ad~\ref{item:fin3}. Let $r$ be as above and let $\varphi\in\bv_{c}(\ball_{r};\R^{N})$ satisfy $\overline{\mathscr{F}}_{u}^{*}[u+\varphi;\ball_{r}]<\infty$. Let $(\vartheta_{j})\subset\mathscr{A}_{u_{0}}^{q}(\ball_{r},\Omega')$ and $(\theta_{j})\subset\mathscr{A}_{u_{0}}^{q}(\Omega\setminus\overline{\ball}_{r},\Omega')$ be generating sequences for $\overline{\mathscr{F}}_{u_{0}}^{*}[\mathbf{v};\ball_{r},\Omega']$ or $\overline{\mathscr{F}}_{u_{0}}^{*}[\mathbf{u}^{(2)};\Omega\setminus\overline{\ball}_{r};\Omega']$, respectively, where $\mathbf{v}$ is given by 
\begin{align*}
\mathbf{v}:=\begin{cases} u+\varphi&\;\text{in}\;\ball_{r},\\
u_{0}&\;\text{in}\;\Omega'\setminus\overline{\ball}_{r} \end{cases}
\end{align*}
and $\mathbf{u}^{(2)}$ is as in~\eqref{eq:drkroger}. Gluing $\vartheta_{j}$ and $\theta_{j}$ similarly as in~\eqref{eq:gluecreate}, we obtain a sequence $(\Theta_{j})\subset\mathscr{A}_{u_{0}}^{q}(\Omega,\Omega')$ such that $\Theta_{j}\stackrel{*}{\rightharpoonup}\mathbf{v}$ in $\bv(\Omega';\R^{N})$. By definition of $\overline{\mathscr{F}}^{*}$, we then obtain 
\begin{align}\label{eq:rdisher}
\overline{\mathscr{F}}_{v}^{*}[u+\varphi;\Omega] & \leq \liminf_{j\to\infty}\int_{\Omega}F(\nabla\Theta_{j})\dif x \leq \overline{\mathscr{F}}_{u}^{*}[u+\varphi;\ball_{r}] + \overline{\mathscr{F}}_{w}^{*}[u;\Omega\setminus\overline{\ball}_{r}]
\end{align}
similarly as in~\eqref{eq:ineqSplitDir1}. We then conclude by use of the $\bv$-minimality of $u$ for compactly supported variations in the second step that 
\begin{align*}
\overline{\mathscr{F}}_{u}^{*}[u;\ball_{r}] & = \overline{\mathscr{F}}_{v}^{*}[u;\Omega]-\overline{\mathscr{F}}_{w}^{*}[u;\Omega\setminus\overline{\ball}_{r}]\;\;\;\;\;(\text{as $r$ is an additivity radius for $\overline{\mathscr{F}}^{*}[u;-]$ by~\ref{item:fin2})} \\ 
& \leq \overline{\mathscr{F}}_{v}^{*}[u+\varphi;\Omega]-\overline{\mathscr{F}}_{w}^{*}[u;\Omega\setminus\overline{\ball}_{r}]  \\ & \!\!\!\stackrel{\eqref{eq:rdisher}}{\leq} \overline{\mathscr{F}}_{u}^{*}[u+\varphi;\ball_{r}],
\end{align*}
and hence $u$ is a $\bv$-minimizer of $\overline{\mathscr{F}}_{u}^{*}[-;\ball_{r}]$ for compactly supported variations. This is~\ref{item:fin3}, and the proof is complete. 
\end{proof} 
The following two remarks equally apply mutatis mutandis to the functionals $\overline{\mathscr{F}}$:
\begin{remark}\label{rem:Columbo}
In the situation of the previous lemma, if $u_{0}\in\sobo^{1,q}(\Omega';\R^{N})$ is as in~\eqref{eq:relaxedviaboundaryvalues}, $(u_{j})\subset\mathscr{A}_{u_{0}}^{q}(\Omega,\Omega')$ a generating sequence for $\overline{\mathscr{F}}_{u_{0}}^{*}[\mathbf{u};\Omega,\Omega']$  and $\lambda\in\mathrm{RM}_{\mathrm{fin}}(\Omega')$ a weak*-limit of a suitable subsequence of $(|Du_{j}|)$, the above proof shows that Lemma~\ref{lem:additivityradii}~\ref{item:fin1} and~\ref{item:fin2} hold provided $\mathcal{M}Du(x_{0},r)+\mathcal{M}\lambda(x_{0},r)<\infty$.

By the construction from the proof of Lemma~\ref{lem:additivityradii}~\ref{item:fin2}, we then deduce validity of 
\begin{align}\label{eq:AdrianMonkSheronaFleming}
\overline{\mathscr{F}}_{v}^{*}[u;\Omega]=\overline{\mathscr{F}}_{u}^{*}[u;\ball_{r}(x_{0})] + \overline{\mathscr{F}}_{u}^{*}[u;\ball_{s}(x_{0})\setminus\overline{\ball}_{r}(x_{0})] + \overline{\mathscr{F}}_{w}^{*}[u;\Omega\setminus\overline{\ball}_{s}(x_{0})]
\end{align}
for all $0<r<s<\dista(x_{0},\partial\Omega)$ with $\mathcal{M}Du(x_{0},t)+\mathcal{M}\lambda(x_{0},t)<\infty$ for $t\in\{r,s\}$, where $w\in\bv(\Omega;\R^{N})$ satisfies $\trace_{\partial\Omega}(v-w)=0$ and $\trace_{\partial\!\ball_{s}(x_{0})}^{-}(w)=\trace_{\partial\!\ball_{s}(x_{0})}^{+}(u)$ $\mathscr{H}^{n-1}$-a.e. on $\partial\Omega$ or $\partial\!\ball_{s}(x_{0})$, respectively.
\end{remark}
\begin{remark}\label{rem:diffshift}
Let $1<p<\infty$. If one considers the functionals $\mathscr{F}_{\locc}$ as in \cite{SchmidtPR} (cf.~\eqref{eq:FonsecaMalySchmidtRelax}), then the proof of the analogue of \eqref{eq:ineqSplitDir2} trivialises by the very definition of $\mathscr{F}_{\locc}$. This is so because in the definition of $\mathscr{F}_{\locc}$ not even locally membership of generating sequences in certain Dirichlet classes is required. On the other hand, it is at the analogue of~\eqref{eq:ineqSplitDir1} for $\mathscr{F}_{\locc}$ where a variant of the good generation theorem is required. In contrast, inequality~\eqref{eq:ineqSplitDir1} in the above proof only uses $\mathcal{M}Du(x_{0},r)<\infty$ (whereby $\trace_{\partial\!\ball_{r}(x_{0})}(u)\in\sobo^{1-1/q,q}(\partial\!\ball_{r}(x_{0});\R^{N})$) and the very definition of relaxations with solid boundary values. As such, depending on the definition of the underlying functionals, the difficulties might shift. Yet, \emph{independently} of the corresponding relaxation, the concepts of local minimizers coincide -- see the following Section~\ref{sec:schmidt}.
\end{remark}
\subsection{Connections with Schmidt's notions of minimality}\label{sec:schmidt}
Let $1<p\leq q<\frac{np}{n-1}$. In this section, we connect the definitions of the relaxed functionals and local minimality for compactly supported variations with those considered by \textsc{Schmidt} \cite{SchmidtPR}. Let us note that \cite{SchmidtPR} deals with unsigned integrands or integrands that are bounded below exclusively, and so we assume throughout this subsection that  the integrand $F$ satisfies the following hypotheses as in~\cite{SchmidtPR}:
\begin{enumerate}[label={(H\arabic{*}')},start=1]
\item\label{item:H1S} There exist $\ell,L>0$ such that $\ell|z|^{p}\leq F(z)\leq L(1+|z|^{q})$ holds for all $z\in\R^{N\times n}$. 
\item\label{item:H2S} For every $m>0$, there exists $\lambda_{m}>0$ such that for all $z\in\R^{N\times n}$ with $|z|\leq m$ and $\varphi\in\hold_{c}^{\infty}(\ball_{1}(0);\R^{N})$ we have 
\begin{align*}
\lambda_{m}\dashint_{\ball_{1}(0)}(1+|\nabla\varphi|^{2})^{\frac{p-2}{2}}|\nabla\varphi|^{2}\dif x \leq \dashint_{\ball_{1}(0)}F(z+\nabla\varphi)-F(z)\dif x.
\end{align*}
\item\label{item:H3S} $F\in\hold(\R^{N\times n};\R_{\geq 0})$. 
\end{enumerate}
Note that~\ref{item:H2S} and~\ref{item:H2p} are equivalent. This follows from the estimate
\begin{align*}
(1+|z'|^{2})^{\frac{1}{2}}\leq (1+|z|^{2}+|z'|^{2})^{\frac{1}{2}}\leq \sqrt{2}\max\{1,m\}(1+|z'|^{2})^{\frac{1}{2}}
\end{align*} 
for all $z,z'\in\R^{N\times n}$ with $|z|\leq m$ and Lemma~\ref{lem:compaDal}.

Subject to \ref{item:H1S}--\ref{item:H3S}, we then recall from~\eqref{eq:FonsecaMalySchmidtRelax} the definition of the \emph{locally relaxed functional}: Given an open subset $\omega\subset\Omega$, define for $u\in\sobo^{1,p}(\omega;\R^{N})$
\begin{align}\label{eq:SchmidtRelax}
\mathscr{F}_{\locc}[u;\omega] := \inf\left\{\liminf_{j\to\infty}\int_{\omega}F(\nabla u_{j})\dif x\colon\; \begin{array}{c} (u_{j})\subset (\sobo_{\locc}^{1,q}\cap\sobo^{1,p})(\omega;\R^{N}), \\ u_{j} \rightharpoonup u\;\text{in}\;\sobo^{1,p}(\omega;\R^{N})\end{array} \right\}.
\end{align}
\begin{proposition}\label{prop:Equivalenceminimizers}
Let $1<p\leq q <\frac{np}{n-1}$ and suppose that $F$ satisfies \emph{\ref{item:H1S}--\ref{item:H3S}}. Then the following are equivalent for $u\in\sobo_{\locc}^{1,p}(\Omega;\R^{N})$: 
\begin{enumerate}
\item\label{item:MINequiv1} $u$ is a \emph{local minimizer of $\mathscr{F}_{\locc}$ for compactly supported variations on $\Omega$} in the sense that every $x_{0}\in\Omega$ has an open neighbourhood $\omega\Subset\Omega$ such that $\mathscr{F}_{\locc}[u;\omega]<\infty$ and 
\begin{align*}
\mathscr{F}_{\locc}[u;\omega]\leq \mathscr{F}_{\locc}[u+\varphi;\omega]\qquad\text{for all}\;\varphi\in\sobo_{c}^{1,p}(\omega;\R^{N}). 
\end{align*}
\item\label{item:MINequiv2} $u$ is a \emph{local minimizer} of $\overline{\mathscr{F}}$ for compactly supported variations on $\Omega$ in the sense of Definition~\ref{def:locmin}~\ref{item:WLM1}.
\end{enumerate} 
\end{proposition}
\begin{proof} Ad~'\ref{item:MINequiv1}$\Rightarrow$\ref{item:MINequiv2}'. It is no loss of generality to assume that the set $\omega$ from~\ref{item:MINequiv1} is an open ball. In fact, let $x_{0}\in\Omega$ and $\omega\Subset\Omega$ be as in~\ref{item:MINequiv1}. As is established in \cite[Lem.~7.10]{SchmidtPR}, if $w\in\sobo^{1,p}(\omega;\R^{N})$ is such that $\mathscr{F}_{\locc}[w;\omega]<\infty$ and $r>0$ is such that both $\ball_{r}(x_{0})\Subset\omega$ and 
\begin{align}\label{eq:SchmidtCondAdd}
\limsup_{\varepsilon\searrow 0}\frac{1}{\varepsilon}\int_{\ball_{r+\varepsilon}(x_{0})\setminus\overline{\ball}_{r-\varepsilon}(x_{0})}|\nabla w|^{p}\dif x <\infty
\end{align}
hold, then we have the additivity property 
\begin{align}\label{eq:TearsForFears}
\mathscr{F}_{\locc}[w;\omega]=\mathscr{F}_{\locc}[w;\ball_{r}(x_{0})]+\mathscr{F}_{\locc}[w;\omega\setminus\overline{\ball}_{r}(x_{0})]. 
\end{align}
Let $r$ be a radius that satisfies~\eqref{eq:SchmidtCondAdd} for $w=u$. Whenever $\varphi\in\sobo^{1,p}(\omega;\R^{N})$ is compactly supported within $\ball_{r}(x_{0})$ and satisfies~$\mathscr{F}_{\locc}[u+\varphi;\ball_{r}(x_{0})]<\infty$, then both $w=u$, $w=u+\varphi$ satisfy~\eqref{eq:SchmidtCondAdd} and hereafter~\eqref{eq:TearsForFears} together with $\mathscr{F}_{\locc}[w;\omega]<\infty$\footnote{Since $u+\varphi$ satisfies~\eqref{eq:SchmidtCondAdd} and $\mathscr{F}_{\locc}[u+\varphi;\ball_{r}(x_{0})]<\infty$, there exists $(z_{j}^{(1)})\subset (\sobo^{1,q}\cap\sobo_{u}^{1,p})(\ball_{r}(x_{0});\R^{N})$ such that $\lim_{j\to\infty}\mathscr{F}[z_{j}^{(1)};\ball_{r}(x_{0})]=\mathscr{F}_{\locc}[u+\varphi;\ball_{r}(x_{0})]<\infty$ by Lemma~\cite[Lem.~7.7]{SchmidtPR}. By \cite[Lem.~7.8]{SchmidtPR}, we equally find a sequence $(z_{j}^{(2)})\subset \sobo^{1,q}(\Omega\setminus\overline{\ball}_{r}(x_{0});\R^{N})$ that attains the same boundary values along $\partial\!\ball_{r}(x_{0})$ as $u$ such that $\lim_{j\to\infty}\mathscr{F}[z_{j}^{(2)};\omega\setminus\overline{\ball}_{r}(x_{0})]=\mathscr{F}_{\locc}[u;\omega\setminus\overline{\ball}_{r}(x_{0})]<\infty$; now it suffices to glue the two sequences along $\partial\!\ball_{r}(x_{0})$.}. Combining the resulting identities and using the minimality condition from~\ref{item:MINequiv1} yields 
\begin{align}\label{eq:Shout}
\mathscr{F}_{\locc}[u;\ball_{r}(x_{0})]\leq\mathscr{F}_{\locc}[u+\varphi;\ball_{r}(x_{0})]\qquad\text{for all}\;\varphi\in\sobo_{c}^{1,p}(\ball_{r}(x_{0});\R^{N}). 
\end{align}
By definition of $\mathscr{F}_{\locc}$ and the non-negativity of $F$, we have $\mathscr{F}_{\locc}[u;\ball_{r}(x_{0})]\leq\mathscr{F}_{\locc}[u;\omega]<\infty$. For this choice of $r$ (cf.~\eqref{eq:SchmidtCondAdd}), ~\cite[Lem.~7.7]{SchmidtPR} provides us with a sequence $(v_{j})$ contained in $(\sobo^{1,q}\cap\sobo_{u}^{1,p})(\ball_{r}(x_{0});\R^{N})$ such that $v_{j}\rightharpoonup u$ in $\sobo^{1,p}(\ball_{r}(x_{0});\R^{N})$ and $\mathscr{F}_{\locc}[u;\ball_{r}(x_{0})]=\lim_{j\to\infty}\int_{\ball_{r}(x_{0})}F(\nabla v_{j})\dif x$. We then deduce from Corollary~\ref{cor:Fubini} that $\trace_{\partial\!\ball_{r}(x_{0})}(u)$ belongs to $\sobo^{1-1/q,q}(\partial\!\ball_{r}(x_{0});\R^{N})$, whereby $\overline{\mathscr{F}}_{u}[-;\ball_{r}(x_{0})]$ is well-defined in view of Corollary~\ref{cor:welldefinedbdryvalues1}. Especially, there exists $\widetilde{v}\in\sobo^{1,q}(\ball_{r}(x_{0});\R^{N})$ such that $\trace_{\partial\!\ball_{r}(x_{0})}(\widetilde{v})=\trace_{\partial\!\ball_{r}(x_{0})}(u)$ $\mathscr{H}^{n-1}$-a.e. on $\partial\!\ball_{r}(x_{0})$, and so Lemma~\ref{lem:DirClasses} yields $(v_{j})\subset \sobo_{\widetilde{v}}^{1,q}(\ball_{r}(x_{0});\R^{N})$. 

Let $\varphi\in\sobo_{c}^{1,p}(\ball_{r}(x_{0});\R^{N})$ be such that $\overline{\mathscr{F}}_{u}[u+\varphi;\ball_{r}(x_{0})]<\infty$ and take a generating sequence $(w_{j})\subset\sobo_{\widetilde{v}}^{1,q}(\ball_{r}(x_{0});\R^{N})$ for $\overline{\mathscr{F}}_{u}[u+\varphi;\ball_{r}(x_{0})]$. Using that the sequence $(v_{j})\subset\sobo_{\widetilde{v}}^{1,q}(\ball_{r}(x_{0});\R^{N})$ satisfies $v_{j}\rightharpoonup u$ in $\sobo^{1,p}(\ball_{r}(x_{0});\R^{N})$ in the first step, we thereby find 
\begin{align*}
\overline{\mathscr{F}}_{u}[u;\ball_{r}(x_{0})] & \leq \liminf_{j\to\infty}\int_{\ball_{r}(x_{0})}F(\nabla v_{j})\dif x  \\ 
& = \mathscr{F}_{\locc}[u;\ball_{r}(x_{0})] \stackrel{\eqref{eq:Shout}}{\leq} \mathscr{F}_{\locc}[u+\varphi;\ball_{r}(x_{0})] \\ 
& \leq \liminf_{j\to\infty}\int_{\ball_{r}(x_{0})}F(\nabla w_{j})\dif x \;\;\;\;\;\;\;\;\;\;\;\Big(\!\!\!\begin{array}{c} \text{by definition of $\mathscr{F}_{\locc}$ and} \\ \text{ $w_{j}\rightharpoonup u+\varphi$ in $\sobo^{1,p}(\ball_{r}(x_{0});\R^{N})$}\end{array}\!\!\!\Big) \\ 
& = \overline{\mathscr{F}}_{u}[u+\varphi;\ball_{r}(x_{0})]. 
\end{align*}
Hence,~\ref{item:MINequiv2} follows. For the remaining implication '\ref{item:MINequiv2}$\Rightarrow$\ref{item:MINequiv1}' we may essentially revert the argument of the previously established direction, and we leave the details to the reader. 
\end{proof}
We conclude this section with a remark on the underlying terminology: 
\begin{remark} \label{rem:weaklocal}
If $u\in\sobo^{1,p}(\omega;\R^{N})$  satisfies the minimality condition from Proposition~\ref{prop:Equivalenceminimizers} for some $\omega\Subset\Omega$, then $u$ is also referred to as a \emph{weak local minimizer} in the terminology of \textsc{Schmidt}~\cite[Def.~6.2]{SchmidtPR}. Here we work with the notions as displayed in Proposition~\ref{prop:Equivalenceminimizers} to avoid confusion with the common usage of the terminology of weak local minimizers; by the classical interpretation of weak local minimality (e.g. \`a la~\cite{KT} by \textsc{Taheri} and the second named author) one would call $u\in\sobo^{1,p}(\Omega;\R^{N})$ a \emph{weak local minimizer} provided there exists $\delta>0$ such that $\mathscr{F}_{\locc}[u;\Omega]\leq \mathscr{F}_{\locc}[\psi;\Omega]$ holds whenever $\psi\in \sobo_{u}^{1,p}(\Omega;\R^{N})$ satisfies $\|\nabla u - \nabla\psi\|_{\lebe^{\infty}(\Omega)}\leq\delta$, and it is easily seen that this notion is not equivalent to the one displayed above.  
\end{remark} 
\section{Estimates for linearisations}\label{sec:linearisation}
In this intermediate section, we now record some auxiliary results on shifted integrands that prove instrumental for the linearisation strategy below. For $F\in\hold^{1}(\R^{N\times n})$ and $w\in\R^{N\times n}$, we define the \emph{shifted integrand} $F_{w}\colon\R^{N\times n}\to\R$ by
\begin{align}\label{eq:shifted}
F_{w}(z):=F(w+z)-F(w)-\langle F'(w),z\rangle,\qquad z\in\R^{N\times n}. 
\end{align}
The following is a straightforward variant of~\cite[Lem.~4.1]{GK1}; also see \cite[Lem.~2.2]{BGIK}.
\begin{lemma}\label{lem:shifted}
Let $1\leq p \leq q < \infty$ and let $F\colon\R^{N\times n}\to\R$ be an integrand which satisfies \emph{\ref{item:H1}}, \emph{\ref{item:H2p}} and  \emph{\ref{item:H3}}. Then for each $m>0$ there exists a constant $c=c(n,N,p,q,m,\ell_{m},L)\in [1,\infty)$ such that for all $w\in\R^{N\times n}$ with $|w|\leq m$ the following hold for all $z\in\R^{N\times n}$: 
\begin{enumerate}
\item\label{item:shifted(a)} $|F_{w}(z)|\leq c(\mathbbm{1}_{\{|z|\leq 1\}}|z|^{2}+\mathbbm{1}_{\{|z|> 1\}}|z|^{q})$,
\item\label{item:shifted(b)} $|F'_{w}(z)|\leq c(\mathbbm{1}_{\{|z|\leq 1\}}|z|+\mathbbm{1}_{\{|z|> 1\}}|z|^{q-1})$, 
\item\label{item:shifted(c)} $|F''_{w}(0)z-F'_{w}(z)|\leq cV_{\max\{1,q-1\}}(z)$. 
\end{enumerate}
\end{lemma} 

Let $F\in\hold(\R^{N\times n})$ satisfy \ref{item:H1} and \ref{item:H2} (or  \ref{item:H2p}). Given open and bounded sets $\Omega,\Omega'\subset\R^{n}$ with Lipschitz boundaries and $\Omega\Subset\Omega'$, $u_{0}\in\sobo_{0}^{1,q}(\Omega';\R^{N})$ and $v\in\sobo^{1,q}(\Omega;\R^{N})$ (or $v\in\bv(\Omega;\R^{N})$ with $\trace_{\partial\Omega}(v)\in\sobo^{1-1/q,q}(\partial\Omega;\R^{n})$) such that $\trace_{\partial\Omega}(u_{0})=\trace_{\partial\Omega}(v)$ $\mathscr{H}^{n-1}$-a.e. on $\partial\Omega$, we define $\overline{\mathscr{F}}_{\nabla a,u_{0}}^{*}[-;\Omega,\Omega']$ and $\overline{\mathscr{F}}_{\nabla a,v}^{*}[-;\Omega]$ as in~\eqref{eq:relaxedviaboundaryvalues} or~\eqref{eq:relaxedviaboundaryvalues1}, now systematically replacing $F$ by $F_{\nabla a}$. With the obvious modifications, one then also introduces the functionals $\overline{\mathscr{F}}_{\nabla a,u_{0}}[-;\Omega,\Omega']$ and $\overline{\mathscr{F}}_{\nabla a,v}[-;\Omega]$ in the $\sobo^{1,p}$-setting. In this situation, the following lemma will be required in Sections~\ref{sec:Cacc} and~\ref{sec:proofmain}. Its proof is elementary and is provided in the appendix, Section~\ref{sec:proofshiftconnect}, for completeness. 
\begin{lemma}\label{lem:shiftconnect}
Let $F\in\hold^{1}(\R^{N\times n})$ satisfy  \emph{\ref{item:H1}} with $1\leq q < \frac{n}{n-1}$ and \emph{\ref{item:H2}}. Moreover, given $m>0$ and an affine-linear map $a\colon\R^{n}\to\R^{N}$ with $|\nabla a|\leq m$, put $\widetilde{u}:=u-a$ for $u\in\bv_{\locc}(\Omega;\R^{N})$. In the situation described above, the following hold: 
\begin{enumerate}
\item\label{item:shiftconnect1} If $x_{0}\in\Omega$ and $r>0$ is a good radius for $Du$ at $x_{0}$ such that $\overline{\mathscr{F}}_{v}^{*}[u;\Omega]\in (-\infty,\infty)$, then $\overline{\mathscr{F}}_{\nabla a,\widetilde{u}}^{*}[\widetilde{u};\ball_{r}(x_{0})]\in (-\infty,\infty)$. Moreover, any additivity radius for $\overline{\mathscr{F}}^{*}$ is an additivity radius for $\overline{\mathscr{F}}_{\nabla a}^{*}$.
\item\label{item:shiftconnect2} The map $u\in\bv_{\locc}(\Omega;\R^{N})$ is a (local) $\bv$-minimizer of $\overline{\mathscr{F}}^{*}$ (for compactly supported variations) if and only if $\widetilde{u}$ is a (local) $\bv$-minimizer (of $\overline{\mathscr{F}}_{\nabla a}^{*}$ for compactly supported variations).
\item\label{item:shiftconnect3} There exists a constant $\ell^{(m)}>0$ such that whenever $\ball_{s}(x_{0})\Subset\Omega$ we have 
\begin{align}\label{eq:Vboundbelow}
\ell^{(m)} \int_{\ball_{s}(x_{0})}V(\nabla\varphi)\dif x \leq \int_{\ball_{s}(x_{0})}F_{\nabla a}(\nabla\varphi)\dif x\;\;\text{for all}\;\varphi\in\sobo_{0}^{1,q}(\ball_{s}(x_{0});\R^{N}). 
\end{align} 
\end{enumerate}
If $1<p\leq q<\frac{np}{n-1}$ and $F\in\hold^{1}(\R^{N\times n})$ satisfies~\emph{\ref{item:H1}} and~\emph{\ref{item:H2p}},~\ref{item:shiftconnect1} and~\ref{item:shiftconnect2} also hold for $\overline{\mathscr{F}}$ and $u\in\sobo_{\locc}^{1,p}(\Omega;\R^{N})$ with the obvious modifications, and~\eqref{eq:Vboundbelow} persists with $V_{p}$ instead of $V$.
\end{lemma}
Let us remark that~\eqref{eq:Vboundbelow} also holds for $F\in\hold(\R^{N\times n})$ in the form 
\begin{align}\label{eq:Vboundbelow1}
\ell^{(m)} \int_{\ball_{s}(x_{0})}V(\nabla\varphi)\dif x \leq \int_{\ball_{s}(x_{0})}F(\nabla\varphi+\nabla a)-F(\nabla a)\dif x
\end{align}
for all $\varphi\in\sobo_{0}^{1,q}(\ball_{s}(x_{0});\R^{N})$, which is directly seen by~\ref{item:H2} and Lemma~\ref{lem:Efunction}~\ref{item:VpCompa1}. 
\section{A Mazur-type lemma and the Euler-Lagrange system}\label{sec:MazurEuler}
This section is devoted to a convergence improvement for certain recovery sequences, which we shall informally refer to as \emph{Mazur-type lemma}; see Proposition~\ref{prop:mazur}. As a main consequence, it will allow us to deduce the validity of the Euler-Lagrange system without appealing to measure representations in Corollary~\ref{cor:EulerLagrange} below. We begin with
\begin{proposition}[of Mazur-type]\label{prop:mazur}
Let $F\in\hold(\R^{N\times n})$ satisfy \emph{\ref{item:H1}} and \emph{\ref{item:H2}} with $1\leq q<\frac{n}{n-1}$. Given an open and bounded set $\Omega\subset\R^{n}$ with Lipschitz boundary $\partial\Omega$, let $v\in\bv(\Omega;\R^{N})$ be such that $\trace_{\partial\Omega}(v)\in\sobo^{1-1/q,q}(\partial\Omega;\R^{N})$ and let $u$ be a $\bv$-minimizer of $\overline{\mathscr{F}}_{v}^{*}[-;\Omega]$ for compactly supported variations. Letting $Du=\nabla u\mathscr{L}^{n}\mres\Omega + D^{s}u$ be the Lebesgue-Radon-Nikod\'{y}m decomposition of $Du$,  \emph{there exists a generating sequence $(u_{j})$ for $\overline{\mathscr{F}}_{v}^{*}[u;\Omega]$} such that
\begin{align}\label{eq:mazurMainClaim}
\nabla u_{j} \to \nabla u\qquad\text{in $\mathscr{L}^{n}$-measure on $\Omega$ as $j\to\infty$}.
\end{align} 
\end{proposition}
\begin{proof}
Let $\mathcal{L}$ be the set of points $x \in \Omega$, where
\begin{enumerate}[label={(P\arabic{*})},start=1]
\item\label{item:PA1} $u$ is approximately differentiable,
\item\label{item:PA2} the approximate gradient $\nabla u$ has a Lebesgue point,  
\item\label{item:PA3} $\lim_{r\searrow 0}\frac{|D^{s}u|(\ball_{r}(x))}{r^n} =0$.
\end{enumerate}
By the results gathered in Section~\ref{sec:functionspaces}, $\mathscr{L}^{n}(\Omega \setminus \mathcal{L}) = 0$. We split the proof into four steps. 

\emph{Step 1. A preliminary estimate.} Let $x_{0}\in\Omega$, $0<r<\frac{1}{2}\dista(x_{0},\partial\Omega)$, $M_{0}>0$ and $a\colon\R^{n}\to\R^{N}$ be affine-linear with $|\nabla a|\leq M_{0}$. Moreover, let $E_{x_{0}}\subset(\tfrac{4}{3}r,\tfrac{5}{3}r)$ satisfy $\mathscr{L}^{1}(E_{x_{0}})=0$. With $\ell_{M_{0}}$ as in~\ref{item:H2}, we claim that there exists a constant $c=c(M_{0},\ell_{M_{0}},L,N,n,q)>1$ and a radius $\widetilde{r}=\widetilde{r}(u,r,x_{0},a)\in (\tfrac{4}{3}r,\tfrac{5}{3}r)\cap E_{x_{0}}^{\complement}$ such that we have
\begin{align}\label{eq:MazurMainAuxIneq}
\begin{split}
\limsup_{j\to\infty}\dashint_{\ball_{\widetilde{r}}(x_{0})}V(\nabla u_{j}-\nabla a)&\dif x \leq c(1+|\nabla a|^{q-1})\times \\ & \times\Big(\dashint_{\ball_{2r}(x_{0})}|\nabla u-\nabla a|\dif y +\frac{|D^{s}u|(\ball_{2r}(x_{0}))}{\mathscr{L}^{n}(\ball_{2r}(x_{0}))} \Big)\\ 
& + c\Big(\dashint_{\ball_{2r}(x_{0})}|\nabla u-\nabla a|\dif y +\frac{|D^{s}u|(\ball_{2r}(x_{0}))}{\mathscr{L}^{n}(\ball_{2r}(x_{0}))} \Big)^{q} \\ 
& + c(1+|\nabla a|^{q-1})\Big(\dashint_{\ball_{2r}(x_{0})} \frac{|u-a|}{r}\dif y\Big) \\ & + c\Big(\dashint_{\ball_{2r}(x_{0})}\frac{|u-a|}{r}\dif y \Big)^{q}
\end{split}
\end{align}
whenever $(u_{j})$ is a generating sequence for $\overline{\mathscr{F}}_{v}^{*}[u;\Omega]$ such that $\trace_{\partial\!\ball_{\widetilde{r}}(x_{0})}(u_{j})=\trace_{\partial\!\ball_{\widetilde{r}}(x_{0})}(u)$ $\mathscr{H}^{n-1}$-a.e. on $\partial\!\ball_{\widetilde{r}}(x_{0})$ for all sufficiently large $j\in\mathbb{N}$.

We apply Lemma~\ref{lem:HLW} to the interval $(\frac{4}{3}r,\frac{5}{3}r)$, the exceptional set 
\begin{align*}
E=\{t\in(\tfrac{4}{3}r,\tfrac{5}{3}r)\colon\;\mathcal{M}Du(x_{0},t)=\infty\}\cup E_{x_{0}} \cup \big(\tfrac{4}{3}r,\tfrac{401}{300}r\big)\cup\big(\tfrac{499}{300}r,\tfrac{5}{3}r\big)
\end{align*}
and the right-continuous function
\begin{align*}
\Theta(s):=\frac{1}{s}\|u-a\|_{\lebe^{1}(\ball_{s}(x_{0}))}+|D(u-a)|(\overline{\ball}_{s}(x_{0})),\qquad \tfrac{4}{3}r<s<\tfrac{5}{3}r.
\end{align*}
Within the framework of Lemma~\ref{lem:HLW}, this corresponds e.g. to the choice $\theta=\frac{1}{25}$. Hence, Lemma~\ref{lem:HLW} provides us with $\widetilde{r},\widetilde{s}\in (\frac{4}{3}r,\frac{5}{3}r)\cap E^{\complement}$, $\widetilde{r}<\widetilde{s}$, such that we have 
\begin{align}\label{eq:againandagain}
\begin{split}
& \tfrac{1}{300}r<\widetilde{r}-\tfrac{4}{3}r,\;3(\widetilde{s}-\widetilde{r})\leq r\leq 24(\widetilde{s}-\widetilde{r}),\\
&\frac{\Theta(\widetilde{r})-\Theta(s)}{\widetilde{r}-s}\leq 3600\frac{\Theta(\frac{5}{3}r)-\Theta(\frac{4}{3}r)}{r}\;\;\text{for}\;\tfrac{4}{3}r<s<\widetilde{r},\\ 
&\frac{\Theta(s)-\Theta(\widetilde{r})}{s-\widetilde{r}}\leq 3600\frac{\Theta(\frac{5}{3}r)-\Theta(\frac{4}{3}r)}{r}\;\;\text{for}\;\widetilde{r}<s<\tfrac{5}{3}r, \\ 
&\frac{\Theta(\widetilde{s})-\Theta(s)}{\widetilde{s}-s}\leq 3600\frac{\Theta(\frac{5}{3}r)-\Theta(\frac{4}{3}r)}{r}\;\;\text{for}\;\tfrac{4}{3}r<s<\widetilde{s}.
\end{split}
\end{align}
For $v\in\bv(\Omega;\R^{N})$ we then define, with the trace-preserving operator $\widetilde{\mathbb{E}}$ from Section~\ref{sec:Fubini},
\begin{align*}
\mathbb{T}v:=\begin{cases} 
\widetilde{\mathbb{E}}_{\ball_{\widetilde{r}}(x_{0})}v&\;\text{in}\;\ball_{\widetilde{r}}(x_{0}),\\
\widetilde{\mathbb{E}}_{\ball_{\widetilde{s}}(x_{0})\setminus\overline{\ball}_{\widetilde{r}}(x_{0})}v&\;\text{in}\;\ball_{\widetilde{s}}(x_{0})\setminus\overline{\ball}_{\widetilde{r}}(x_{0}), \\ 
v&\;\text{in}\;\Omega\setminus\overline{\ball}_{\widetilde{s}}(x_{0}),
\end{cases}
\end{align*}
so that Lemma~\ref{lem:extensionoperator} implies  $\mathbb{T}v|_{\ball_{\widetilde{s}}(x_{0})}\in\sobo^{1,1}(\ball_{\widetilde{s}}(x_{0});\R^{N})$ whenever $\mathcal{M}Dv(x_{0},\widetilde{r})<\infty$, again being a consequence of Lemma~\ref{lem:extensionoperator},~\eqref{eq:interiortraces1} and~\eqref{eq:traceequal}. This, in particular applies to $u-a$ since $\widetilde{r}\in E^{\complement}$. Moreover, $\mathbb{T}$ is the identity on the affine-linear maps by Lemma~\ref{lem:extensionoperator}~\ref{item:extension3a}. Towards~\eqref{eq:MazurMainAuxIneq}, we now claim that there exists a constant $c=c(N,n,q)>0$ such that we have for $j\in\{0,1\}$ 
\begin{align}\label{eq:lewandowski}
\begin{split}
\dashint_{\ball_{\widetilde{s}}(x_{0})}|D^{j}\mathbb{T}v|&\leq c \dashint_{\ball_{2r}(x_{0})}|D^{j}v|, \\
\dashint_{\ball_{\widetilde{s}}(x_{0})}\left\vert \frac{D^{j}\mathbb{T}(u-a)}{\widetilde{s}^{1-j}}\right\vert^{q}&\leq c \Big(\dashint_{\ball_{2r}(x_{0})}\left\vert \frac{u-a}{r}\right\vert\dif x+\dashint_{\ball_{2r}(x_{0})}|D(u-a)|\Big)^{q}. 
\end{split}
\end{align}
Estimate~$\eqref{eq:lewandowski}_{1}$ is a direct consequence of Lemma~\ref{lem:extensionoperator}~\ref{item:extension2} and $\widetilde{s}\sim r$. For $\eqref{eq:lewandowski}_{2}$, we choose $\varepsilon=\frac{r}{1000}$ in Lemma~\ref{lem:extensionoperator}~\ref{item:extension5}. In particular, for this choice of $\varepsilon$ we have $\widetilde{r}-\varepsilon>\tfrac{4}{3}r$ by $\eqref{eq:againandagain}_{1}$ and so $\eqref{eq:againandagain}_{2}$ shall be available for any $s\in (\widetilde{r}-\varepsilon,\widetilde{r})$. 

Let $j\in\{0,1\}$. With $\overline{c}=\overline{c}(n,N,q)>0$ being the maximum of the constants provided by Lemma~\ref{lem:extensionoperator}~\ref{item:extension4} and \ref{item:extension5}, we thus find 
\begin{align*}
\int_{\ball_{\widetilde{s}}(x_{0})}\left\vert \frac{D^{j}\mathbb{T}(u-a)}{\widetilde{s}^{1-j}}\right\vert^{q} & = \int_{\ball_{\widetilde{r}}(x_{0})}\left\vert \frac{D^{j}\mathbb{T}(u-a)}{\widetilde{s}^{1-j}}\right\vert^{q} + \int_{\ball_{\widetilde{s}}(x_{0})\setminus\ball_{\widetilde{r}}(x_{0})}\left\vert \frac{D^{j}\mathbb{T}(u-a)}{\widetilde{s}^{1-j}}\right\vert^{q}\\ 
& \leq \overline{c}\frac{r^{n(1-q)+q}}{\varepsilon^{q}}\Big(\int_{\ball_{\widetilde{r}}(x_{0})}\left\vert \frac{D^{j}(u-a)}{\widetilde{s}^{1-j}}\right\vert \Big)^{q} \\ &+\overline{c}r^{n(1-q)+q}\Big(\sup_{0<\delta<\varepsilon}\frac{1}{\delta}\int_{(\ball_{\widetilde{r}}(x_{0}))_{\delta}^{\complement}}\left\vert \frac{D^{j}(u-a)}{\widetilde{s}^{1-j}}\right\vert\Big)^{q}\\ 
&+ \overline{c}r^{n(1-q)+q}\Big(\sup_{0<\delta<\widetilde{s}-\widetilde{r}}\frac{1}{\delta}\int_{(\ball_{\widetilde{s}}(x_{0})\setminus\overline{\ball}_{\widetilde{r}}(x_{0}))_{\delta}^{\complement}}\left\vert \frac{D^{j}(u-a)}{\widetilde{s}^{1-j}}\right\vert\Big)^{q}\\ 
& \leq c\Big(r^{n(1-q)}\Big(\int_{\ball_{2r}(x_{0})}\left\vert \frac{D^{j}(u-a)}{r^{1-j}}\right\vert \Big)^{q} + r^{n(1-q)+q}\Big(\frac{\Theta(2r)}{r} \Big)^{q}\Big), 
\end{align*}
the ultimate inequality being valid by our choice of $\varepsilon$, $\widetilde{r}-\varepsilon>\frac{4}{3}r$, $\widetilde{s}\sim r$ and $\eqref{eq:againandagain}_{2}$--$\eqref{eq:againandagain}_{4}$; here, the lower order terms corresponding to $j=0$ are estimated similarly as in~\eqref{eq:pinacolada}. Still, $c>0$ as in the ultimate line only depends on $n,N$ and $q$. From here,~\eqref{eq:lewandowski} is immediate by the definition of $\Theta$. 

Recall that we have assumed $(u_{j})$ to be a generating sequence for $\overline{\mathscr{F}}_{v}^{*}[u;\Omega]$ such that we have $\trace_{\partial\!\ball_{\widetilde{r}}(x_{0})}(u_{j})=\trace_{\partial\!\ball_{\widetilde{r}}(x_{0})}(u)$ $\mathscr{H}^{n-1}$-a.e. on $\partial\!\ball_{\widetilde{r}(x_{0})}$ for all sufficiently large $j\in\mathbb{N}$. We choose a cut-off function $\rho\in\hold_{c}^{\infty}(\Omega;[0,1])$ with $\mathbbm{1}_{\ball_{\widetilde{r}}(x_{0})}\leq \rho \leq \mathbbm{1}_{\ball_{\widetilde{s}}(x_{0})}$ and $|\nabla\rho|\leq \frac{48}{r}$, cf.~$\eqref{eq:againandagain}_{1}$. Define 
\begin{align*}
\varphi_{j}:=\begin{cases} 
u_{j}-a &\;\text{in}\;\ball_{\widetilde{r}}(x_{0}),\\ 
\rho\widetilde{\mathbb{E}}_{\ball_{\widetilde{s}}\setminus\overline{\ball}_{\widetilde{r}}}(u-a)&\;\text{in}\;\ball_{\widetilde{s}}(x_{0})\setminus\overline{\ball}_{\widetilde{r}}(x_{0}), 
\end{cases}
\end{align*}
so that $\varphi_{j}\in\sobo_{c}^{1,q}(\ball_{\widetilde{s}}(x_{0});\R^{N})$ for all sufficiently large $j\in\mathbb{N}$. Since $u$ is a $\bv$-minimizer of $\overline{\mathscr{F}}_{v}^{*}[-;\Omega]$ for compactly supported variations and $\trace_{\partial\!\ball_{\widetilde{r}}(x_{0})}(\mathbb{T}u)=\trace_{\partial\!\ball_{\widetilde{r}}(x_{0})}(u)$, we have 
\begin{align}\label{eq:MazurBVMinimality}
\limsup_{j\to\infty}\int_{\ball_{\widetilde{r}}(x_{0})}F(\nabla u_{j})-F(\nabla\mathbb{T}u)\dif x \leq 0. 
\end{align}
Indeed, setting 
\begin{align*}
w:=\begin{cases} u&\;\text{in}\;\Omega\setminus\overline{\ball}_{\widetilde{r}}(x_{0}), \\ 
\mathbb{T}u&\;\text{in}\;\ball_{\widetilde{r}}(x_{0}), \end{cases}\;\;\text{and}\;\;w_{j}:=\begin{cases} u_{j}&\;\text{in}\;\Omega\setminus\overline{\ball}_{\widetilde{r}}(x_{0}), \\ 
\mathbb{T}u&\;\text{in}\;\ball_{\widetilde{r}}(x_{0}), \end{cases}
\end{align*}
we have $w_{j}\stackrel{*}{\rightharpoonup}w$ in $\bv(\Omega;\R^{N})$. Thus, since $(u_{j})$ is generating, $w-u\in\bv_{c}(\Omega;\R^{N})$ and $u$ is a $\bv$-minimizer of $\overline{\mathscr{F}}_{v}^{*}[-;\Omega]$ for compactly supported variations, 
\begin{align*}
\limsup_{j\to\infty}\int_{\Omega}F(\nabla u_{j})\dif x & = \overline{\mathscr{F}}_{v}^{*}[u;\Omega] \leq \overline{\mathscr{F}}_{v}^{*}[w;\Omega] \leq \liminf_{j\to\infty}\int_{\Omega}F(\nabla w_{j})\dif x, 
\end{align*}
so that~\eqref{eq:MazurBVMinimality} follows by virtue of 
\begin{align*} 
\limsup_{j\to\infty} \int_{\ball_{\widetilde{r}}(x_{0})} F(\nabla u_{j})-F(\nabla\mathbb{T}u)\dif x & \leq \limsup_{j\to\infty} \int_{\Omega} F(\nabla u_{j})\dif x \\ & + \limsup_{j\to\infty}\Big(\int_{\Omega}-F(\nabla w_{j})\dif x\Big) \\ 
& = \limsup_{j\to\infty} \int_{\Omega} F(\nabla u_{j})\dif x \\ & - \liminf_{j\to\infty}\Big(\int_{\Omega}F(\nabla w_{j})\dif x\Big) \leq 0. 
\end{align*}
Next recall that $|\nabla a|\leq M_{0}$, so that the strong quasiconvexity from~\ref{item:H2} is at our disposal with $\ell_{M_{0}}>0$. In combination with the Lipschitz-type bound~\eqref{eq:lipschitz} we therefore conclude
\begin{align*}
\ell^{(M_{0})}\limsup_{j\to\infty}& \int_{\ball_{\widetilde{s}}(x_{0})}V(\nabla\varphi_{j})\dif x  \stackrel{\eqref{eq:Vboundbelow1}}{\leq} \limsup_{j\to\infty}\int_{\ball_{\widetilde{s}}(x_{0})}F(\nabla\varphi_{j}+\nabla a)-F(\nabla a)\dif x \\ 
& \leq\limsup_{j\to\infty} \int_{\ball_{\widetilde{r}}(x_{0})}F(\nabla u_{j})-F(\nabla a)\dif x \\ 
& + \int_{\ball_{\widetilde{s}}(x_{0})\setminus\overline{\ball}_{\widetilde{r}}(x_{0})}F(\nabla (\rho\mathbb{T}(u-a))+\nabla a)-F(\nabla a)\dif x \\ 
& \!\!\!\!\!\!\!\!\!\stackrel{\eqref{eq:MazurBVMinimality},~\mathbb{T}a=a}{\leq} \int_{\ball_{\widetilde{r}}(x_{0})}F(\nabla \mathbb{T}u)-F(\nabla\mathbb{T}a)\dif x \\ 
& + \int_{\ball_{\widetilde{s}}(x_{0})\setminus\overline{\ball}_{\widetilde{r}}(x_{0})}F(\nabla (\rho\mathbb{T}(u-a))+\nabla a)-F(\nabla a)\dif x \\ 
& \!\!\!\stackrel{\eqref{eq:lipschitz}}{\leq} c\int_{\ball_{\widetilde{r}}(x_{0})}(1+|\nabla\mathbb{T}(u-a)|^{q-1}+|\nabla\mathbb{T}a|^{q-1})|\nabla\mathbb{T}(u-a)|\dif x \\ 
& + c\int_{\ball_{\widetilde{s}}(x_{0})\setminus\overline{\ball}_{\widetilde{r}}(x_{0})}(1+|\nabla(\rho\mathbb{T}(u-a))|^{q-1}+|\nabla a|^{q-1})|\nabla(\rho\mathbb{T}(u-a))|\dif x \\ 
& =: \mathrm{I} + \mathrm{II}. 
\end{align*}
By \eqref{eq:lewandowski} and $r<\widetilde{r}<\widetilde{s}<2r$, we have  
\begin{align*}
\mathrm{I} & \leq c(1+|\nabla a|^{q-1})\int_{\ball_{2r}(x_{0})}|D(u-a)| + cr^{n}\Big(\dashint_{\ball_{2r}(x_{0})}\left\vert \frac{u-a}{r}\right\vert\dif x+\dashint_{\ball_{2r}(x_{0})}|D^{j}(u-a)|\Big)^{q}.
\end{align*}
On the other hand, our assumptions on $\rho$ and \eqref{eq:lewandowski} imply 
\begin{align*}
\mathrm{II} & \leq c(1+|\nabla a|^{q-1})\Big(\int_{\ball_{2r}(x_{0})}\left\vert \frac{u-a}{r}\right\vert\dif x+\int_{\ball_{2r}(x_{0})}|D(u-a)| \Big)\\
& + cr^{n}\Big(\dashint_{\ball_{2r}(x_{0})}\left\vert \frac{u-a}{r}\right\vert\dif x+\dashint_{\ball_{2r}(x_{0})}|D(u-a)| \Big)^{q}.
\end{align*}
Combining the estimations for $\mathrm{I}$, $\mathrm{II}$ and recalling that $\ell^{(M_{0})}=\ell_{M_{0}}/(16(1+|M_{0}|^{2})^{\frac{3}{2}})$  consequently yields~\eqref{eq:MazurMainAuxIneq}. 

\emph{Step 2. Construction of a suitable generating sequence.} Let $0<\varepsilon<1$ be given. For each $x\in\mathcal{L}$, we set $M_{0,x}:=|\nabla u(x)|$. This fixes the constant $c=c(M_{0,x},\ell_{M_{0,x}},L,N,n,q)>0$ from step 1 underlying inequality~\eqref{eq:MazurMainAuxIneq}; this constant explicitely depends on $x$. Properties \ref{item:PA1}--\ref{item:PA3} let us conclude the existence of some $0<r_{x}<\tfrac{1}{2}\dista(x,\partial\Omega)$ such that we have for all $0<r<r_{x}$
\begin{align}\label{eq:betterallthetime}
\begin{split}
&(1+|\nabla u(x)|^{q-1})\Big(\dashint_{\ball_{2r}(x)}|\nabla u-\nabla u(x)|\dif y + \frac{|D^{s}u|(\overline{\ball}_{2r}(x))}{\mathscr{L}^{n}(\ball_{2r}(x))} \Big)<\frac{2^{n}\varepsilon}{3^{n+2}c},\\
& (1+|\nabla u(x)|^{q-1})\Big(\dashint_{\ball_{2r}(x)}\frac{|u(y)-\nabla u(x)(y-x)-u(x)|}{r}\dif y\Big) < \frac{2^{n}\varepsilon}{3^{n+2}c}.
\end{split}
\end{align}
Let $\Omega'$ be an open and bounded set with Lipschitz boundary such that $\Omega\Subset\Omega'$. As usual, let $u_{0}\in\sobo_{0}^{1,q}(\Omega';\R^{N})$ be such that $u_{0}$ attains the same traces along $\partial\Omega$ as $v$, and let $(v_{j})\subset\mathscr{A}_{u_{0}}^{q}(\Omega,\Omega')$ be a generating sequence for $\overline{\mathscr{F}}_{u_{0}}^{*}[\mathbf{u};\Omega,\Omega']$. Here, as in~\eqref{eq:bfUdef}, $\mathbf{u}$ denotes the extension of $u$ to $\Omega'$ by $u_{0}$. Next let $\lambda'\in\mathrm{RM}_{\mathrm{fin}}(\Omega')$ be a weak*-limit of a suitable subsequence of $(|Dv_{j}|)$. 

For each $x\in\mathcal{L}$ and $0<r<r_{x}$, we apply step 1 to the particular choice $M_{0}=M_{0,x}$, $E_{x}:=\{t\in(\tfrac{4}{3}r,\tfrac{5}{3}r)\colon\;\mathcal{M}Du(x,t)=\infty\;\text{or}\;\mathcal{M}\lambda'(x,t)=\infty\}$ and $a(y)=\nabla u(x)\cdot(y-x) + u(x)$. We thus obtain the existence of some\footnote{In the terminology of step 1, we have $\widetilde{r}=\widetilde{r}(u,r,x,\nabla u(x)(\cdot-x)+u(x))$, but since $\nabla u(x)(\cdot-x)+u(x)$ is termined by $u$ and $x$ and required to belong to $E_{x}^{\complement}$, the latter being defined in terms of $|Du|$, hence $u$, and $\lambda'$, we write $\widetilde{r}=\widetilde{r}(u,r,x,\lambda')$.} $\widetilde{r}=\widetilde{r}(u,r,x,\lambda')\in(\tfrac{4}{3}r,\tfrac{5}{3}r)\cap E_{x}^{\complement}$ such that~\eqref{eq:MazurMainAuxIneq} with $x_{0}$ replaced by $x$ holds for any generating sequence $(u_{j})$ for $\overline{\mathscr{F}}_{v}^{*}[u;\Omega]$ that eventually attains the same traces along $\partial\!\ball_{\widetilde{r}}(x)$ as $u$. Now consider the family $\mathscr{V}$ of closed balls given by
\begin{align}
\mathscr{V}:=\Big\{\overline{\ball}_{\widetilde{r}(u,r,x,\lambda')}(x)\colon\;x\in\mathcal{L},\;0<r<r_{x} \Big\}.
\end{align}
Then $\mathscr{V}$ is a Vitali covering for $\mathcal{L}$ and so, by the Vitali covering theorem~\cite[Thm.~2.19]{AFP}, there exists $\mathbb{K}\subseteq\mathbb{N}$ and a sequence $(\overline{\ball}_{\widetilde{r}(u,r_{k},x_{k},\lambda')}(x_{k}))_{k\in\mathbb{K}}\subset\mathscr{V}$ of pairwise disjoint closed balls such that 
\begin{align}\label{eq:exhaustion}
\mathscr{L}^{n}\Big(\mathcal{L}\setminus\bigcup_{k\in\mathbb{K}}\overline{\ball}_{\widetilde{r}(u,r_{k},x_{k},\lambda')}(x_{k}) \Big)=0\;\;\text{and so}\;\;\mathscr{L}^{n}\Big(\Omega\setminus\bigcup_{k\in\mathbb{K}}\overline{\ball}_{\widetilde{r}(u,r_{k},x_{k},\lambda')}(x_{k}) \Big)=0
\end{align}
because of $\mathscr{L}^{n}(\Omega\setminus\mathcal{L})=0$. Corollary~\ref{cor:allalong} then allows us to modify $(v_{j})$ to a sequence $(u_{j})$ which is still generating for $\overline{\mathscr{F}}_{v}^{*}[u;\Omega]$ and, for each $k\in\mathbb{K}$, there exists $j_{k}\in\mathbb{N}$ such that $\trace_{\partial\!\ball_{\widetilde{r}(u,r_{k},x_{k},\lambda')}(x_{k})}(u_{j})=\trace_{\partial\!\ball_{\widetilde{r}(u,r_{k},x_{k},\lambda')}(x_{k})}(u)$ $\mathscr{H}^{n-1}$-a.e. on $\partial\!\ball_{\widetilde{r}(u,r_{k},x_{k},\lambda')}(x_{k})$ for all $j\geq j_{k}$. For brevity, we now write $\ball^{(k)}:=\ball_{\widetilde{r}(u,r_{k},x_{k},\lambda')}(x_{k})$. 

\emph{Step 3. Bounds for the sequence from step 2.} By construction, the sequence $(u_{j})$ satisfies the requirements from step 1 for any ball $\ball^{(k)}$, $k\in\mathbb{K}$. Rewriting~\eqref{eq:betterallthetime} for $x=x_{k}$ and $r=r_{k}$, we have for all $k\in\mathbb{K}$
\begin{align}\label{eq:betterallthetime1}
\begin{split}
&(1+|\nabla u(x_{k})|^{q-1})\Big(\dashint_{\ball_{2r_{k}}(x_{k})}|\nabla u-\nabla u(x_{k})|\dif y + \frac{|D^{s}u|(\overline{\ball}_{2r_{k}}(x_{k}))}{\mathscr{L}^{n}(\ball_{2r_{k}}(x_{k}))} \Big)<\frac{2^{n}\varepsilon}{3^{n+2}c_{k}},\\
&(1+|\nabla u(x_{k})|^{q-1})\Big(\dashint_{\ball_{2r_{k}}(x_{k})}\frac{|u(y)-\nabla u(x_{k})(y-x_{k})-u(x_{k})|}{r_{k}}\dif y\Big) < \frac{2^{n}\varepsilon}{3^{n+2}c_{k}}
\end{split}
\end{align}
where we denote the corresponding constants from~\eqref{eq:betterallthetime} by $c_{k}$ to indicate the additional dependence on $x_{k}$ and $r_{k}$. Thus, estimate~\eqref{eq:MazurMainAuxIneq} in conjunction with $q\geq 1$ and $0<\varepsilon<1$ implies  
\begin{align}\label{eq:MazurMainAuxIneq1}
\limsup_{j\to\infty}\int_{\ball^{(k)}}V(\nabla u_{j}-\nabla u(x_{k}))\dif y \leq \frac{\varepsilon}{2}\mathscr{L}^{n}(\ball^{(k)})
\end{align}
for all $k\in\mathbb{K}$. Since $V$ is $1$-Lipschitz by Lemma~\ref{lem:Efunction}~\ref{item:EfctE}, we moreover have for all $j\in\mathbb{N}$
\begin{align}\label{eq:MazurLipschitz}
\begin{split}
\int_{\ball^{(k)}}V(\nabla u_{j}-\nabla u)\dif x & \leq \int_{\ball^{(k)}}V(\nabla u_{j}-\nabla u(x_{k}))\dif x \\ & + \int_{\ball^{(k)}}|\nabla u-\nabla u(x_{k})|\dif x 
\end{split}
\end{align}
and from $\frac{4}{3}r_{k}<\widetilde{r}(u,r_{k},x_{k},\lambda')$ we deduce by $c_{k}>1$ that 
\begin{align}\label{eq:sgtpepper}
\frac{\mathscr{L}^{n}(\ball_{2r_{k}}(x_{k}))}{\mathscr{L}^{n}(\ball^{(k)})}\leq\Big(\frac{3}{2}\Big)^{n},\;\;\text{so}\;\;\; \int_{\ball^{(k)}}|\nabla u-\nabla u(x_{k})|\dif x \stackrel{\eqref{eq:betterallthetime1}_{1}}{\leq}\frac{\varepsilon}{2}\mathscr{L}^{n}(\ball^{(k)}).
\end{align}
Now define, for $t>0$, 
\begin{align}
E_{j}^{t}:=\{x\in\Omega\colon\;|\nabla u_{j}(x)-\nabla u(x)|>t\}.
\end{align}
Then, by Ceby\v{s}ev's inequality, 
\begin{align}\label{eq:MazurCebysev}
\begin{split}
\limsup_{j\to\infty}\mathscr{L}^{n}(\ball^{(k)}&\cap E_{j}^{t})\leq \limsup_{j\to\infty}\frac{1}{V(t)}\int_{\ball^{(k)}}V(\nabla u_{j}-\nabla u)\dif x\\ 
& \!\!\!\!\!\!\stackrel{\eqref{eq:MazurLipschitz}}{\leq} \limsup_{j\to\infty}\frac{1}{V(t)}\int_{\ball^{(k)}}V(\nabla u_{j}-\nabla u(x_{k}))\dif x \\ & \!\!\!\!\,+ \frac{1}{V(t)}\int_{\ball^{(k)}}|\nabla u-\nabla u(x_{k})|\dif x \\ 
& \!\!\!\!\!\!\!\!\!\!\!\!\!\stackrel{\eqref{eq:MazurMainAuxIneq1},~\eqref{eq:sgtpepper}}{\leq} \frac{\varepsilon}{V(t)}\mathscr{L}^{n}(\ball^{(k)}). 
\end{split}
\end{align}
By \eqref{eq:exhaustion}, we may pick a finite subset $\mathbb{K}'\subset\mathbb{K}$ such that $\mathscr{L}^{n}(\Omega\setminus\bigcup_{k\in\mathbb{K}'}\ball^{(k)})<\varepsilon$. We then conclude 
\begin{align*}
\limsup_{j\to\infty}\mathscr{L}^{n}(E_{j}^{t}) & = \limsup_{j\to\infty}\Big(\mathscr{L}^{n}\Big(E_{j}^{t}\setminus\bigcup_{k\in\mathbb{K}'}\ball^{(k)}\Big)+ \mathscr{L}^{n}\Big(E_{j}^{t}\cap\bigcup_{k\in\mathbb{K}'}\ball^{(k)} \Big)\Big)\\ 
& \leq \varepsilon + \limsup_{j\to\infty}\sum_{k\in\mathbb{K}'}\mathscr{L}^{n}(E_{j}^{t}\cap\ball^{(k)}) \\ 
& \!\!\!\stackrel{\eqref{eq:MazurCebysev}}{\leq} \varepsilon\Big(1+\frac{1}{V(t)}\sum_{k\in\mathbb{K}}\mathscr{L}^{n}(\ball^{(k)}) \Big)\\
& \leq \varepsilon\Big(1+\frac{\mathscr{L}^{n}(\Omega)}{V(t)}\Big)
\end{align*}
since the $\ball^{(k)}$'s are pairwise disjoint and contained in $\Omega$.

\emph{Step 4. Diagonal sequence and conclusion.} By step 3, for every $\ell\in\mathbb{N}_{\geq 2}$  there exists a generating sequence $(u_{j}^{(\ell)})$ such that 
\begin{align*}
\limsup_{j\to\infty}\mathscr{L}^{n}(\{x\in\Omega\colon\;|\nabla u_{j}^{(\ell)}(x)-\nabla u(x)|>\tfrac{1}{k}\})\leq \frac{1}{\ell}\Big(1+\frac{\mathscr{L}^{n}(\Omega)}{V(\tfrac{1}{k})}\Big)
\end{align*}
holds for all $k\in\mathbb{N}_{\geq 2}$. Thus, for any $\ell\in\mathbb{N}_{\geq 2}$ and $k\in\mathbb{N}_{\geq 2}$ there exists a subsequence $(u_{j_{k,\ell}(i)}^{(\ell)})\subset (u_{j}^{(\ell)})$ such that 
\begin{align}\label{eq:AgainAlmostThere}
\begin{split}
&\mathscr{L}^{n}(\{x\in\Omega\colon\;|\nabla u_{j_{k,\ell}(i)}^{(\ell)}(x)-\nabla u(x)|>\tfrac{1}{k}\})\leq \frac{1}{\ell}\Big(1+\frac{\mathscr{L}^{n}(\Omega)}{V(\tfrac{1}{k})}\Big)+2^{-i},\\
&\;\;\;\;\;\;\;\;\;\;\;\;\;\;\;\;\;\left\vert\overline{\mathscr{F}}_{v}^{*}[u;\Omega]-\int_{\Omega}F(\nabla u_{j_{k,\ell}(i)}^{(\ell)})\dif x\right\vert \leq 2^{-i}
\end{split}
\end{align}
hold for all $i\in\mathbb{N}$. Given $k\in\mathbb{N}_{\geq 2}$, we set $\ell=k^{3}$ and $i=k$, giving us the sequence $(u_{k})$ defined by $u_{k}:=u_{j_{k,k^{3}}(k)}^{(k^{3})}$. Now, for an arbitrary $t>0$, pick $k_{0}\in\mathbb{N}_{\geq 2}$ such that $\frac{1}{k}\leq t$ holds for all $k\geq k_{0}$. Recalling that $(\sqrt{2}-1)|z|^{2}\leq V(z)$ for all $|z|\leq 1$ (cf.~Lemma~\ref{lem:Efunction}~\ref{item:EfctA}), we have for $k\geq k_{0}$
\begin{align*}
\mathscr{L}^{n}(\{x\in\Omega\colon\;|\nabla u_{k}(x)-\nabla u(x)|>t\}) & \leq \mathscr{L}^{n}(\{x\in\Omega\colon\;|\nabla u_{k}(x)-\nabla u(x)|>\tfrac{1}{k}\})\\
& \!\!\!\!\stackrel{\eqref{eq:AgainAlmostThere}_{1}}{\leq} \frac{1}{k^{3}}\Big(1+k^{2}\frac{\mathscr{L}^{n}(\Omega)}{\sqrt{2}-1}\Big)+2^{-k}\to 0
\end{align*}
as $k\to\infty$. Lastly, $\eqref{eq:AgainAlmostThere}_{2}$ implies that $(u_{k})$ is generating for $\overline{\mathscr{F}}_{v}^{*}[u;\Omega]$, and the proof is complete. 
\end{proof}
While the proof is essentially the same, we explicitely state the analogous result for $p>1$: 
\begin{corollary}[of Mazur-type]\label{cor:mazur}
Let $F\in\hold(\R^{N\times n})$ satisfy \emph{\ref{item:H1}} and \emph{\ref{item:H2p}} with $1<p\leq q<\frac{np}{n-1}$. Given an open and bounded set $\Omega\subset\R^{n}$ with Lipschitz boundary $\partial\Omega$, let $v\in\sobo^{1,p}(\Omega;\R^{N})$ be such that $\trace_{\partial\Omega}(v)\in\sobo^{1-1/q,q}(\partial\Omega;\R^{N})$ and let $u$ be a minimizer of $\overline{\mathscr{F}}_{v}[-;\Omega]$ for compactly supported variations on $\Omega$. Then \emph{there exists a generating sequence $(u_{j})$ for $\overline{\mathscr{F}}_{v}[u;\Omega]$} such that
\begin{align}\label{eq:mazurMainClaim1}
\nabla u_{j} \to \nabla u\qquad\text{in $\mathscr{L}^{n}$-measure on $\Omega$ as $j\to\infty$}.
\end{align}
\end{corollary}
\begin{remark}[Trivialisation for $1<p=q$]
Proposition~\ref{prop:mazur} and Corollary~\ref{cor:mazur} can be interpreted in the way that the \textsc{Lebesgue-Serrin-Marcellini} extension always selects \emph{good} generating sequences, where \emph{good} refers to a less oscillatory behaviour. Note that, when $p=q\in (1,\infty)$ and so no proper relaxation is required, Proposition~\ref{prop:mazur} is in line with the constant sequence $(u)$ already being generating for $\int_{\Omega}F(\nabla u)\dif x$, $u\in\sobo^{1,p}(\Omega;\R^{N})$. 
\end{remark}

\begin{remark}\label{rem:MazurDiscuss}
Establishing the convergence $\nabla u_{j}\to\nabla u$ in $\mathscr{L}^{n}$-measure for local $\bv$-minimi- zers $u\in\bv_{\locc}(\Omega;\R^{N})$, Proposition~\ref{prop:mazur} limits the oscillatory behaviour of suitable generating sequences by use of minimality. Yet, it does not rule out concentration effects, which can be seen via the classical one-dimensional example $F=|\cdot|$. This choice of $F\in\hold(\R)$ satisfies \ref{item:H1} and \ref{item:H2}, and $u=\mathrm{sgn}$ is a local minimizer. 

In the above proof, minimality for compactly supported variations enters in step 1. We wish to point out that, in general, even if $q=1$ and $u,u_{j}\in\bv(\Omega;\R^{N})$, the slightly weaker estimate 
\begin{align*}
\limsup_{j\to\infty}\dashint_{\ball_{r}(x_{0})}V(\nabla u_{j}-\nabla a)\dif x \leq (\text{right-hand side of~\eqref{eq:MazurMainAuxIneq}})
\end{align*}
is not a plain consequence of the convergence $u_{j}\stackrel{*}{\rightharpoonup}u$ in $\bv(\Omega;\R^{N})$. To see this, consider the following scenario with $n=N=1$: Set $\Omega=(-1,1)$, $v\equiv 0$ and let $h$ be the $1$-periodic extension of $[0,1]\ni x\mapsto (\frac{1}{2}-|x-\frac{1}{2}|)$ to $\R$. Letting $u_{j}$ be the restriction of $h_{j}(x):=\frac{1}{j}h(jx)$ to $(-1,1)$, we have $(u_{j})\subset\sobo_{0}^{1,1}(\Omega)$ and $u_{j}\stackrel{*}{\rightharpoonup}u\equiv 0$ in $\bv(\Omega)$, but with $a\equiv 0$ we have for all $0<r<1$
\begin{align*}
(\sqrt{2}-1) = \limsup_{j\to\infty}\dashint_{(-r,r)}V(u'_{j})\dif x,\;\;\;\text{whereas}\;\|u\|_{\bv((-1,1))}=0. 
\end{align*}
\end{remark}
As a consequence of the previous proposition, we obtain the validity of the corresponding Euler-Lagrange system: 
\begin{corollary}[Euler-Lagrange system I]\label{cor:EulerLagrange}
Let $1\leq p \leq q <\min\{\frac{np}{n-1},p+1\}$, $\Omega\subset\R^{n}$ be open and bounded with Lipschitz boundary $\partial\Omega$. Moreover, 
\begin{enumerate}
\item\label{item:EL1} if $p=1$, let $F\in\hold(\R^{N\times n})$ be an integrand that satisfies \emph{\ref{item:H1} and~\ref{item:H2}}, let $v\in\bv(\Omega;\R^{N})$ satisfy $\trace_{\partial\Omega}(v)\in\sobo^{1-1/q,q}(\partial\Omega;\R^{N})$, and suppose that $u$ is a $\bv$-minimi- zer of $\overline{\mathscr{F}}_{v}^{*}[-;\Omega]$ for compactly supported variations. 
\item\label{item:EL2} if $p>1$, let $F\in\hold(\R^{N\times n})$ be an integrand that satisfies \emph{\ref{item:H1} and~\ref{item:H2p}}, let $v\in\sobo^{1,p}(\Omega;\R^{N})$ satisfy $\trace_{\partial\Omega}(v)\in\sobo^{1-1/q,q}(\partial\Omega;\R^{N})$ and suppose that $u$ is a minimizer of $\overline{\mathscr{F}}_{v}[-;\Omega]$ for compactly supported variations. 
\end{enumerate}
Denoting $Du=\nabla u\mathscr{L}^{n}\mres\Omega + D^{s}u$ the Lebesgue-Radon-Nikod\'{y}m decomposition of $Du$, in each of the cases~\ref{item:EL1} and~\ref{item:EL2} we have 
\begin{align}\label{eq:absmin}
\int_{\Omega} \! F(\nabla u) \, \dif x \leq \int_{\Omega} \! F(\nabla u +\nabla \varphi ) \, \dif x \qquad\text{for all}\;\varphi\in\sobo_{0}^{1,s}(\Omega;\R^{N})
\end{align}
whenever $\frac{p}{p-q+1}\leq s \leq\infty$. In each of the cases~\ref{item:EL1} and~\ref{item:EL2} it follows that if the integrand $F \colon \R^{N\times n} \to \R$ moreover is of class $\hold^{1}(\R^{N\times n})$, then we have validity of the  \emph{Euler-Lagrange system}
\begin{align}\label{eq:ELLip}
\int_{\Omega}\langle F'(\nabla u),\nabla\varphi\rangle\dif x = 0\qquad\text{for all}\;\varphi\in\sobo_{0}^{1,s}(\Omega;\R^{N})
\end{align}
provided $\frac{p}{p-q+1}\leq s \leq\infty$. In particular, $F^{\prime}( \nabla u)$ is \emph{row-wise divergence free in} $\Omega$, so satisfies 
\begin{equation}\label{EL}
\mathrm{div}(F^{\prime}(\nabla u))= 0 
\end{equation}
in the sense of distributions on $\Omega$.
\end{corollary}

\begin{proof}
Let $\varphi \in \sobo_{c}^{1,\infty}( \Omega ; \R^{N} )$.  By Proposition~\ref{prop:mazur} and passing to a subsequence, there exists a generating sequence $(u_{j})$ for $\overline{\mathscr{F}}_{v}^{*}[u;\Omega]$ if $p=1$ (or, by Corollary~\ref{cor:mazur}, for $\overline{\mathscr{F}}_{v}[u;\Omega]$ if $1<p<\infty$) that satisfies $\nabla u_{j}\to\nabla u$ $\mathscr{L}^{n}$-a.e. in $\Omega$; by continuity of $F$, we also have $F(\nabla u_{j}+\nabla\varphi)-F(\nabla u_{j})\to F(\nabla u+\nabla\varphi)-F(\nabla u)$ $\mathscr{L}^{n}$-a.e. in $\Omega$. By~\ref{item:H2} for $p=1$ or~\ref{item:H2p} for $p>1$, respectively, $(u_{j})$ is uniformly bounded in $\sobo^{1,p}(\Omega;\R^{N})$ by the same argument as in the proof of Proposition~\ref{prop:existenceminimizers}. Since $q<p+1$, the Lipschitz-type estimate~\eqref{eq:lipschitz} yields for all $\omega\subset\Omega$
\begin{align}\label{eq:MazurEqui}
\begin{split}
\int_{\omega}&|F(\nabla u_{j}+\nabla\varphi)-F(\nabla u_{j})|\dif x \leq c\int_{\omega}(1+|\nabla u_{j}|^{q-1}+|\nabla\varphi|^{q-1})|\nabla\varphi|\dif x \\ 
& \leq c\int_{\omega}|\nabla\varphi|+|\nabla\varphi|^{q}\dif x + c\Big(\int_{\Omega}|\nabla u_{j}|^{p}\Big)^{\frac{q-1}{p}}\Big(\int_{\omega}|\nabla\varphi|^{\frac{p}{p-q+1}}\dif x\Big)^{\frac{p-q+1}{p}},
\end{split}
\end{align}
where $c=c(n,N,L,q)>0$. From~\eqref{eq:MazurEqui} and because of $\sup_{j\in\mathbb{N}}\|\nabla u_{j}\|_{\lebe^{p}(\Omega)}<\infty$, we conclude that 
 $(F(\nabla u_{j}+\nabla\varphi)-F(\nabla u_{j}))$ is uniformly integrable. Now, by Vitali's convergence theorem and since the sequence $(u_{j}+\varphi )$ is admissible for
$u+\varphi$ we get
\begin{align*}
  0 & \leq \overline{\mathscr{F}}_{v}^{*}[u+\varphi ;\Omega]-\overline{\mathscr{F}}_{v}^{*}[u ;\Omega] \\ & \leq \liminf_{j \to \infty} \int_{\Omega} \! \bigl( F(\nabla u_{j}+\nabla \varphi ) - F(\nabla u_{j}) \bigr) \, \dif x = \int_{\Omega} \! \bigl( F(\nabla u + \nabla \varphi ) -F(\nabla u) \bigr) \, \dif x.
\end{align*}
Thus \eqref{eq:absmin} holds for $\varphi \in \sobo^{1,\infty}_{c}( \Omega ; \R^{N} )$, and then extends to $\varphi\in\sobo_{0}^{1,s}(\Omega;\R^{N})$ by \eqref{eq:lipschitz} and validity of \eqref{eq:absmin} for $\varphi\in\sobo_{c}^{1,\infty}(\Omega;\R^{N})$. Finally, when $F$ is $\hold^1$, then the Euler-Lagrange system \eqref{EL} is a standard consequence of \eqref{eq:lipschitz} and \eqref{eq:absmin}, recalling that 
\begin{align*}
(-1,1)\setminus\{0\}\ni \varepsilon \mapsto \frac{F(\nabla u+\varepsilon\nabla\varphi)-F(\nabla u)}{\varepsilon}
\end{align*}
is dominated by a multiple of $(1+|\nabla u|^{q-1}+|\nabla\varphi|^{q-1})|\nabla\varphi|\in\lebe^{1}$. The proof is complete.
\end{proof}
\begin{remark}\label{rem:measuredensityEuler}
In the case $p=1$ and even for $1\leq q<\frac{n}{n-1}$, the preceding corollary yields that the \emph{approximate gradients} $\nabla u$ (and not only the full gradients $Du$) of $\bv$-minimizers $u$ (for compactly supported variations) satisfy the weaker minimality property~\eqref{eq:absmin} and hereafter the Euler-Lagrange system~\eqref{eq:ELLip}. This might seem surprising since the approximate gradients of $\bv$-functions do not have gradient structure in general; by a result of \textsc{Alberti} \cite{Alberti1}, any $w\in\lebe^{1}(\Omega;\R^{N\times n})$ arises as the approximate gradient of a $\bv$-function. In the case $p=q=1$, the Euler-Lagrange system  can be seen by employing the integral representation of the relaxed functional (cf.~\eqref{eq:lingrowth}) and testing  the corresponding Euler-Lagrange system with $\sobo^{1,1}$-maps, but Proposition~\ref{prop:mazur} and Corollary~\ref{cor:EulerLagrange} show that integral representations are not even required for this conclusion. 

In the context of non-negative integrands and $p>1$, where the measure representations of \textsc{Fonseca \& Mal\'{y}}~\cite{FM,BFM} are available, \textsc{Schmidt}~\cite[Lems.~7.1--7.3]{SchmidtPR} derives the Euler-Lagrange system for the functionals $\mathscr{F}_{\locc}$ (cf.~\eqref{eq:SchmidtRelax}) as a consequence of the representation 
\begin{align}\label{eq:WaterColors} 
\mathscr{F}_{\locc}[u+\varphi;\omega]-\mathscr{F}_{\locc}[u;\omega] = \int_{\omega}F(\nabla u+\nabla\varphi)-F(\nabla u)\dif x 
\end{align}
for open $\omega\Subset\Omega$ and $\varphi\in\sobo^{1,p/(p-q+1)}(\Omega;\R^{N})$ provided  $\mathscr{F}_{\locc}[u;\Omega]<\infty$. This may be interpreted in the sense that if $u$ is perturbed by a sufficiently regular $\varphi$, then the singular part of the measure representing $\mathscr{F}_{\locc}[u;-]$ does not increase when passing to $\mathscr{F}_{\locc}[u+\varphi;-]$ on open sets $\omega\Subset\Omega$. We expect that~\eqref{eq:WaterColors} persists for the functionals $\overline{\mathscr{F}}^{*}$ or $\overline{\mathscr{F}}$ despite of the present lack of measure representations. However, we refrain from delving more into this matter here as~\eqref{eq:WaterColors} is not required for the partial regularity proof below, and shall pursue validity of~\eqref{eq:WaterColors} for generic maps $u$ elsewhere.
\end{remark}
We conclude this section by stating a variant of Corollary~\ref{cor:mazur} for local ($\bv$-)minimizers for compactly supported variations:
\begin{corollary}[Euler-Lagrange system II]\label{cor:ELMAIN}
In the situation of 
\begin{enumerate}
\item Corollary~\ref{cor:mazur}~\ref{item:EL1}, let $u$ be a local $\bv$-minimizer of $\overline{\mathscr{F}}^{*}$ for compactly supported variations. 
\item Corollary~\ref{cor:mazur}~\ref{item:EL2}, let $u$ be a local minimizer of $\overline{\mathscr{F}}$ for compactly supported variations. 
\end{enumerate}
Given $\frac{p}{p-q+1}\leq s\leq \infty$, we then have 
\begin{align}\label{eq:ELNEW}
\int_{\Omega}\langle F'(\nabla u),\nabla\varphi\rangle\dif x = 0\qquad\text{for all}\;\varphi\in\sobo_{c}^{1,s}(\Omega;\R^{N}). 
\end{align}
\end{corollary}
\begin{proof} Let $\varphi\in\sobo_{c}^{1,s}(\Omega;\R^{N})$ be arbitrary. For any $x_{0}\in\spt(\varphi)$, there exists an open neighbourhood $\omega^{x_{0}}\Subset\Omega$ such that $u$ is a ($\bv$-)minimizer on $\omega^{x_{0}}$ for compactly supported variations. By compactness of $\spt(\varphi)$, there exist finitely many $x_{0}^{1},...,x_{0}^{m}\in\spt(\varphi)$ such that $\spt(\varphi)\subset\bigcup_{j=1}^{m}\omega^{x_{0}^{j}}$. Let $(\rho_{j})$ be a partition of unity of $\spt(\varphi)$ subject to $(\omega^{x_{0}^{j}})_{j=1,...,m}$ such that $\spt(\rho_{j})\subset\omega^{x_{0}^{j}}$ for all $j\in\{1,...,m\}$. We then conclude by Corollary~\ref{cor:mazur}
\begin{align*}
\int_{\Omega}\langle F'(\nabla u),\nabla\varphi\rangle\dif x = \sum_{j=1}^{m}\int_{\omega^{x_{0}^{j}}}\langle F'(\nabla u),\nabla(\rho_{j}\varphi)\rangle\dif x =0, 
\end{align*}
which is~\eqref{eq:ELNEW}. The proof is complete. 
\end{proof}

\section{The Caccioppoli inequality of the second kind}\label{sec:Cacc}
\subsection{The Caccioppoli inequality for $(1,q)$-growth}
In this section we establish one of the core results of the paper, namely the Caccioppoli inequality of the second kind. We state the theorem for $\bv$-minimizers with respect to compactly supported variations and hereafter the $(1,q)$-growth regime; the $(p,q)$-case for $1<p<q<\frac{np}{n-1}$ is addressed in Section~\ref{sec:Caccpq}.
\begin{theorem}[Caccioppoli inequality of the second kind]\label{thm:Cacc}
Let $1<q<\frac{n}{n-1}$. Assume that $F\colon\R^{N\times n}\to\R$ satisfies \emph{\ref{item:H1}--\ref{item:H3}} and suppose that $x_{0}\in\Omega$ and $R>0$ are such that $\ball_{2R}(x_{0})\Subset\Omega$. Let $u\in\bv(\ball_{2R}(x_{0});\R^{N})$ be a \emph{$\bv$-minimizer} of $\overline{\mathscr{F}}_{u}^{*}[-;\ball_{2R}(x_{0})]$ for compactly supported variations; in particular, $\trace_{\partial\!\ball_{2R}(x_{0})}(u)\in\sobo^{1-1/q,q}(\partial\!\ball_{R}(x_{0});\R^{N})$. Moreover, given $m>0$, let $a\colon\R^{n}\to\R^{N}$ be an affine-linear map with $|\nabla a|\leq m$. 

Then there exists a constant $c=c(n,N,m,L,\ell_{m})>0$ such that 
\begin{align}\label{eq:Caccmain}
\begin{split}
\dashint_{\ball_{R}(x_{0})}V(D(u-a)) & \leq c\left[ \dashint_{\ball_{2R}(x_{0})}V \left(\frac{u-a}{R}\right) \right. \\ & \;\;\;\;\;\;\;\;\;\;\; \left.+ \sum_{j\in\{0,1\}}\left(\dashint_{\ball_{2R}(x_{0})}V\left(\frac{D^{j}(u-a)}{R^{1-j}}\right) \right)^{q}\right].
\end{split}
\end{align}
\end{theorem}
For the proof of Theorem~\ref{thm:Cacc} and its variant for $p>1$, cf.~Corollary~\ref{cor:Caccplargerthan1} below, we require the following iteration lemma that appears as a reformulation of \cite[Lem.~6.6]{SchmidtPR1} for a different choice of the auxiliary $V$-functions. It  can be obtained by a slight variation of the proof of \cite[Lem.~3.1, \S 5]{Giaquinta} by routine means and so is omitted here:
\begin{lemma}\label{lem:iteration}
Let $1\leq p<\infty$, $x_{0}\in\R^{n}$ and $R>0$. Suppose that $v\in\lebe^{p}(\ball_{2R}(x_{0});\R^{N})$ and that $h\colon [R,2R]\to\R_{>0}$ is a bounded function. Let $\lambda_{1},\lambda_{2},\lambda_{3},\gamma_{1},\gamma_{2},\gamma_{3}\geq 0$ and $0\leq\theta<1$ be constants such that for all $R\leq r < s \leq 2R$ there holds
\begin{align}\label{eq:therock1}
\begin{split}
h(r) & \leq \theta h(s) + \lambda_{1}(s-r)^{\gamma_{1}} \\ & + \lambda_{2}\int_{\ball_{2R}(x_{0})}V_{p}\Big(\frac{v}{s-r}\Big) + \lambda_{3}(s-r)^{\gamma_{2}}\Big(\int_{\ball_{2R}(x_{0})}V_{p}\Big(\frac{v}{s-r}\Big)\Big)^{\gamma_{3}}. 
\end{split}
\end{align} 
Then we have, with a constant $c=c(\gamma_{1},\gamma_{2},\gamma_{3},\theta,p)>0$, 
\begin{align}\label{eq:therock2}
h(R) \leq c\Big(\lambda_{1}R^{\gamma_{1}} & + \lambda_{2}\int_{\ball_{2R}(x_{0})}V_{p}\Big(\frac{v}{R}\Big) +\lambda_{3}R^{\gamma_{2}}\Big(\int_{\ball_{2R}(x_{0})}V_{p}\Big(\frac{v}{R}\Big)\Big)^{\gamma_{3}}\Big). 
\end{align}
\end{lemma}
\begin{proof}[Proof of Theorem~\ref{thm:Cacc}]
Let $x_{0}\in\Omega$ and $0<R\leq r<s\leq 2R<\dista(x_{0},\partial\Omega)$. Moreover, let $m>0$ and $a\colon\R^{n}\to\R^{N}$ be an affine-linear map with $|\nabla a|\leq m$. For ease of notation, we abbreviate $\widetilde{u}:=u-a$ and $\ball_{t}:=\ball_{t}(x_{0})$ for $0<t<2R$ in the sequel. As above, we let $\lambda\in\mathrm{RM}_{\mathrm{fin}}(\Omega')$ be a weak*-limit of the total variation sequence of a suitable generating sequence for $\overline{\mathscr{F}}_{\nabla a,v_{0}}^{*}[\widetilde{\mathbf{u}};\ball_{2R}(x_{0}),\Omega']$, where $\Omega'$ is open and bounded with $\ball_{R}(x_{0})\Subset\Omega'$ and $\widetilde{\mathbf{u}}$ is the extension of $\widetilde{u}$ to $\Omega'$ by some $v_{0}\in\sobo^{1,q}(\Omega';\R^{N})$ such that $\trace_{\partial\!\ball_{2R}}(v_{0})=\trace_{\partial\!\ball_{2R}}(\widetilde{u})$ $\mathscr{H}^{n-1}$-a.e. on $\partial\!\ball_{2R}$. 

\emph{Step 1. Choosing good radii.} As above, $\mathscr{L}^{1}$-almost all $t\in(r,s)$ satisfy $\mathcal{M}Du(x_{0},t)<\infty$ and are additivity radii for $\overline{\mathscr{F}}_{\nabla a,\widetilde{u}}^{*}[\widetilde{u};-]$. We then denote the exceptional set $\mathscr{N}(r,s):=\{t\in(r,s)\colon\;\mathcal{M}Du(x_{0},t)=\infty\;\text{or}\;\mathcal{M}\lambda(x_{0},t)=\infty\}$ and define
\begin{align}\label{eq:ThetaDefCacc}
\Theta(t):=\int_{\overline{\ball}_{t}}V\Big(\frac{\widetilde{u}}{s-r}\Big)\dif x +\int_{\overline{\ball}_{t}}V(D\widetilde{u}),\qquad t\in [r,s]. 
\end{align}
Since $V(D\widetilde{u})$ is a Radon measure on $\ball_{2R}$, $\Theta\colon [r,s]\to [0,\infty)$  still is non-decreasing and right-continuous. Thus we may apply  Lemma \ref{lem:HLW} to the particular choice 
\begin{align}\label{eq:CaccEdef}
E:=\mathscr{N}(r,s)\cup (r,\tfrac{19}{20}r+\tfrac{1}{20}s)\cup (\tfrac{19}{20}s+\tfrac{1}{20}r,s)
\end{align}
and $f=\Theta$. Then we have $\mathscr{L}^{1}(E)=\frac{1}{10}(s-r)<\theta(s-r)$ with $\theta=\frac{9}{80}$. By definition of $E$ and in light of Lemma~\ref{lem:additivityradii}, Lemma~\ref{lem:HLW} provides us with additivity radii $\widetilde{r},\widetilde{s}\in (r,s)\cap E^{\complement}$ for $\overline{\mathscr{F}}_{\nabla a,\widetilde{u}}^{*}[\widetilde{u};-]$ satisfying $\widetilde{r}<\widetilde{s}$, $\tfrac{1}{8}(s-r)\leq \widetilde{s}-\widetilde{r}\leq s-r$, $\mathcal{M}Du(x_{0},\widetilde{r}),\mathcal{M}Du(x_{0},\widetilde{s})<\infty$ together with
\begin{align}\label{eq:uniformlycomparable}
\begin{split}
&\frac{\Theta(\tau)-\Theta(\widetilde{r})}{\tau-\widetilde{r}}\leq 8000 \frac{\Theta(s)-\Theta(r)}{s-r} \qquad\text{for all}\;\tau\in (\widetilde{r},s),\\
&\frac{\Theta(\widetilde{r})-\Theta(\tau)}{\widetilde{r}-\tau}\leq 8000 \frac{\Theta(s)-\Theta(r)}{s-r} \qquad\text{for all}\;\tau\in (r,\widetilde{r}),\\
& \frac{\Theta(\widetilde{s})-\Theta(\tau)}{\widetilde{s}-\tau}\leq 8000 \frac{\Theta(s)-\Theta(r)}{s-r} \qquad\text{for all}\;\tau\in (r,\widetilde{s}),\\
& \frac{\Theta(\tau)-\Theta(\widetilde{s})}{\tau-\widetilde{s}}\leq 8000 \frac{\Theta(s)-\Theta(r)}{s-r} \qquad\text{for all}\;\tau\in (\widetilde{s},s).
\end{split}
\end{align}
\emph{Step 2. Estimating layer terms.} For later usage, we now address the estimation of layer terms which arise in the construction of certain competitor maps. Let $r<\varrho_{1}<\varrho_{2}<s$ be such that 
\begin{align}\label{eq:uniformchoose}
\sup_{t\in (\varrho_{1},\varrho_{2})}\frac{\Theta(t)-\Theta(\varrho_{1})}{t-\varrho_{1}} + \sup_{t\in (\varrho_{1},\varrho_{2})}\frac{\Theta(\varrho_{2})-\Theta(t)}{\varrho_{2}-t}&\leq \kappa\frac{\Theta(s)-\Theta(r)}{s-r},
\end{align}
for some $0<\kappa<\infty$. Let $(\ball^{i})$ be a Whitney covering for $\ball_{\varrho_{2}}\setminus\overline{\ball}_{\varrho_{1}}$ satisfying \ref{item:Whitney1}--\ref{item:Whitney3} from Section~\ref{sec:Fubini}.  Define index sets $\mathcal{I}_{\leq},\mathcal{J}_{\leq}$ and $\mathcal{I}_{>},\mathcal{J}_{>}$ by 
\begin{align*}
& \mathcal{I}_{\leq}:=\big\{i\in\mathbb{N}\colon\; \dashint_{\Lambda\!\ball^{i}}|D\widetilde{u}|\leq 1 \big\},\;\;\;\mathcal{I}_{>}:=\mathbb{N}\setminus\mathcal{I}_{\leq},\\
& \mathcal{J}_{\leq}:=\big\{i\in\mathbb{N}\colon\; \dashint_{\Lambda\!\ball^{i}}\left\vert\frac{\widetilde{u}}{s-r}\right\vert\dif x \leq 1 \big\},\;\;\;\mathcal{J}_{>}:=\mathbb{N}\setminus\mathcal{J}_{\leq}
\end{align*}
and introduce a convex and differentiable function $m_{q}\colon \R\to[0,\infty)$ by 
\begin{align}\label{eq:mqdef}
m_{q}(z):=\begin{cases} 
|z|^{2} &\;\text{for}\;|z|\leq 1,\\ 
\frac{2}{q}|z|^{q} + 1 - \frac{2}{q}&\;\text{for}\;|z|>1. 
\end{cases} 
\end{align}
An elementary computation establishes that for any finite dimensional, real normed space $(X,\|\cdot\|_{X})$ the uniform equivalences
\begin{align}\label{eq:auxMq}
\begin{split}
m_{q}(\|z\|_{X})&\sim_{q}\min\{\|z\|_{X}^{2},\|z\|_{X}^{q}\}\;\;\text{for all $z\in X$}\qquad\text{if $1\leq q\leq 2$},\\ 
m_{q}(\|z\|_{X})&\sim_{q}\max\{\|z\|_{X}^{2},\|z\|_{X}^{q}\}\;\;\text{for all $z\in X$}\qquad\!\text{if $2\leq q< \infty$}
\end{split}
\end{align}
hold. Here and in what follows, we put $\mathbb{E}_{\varrho_{1},\varrho_{2}}:=\mathbb{E}_{\ball_{\varrho_{2}}\setminus\overline{\ball}_{\varrho_{1}}}$ with the trace-preserving operator from Section~\ref{sec:Fubini}. Given $x\in \ball_{\varrho_{2}}\setminus\overline{\ball}_{\varrho_{1}}$, we distinguish two cases: If $x\in\ball^{i}$ with $i\in\mathcal{I}_{\leq}$, then by Lemma~\ref{lem:extensionoperator}~\ref{item:extension1} we have because of $1\leq q <\frac{n}{n-1}$
\begin{align}\label{eq:Caccest1}
\begin{split}
m_{q}(|\nabla\mathbb{E}_{\varrho_{1},\varrho_{2}}\widetilde{u}(x)|) & \leq c m_{q}\Big(\dashint_{\Lambda\!\ball^{i}}|D\widetilde{u}| \Big) \,\;\;\;\;\;\;\;\;\text{(by Lemma~\ref{lem:extensionoperator}~\ref{item:extension1})}\\ 
&\leq c \Big(\dashint_{\Lambda\!\ball^{i}}|D\widetilde{u}| \Big)^{2} \,\;\;\;\;\;\;\;\;\;\;\;\text{(by ~ $\eqref{eq:auxMq}_{1}$ and $i\in\mathcal{I}_{\leq}$)}\\ & \leq c \dashint_{\Lambda\!\ball^{i}}V(D\widetilde{u})\;\;\;\;\;\;\;\;\;\;\;\;\;\;\text{(by Lemma~\ref{lem:Efunction}, $i\in\mathcal{I}_{\leq}$ and Jensen)}.
\end{split}
\end{align}
On the other hand, if $x\in\ball^{i}$ is such that $i\in\mathcal{I}_{>}$, then we have 
\begin{align}\label{eq:Caccest2}
\begin{split}
m_{q}(|\nabla\mathbb{E}_{\varrho_{1},\varrho_{2}}\widetilde{u}(x)|) & \leq c m_{q}\Big(\dashint_{\Lambda\!\ball^{i}}|D\widetilde{u}| \Big) \,\;\;\;\;\;\;\;\;\text{(by Lemma~\ref{lem:extensionoperator}~\ref{item:extension1})}\\ &\leq c \Big(\dashint_{\Lambda\!\ball^{i}}|D\widetilde{u}| \Big)^{q}\;\;\;\;\;\;\;\;\;\;\;\;\text{(by~$\eqref{eq:auxMq}_{1}$ and $i\in\mathcal{I}_{>}$)}\\ & \leq c \Big(\dashint_{\Lambda\!\ball^{i}}V(D\widetilde{u}) \Big)^{q}\,\;\;\;\;\;\;\;\,\text{(by Lemma~\ref{lem:Efunction} and Jensen)}.
\end{split}
\end{align}
We now combine the previous estimates to arrive at 
\begin{align}\label{eq:I0bound}
\begin{split}
\int_{\ball_{\varrho_{2}}\setminus\overline{\ball}_{\varrho_{1}}}m_{q}& (|\nabla\mathbb{E}_{\varrho_{1},\varrho_{2}} \widetilde{u}|)\dif x \leq \sum_{i\in\mathbb{N}}\int_{\ball^{i}}m_{q}(|\nabla\mathbb{E}_{\varrho_{1},\varrho_{2}}\widetilde{u}|)\dif x \\
& = \sum_{i\in\mathcal{I}_{\leq}}\int_{\ball^{i}}m_{q}(|\nabla\mathbb{E}_{\varrho_{1},\varrho_{2}}\widetilde{u}|)\dif x + \sum_{i\in\mathcal{I}_{>}}\int_{\ball^{i}}m_{q}(|\nabla\mathbb{E}_{\varrho_{1},\varrho_{2}}\widetilde{u}|)\dif x \\
& \!\!\!\!\!\!\!\stackrel{\eqref{eq:Caccest1},\eqref{eq:Caccest2}}{\leq} c\Big(\sum_{i\in\mathcal{I}_{\leq}}\int_{\Lambda\!\ball^{i}}V(D\widetilde{u}) + \sum_{i\in\mathcal{I}_{>}}\int_{\ball^{i}}\Big(\dashint_{\Lambda\!\ball^{i}}V(D\widetilde{u}) \Big)^{q}\Big)\dif x \\
& \!\!\!\!\!\!\!\!\!\!\!\!\stackrel{\text{\ref{item:Whitney2}},\,\ell^{1}\hookrightarrow\ell^{q}}{\leq} c\Big(\int_{\ball_{\varrho_{2}}\setminus\overline{\ball}_{\varrho_{1}}}V(D\widetilde{u}) + \Big(\sum_{i\in\mathcal{I}_{>}}r_{i}^{n(\frac{1}{q}-1)}\int_{\Lambda\!\ball^{i}}V(D\widetilde{u})\Big)^{q}\Big)\\
& \!\!=: c(\mathrm{I}_{a} + \mathrm{I}_{b}).
\end{split}
\end{align}
As to $\mathrm{I}_{b}$, we introduce as in~\eqref{eq:Imdef} for $m\in\mathbb{N}_{0}$
\begin{align*}
\mathcal{I}_{>}^{m}:=\big\{i\in\mathcal{I}_{>}\colon\;2^{-m-1}(\varrho_{2}-\varrho_{1})\leq r_{i}<2^{-m}(\varrho_{2}-\varrho_{1}) \big\} 
\end{align*}
so that, with a constant $c'>0$ independent of of $i,m$, $\varrho_{1},\varrho_{2}$ and $\widetilde{u}$, $\bigcup_{i\in\mathcal{I}_{>}^{m}}\Lambda\!\ball^{i}\subset S^{m}$, where we put as in the proof of Lemma~\ref{lem:extensionoperator}
\begin{align*}
S^{m}:=\Big\{x\in\ball_{\varrho_{2}}\setminus\overline{\ball}_{\varrho_{1}}\colon\;\frac{1}{c'}2^{-m}(\varrho_{2}-\varrho_{1})<\mathrm{dist}(x,\partial(\ball_{\varrho_{2}}\setminus\overline{\ball}_{\varrho_{1}}))<c'2^{-m}(\varrho_{2}-\varrho_{1})\Big\}.
\end{align*}
As in the proof of Lemma~\ref{lem:extensionoperator}, we then arrive at the following string of inequalities:
\begin{align}\label{eq:Ibest}
\begin{split}
\mathrm{I}_{b} & \leq c\Big(\sum_{m=0}^{\infty}\sum_{i\in\mathcal{I}_{>}^{m}}r_{i}^{n(\frac{1}{q}-1)}\int_{\Lambda\!\ball_{i}}V(D\widetilde{u}) \Big)^{q} \\
& \leq  c\Big(\sum_{m=0}^{\infty}\Big(\Big(2^{-m}(\varrho_{2}-\varrho_{1})\Big)^{n(\frac{1}{q}-1)}\frac{(\varrho_{2}-\varrho_{1})}{2^{m}}\Big)\times \frac{2^{m}}{(\varrho_{2}-\varrho_{1})}\int_{S^{m}}V(D\widetilde{u})\Big)\Big)^{q} \\
& \leq c\Big((\varrho_{2}-\varrho_{1})^{n(\frac{1}{q}-1)+1}\times \Big. \\ & \Big. \times\sum_{m=0}^{\infty} 2^{m (n(1-\frac{1}{q})-1)} \Big(\sup_{t\in (\varrho_{1},\varrho_{2})}\frac{\Theta(t)-\Theta(\varrho_{1})}{t-\varrho_{1}} + \sup_{t\in (\varrho_{1},\varrho_{2})}\frac{\Theta(\varrho_{2})-\Theta(t)}{\varrho_{2}-t}\Big)\Big)^{q} \\ & =: \mathrm{II}. 
\end{split}
\end{align} 
Again, due to our assumption $q<\frac{n}{n-1}$, $n(\frac{1}{q}-1)+1 >0$. Thus by convergence of the geometric series and \eqref{eq:uniformchoose},
\begin{align}\label{eq:IIest}
\begin{split}
\mathrm{II} & \!\!\stackrel{\eqref{eq:uniformchoose}}{\leq} c\Big(\kappa(\varrho_{2}-\varrho_{1})^{n(\frac{1}{q}-1)+1}\frac{\Theta(s)-\Theta(r)}{s-r}\Big)^{q} \\
& \!\!\!\!\!\!\!\!\!\!\stackrel{\varrho_{2}-\varrho_{1}\leq s-r}{\leq} c\kappa^{q}\Big((s-r)^{n(\frac{1}{q}-1)}\Big(\int_{\overline{\ball}_{s}\setminus\overline{\ball}_{r}} V \Big(\frac{\widetilde{u}}{s-r}\Big)\dif x + V(D\widetilde{u}) \Big)\Big)^{q} \\
& = c\kappa^{q}(s-r)^{n}\Big((s-r)^{-n}\int_{\overline{\ball}_{s}\setminus\overline{\ball}_{r}} V \Big(\frac{\widetilde{u}}{s-r}\Big)\dif x + V(D\widetilde{u})  \Big)^{q}.
\end{split}
\end{align}
For the bounds of lower order terms, we replace $\nabla\mathbb{E}_{\varrho_{1},\varrho_{2}}\widetilde{u}$ in \eqref{eq:Caccest1},~\eqref{eq:Caccest2} and~\eqref{eq:I0bound}\emph{ff}.  systematically by $\frac{\mathbb{E}_{\varrho_{1},\varrho_{2}}\widetilde{u}}{s-r}$ and then use Lemma~\ref{lem:extensionoperator}~\ref{item:extension1} with $j=0$. In consequence, 
\begin{align}\label{eq:mainstepCacc}
\begin{split}
\sum_{j\in\{0,1\}}\int_{\ball_{\varrho_{2}}\setminus\overline{\ball}_{\varrho_{1}}}&m_{q}\Big(\frac{|\nabla^{j}\mathbb{E}_{\varrho_{1},\varrho_{2}}\widetilde{u}|}{(s-r)^{1-j}}\Big)\dif x \leq c\Big(\int_{\overline{\ball}_{s}\setminus \overline{\ball}_{r}}\Big(V \Big(\frac{\widetilde{u}}{s-r}\Big)\dif x + V(D\widetilde{u})\Big) \Big. \\
& \Big. + (s-r)^{n}\Big((s-r)^{-n}\int_{\overline{\ball}_{s}\setminus\overline{\ball}_{r}} \Big(V \Big(\frac{\widetilde{u}}{s-r}\Big)\dif x + V(D\widetilde{u})\Big) \Big)^{q}\Big) 
\end{split}
\end{align}
with a constant $c=c(q,N,n,\kappa)>0$.

\emph{Step 3. Minimality and derivation of \eqref{eq:Caccmain}.} Put $\sigma:=\frac{\widetilde{s}-\widetilde{r}}{10}$. We pick a smooth cut-off function $\rho\in\hold_{c}^{\infty}(\ball_{\widetilde{s}};[0,1])$ with 
\begin{align}\label{eq:rhoprops}
\mathbbm{1}_{\ball_{\widetilde{r}+\sigma}}\leq\rho\leq\mathbbm{1}_{\ball_{\widetilde{s}-\sigma}}\;\;\;\text{and}\;\;\;|\nabla\rho|\leq \frac{10}{\widetilde{s}-\widetilde{r}}.
\end{align}
We then put
\begin{align}\label{eq:defT}
\mathbb{T}\widetilde{u} :=\begin{cases} 
\widetilde{u}&\;\text{in}\;\ball_{\widetilde{r}}\\
\mathbb{E}_{\widetilde{r},\widetilde{s}}\widetilde{u}&\;\text{in}\;\ball_{\widetilde{s}}\setminus\overline{\ball}_{\widetilde{r}},\\
\widetilde{u}&\;\text{in}\;\ball_{2R}\setminus\overline{\ball}_{\widetilde{s}},  
\end{cases} 
\end{align}
where $\mathbb{E}_{\widetilde{r},\widetilde{s}}:=\mathbb{E}_{\ball_{\widetilde{s}}\setminus\overline{\ball}_{\widetilde{r}}}$ is given as in \eqref{eq:extoperator}, and define the future competitor map 
\begin{align*}
\psi:=(1-\rho)\mathbb{T}\widetilde{u}-\widetilde{u}\in\bv_{c}(\ball_{2R};\R^{N}).
\end{align*}
By Lemma~\ref{lem:shiftconnect} and as $u$ is a $\bv$-minimizer of $\overline{\mathscr{F}}^{*}[-;\ball_{2R}]$ for compactly supported variations,  
\begin{align}\label{eq:modBVminimality}
\overline{\mathscr{F}}_{\nabla a,\widetilde{u}}^{*}[\widetilde{u};\ball_{2R}]\leq \overline{\mathscr{F}}_{\nabla a,\widetilde{u}}^{*}[\widetilde{u}+\psi;\ball_{2R}].
\end{align}
\emph{Step 3a. Finiteness.}
For~\eqref{eq:modBVminimality} to be non-vacuous, we require $\overline{\mathscr{F}}_{\nabla a,\widetilde{u}}^{*}[\widetilde{u}+\psi;\ball_{2R}]<\infty$. This is seen as follows: Since $\overline{\mathscr{F}}_{\nabla a,\widetilde{u}}^{*}[\widetilde{u};\ball_{2R}\setminus\overline{\ball}_{\widetilde{s}}]<\infty$ by Lemmas~\ref{lem:additivityradii},~\ref{lem:shiftconnect} and Remark~\ref{rem:Columbo}, 
there exists $w_{0}\in \sobo_{0}^{1,q}(\Omega';\R^{N})$ such that $\trace_{\partial\!\ball_{\widetilde{r}}}(w_{0})=\trace_{\partial\!\ball_{\widetilde{r}}}(\widetilde{u})$, $\trace_{\partial\!\ball_{\widetilde{s}}}(w_{0})=\trace_{\partial\!\ball_{\widetilde{s}}}(\widetilde{u})$ and $\trace_{\partial\!\ball_{2R}}(w_{0})=\trace_{\partial\!\ball_{2R}}(\widetilde{u})$ $\mathscr{H}^{n-1}$-a.e. on $\partial\!\ball_{\widetilde{r}},\partial\!\ball_{\widetilde{s}}$ or $\partial\!\ball_{2R}$, respectively. We then find a generating sequence $(v_{j})\subset\mathscr{A}_{w_{0}}^{q}(\ball_{2R}\setminus\overline{\ball}_{\widetilde{s}},\Omega')$ for $\overline{\mathscr{F}}_{\nabla a,\widetilde{u}}^{*}[\widetilde{u};\ball_{2R}\setminus\overline{\ball}_{\widetilde{s}}]$. 
Now define, for $j\in\mathbb{N}$, 
\begin{align*}
 w_{j}:=\begin{cases} 
(1-\rho)\mathbb{E}_{\widetilde{r},\widetilde{s}}\widetilde{u}&\;\text{in}\;\ball_{\widetilde{s}},\\
v_{j}&\;\text{in}\;\ball_{2R}\setminus\overline{\ball}_{\widetilde{s}}, \\ 
w_{0}&\;\text{in}\;\Omega'\setminus\overline{\ball}_{2R},\end{cases}\;\;\;\text{and}\;\;\;w:=\begin{cases} 
(1-\rho)\mathbb{E}_{\widetilde{r},\widetilde{s}}\widetilde{u}&\;\text{in}\;\ball_{\widetilde{s}},\\
\widetilde{u}&\;\text{in}\;\ball_{2R}\setminus\overline{\ball}_{\widetilde{s}},\\
w_{0}&\;\text{in}\;\Omega'\setminus\overline{\ball}_{2R}.
\end{cases}
\end{align*}
Then we have $(w_{j})\subset\mathscr{A}_{w_{0}}^{q}(\ball_{2R},\Omega')$ and $w_{j}\stackrel{*}{\rightharpoonup}w$ in $\bv(\Omega';\R^{N})$. On the other hand, we have $\mathbb{E}_{\widetilde{r},\widetilde{s}}\widetilde{u}\in\sobo^{1,q}(\ball_{\widetilde{s}}\setminus\overline{\ball}_{\widetilde{r}};\R^{N})$; this is a consequence of step 2 applied to $\varrho_{1}=\widetilde{r}$, $\varrho_{2}=\widetilde{s}$ and the definition of $m_{q}$.  By the support properties of $\rho$, we thus have $(1-\rho)\mathbb{T}u\in\sobo^{1,q}(\ball_{\widetilde{s}};\R^{N})$, and since $(1-\rho)\mathbb{T}u$ attains the same boundary values along $\partial\!\ball_{\widetilde{s}}$ as $\widetilde{u}$, Lemma~\ref{lem:consistency} gives 
\begin{align}\label{eq:justonemorething}
\overline{\mathscr{F}}_{\nabla a,\widetilde{u}}^{*}[(1-\rho)\mathbb{T}\widetilde{u};\ball_{\widetilde{s}}] = \int_{\ball_{\widetilde{s}}}F_{\nabla a}(\nabla((1-\rho)\mathbb{T}u))\dif x.
\end{align}
In consequence,  
\begin{align}\label{eq:ukraine}
\begin{split}
\overline{\mathscr{F}}_{\nabla a,\widetilde{u}}^{*}[\widetilde{u}+\psi;\ball_{2R}] & \stackrel{\substack{(w_{j})\subset\mathscr{A}_{w_{0}}^{q}(\ball_{2R},\Omega'),\\\,w_{j}\stackrel{*}{\rightharpoonup}w}}{\leq} \liminf_{j\to\infty}\Big(\int_{\ball_{\widetilde{s}}\setminus\overline{\ball}_{\widetilde{r}}}F_{\nabla a}(\nabla ((1-\rho)\mathbb{E}_{\widetilde{r},\widetilde{s}}\widetilde{u}))\dif x \Big. \\ & \Big. \;\;\;\;\;\;\;\;\;\;\;\;\;\;\;\;\;\;\;\;\;\;\;\;\;\;\;\; + \int_{\ball_{2R}\setminus\overline{\ball}_{\widetilde{s}}}F_{\nabla a}(\nabla  v_{j})\dif x\Big) \\  
& \;\;\;\;\; \;\;\;\;\;\;= \overline{\mathscr{F}}_{\nabla a,\widetilde{u}}^{*}[(1-\rho)\mathbb{T}\widetilde{u};\ball_{\widetilde{s}}] + \overline{\mathscr{F}}_{\nabla a,\widetilde{u}}^{*}[\widetilde{u};\ball_{2R}\setminus\overline{\ball}_{\widetilde{s}}], 
\end{split}
\end{align}
where we additionally used $F_{\nabla a}(0)=0$. In particular, $\overline{\mathscr{F}}_{\nabla a,\widetilde{u}}^{*}[u+\psi;\ball_{2R}]<\infty$.

\emph{Step 3b. Derivation of~\eqref{eq:Caccmain}.} With $w_{0}$ as in step 3a, we let $(\widetilde{u}_{j})\subset\mathscr{A}_{w_{0}}^{q}(\ball_{\widetilde{r}},\Omega')$ be a generating sequence for $\overline{\mathscr{F}}_{\nabla a,\widetilde{u}}^{*}[\widetilde{u};\ball_{\widetilde{r}}]$. This value is finite and $\widetilde{r}$ is an additivity radius for $\overline{\mathscr{F}}^{*}[u;-]$, which can be seen by Lemma~\ref{lem:additivityradii}~\ref{item:fin1},~\ref{item:fin2}, Remark~\ref{rem:Columbo} and $\widetilde{r}\in E^{\complement}$. With the cut-off function $\rho$ from above, we then define 
\begin{align*}
\varphi_{j}:=\begin{cases} \rho\widetilde{u}_{j}&\;\text{in}\;\ball_{\widetilde{r}},\\
\rho\mathbb{E}_{\widetilde{r},\widetilde{s}}\widetilde{u}&\;\text{in}\;\ball_{\widetilde{s}}\setminus\overline{\ball}_{\widetilde{r}},\\
0&\;\text{in}\;\ball_{2R}\setminus\overline{\ball}_{\widetilde{s}}, 
\end{cases}
\end{align*}
so that $\varphi_{j}\in\sobo_{c}^{1,q}(\ball_{\widetilde{s}};\R^{N})$ and $\varphi_{j}\to\rho\mathbb{T}\widetilde{u}$ in $\lebe_{\locc}^{1}(\ball_{\widetilde{s}};\R^{N})$. Moreover, since $|\nabla a|\leq m$ by assumption, we have Lemma~\ref{lem:shiftconnect}~\ref{item:shiftconnect3} and hence~\eqref{eq:Vboundbelow} at our disposal. We thus conclude 
\begin{align*}
\ell^{(m)}\int_{\overline{\ball}_{r}}V(D\widetilde{u}) &  \leq \ell^{(m)}\liminf_{j\to\infty}\int_{\ball_{\widetilde{s}}}V(\nabla \varphi_{j})\dif x \;\;\;\;\;\;\;\;\;\;\;\;\;\;\;\;(\text{by Lemma~\ref{lem:reshetnyak}})\\ & \!\!\stackrel{\eqref{eq:Vboundbelow}}{\leq} \liminf_{j\to\infty}\int_{\ball_{\widetilde{s}}}F_{\nabla a}(\nabla\varphi_{j})\dif x \\ 
& = \Big(\liminf_{j\to\infty}\int_{\ball_{\widetilde{r}}}F_{\nabla  a}(\nabla \widetilde{u}_{j})\dif x\Big) + \int_{\ball_{\widetilde{s}}\setminus\overline{\ball}_{\widetilde{r}}}F_{\nabla a}(\nabla(\rho\mathbb{T}\widetilde{u}))\dif x\\ 
& \!\stackrel{(*)}{=} \overline{\mathscr{F}}_{\nabla a,\widetilde{u}}^{*}[\widetilde{u};\ball_{2R}] - \overline{\mathscr{F}}_{\nabla a,\widetilde{u}}^{*}[\widetilde{u};\ball_{2R}\setminus\overline{\ball}_{\widetilde{r}}] + \int_{\ball_{\widetilde{s}}\setminus\overline{\ball}_{\widetilde{r}}}F_{\nabla a}(\nabla(\rho\mathbb{T}\widetilde{u}))\dif x \\ 
& \!\!\!\stackrel{\eqref{eq:modBVminimality}}{\leq}  \overline{\mathscr{F}}_{\nabla a,\widetilde{u}}^{*}[\widetilde{u}+\psi;\ball_{2R}] - \overline{\mathscr{F}}_{\nabla a,\widetilde{u}}^{*}[\widetilde{u};\ball_{2R}\setminus\overline{\ball}_{\widetilde{r}}] + \int_{\ball_{\widetilde{s}}\setminus\overline{\ball}_{\widetilde{r}}}F_{\nabla a}(\nabla(\rho\mathbb{T}\widetilde{u}))\dif x \\
& \!\!\!\stackrel{\eqref{eq:ukraine}}{\leq}   \overline{\mathscr{F}}_{\nabla a,\widetilde{u}}^{*}[(1-\rho)\mathbb{T}\widetilde{u};\ball_{\widetilde{s}}] + (\overline{\mathscr{F}}_{\nabla a,\widetilde{u}}^{*}[\widetilde{u};\ball_{2R}\setminus\overline{\ball}_{\widetilde{s}}] - \overline{\mathscr{F}}_{\nabla a,\widetilde{u}}^{*}[\widetilde{u};\ball_{2R}\setminus\overline{\ball}_{\widetilde{r}}]) \\ 
& + \int_{\ball_{\widetilde{s}}\setminus\overline{\ball}_{\widetilde{r}}}F_{\nabla a}(\nabla(\rho\mathbb{T}\widetilde{u}))\dif x \\ 
& =: \mathrm{III}+\mathrm{IV}+\mathrm{V}, 
\end{align*}
where we used at $(*)$ that $\widetilde{r}$ is an additivity radius for $\overline{\mathscr{F}}_{\nabla a ,\widetilde{u}}^{*}[u;-]$ and $(\widetilde{u}_{j})$ is generating for $\overline{\mathscr{F}}_{\nabla a,\widetilde{u}}^{*}[\widetilde{u};\ball_{\widetilde{r}}]$. We now turn to the estimation of $\mathrm{III}, \mathrm{IV}$ and $\mathrm{V}$. 

Ad $\mathrm{III}$ and $\mathrm{V}$. Note that $(1-\rho)\mathbb{T}\widetilde{u}$ is supported away from $\overline{\ball}_{\widetilde{r}}$ and its restriction to $\ball_{\widetilde{s}}\setminus\overline{\ball}_{\widetilde{r}}$ belongs to $\sobo^{1,q}(\ball_{\widetilde{s}}\setminus\overline{\ball}_{\widetilde{r}};\R^{N})$. Hence, by Lemma~\ref{lem:consistency} and the definition of $\mathbb{T}$, 
\begin{align}\label{eq:Cacc1}
\begin{split}
\mathrm{III} & \leq \int_{\ball_{\widetilde{s}}\setminus\overline{\ball}_{\widetilde{r}}}F_{\nabla a}(\nabla ((1-\rho)\mathbb{E}_{\widetilde{r},\widetilde{s}}\widetilde{u}))\dif x \\ 
& \!\!\!\!\!\!\!\!\!\stackrel{\text{Lem.}~\ref{lem:shifted}\ref{item:shifted(a)}}{\leq} c\int_{\ball_{\widetilde{s}}\setminus\overline{\ball}_{\widetilde{r}}}m_{q}((1-\rho)\nabla\mathbb{E}_{\widetilde{r},\widetilde{s}}\widetilde{u}-\mathbb{E}_{\widetilde{r},\widetilde{s}}\widetilde{u}\otimes\nabla\rho)\dif x \\
& \leq c\int_{\ball_{\widetilde{s}}\setminus\overline{\ball}_{\widetilde{r}}}m_{q}(\nabla\mathbb{E}_{\widetilde{r},\widetilde{s}}\widetilde{u})\dif x + c\int_{\ball_{\widetilde{s}}\setminus\overline{\ball}_{\widetilde{r}}}m_{q}\left(\frac{\mathbb{E}_{\widetilde{r},\widetilde{s}}\widetilde{u}}{s-r}\right)\dif x
\end{split}
\end{align}
by~\eqref{eq:rhoprops} and $(\widetilde{s}-\widetilde{r})\sim(s-r)$ by~\eqref{eq:CaccEdef}\emph{ff.}. Here, $c=c(n,N,L,\ell^{(m)},m,q)>0$. In a similar vein, we obtain 
\begin{align}\label{eq:Cacc1A}
\mathrm{V} \leq c\int_{\ball_{\widetilde{s}}\setminus\overline{\ball}_{\widetilde{r}}}m_{q}(\nabla\mathbb{E}_{\widetilde{r},\widetilde{s}}\widetilde{u})\dif x + c\int_{\ball_{\widetilde{s}}\setminus\overline{\ball}_{\widetilde{r}}}m_{q}\left(\frac{\mathbb{E}_{\widetilde{r},\widetilde{s}}\widetilde{u}}{s-r}\right)\dif x. 
\end{align}
We may then combine~\eqref{eq:Cacc1} and~\eqref{eq:Cacc1A}. To this end, we recall that by ~\eqref{eq:uniformlycomparable},~\eqref{eq:uniformchoose} is satisfied with $\varrho_{1}=\widetilde{r}$, $\varrho_{2}=\widetilde{s}$ and $\kappa=16000$. By step 2,~\eqref{eq:mainstepCacc} is applicable and thus yields
\begin{align}\label{eq:TabeaTscherpelTababy}
\begin{split}
\mathrm{III}+\mathrm{V} & \leq c\Big(\int_{\overline{\ball}_{s}\setminus \overline{\ball}_{r}}\Big(V \Big(\frac{\widetilde{u}}{s-r}\Big)\dif x + V(D\widetilde{u})\Big) \Big. \\
& \Big. + (s-r)^{n}\Big((s-r)^{-n}\int_{\overline{\ball}_{s}\setminus\overline{\ball}_{r}} \Big(V \Big(\frac{\widetilde{u}}{s-r}\Big)\dif x + V(D\widetilde{u})\Big) \Big)^{q}\Big). 
\end{split}
\end{align} 
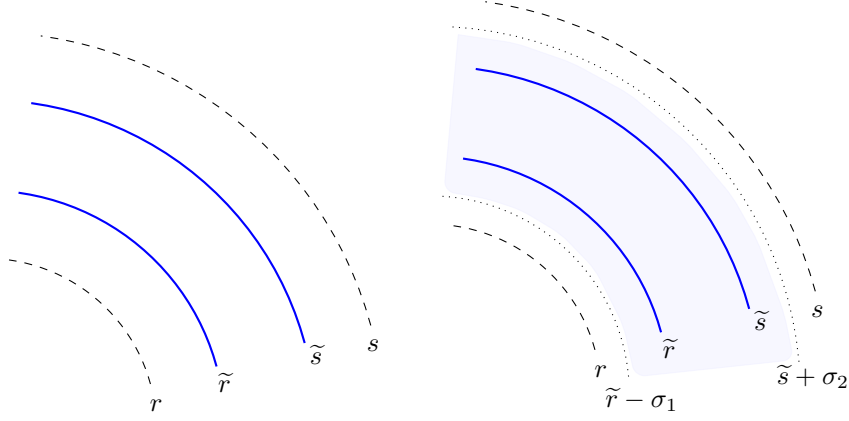
\begin{figure}
\begin{tikzpicture}[scale=0.75]
\draw[black,dashed,domain=15:82] plot ({3*cos(\x)},{3*sin(\x)});
\draw[black,dashed,domain=15:82] plot ({7*cos(\x)},{7*sin(\x)});
\draw[blue,thick,domain=15:82] plot ({4.2*cos(\x)},{4.2*sin(\x)});
\draw[blue,thick,domain=15:82] plot ({5.8*cos(\x)},{5.8*sin(\x)});
\draw (3,0.4) node {$r$};
\draw (4.2,0.8) node {$\widetilde{r}$};
\draw (5.8,1.3) node {$\widetilde{s}$};
\draw (6.8,1.46) node {$s$};
\end{tikzpicture}
\hspace{0.5cm}
\begin{tikzpicture}[scale=0.75]
\draw[blue!10!white, rounded corners, fill=blue!10!white,   opacity=0.3] (0.5,6.35)--(1.5,6.2)--(2.5,5.85)--(3.5,5.3)--(4.5,4.45)--(5.3,3.5)-- (5.75,2.7)--(6.22,1.4)--(6.38,0.6)--(3.6,0.3)--(3.4,1.2)--(2.9,2.2)--(2.3,2.825)--(1.6,3.25)--(1.0,3.45)--(0.25,3.58)--(0.5,6.35);
\draw[black,dashed,domain=15:82] plot ({3*cos(\x)},{3*sin(\x)});
\draw[black,dashed,domain=15:82] plot ({7*cos(\x)},{7*sin(\x)});
\draw[blue,thick,domain=15:82] plot ({4.2*cos(\x)},{4.2*sin(\x)});
\draw[blue,thick,domain=15:82] plot ({5.8*cos(\x)},{5.8*sin(\x)});
\draw[black,dotted,domain=5:87] plot ({3.5*cos(\x)},{3.5*sin(\x)});
\draw[black,dotted,domain=5:87] plot ({6.5*cos(\x)},{6.5*sin(\x)});
\draw (3.72,-0.1) node {$\widetilde{r}-\sigma_{1}$};
\draw (6.72,0.3) node {$\widetilde{s}+\sigma_{2}$};
\draw (3,0.4) node {$r$};
\draw (4.2,0.8) node {$\widetilde{r}$};
\draw (5.8,1.3) node {$\widetilde{s}$};
\draw (6.8,1.46) node {$s$};
\end{tikzpicture}
\caption{Radius notation and construction of competitor maps in the proof of Theorem~\ref{thm:Cacc}.}
\end{figure}
Ad $\mathrm{IV}$. We proceed similarly as in step 1 and invoke Lemma~\ref{lem:HLW} to find $\sigma_{1},\sigma_{2}>0$ with the following properties: 
\begin{align}\label{eq:secondTF}
\begin{split}
&\sigma_{1} \leq \widetilde{r}-r \leq 8\sigma_{1},\;\;\;\mathcal{M}Du(x_{0},\widetilde{r}-\sigma_{1})<\infty,\\
&\frac{\Theta(\tau)-\Theta(\widetilde{r}-\sigma_{1})}{\tau-r+\sigma_{1}}\leq 800 \frac{\Theta(\widetilde{r})-\Theta(r)}{\widetilde{r}-r}\qquad\text{for all}\;\tau\in (\widetilde{r}-\sigma_{1},\widetilde{r}),\\
&\sigma_{2} \leq (s-\widetilde{s}) \leq 8\sigma_{2},\;\;\;\mathcal{M}Du(x_{0},\widetilde{s}+\sigma_{2})<\infty,\\ 
& \frac{\Theta(\widetilde{s}+\sigma_{2})-\Theta(\tau)}{\widetilde{s}+\sigma_{2}-\tau}\leq 800\frac{\Theta(s)-\Theta(\widetilde{s})}{s-\widetilde{s}}\qquad\text{for all}\;\tau\in (\widetilde{s},\widetilde{s}+\sigma_{2}).
\end{split}
\end{align}
Before we continue, we collect some consequences of \eqref{eq:secondTF}: Since $\widetilde{r}\notin E$ with $E$ as in \eqref{eq:CaccEdef}, we have $
\frac{1}{20}(s-r)\leq \widetilde{r}-r\leq s-r$. Thus, 
by $\eqref{eq:secondTF}_{2}$ combined with $\Theta(\widetilde{r})\leq\Theta(s)$:  
\begin{align}\label{eq:secondTF1}
\frac{\Theta(\tau)-\Theta(\widetilde{r}-\sigma_{1})}{\tau-\widetilde{r}+\sigma_{1}}\leq 800 \frac{\Theta(\widetilde{r})-\Theta(r)}{\widetilde{r}-r} \leq 16000 \frac{\Theta(s)-\Theta(r)}{s-r}
\end{align}
for all $\tau\in (\widetilde{r}-\sigma_{1},\widetilde{r})$. Similarly, $\frac{1}{20}(s-r)\leq s-\widetilde{s}\leq s-r$ and so we have 
\begin{align}\label{eq:secondTF2}
\frac{\Theta(\widetilde{s}+\sigma_{2})-\Theta(\tau)}{\widetilde{s}+\sigma_{2}-\tau}\leq 16000\frac{\Theta(s)-\Theta(r)}{s-r}\qquad\text{for all}\;\tau\in (\widetilde{s},\widetilde{s}+\sigma_{2}). 
\end{align} 
By $\eqref{eq:secondTF}_{1}$ and $\eqref{eq:secondTF}_{3}$, we choose a cut-off function $\eta\in\hold_{c}^{\infty}(\ball_{2R};[0,1])$ such that
\begin{align}\label{eq:etaest}
\begin{split}
&\mathbbm{1}_{\ball_{\widetilde{s}}\setminus\overline{\ball}_{\widetilde{r}}} \leq \eta \leq \mathbbm{1}_{\ball_{\widetilde{s}+\sigma_{2}}\setminus\overline{\ball}_{\widetilde{r}-\sigma_{1}}}, \\ 
& |\nabla \eta| \leq 100\max\left\{\frac{1}{\widetilde{r}-r},\frac{1}{s-\widetilde{s}}\right\},\;\text{so that}\;|\nabla\eta|\leq \frac{2000}{s-r}.
\end{split}
\end{align}
We now define a new competitor map $v\in\bv_{c}(\ball_{2R};\R^{N})$ by 
\begin{align*}
v:=\begin{cases}
\eta\mathbb{E}_{\widetilde{r}-\sigma_{1},\widetilde{r}}\widetilde{u}&\;\text{on}\;\ball_{\widetilde{r}}\setminus\overline{\ball}_{\widetilde{r}-\sigma_{1}},\\
\widetilde{u}&\;\text{on}\;\ball_{\widetilde{s}}\setminus\overline{\ball}_{\widetilde{r}},\\
\eta\mathbb{E}_{\widetilde{s},\widetilde{s}+\sigma_{2}}\widetilde{u}&\;\text{on}\;\ball_{\widetilde{s}+\sigma_{2}}\setminus \overline{\ball}_{\widetilde{s}},\\
0&\;\text{otherwise}. 
\end{cases}
\end{align*} 
Similarly as in Step 3a, $\overline{\mathscr{F}}_{\nabla a,v}^{*}[v;\ball_{s}]\in (-\infty,\infty)$. Moreover, as a substitute of~\eqref{eq:ukraine}, we find 
\begin{align}\label{eq:koppnick}
\begin{split}
\overline{\mathscr{F}}_{\nabla a,v}^{*}[v;\ball_{s}] & \leq \overline{\mathscr{F}}_{\nabla a,v}^{*}[\eta\mathbb{E}_{\widetilde{r}-\sigma_{1},\widetilde{r}}\widetilde{u};\ball_{\widetilde{r}}\setminus\overline{\ball}_{\widetilde{r}-\sigma_{1}}] + \overline{\mathscr{F}}_{\nabla a,\widetilde{u}}^{*}[\widetilde{u};\ball_{\widetilde{s}}\setminus\overline{\ball}_{\widetilde{r}}] \\ & + \overline{\mathscr{F}}_{\nabla a,v}^{*}[\eta\mathbb{E}_{\widetilde{s},\widetilde{s}+\sigma_{2}}\widetilde{u};\ball_{\widetilde{s}+\sigma_{2}}\setminus\overline{\ball}_{\widetilde{s}}].
\end{split}
\end{align}
By Lemma~\ref{lem:additivityradii}~\ref{item:fin2}, Remark~\ref{rem:Columbo}, cf.~\eqref{eq:AdrianMonkSheronaFleming}, and $\widetilde{r},\widetilde{s}\in E^{\complement}$ we then conclude that
\begin{align}\label{eq:homerun}
\begin{split}
\overline{\mathscr{F}}_{\nabla a,\widetilde{u}}^{*}[\widetilde{u};\ball_{2R}\setminus\overline{\ball}_{\widetilde{s}}] & +\overline{\mathscr{F}}_{\nabla a,\widetilde{u}}^{*}[\widetilde{u};\ball_{\widetilde{s}}\setminus\overline{\ball}_{\widetilde{r}}] + \overline{\mathscr{F}}_{\nabla a,\widetilde{u}}^{*}[\widetilde{u};\ball_{\widetilde{r}}] \\ & = \overline{\mathscr{F}}_{\nabla a,\widetilde{u}}^{*}[\widetilde{u};\ball_{2R}] = \overline{\mathscr{F}}_{\nabla a,\widetilde{u}}^{*}[\widetilde{u};\ball_{2R}\setminus\overline{\ball}_{\widetilde{r}}] + \overline{\mathscr{F}}_{\nabla a,\widetilde{u}}^{*}[\widetilde{u};\ball_{\widetilde{r}}]
\end{split}
\end{align}
whereby 
\begin{align}\label{eq:Cacc2}
\begin{split}
\mathrm{IV}\;\; & \!\!\stackrel{\eqref{eq:homerun}}{=} -\overline{\mathscr{F}}_{\nabla a,\widetilde{u}}^{*}[\widetilde{u};\ball_{\widetilde{s}}\setminus\overline{\ball}_{\widetilde{r}}] \\ 
& \!\!\stackrel{\eqref{eq:koppnick}}{\leq} -\overline{\mathscr{F}}_{\nabla a,v}^{*}[v;\ball_{s}] \\ & + (\overline{\mathscr{F}}_{\nabla a,v}^{*}[\eta\mathbb{E}_{\widetilde{r}-\sigma_{1},\widetilde{r}}\widetilde{u};\ball_{\widetilde{r}}\setminus\overline{\ball}_{r}]+\overline{\mathscr{F}}_{\nabla a,v}^{*}[\eta\mathbb{E}_{\widetilde{s},\widetilde{s}+\sigma_{2}}\widetilde{u};\ball_{s}\setminus\overline{\ball}_{\widetilde{s}}]) \\ 
& =: -\mathrm{IV}_{a} + \mathrm{IV}_{b}.
\end{split}
\end{align}
To give an estimate on $(-\mathrm{IV}_{a})$, we note that by \ref{item:H2} applied to $\nabla a$ with $|\nabla a|\leq m$ and the compact support of $v\in\bv(\ball_{2R'};\R^{N})$, \eqref{eq:Vboundbelow} implies 
\begin{align*}
\ell^{(m)}\int_{\ball_{s}}V(Dv) \leq \overline{\mathscr{F}}_{\nabla a,v}^{*}[v;\ball_{s}]. 
\end{align*}
Therefore, 
\begin{align}\label{eq:mainsign}
-\mathrm{IV}_{a}=-\overline{\mathscr{F}}_{\nabla a,v}^{*}[v;\ball_{s}] & \leq - \ell^{(m)}\int_{\ball_{s}}V(Dv)\leq 0. 
\end{align}
We proceed by estimating the layer term $\mathrm{IV}_{b}$. As in the proofs of Proposition~\ref{prop:goodminseqs} or~\ref{prop:mazur}, using Lemma~\ref{lem:extensionoperator}~\ref{item:extension4}, we conclude that the restriction of $v$ to the annulus $(\ball_{\widetilde{r}}\setminus\overline{\ball}_{\widetilde{r}-\sigma_{1}})\cup(\ball_{\widetilde{s}+\sigma_{2}}\setminus\overline{\ball}_{\widetilde{s}})$ belongs to $\sobo^{1,q}((\ball_{\widetilde{r}}\setminus\overline{\ball}_{\widetilde{r}-\sigma_{1}})\cup(\ball_{\widetilde{s}+\sigma_{2}}\setminus\overline{\ball}_{\widetilde{s}});\R^{N})$. Lemma~\ref{lem:consistency} then gives us
\begin{align*}
\mathrm{IV}_{b} & \leq \int_{\ball_{\widetilde{r}}\setminus\overline{\ball}_{\widetilde{r}-\sigma_{1}}} F_{\nabla a}(\nabla(\eta\mathbb{E}_{\widetilde{r}-\sigma_{1},\widetilde{r}}\widetilde{u}))\dif x + \int_{\ball_{\widetilde{s}+\sigma_{2}}\setminus\overline{\ball}_{\widetilde{s}}} F_{\nabla a}(\nabla(\eta\mathbb{E}_{\widetilde{s},\widetilde{s}+\sigma_{2}}\widetilde{u}))\dif x\\ 
& \!\!\!\!\!\!\!\!\!\stackrel{\text{Lem.}~\ref{lem:shifted}\ref{item:shifted(a)}}{\leq} c\Big(\int_{\ball_{\widetilde{r}}\setminus\overline{\ball}_{\widetilde{r}-\sigma_{1}}} m_{q}(\nabla(\eta\mathbb{E}_{\widetilde{r}-\sigma_{1},\widetilde{r}}\widetilde{u}))\dif x \Big. \\ & \;\;\;\;\;\;\;\;\;\;\;\;\;\;\;\;\;\;\;\;\;\;\;\;\Big. + \int_{\ball_{\widetilde{s}+\sigma_{2}}\setminus\overline{\ball}_{\widetilde{s}}} m_{q}(\nabla(\eta\mathbb{E}_{\widetilde{s},\widetilde{s}+\sigma_{2}}\widetilde{u}))\dif x\Big) \\ 
& \!\!\!\!\stackrel{\eqref{eq:etaest}_{2}}{\leq}  c\left(\int_{\ball_{\widetilde{r}}\setminus\overline{\ball}_{\widetilde{r}-\sigma_{1}}} m_{q}(\nabla\mathbb{E}_{\widetilde{r}-\sigma_{1},\widetilde{r}}\widetilde{u})\dif x + \int_{\ball_{\widetilde{r}}\setminus\overline{\ball}_{\widetilde{r}-\sigma_{1}}} m_{q}\left(\frac{\mathbb{E}_{\widetilde{r}-\sigma_{1},\widetilde{r}}\widetilde{u}}{s-r}\right)\dif x\right) \\ 
& + c\left(\int_{\ball_{\widetilde{s}+\sigma_{2}}\setminus\overline{\ball}_{\widetilde{s}}} m_{q}(\nabla\mathbb{E}_{\widetilde{s},\widetilde{s}+\sigma_{2}}\widetilde{u})\dif x + \int_{\ball_{\widetilde{s}+\sigma_{2}}\setminus\overline{\ball}_{\widetilde{s}}} m_{q}\left(\frac{\mathbb{E}_{\widetilde{s},\widetilde{s}+\sigma_{2}}\widetilde{u}}{s-r}\right)\dif x\right)\\
& =: \mathrm{IV}_{c} + \mathrm{IV}_{d}. 
\end{align*}
To estimate $\mathrm{IV}_{c}$, we employ step 2 with $\varrho_{1}=\widetilde{r}-\sigma_{1}$ and $\varrho_{2}=\widetilde{r}$. By~\eqref{eq:uniformlycomparable} and \eqref{eq:secondTF1}, \eqref{eq:uniformchoose} is satisfied for these particular choices of radii with $\kappa=24000$. Thus \eqref{eq:mainstepCacc} is available for these choices of parameters; the term $\mathrm{IV}_{d}$ can be estimated by analogous means, where we now set $\varrho_{1}=\widetilde{s}$, $\varrho_{2}=\widetilde{s}+\sigma_{2}$ and utilise \eqref{eq:uniformlycomparable} and \eqref{eq:secondTF2} to arrive at the requisite form of \eqref{eq:uniformchoose}. This equally yields \eqref{eq:mainstepCacc} in this situation. In combination with \eqref{eq:Cacc2} and \eqref{eq:mainsign}, we then obtain 
\begin{align}\label{eq:Cacc2A}
\begin{split}
\mathrm{IV} & \leq c\Big(\int_{\ball_{s}\setminus \overline{\ball}_{r}}V\Big(\frac{\widetilde{u}}{s-r}\Big)\mathscr{L}^{n}+ V(D\widetilde{u}) \Big. \\
& \Big. \;\;\;\;\;\;\;\;\;\;\;\;\;\;\;\;+ (s-r)^{n}\Big((s-r)^{-n}\int_{\overline{\ball}_{s}\setminus\overline{\ball}_{r}} V\left(\frac{\widetilde{u}}{s-r}\right)\mathscr{L}^{n}+ V(D\widetilde{u}) ) \Big)^{q}\Big). 
\end{split}
\end{align}
Gathering the estimates from \eqref{eq:TabeaTscherpelTababy}~and~\eqref{eq:Cacc2A}, we obtain 
\begin{align*}
\int_{\overline{\ball}_{r}}V(D\widetilde{u}) & \leq  c\sum_{j\in\{0,1\}}\Big(\int_{\overline{\ball}_{s}\setminus \overline{\ball}_{r}}V((s-r)^{j-1}D^{j}\widetilde{u})\Big. \\
& \Big. \;\;\;\;\;\;\;\;\;\;\;\;\;+ (s-r)^{n}\Big((s-r)^{-n}\int_{\overline{\ball}_{s}\setminus\overline{\ball}_{r}} V((s-r)^{j-1} D^{j}\widetilde{u}) \Big)^{q}\Big), 
\end{align*}
where $c=c(n,N,L,\ell^{(m)},m,q)>0$. Fixing this constant $c$, we fill the hole on the right hand side and add $c\int_{\overline{\ball}_{r}}V(D\widetilde{u})$ to both sides. We thus obtain with $\theta=\frac{c}{c+1}$ 
\begin{align*}
\int_{\overline{\ball}_{r}}V(D\widetilde{u}) & \leq \theta \int_{\overline{\ball}_{s}}V(D\widetilde{u}) + (s-r)^{n}\Big((s-r)^{-n}\int_{\ball_{2R}} V(D\widetilde{u}) \Big)^{q} \\ 
&  + \int_{\overline{\ball}_{s}\setminus\overline{\ball}_{r}}V\left(\frac{\widetilde{u}}{s-r}\right)\dif x + (s-r)^{n}\Big((s-r)^{-n}\int_{\ball_{s}\setminus\overline{\ball}_{r}}V\left(\frac{\widetilde{u}}{s-r}\right)\dif x\Big)^{q}
\end{align*}
Now Lemma~\ref{lem:iteration} yields \eqref{eq:Caccmain}, and the proof is complete. 
\end{proof}
\subsection{The Caccioppoli inequality for $(p,q)$-growth}\label{sec:Caccpq}
We now briefly comment on the requisite modifications in the growth regime $1<p\leq q<\frac{np}{n-1}$ that lead to 
\begin{corollary}\label{cor:Caccplargerthan1}
Let $1<p\leq q<\frac{np}{n-1}$. Assume that $F\colon\R^{N\times n}\to\R$ satisfies \emph{~\ref{item:H1},~\ref{item:H2p} and~\ref{item:H3}} and suppose that $x_{0}\in\Omega$ and $R>0$ are such that $\ball_{2R}(x_{0})\Subset\Omega$. Let $u\in\sobo^{1,p}(\ball_{2R}(x_{0});\R^{N})$ be a \emph{minimizer} of $\overline{\mathscr{F}}_{u}[-;\ball_{2R}(x_{0})]$ for compactly supported variations. Moreover, given $m>0$, let $a\colon\R^{n}\to\R^{N}$ be an affine-linear map with $|\nabla a|\leq m$. 

Then there exists a constant $c=c(n,N,m,L,\ell_{m},p,q)>0$ such that 
\begin{align}\label{eq:CaccmainP}
\begin{split}
\dashint_{\ball_{R}(x_{0})}V_{p}(\nabla(u-a))\dif x & \leq c\left[ \dashint_{\ball_{2R}(x_{0})}V_{p}\left(\frac{u-a}{R}\right) \right. \\ & \;\;\;\;\;\;\;\;\;\;\; \left.+ \sum_{j\in\{0,1\}}\left(\dashint_{\ball_{2R}(x_{0})}V_{p}\left(\frac{\nabla^{j}(u-a)}{R^{1-j}}\right)\dif x\right)^{\frac{q}{p}}\right].
\end{split}
\end{align}
\end{corollary}
\begin{proof}
We only outline the main modifications, where we now systematically work with $V_{p}$ instead of $V=V_{1}$. We keep the set-up of step 1, where we now take 
\begin{align*}
\Theta_{p}(t):=\int_{\ball_{t}}V_{p}\Big(\frac{\widetilde{u}}{s-r}\Big)+V_{p}(\nabla\widetilde{u})\dif x
\end{align*}
instead of~\eqref{eq:ThetaDefCacc}; note that the non-negative measure $\lambda\in\mathrm{RM}_{\mathrm{fin}}(\Omega')$ is now taken to be a weak*-limit of a subsequence of the sequence $(|\nabla u_{j}|^{p}\mathscr{L}^{n}\mres\Omega')$ of the total variation measures for some respective generating sequence $(u_{j})$. The main distinction in the proof enters in step 2. Estimate~\eqref{eq:Caccest1} holds true, whereas~\eqref{eq:Caccest2} is modified by 
\begin{align}\label{eq:Caccest2A1}
\begin{split}
m_{q}(|\nabla\mathbb{E}_{\varrho_{1},\varrho_{2}}\widetilde{u}(x)|) & \leq cm_{q}\Big(\dashint_{\Lambda\!\ball^{i}}|\nabla\widetilde{u}|\dif x \Big) \;\;(\text{by Lemma~\ref{lem:extensionoperator}~\ref{item:extension1}})\\  & \leq c\Big(\dashint_{\Lambda\!\ball^{i}}V_{p}(\nabla\widetilde{u})\dif x\Big)^{\frac{q}{p}} (\text{by Lemma~\ref{lem:Efunction}~\ref{item:EfctA}, $i\in\mathcal{I}_{>}$ and Jensen}).
\end{split}
\end{align}
When we place this estimate into~\eqref{eq:I0bound} and hereafter~\eqref{eq:Ibest}, now using that $\ell^{1}\hookrightarrow\ell^{\frac{q}{p}}$ and $q<\frac{np}{n-1}$, we end up with 
\begin{align}\label{eq:Caccest3A1}
\begin{split}
&\sum_{j\in\{0,1\}}\int_{\ball_{\varrho_{2}}\setminus\overline{\ball}_{\varrho_{1}}} m_{q}\Big(\frac{|\nabla^{j}\mathbb{E}_{\varrho_{1},\varrho_{2}}\widetilde{u}|}{(s-r)^{1-j}} \Big)\dif x  \leq c\Big(\int_{\ball_{s}\setminus\overline{\ball}_{r}}V_{p}\Big(\frac{\widetilde{u}}{s-r}\Big)\dif x \Big.\\ &\Big. + \int_{\ball_{s}\setminus\overline{\ball}_{r}}V_{p}(\nabla\widetilde{u})\dif x+(s-r)^{n}\Big((s-r)^{-n}\int_{\ball_{s}\setminus\overline{\ball}_{r}}V_{p}\Big(\frac{\widetilde{u}}{s-r}\Big)+V_{p}(\nabla\widetilde{u})\dif x \Big)^{\frac{q}{p}}\Big)
\end{split}
\end{align}
as a substitute for~\eqref{eq:mainstepCacc}. Using~\ref{item:H2p} instead of ~\ref{item:H2} in step 3b, the proof now evolves exactly as above. 
\end{proof}
\section{Proof of Theorems~\ref{thm:main1} and \ref{thm:main2}}\label{sec:proofmain}
This final section is devoted to the proof of Theorems~\ref{thm:main1} and~\ref{thm:main2}. Based on the Euler-Lagrange system from Corollary~\ref{cor:EulerLagrange} and the Caccioppoli inequality of the second kind from Corollary~\ref{cor:Caccplargerthan1}, the superlinear growth case from Theorem~\ref{thm:main2} follows by $\mathbb{A}$-harmonic approximation along the same lines as in \cite[Sec.~7.2,~7.3]{SchmidtPR1}; also see Remark~\ref{rem:plarger1excess} below. In the linear growth regime $p=1$, the corresponding comparison systems can only be solved on specific balls, which is why we address this case explicitely and give the detailled proof. 
\subsection{Improved distance estimates to an $\A$-harmonic comparison map}\label{sec:improveddistance} By our above discussion, we now suppose that $1=p\leq q < \frac{n}{n-1}$. Towards the excess decay in Section~\ref{sec:excess}, we require an improved estimate of the local $\bv$-minimizers for compactly supported variations to $\A$-harmonic comparison maps on good balls. Our strategy is inspired by the precursors \cite{Gm20,GK1} of the present paper, now being further systematised by the maximal condition \ref{item:bdry1} below:
\begin{proposition}\label{prop:improvement}
Let $1\leq q < \frac{n}{n-1}$ and $\ball_{R}(x_{0})\Subset\Omega$. Moreover, let $F\colon\R^{N\times n}\to\R$ satisfy \emph{\ref{item:H1}--\ref{item:H3}} and suppose that $u\in\bv_{\locc}(\Omega;\R^{N})$ is a local minimizer of $\overline{\mathscr{F}}^{*}$ for compactly supported variations in $\Omega$. 
Then for any $m>0$ there exists a constant $C=C(m,n,N,q,L,\ell_{m})>0$ (with $\ell_{m}>0$ as in~\emph{\ref{item:H2}})  such that the following holds: If
\begin{enumerate}
\item \label{item:bdry1} $\mathcal{M}Du(x_{0},R)<\infty$, and so   
\item \label{item:bdry2} $\trace_{\partial\!\ball_{R}(x_{0})}(u)\in \sobo^{\frac{1}{2n+1},\frac{2n+1}{2n}}(\partial\!\ball_{R}(x_{0});\R^{N})$ by Corollary~\ref{cor:Fubini},
\end{enumerate}
and $a\colon\R^{n}\to\R^{N}$ is affine-linear with $|\nabla a|\leq m$, then we have with $\widetilde{u}:=u-a$
\begin{align}\label{eq:absolutemain}
\dashint_{\ball_{R}(x_{0})}V\Big(\frac{\widetilde{u}-\widetilde{h}}{R}\Big)\dif x \leq C \Big(\dashint_{\ball_{R}(x_{0})}V(D\widetilde{u}) \Big)^{\frac{2n}{2n-1}}, 
\end{align}
where $\widetilde{h}$ is the unique solution of the Legendre-Hadamard elliptic system
\begin{align}\label{eq:linsysmain}
\begin{cases} 
-\di( F''(\nabla a)\nabla \widetilde{h}) = 0&\;\text{in}\;\ball_{R}(x_{0}),\\
\widetilde{h} = \trace_{\partial\!\ball_{R}(x_{0})}(\widetilde{u}) &\;\text{on}\;\partial\!\ball_{R}(x_{0}). 
\end{cases}
\end{align}
\end{proposition}
\begin{proof}
Let $u\in\bv_{\locc}(\Omega;\R^{N})$, $\ball_{R}(x_{0})\Subset\Omega$ and $a\colon\R^{n}\to\R^{N}$ be as in the proposition. By~\eqref{eq:traceequal} and Corollary~\ref{cor:Fubini}, \ref{item:bdry1} implies that the right- and left-sided traces of $u$ along $\partial\!\ball_{R}(x_{0})$ coincide and that \ref{item:bdry2} holds because of $1<\vartheta:=\frac{2n+1}{2n}<\frac{n}{n-1}$. 

Since $\mathbb{A}:=F''(\nabla a)$ is Legendre-Hadamard elliptic by~\ref{item:H2} and the trace $\trace_{\partial\!\ball_{R}(x_{0})}(u)$ belongs to $\sobo^{1-1/\vartheta,\vartheta}(\partial\!\ball_{R}(x_{0});\R^{N})$ by~\ref{item:bdry2}, the system \eqref{eq:linsysmain} has a unique weak solution $\widetilde{h}\in\sobo^{1,\vartheta}(\ball_{R}(x_{0});\R^{N})$ by Lemma~\ref{lem:linearsystems}~\ref{item:linsys1}. For the following, it is convenient to rescale to the unit ball, and so we define for $x\in\ball_{1}(0)$ 
\begin{align*}
\Psi(x):=\frac{\psi(x_{0}+Rx)}{R},\;\;\;\Phi(x):=\frac{\varphi(x_{0}+Rx)}{R}\;\;\;\text{and}\;\;\;\widetilde{U}(x):=\frac{\widetilde{u}(x_{0}+Rx)}{R}, 
\end{align*}
where $\psi:=\widetilde{u}-\widetilde{h}$ and $\varphi\in(\hold_{0}\cap\hold^{1})(\ball_{R}(x_{0});\R^{N})$ is arbitrary. Now define
\begin{align*}
\mathbb{T}\colon \R^{N}\ni z \mapsto \begin{cases} z&\;\text{if}\;|z|\leq 1,\\
\frac{z}{|z|}&\;\text{if}\;|z|>1, 
\end{cases}
\end{align*}
so that $\mathbb{T}(\Psi)\in\lebe^{\infty}(\ball_{1}(0);\R^{N})$. We then consider the auxiliary Legendre-Hadamard elliptic system
\begin{align}\label{eq:linear2}
\begin{cases}
-\di(\A\nabla\Phi)= \mathbb{T}(\Psi)&\;\text{in}\;\ball_{1}(0),\\
\Phi = 0 &\;\text{on}\;\partial\!\ball_{1}(0).
\end{cases}
\end{align}
Since $\mathbb{T}(\Psi)\in\lebe^{s}(\ball_{1}(0);\R^{N})$ for any $1<s<\infty$, we invoke Lemma~\ref{lem:linearsystems}~\ref{item:linsys2} to find that there exists a weak solution $\Phi$ of~\eqref{eq:linear2} that satisfies $\Phi\in(\sobo_{0}^{1,s}\cap\sobo^{2,s})(\ball_{1}(0);\R^{N})$ for all $1<s<\infty$ together with the continuity bounds
\begin{align}\label{eq:contboundsAAA}
\|\Phi\|_{\sobo^{2,s}(\ball_{1}(0))}\leq c\|\mathbb{T}(\Psi)\|_{\lebe^{s}(\ball_{1}(0))}, 
\end{align}
where $c=c(s,m,n,N,L,\ell_{m})>0$. We apply~\eqref{eq:contboundsAAA} to the particular choice $s:=2n$ and use Morrey's embedding $\sobo^{1,2n}(\ball_{1}(0);\R^{N})\hookrightarrow\hold^{0,\frac{1}{2}}(\overline{\ball}_{1}(0);\R^{N})$ to find 
\begin{align}\label{eq:contboundsBBB}
\|\Phi\|_{\hold^{0,\frac{1}{2}}(\ball_{1}(0))}+\|\nabla\Phi\|_{\hold^{0,\frac{1}{2}}(\ball_{1}(0))} & \stackrel{\eqref{eq:contboundsAAA}}{\leq} c\|\mathbb{T}(\Psi)\|_{\lebe^{2n}(\ball_{1}(0))} \leq c\|V(\Psi)\|_{\lebe^{1}(\ball_{1}(0))}^{\frac{1}{2n}}, 
\end{align}
where now $c=c(m,n,N,L,\ell_{m})>0$; here, the second inequality is a consequence of the definition of $\mathbb{T}$ and $V$, cf.~Lemma~\ref{lem:Efunction}~\ref{item:EfctA}. Now, since $u$ is a local $\bv$-minimizer of $\overline{\mathscr{F}}^{*}$ for compactly supported variations in $\Omega$, so is $\widetilde{u}$ for $\overline{\mathscr{F}}_{\nabla a}^{*}$ by Lemma~\ref{lem:shiftconnect}. Thus, with the Lebesgue-Radon-Nikod\'{y}m decomposition $D\widetilde{U}=\nabla\widetilde{U}\mathscr{L}^{n}\mres\ball_{1}(0)+D^{s}U$, we find
\begin{align}\label{eq:oliverkahn}
\begin{split}
\left\vert \int_{\ball_{1}(0)}\langle \A D\widetilde{U},\nabla\Phi\rangle\right\vert & \leq \left\vert \int_{\ball_{1}(0)} \langle\A\nabla\widetilde{U},\nabla\Phi\rangle\dif x\right\vert + \left\vert \int_{\ball_{1}(0)} \langle\A D^{s}U,\nabla\Phi\rangle\right\vert \\
& \leq \left\vert \int_{\ball_{1}(0)}\langle(F''_{\nabla a}(0)\nabla\widetilde{U}-F'_{\nabla a}(\nabla\widetilde{U})),\nabla\Phi\rangle\dif x\right\vert \\ &  + \left\vert\int_{\ball_{1}(0)}\langle F'_{\nabla a}(\nabla\widetilde{U}),\nabla\Phi\rangle\dif x \right\vert + \left\vert \int_{\ball_{1}(0)}\langle\A D^{s}U,\nabla\Phi\rangle\right\vert \\
& =: \mathrm{I} + \mathrm{II} + \mathrm{III}, 
\end{split}
\end{align} 
and since $\nabla\Phi$ is continuous up to the boundary by~\eqref{eq:contboundsBBB}, all terms are meaningful. For term $\mathrm{I}$, we use Lemma~\ref{lem:shifted}~\ref{item:shifted(c)} and $q<\frac{n}{n-1}\leq 2$ to find 
\begin{align}\label{eq:contboundsCCC}
\mathrm{I} \leq c\int_{\ball_{1}(0)}V(\nabla\widetilde{U})\dif x \|\nabla\Phi\|_{\lebe^{\infty}(\ball_{1}(0))} \stackrel{\eqref{eq:contboundsBBB}}{\leq}  cV(D\widetilde{U})(\ball_{1}(0))\|V(\Psi)\|_{\lebe^{1}(\ball_{1}(0))}^{\frac{1}{2n}}
\end{align}
From Lemma~\ref{lem:shiftconnect}~\ref{item:shiftconnect2} and Corollary~\ref{cor:ELMAIN}\footnote{We extend $\varphi\in\sobo_{0}^{1,\infty}(\ball_{R}(x_{0});\R^{N})$ given by $\varphi(x)=R\Phi(\frac{x-x_{0}}{R})$ by zero to $\varphi\in\sobo_{c}^{1,\infty}(\Omega;\R^{N})$, apply Corollary~\ref{cor:ELMAIN} to $F_{\nabla a}$, $\widetilde{u}$ and then scale back to the unit ball.},  we infer $\mathrm{II}=0$. For term $\mathrm{III}$, we recall~\eqref{eq:functionsofmeasures} to find 
\begin{align}\label{eq:contboundsDDD}
\begin{split}
\mathrm{III} & \leq c(\A)|D^{s}U|(\ball_{1}(0))\|\nabla\Phi\|_{\lebe^{\infty}(\ball_{1}(0))} \stackrel{\eqref{eq:contboundsBBB}}{\leq} c(\A)V(D\widetilde{U})(\ball_{1}(0))\|V(\Psi)\|_{\lebe^{1}(\ball_{1}(0))}^{\frac{1}{2n}}.
\end{split}
\end{align}
In consequence,~\eqref{eq:oliverkahn}--\eqref{eq:contboundsDDD} combine to 
\begin{align}\label{eq:contboundsEEE}
\left\vert \int_{\ball_{1}(0)}\langle \A D\widetilde{U},\nabla\Phi\rangle\right\vert \leq cV(D\widetilde{U})(\ball_{1}(0))\|V(\Psi)\|_{\lebe^{1}(\ball_{1}(0))}^{\frac{1}{2n}}.
\end{align}
Since, in particular, $\Phi\in(\hold_{0}\cap\hold^{1})(\ball_{1}(0);\R^{N})$ and $\Psi\in\bv_{0}(\ball_{1}(0);\R^{N})$, \eqref{eq:linear2} gives 
\begin{align}\label{eq:test1}
\int_{\ball_{1}(0)}\langle \A\nabla\Phi, D\Psi\rangle = \int_{\ball_{1}(0)}\langle\mathbb{T}(\Psi),\Psi\rangle\dif x.
\end{align}
Using that $\langle\mathbb{T}(z), z\rangle \leq cV(z)$ (cf.~Lemma~\ref{lem:Efunction}~\ref{item:EfctA}) in the second step, we therefore obtain 
\begin{align*}
\int_{\ball_{1}(0)}V(\Psi)\dif x & \;\stackrel{\text{Lemma~\ref{lem:Efunction}\ref{item:EfctA}}}{\leq} c \int_{\ball_{1}(0)}\min\{|\Psi|,|\Psi|^{2}\}\dif x \\
& \;\;\;\;\;\;\,= c \int_{\ball_{1}(0)}\langle\mathbb{T}(\Psi),\Psi\rangle\dif x\\
& \;\;\;\;\;\stackrel{\eqref{eq:test1}}{=} c \int_{\ball_{1}(0)} \langle\A \nabla\Phi, D\Psi\rangle\\  
& \stackrel{\A\;\text{ is symmetric}}{=} c \int_{\ball_{1}(0)} \langle\A D\Psi,\nabla \Phi\rangle\\ 
& \;\;\;\;\;\,\stackrel{\eqref{eq:linsysmain}}{=} c \int_{\ball_{1}(0)} \langle\A D\widetilde{U},\nabla\Phi\rangle \stackrel{\eqref{eq:contboundsEEE}}{\leq} cV(D\widetilde{U})(\ball_{1}(0))\|V(\Psi)\|_{\lebe^{1}(\ball_{1}(0))}^{\frac{1}{2n}}, 
\end{align*}
where $c=c(m,N,n,q,L,\ell_{m})>0$ is a constant. We may assume that $\|V(\Psi)\|_{\lebe^{1}(\ball_{1}(0))}>0$. Dividing the previous overall inequality by $\|V(\Psi)\|_{\lebe^{1}(\ball_{1}(0))}^{\frac{1}{2n}}$ and raising the resulting inequality to the power $\frac{2n}{2n-1}$ yields 
\begin{align}\label{eq:unscaledComparison}
\int_{\ball_{1}(0)}V(\Psi)\dif x \leq c \Big(\int_{\ball_{1}(0)}V(D\widetilde{U})\Big)^{\frac{2n}{2n-1}}.
\end{align}
Now we scale back to $\ball_{R}(x_{0})$ to obtain \eqref{eq:absolutemain}, and the proof is complete. 
\end{proof}

\subsection{Preliminary excess decay}\label{sec:excess}
Based on Theorem~\ref{thm:Cacc} and Proposition~\ref{prop:improvement}, we are now in position to conclude a preliminary excess decay estimate. This estimate will be subsequently iterated in Corollary~\ref{cor:iterate} and thereby imply Theorem~\ref{thm:main1} in Section~\ref{sec:proof1}. For the remainder of this section, we assume $1=p \leq q < \frac{n}{n-1}$. Given $u\in\bv_{\locc}(\Omega)$ and $\ball_{r}(x_{0})\Subset\Omega$, we define the \emph{excess} via 
\begin{align}\label{eq:excessdef}
\mathbf{E}[u;x_{0},r]:=\dashint_{\ball_{r}(x_{0})}V(Du-(Du)_{\ball_{r}(x_{0})}),\;\text{where}\;\;\;(Du)_{\ball_{r}(x_{0})}:=\frac{Du(\ball_{r}(x_{0}))}{\omega_{n}r^{n}}.
\end{align} 
\begin{proposition}[Preliminary excess decay]\label{prop:dec}
Let $1\leq q <\frac{n}{n-1}$ and suppose that $F\colon\R^{N\times n}\to\R$ satisfies \emph{\ref{item:H1}--\ref{item:H3}}. Moreover, let $\omega\Subset\Omega$ be open with Lipschitz boundary and let $u$ be a $\bv$-minimizer of $\overline{\mathscr{F}}_{u}^{*}[-;\omega]$ for compactly supported variations. Then the following hold for all $x_{0}\in\omega$ and $0<R_{0}\leq 1$ such that $\ball_{R_{0}}(x_{0})\Subset\omega$: 
\begin{enumerate}
\item\label{item:dec1} For any $M_{0}>0$ there exists a constant $c=c(n,N,M_{0},\ell_{M_{0}},L,q)>0$ with the following property: If 
\begin{align}\label{eq:smallness0}
|(Du)_{\ball_{R_{0}}(x_{0})}|\leq M_{0}\;\;\;\text{and}\;\;\;\mathbf{E}[u;x_{0},R_{0}]\leq 1, 
\end{align}
then we have for all $0<\sigma<\frac{1}{10}$ with $\Psi_{q}(t):=t+t^{q}$
\begin{align}\label{eq:prelimdec}
\begin{split}
\mathbf{E}[u;x_{0},\sigma R_{0}] & \leq c\Big[\Psi_{q}\Big(\frac{1}{\sigma^{n+2}}\Big(\mathbf{E}[u;x_{0},R_{0}]\Big)^{\frac{1}{2n-1}}\Big)+\sigma^{2}+ \frac{1}{\sigma^{nq}}\mathbf{E}[u;x_{0},R_{0}]^{q-1}\Big]\times \\ 
& \times \mathbf{E}[u;x_{0},R_{0}].
\end{split}
\end{align}
\item\label{item:dec2} For any $M_{0}>0$ and any $0<\alpha<1$, there exist parameters $\sigma=\sigma(n,N,M_{0},\ell_{M_{0}},L,q)\in (0,\frac{1}{10})$ and $\overline{\varepsilon}=\overline{\varepsilon}(n,N,M_{0},\ell_{M_{0}},L,q)\in (0,1)$ such that 
\begin{align}
|(Du)_{\ball_{R_{0}}(x_{0})}|\leq M_{0}\;\;\;\text{and}\;\;\;\mathbf{E}[u;x_{0},R_{0}]\leq \overline{\varepsilon}^{2}
\end{align}
imply 
\begin{align}\label{eq:prelimexcessdecayA}
\mathbf{E}[u;x_{0},\sigma R_{0}]\leq \sigma^{1+\alpha}\mathbf{E}[u;x_{0},R_{0}].
\end{align}
\end{enumerate}
\end{proposition}
\begin{proof} 
Let $x_{0}\in\omega$. Ad~\ref{item:dec1}. Let $M_{0}>0$ and let $\ball_{R_{0}}:=\ball_{R_{0}}(x_{0})\Subset\omega$ be such that \eqref{eq:smallness0} holds. By Jensen's inequality, this implies that
\begin{align}\label{eq:smallness1}
|(Du)_{\ball_{R_{0}}}|\leq M_{0}\;\;\;\text{and}\;\;\;\dashint_{\ball_{R_{0}}}|Du-(Du)_{\ball_{R_{0}}}|\leq 2
\end{align}
are both satisfied. We then fix an affine-linear map $a\colon\R^{n}\to\R^{N}$ with $\nabla a=(Du)_{\ball_{R_{0}}}$ and then define $\widetilde{u}:=u-a$. Recalling the accordingly linearised integrand $F_{\nabla a}$ from~\eqref{eq:shifted}, Lemma~\ref{lem:shiftconnect} yields that $\widetilde{u}$ is a $\bv$-minimizer of $\overline{\mathscr{F}}_{\nabla a,\widetilde{u}}^{*}[-;\omega]$ for compactly supported variations. Put $\vartheta:=\frac{2n+1}{2n}$ so that $1<\vartheta<\frac{n}{n-1}$. Corollary~\ref{cor:Fubini} then provides us with $\frac{17}{20}R_{0}<R<\frac{19}{20}R_{0}$ such that
\begin{itemize}
\item[(i)] $\mathcal{M}Du(x_{0},R)<\infty$ and 
\item[(ii)] $\trace_{\partial\!\ball_{R}}(\widetilde{u})\in\sobo^{1-\frac{1}{\vartheta},\vartheta}(\partial\!\ball_{R};\R^{N})$ together with
\begin{align}\label{eq:MorningGlory}
\Big(\dashint_{\partial\!\ball_{R}}\int_{\partial\!\ball_{R}}\frac{|\trace_{\partial\!\ball_{R}}(\widetilde{u})(x)-\trace_{\partial\!\ball_{R}}(\widetilde{u})(y)|^{\vartheta}}{|x-y|^{n-2+\vartheta}}\dif^{n-1}x\dif^{n-1}y\Big)^{\frac{1}{\vartheta}}\leq c(n)R_{0}^{\frac{1}{\vartheta}}\dashint_{\overline{\ball}_{\frac{19}{20}R_{0}}}|D\widetilde{u}|,
\end{align} 
\end{itemize}
Define $\A:=F''(\nabla a)$, so that $\A$ induces a Legendre-Hadamard elliptic bilinear form by virtue of~\ref{item:H1} and~\ref{item:H2} (see~\eqref{eq:LHE}\emph{ff.}). By Lemma~\ref{lem:linearsystems} and (ii), the system 
\begin{align}\label{eq:linsyscompare}
\begin{cases} 
-\mathrm{div}(\A\nabla\widetilde{h})\, = 0&\;\text{in}\;\ball_{R},\\ 
\;\;\;\;\;\;\;\;\;\;\;\;\;\;\;\;\,\widetilde{h}=\trace_{\partial\!\ball_{R}}(\widetilde{u})&\;\text{on}\;\partial\!\ball_{R} 
\end{cases}
\end{align}
has a unique solution $\widetilde{h}\in(\hold^{\infty}\cap\sobo^{1,\vartheta})(\ball_{R};\R^{N})$. Note that, if $\widetilde{h}$ solves \eqref{eq:linsyscompare}, then $H:=\widetilde{h}-(\trace_{\partial\!\ball_{R}}(\widetilde{u}))_{\partial\!\ball_{R}}$ solves
\begin{align}\label{eq:linsyscompare1}
\begin{cases} 
-\mathrm{div}(\A\nabla H)\, = 0&\;\text{in}\;\ball_{R},\\ 
\;\;\;\;\;\;\;\;\;\;\;\;\;\;\;\;\,H=\trace_{\partial\!\ball_{R}}(\widetilde{u}-(\trace_{\partial\!\ball_{R}}(\widetilde{u}))_{\partial\!\ball_{R}})&\;\text{on}\;\partial\!\ball_{R}.
\end{cases}
\end{align}
The continuity bounds from Lemma~\ref{lem:linearsystems}, the fractional Poincar\'{e} inequality \eqref{eq:fracPoinc} and $\eqref{eq:MorningGlory}$ give 
\begin{align}\label{eq:PoincaPoinca}
\begin{split}
R\Big(\dashint_{\ball_{R}}|\nabla\widetilde{h}(x)&|^{\vartheta}\dif x\Big)^{\frac{1}{\vartheta}} = R\Big(\dashint_{\ball_{R}}|\nabla H(x)|^{\vartheta}\dif x\Big)^{\frac{1}{\vartheta}} \\ 
& \!\!\!\!\!\!\!\!\!\!\!\!\!\!\!\!\!\!\!\!\!\!\!\!\stackrel{\text{Lem.~\ref{lem:linearsystems}}}{\leq} c\Big(\Big(\dashint_{\partial\!\ball_{R}}|\trace_{\partial\!\ball_{R}}(\widetilde{u})(x)-(\trace_{\partial\!\ball_{R}}(\widetilde{u}))_{\partial\!\ball_{R}}|^{\vartheta}\dif^{n-1}x \Big)^{\frac{1}{\vartheta}}\Big. \\ &\Big. \!\!\!\!\!\!\!\!\!\!\!\!\!\!\!\!\!\!+R^{1-\frac{1}{\vartheta}}\Big(\dashint_{\partial\!\ball_{R}}\int_{\partial\!\ball_{R}}\frac{|\trace_{\partial\!\ball_{R}}(\widetilde{u})(x)-\trace_{\partial\!\ball_{R}}(\widetilde{u})(y)|^{\vartheta}}{|x-y|^{n-2+\vartheta}}\dif^{n-1}x\dif^{n-1}y\Big)^{\frac{1}{\vartheta}}\Big)\\ 
& \!\!\!\!\!\!\!\!\!\!\!\!\!\!\!\!\!\!\!\stackrel{\eqref{eq:fracPoinc}}{\leq} cR^{1-\frac{1}{\vartheta}}\Big(\dashint_{\partial\!\ball_{R}}\int_{\partial\!\ball_{R}}\frac{|\trace_{\partial\!\ball_{R}}(\widetilde{u})(x)-\trace_{\partial\!\ball_{R}}(\widetilde{u})(y)|^{\vartheta}}{|x-y|^{n-2+\vartheta}}\dif^{n-1}x\dif^{n-1}y\Big)^{\frac{1}{\vartheta}}\\
& \!\!\!\!\!\!\!\!\!\!\!\!\!\!\!\!\!\!\!\!\!\!\!\!\!\!\!\!\!\!\!\!\!\!\!\!\stackrel{\eqref{eq:MorningGlory},\,\frac{17}{20}R_{0}<R<R_{0}}{\leq} cR_{0}\Big(\dashint_{\ball_{R_{0}}}|D\widetilde{u}|\Big)
\end{split}
\end{align}
where $c=c(n,N,L,\ell_{M_{0}},M_{0},q)>0$. We define $H_{0}(x):=\widetilde{h}(x_{0})+ \nabla\widetilde{h}(x_{0})\cdot(x-x_{0})$ and $H_{1}(x):=a(x)+H_{0}(x)$. Then we find with a constant $c=c(n,N,L,\ell_{M_{0}},M_{0},q)>0$
\begin{align}\label{eq:PoincaPoinca1}
\begin{split}
|\nabla H_{1}(x)| & \stackrel{\eqref{eq:smallness1}}{\leq} M_{0} + |\nabla \widetilde{h}(x_{0})| \leq M_{0}+\sup_{x\in\ball_{R/4}}|\nabla\widetilde{h}(x)| \\ 
& \!\!\!\stackrel{\text{Lem.~\ref{lem:linearsystems}}}{\leq} M_{0} + c\Big(\dashint_{\ball_{R/2}}|\nabla\widetilde{h}(x)|^{\vartheta}\dif x\Big)^{\frac{1}{\vartheta}}\\ 
& \stackrel{\eqref{eq:PoincaPoinca}}{\leq} M_{0} + c\Big(\dashint_{\ball_{R_{0}}}|D\widetilde{u}|\Big)\\ 
& \stackrel{\eqref{eq:smallness1}}{\leq} M_{0} + 2c =: m
\end{split}
\end{align}
for all $x\in\ball_{R}(x_{0})$. Let $0<\sigma<\frac{1}{10}$. At this stage, we use the Caccioppoli inequality from Theorem~\ref{thm:Cacc} with $m>0$ as in \eqref{eq:PoincaPoinca1} and the affine-linear map $H_{1}$; this is admissible by~ \eqref{eq:PoincaPoinca1}. We choose a radius $\sigma R_{0}<\widetilde{R}<\frac{3}{2}\sigma R_{0}$ such that $u$ is a $\bv$-minimizer of $\overline{\mathscr{F}}_{u}^{*}[-;\ball_{2\widetilde{R}}(x_{0})]$ for compactly supported variations in $\ball_{2\widetilde{R}}(x_{0})$; $\mathscr{L}^{1}$-a.e. $\widetilde{R}\in(\sigma R_{0},\frac{3}{2}\sigma R_{0})$ will do by Lemma~\ref{lem:additivityradii}~\ref{item:fin3}. Therefore, Jensen's inequality and Lemma~\ref{lem:Efunction} then give us with the convex function $\Psi_{q}(t):=t+t^{q}$
\begin{align}\label{eq:blurparklife}
\begin{split}
\mathbf{E}[u;x_{0},\sigma R_{0}] & = \dashint_{\ball_{\sigma R_{0}}}V(Du-(Du)_{\ball_{\sigma R_{0}}}) \\ & \!\!\!\stackrel{\text{Jensen}}{\leq} c\dashint_{\ball_{\sigma R_{0}}}V(D(u-H_{1})) \leq c\dashint_{\ball_{\widetilde{R}}}V(D(u-H_{1})) \\ 
& \!\!\!\!\!\stackrel{\text{Thm.~\ref{thm:Cacc}}}{\leq}  c\Psi_{q}\Big( \dashint_{\ball_{2\widetilde{R}}}V\Big(\frac{u-H_{1}}{\widetilde{R}}\Big)\dif x \Big)+c\Big(\dashint_{\ball_{2\widetilde{R}}}V(D(u-H_{1})) \Big)^{q} \\ 
& \!\!\!\!\!\!\!\!\!\!\!\!\stackrel{\frac{17}{20}R_{0}<R<R_{0}}{\leq}  c\Psi_{q}\Big( \dashint_{\ball_{2\widetilde{R}}}V\Big(\frac{u-H_{1}}{\sigma R_{0}}\Big)\dif x \Big)+c\Big(\dashint_{\ball_{2\widetilde{R}}}V(D(u-H_{1})) \Big)^{q} \\ 
& =: \mathrm{I}+\mathrm{II},
\end{split}
\end{align}
where $c=c(n,N,m,\ell_{m},L,q)>0$ is a constant. To estimate $\mathrm{I}$, recall that $u$ and $\ball_{R}=\ball_{R}(x_{0})$ satisfy the requirements of Proposition~\ref{prop:improvement}. Therefore, applying Proposition~\ref{prop:improvement} and noting that $\sigma R_{0}<\widetilde{R}<\frac{3}{2}\sigma R_{0}$ together with $\frac{17}{20}R_{0}<R<\frac{19}{20}R_{0}$, we first obtain 
\begin{align}\label{eq:SM1}
\begin{split}
\dashint_{\ball_{2\widetilde{R}}}V\Big(\frac{\widetilde{u}-\widetilde{h}}{\sigma R_{0}}\Big)\dif x & \stackrel{\text{Lem.~\ref{lem:Efunction}~\ref{item:EfctB}
},\,\ball_{2\widetilde{R}}\subset\ball_{3\sigma R_{0}}}{\leq} \frac{c}{\sigma^{2}}\dashint_{\ball_{3\sigma R_{0}}} V\Big(\frac{\widetilde{u}-\widetilde{h}}{R}\Big)\dif x \\ 
& \;\;\;\stackrel{\ball_{3\sigma R_{0}}\subset \ball_{R},\;R\sim R_{0}}{\leq} \frac{c}{\sigma^{n+2}}\dashint_{\ball_{R}}V\Big(\frac{\widetilde{u}-\widetilde{h}}{R}\Big)\dif x \\ 
& \;\;\;\;\;\;\;\;\;\;\,\stackrel{\text{Prop.~\ref{prop:improvement}}}{\leq} \frac{c}{\sigma^{n+2}}\Big(\dashint_{\ball_{R}}V(D\widetilde{u})\Big)^{\frac{2n}{2n-1}} \\ & \;\;\;\;\;\;\;\;\;\;\;\;\;\; \leq  \frac{c}{\sigma^{n+2}}\Big(\dashint_{\ball_{R_{0}}}V(D\widetilde{u})\Big)^{\frac{2n}{2n-1}}.
\end{split}
\end{align}
Moreover, if $x\in\ball_{2\widetilde{R}}\Subset\ball_{R/2}$, a second order Taylor expansion and inverse estimates for $\A$-harmonic maps (cf.~Lemma~\ref{lem:linearsystems}) yield
\begin{align}\label{eq:harmonicFinal1}
\begin{split}
\left\vert\frac{\widetilde{h}(x)-H_{0}(x)}{\sigma R_{0}} \right\vert & \leq c\sigma R\max_{\overline{\ball}_{R/2}}|\nabla^{2}\widetilde{h}| \\ & \!\!\!\!\!\stackrel{\text{Lem.~\ref{lem:linearsystems}}}{\leq} c\sigma\Big(\dashint_{\ball_{R}}|\nabla\widetilde{h}|^{\vartheta}\Big)^{\frac{1}{\vartheta}} \stackrel{\eqref{eq:PoincaPoinca}}{\leq} c\sigma \dashint_{\ball_{R_{0}}}|D\widetilde{u}|, 
\end{split}
\end{align}
again by use of $\frac{17}{20}R_{0}<R<\frac{19}{20}R_{0}$, where $c=c(n,N,L,\ell_{M_{0}},M_{0},q)>0$. Now, by~\eqref{eq:smallness1}, we may apply Lemma~\ref{lem:Efunction}~\ref{item:VpCompa2} with $\mathbf{m}:=2c$. Thus, enlarging the previous constant, there exists $c=c(n,N,L,\ell_{M_{0}},M_{0},q)>0$ such that 
\begin{align}\label{eq:termtwodeal}
\dashint_{\ball_{2\widetilde{R}}}V\Big(\frac{\widetilde{h}-H_{0}}{\sigma R_{0}}\Big)\dif x \leq  c\sigma^{2}\dashint_{\ball_{R_{0}}}V(D\widetilde{u}). 
\end{align}
Therefore, Lemma~\ref{lem:Efunction}~\ref{item:EfctC} and estimates \eqref{eq:SM1} and \eqref{eq:termtwodeal} give
\begin{align}\label{eq:SM3}
\begin{split}
\mathrm{I} & \stackrel{\text{Lem.~\ref{lem:Efunction}~\ref{item:EfctC}}}{\leq} c\Psi_{q}\Big(\dashint_{\ball_{2\widetilde{R}}} V\Big(\frac{\widetilde{u}-\widetilde{h}}{\sigma R_{0}}\Big)\dif x + \dashint_{\ball_{2\widetilde{R}}}V\Big(\frac{\widetilde{h}-H_{0}}{\sigma R_{0}}\Big)\dif x \Big) \\ 
& \!\stackrel{\eqref{eq:SM1},~\eqref{eq:termtwodeal}}{\leq} c\Psi_{q}\Big(\frac{1}{\sigma^{n+2}}\Big(\dashint_{\ball_{R_{0}}}V(D\widetilde{u}) \Big)^{\frac{2n}{2n-1}}\Big)+ c\Psi_{q}\Big(\sigma^{2}\dashint_{\ball_{R_{0}}}V(D\widetilde{u})\Big).
\end{split}
\end{align}
In view of term $\mathrm{II}$, we first note that similarly to~\eqref{eq:harmonicFinal1}, 
\begin{align}\label{eq:bittersweetsymphony}
V(\nabla\widetilde{h}(x_{0})) \leq cV\Big(\Big(\dashint_{\ball_{R/2}}|\nabla\widetilde{h}|^{\vartheta}\dif x\Big)^{\frac{1}{\vartheta}}\Big) \stackrel{\eqref{eq:PoincaPoinca}}{\leq} cV\Big(\dashint_{\ball_{R_{0}}}|D\widetilde{u}|\Big)\leq c\dashint_{\ball_{R_{0}}}V(D\widetilde{u}).
\end{align}
In consequence, by Lemma~\ref{lem:Efunction}~\ref{item:EfctC}, 
\begin{align}\label{eq:SM4}
\begin{split}
\mathrm{II}=c\Big(\dashint_{\ball_{2\widetilde{R}}}V(D(u-H_{1})) \Big)^{q} & \leq c\Big(\dashint_{\ball_{2\widetilde{R}}}V(D\widetilde{u}) \Big)^{q} + c\Big(\dashint_{\ball_{2\widetilde{R}}}V(\nabla \widetilde{h}(x_{0}))\Big)^{q} \\ 
& \leq c\Big(\frac{R_{0}}{\widetilde{R}}\Big)^{nq}\Big(\dashint_{\ball_{R_{0}}}V(D\widetilde{u}) \Big)^{q} + cV(\nabla\widetilde{h}(x_{0}))^{q}\\
& \!\!\!\!\!\!\!\!\!\!\!\!\stackrel{\eqref{eq:bittersweetsymphony},\,0<\sigma<1}{\leq} \frac{c}{\sigma^{nq}}\Big(\dashint_{\ball_{R_{0}}}V(D\widetilde{u})) \Big)^{q}.
\end{split}
\end{align}
Using that 
\begin{align}\label{eq:flashFM}
\Psi_{q}(ab)\leq \Psi_{q}(a)b\qquad\text{for all}\;a\geq 0,\;0\leq b\leq 1, 
\end{align}
we work from \eqref{eq:blurparklife} to successively obtain
\begin{align*}
\mathbf{E}[u;x_{0},\sigma R_{0}] & \stackrel{\eqref{eq:blurparklife},~\eqref{eq:SM3},~\eqref{eq:SM4}}{\leq} c\Psi_{q}\Big(\frac{1}{\sigma^{n+2}}\Big(\dashint_{\ball_{R_{0}}}V(D\widetilde{u}) \Big)^{\frac{2n}{2n-1}}\Big)+ c\Psi_{q}\Big(\sigma^{2}\dashint_{\ball_{R_{0}}}V(D\widetilde{u})\Big) \\ 
& \;\;\;\;\;\;\;\;\;\;+ \frac{c}{\sigma^{nq}}\Big(\dashint_{\ball_{R_{0}}}V(D\widetilde{u})) \Big)^{q}\\ 
& \;\;\;\;\stackrel{\eqref{eq:smallness0},~\eqref{eq:flashFM}}{\leq} c\Big[\Psi_{q}\Big(\frac{1}{\sigma^{n+2}}\Big(\mathbf{E}[u;x_{0},R_{0}]\Big)^{\frac{1}{2n-1}}\Big)+\sigma^{2}+ \frac{1}{\sigma^{nq}}\mathbf{E}[u;x_{0},R_{0}]^{q-1}\Big]\times \\ 
& \;\;\;\;\;\;\;\;\;\; \times \mathbf{E}[u;x_{0},R_{0}]. 
\end{align*}
This is~\eqref{eq:prelimdec}, and the proof of~\ref{item:dec1} is complete.

Ad~\ref{item:dec2}. Let $M_{0}>0$ be given and let $c=c(n,N,M_{0},\ell_{M_{0}},L,q)>0$ be  as in \ref{item:dec1}. We then choose $\sigma\in (0,\frac{1}{10})$ such that 
\begin{align}\label{eq:sigmachoose1}
0<\sigma<\frac{1}{(5c)^{\frac{1}{1-\alpha}}}\;\;\;\text{so that}\;\;\;\sigma^{2}<\frac{\sigma^{1+\alpha}}{5c}.
\end{align}
The function $\Psi_{q}\colon (0,\infty)\to(0,\infty)$ is invertible. With $\sigma$ as in~\eqref{eq:sigmachoose1}, choose $\overline{\varepsilon}>0$ with 
\begin{align*}
0<\overline{\varepsilon}<\min\left\{1,\Big(\sigma^{n+2}\Psi_{q}^{-1}\Big(\frac{\sigma^{1+\alpha}}{5c} \Big)\Big)^{\frac{2n-1}{2}},\Big(\frac{\sigma^{1+\alpha+nq}}{5c}\Big)^{\frac{1}{2(q-1)}} \right\}, 
\end{align*}
where we use the convention $\frac{1}{\infty}=0$ that corresponds to $q=1$. Then \eqref{eq:smallness0} is satisfied, and inserting these choices into \eqref{eq:prelimdec} yields~\eqref{eq:prelimexcessdecayA}. This is~\ref{item:dec2}, and the proof is complete. 
\end{proof}
The following corollary is a straightforward consequence of the previous Proposition~\ref{prop:dec}:
\begin{corollary}[Excess decay]\label{cor:iterate}
Let $1\leq q <\frac{n}{n-1}$ and suppose that $F\colon\R^{N\times n}\to\R$ satisfies \emph{\ref{item:H1}--\ref{item:H3}}. Moreover, let $\omega\Subset\Omega$ be open with Lipschitz boundary and let $u$ be a $\bv$-minimizer of $\overline{\mathscr{F}}_{u}^{*}[-;\omega]$ for compactly supported variations. For any $0<\alpha<1$ and $M_{0}>0$, there exist $\varepsilon=\varepsilon(n,N,M_{0},\ell_{M_{0}},L,q,\alpha)\in (0,1)$ and $R_{0}=R_{0}(n,N,M_{0},\ell_{M_{0}},L,q,\alpha)\in(0,1)$ with the following property: If $x_{0}\in\omega$ and $0<R<R_{0}$ are such that $\ball_{R_{0}}(x_{0})\Subset\omega$, then 
\begin{align}\label{eq:smallinsky1}
\mathbf{E}[u;x_{0},R]<\varepsilon^{2}\;\;\;\text{and}\;\;\;|(Du)_{\ball_{R}(x_{0})}|<\frac{M_{0}}{2}
\end{align}
imply with a constant $c=c(n,N,M_{0},\ell_{M_{0}},L,q,\alpha)>0$ 
\begin{align}\label{eq:smallinsky2}
\mathbf{E}[u;x_{0},r]\leq c\Big(\frac{r}{R}\Big)^{2\alpha}\mathbf{E}[u;x_{0},R]\qquad\text{for all}\; 0<r<R. 
\end{align}
\end{corollary}
\begin{remark}\label{rem:plarger1excess}
In the case $1<p\leq q <\min\{\frac{np}{n-1},p+1\}$, the previous corollary follows by standard $\mathbb{A}$-harmonic approximation as in \cite[Lem.~7.10]{SchmidtPR1} where one then works with $\mathbf{E}_{p}[u;x_{0},R]$ which is defined as in~\eqref{eq:excessdef}, now replacing $V=V_{1}$ by $V_{p}$. In fact, the only two ingredients to make the approach of \cite{SchmidtPR1} work out even in the signed case and for the amplified exponent range $1<p<q<\min\{\frac{np}{n-1},p+1\}$ are the validity of the Euler-Lagrange system from Corollary~\ref{cor:EulerLagrange} and the Caccioppoli inequality of the second kind from Corollary~\ref{cor:Caccplargerthan1}; signed integrands do not harm this proof strategy, as already visible in the standard $p$-growth regime considered by \textsc{Duzaar} et al. \cite{Duzaar1A,Duzaar2} within the context of $\omega$-minimizers. Even though \textsc{Schmidt} works with the slightly different auxiliary functions $z\mapsto (1+|z|^{2})^{\frac{p-2}{2}}|z|^{2}$, these are uniformly comparable to our choices $V_{p}$. Still, in the growth regime $1=p\leq q<\frac{n}{n-1}$, the $\mathbb{A}$-harmonic approximation lemma would still require the construction of suitable $\mathbb{A}$-harmonic comparison maps; this is not possible for arbitrary balls as the trace space of $\bv$ along spheres is $\lebe^{1}$. 
\end{remark}
\subsection{Proof of Theorem~\ref{thm:main1}}\label{sec:proof1}
Working from Corollary~\ref{cor:iterate} and Remark~\ref{rem:plarger1excess}, the proof of Theorems~\ref{thm:main1} and~\ref{thm:main2} is routine and we give it here for completeness only. Let $\omega\Subset\Omega$ be open and bounded with Lipschitz boundary such that $u$ is a $\bv$-minimizer of $\overline{\mathscr{F}}_{u}^{*}[-;\omega]$ for compactly supported variations. Let $0<\alpha<1$ and, with $M>0$ as in Theorem~\ref{thm:main1}, put $M_{0}:=10\max\{M,1\}$.  Corollary~\ref{cor:iterate} then provides us with $\varepsilon=\varepsilon(n,N,M,\ell_{M},L,q,\alpha)>0$, $R_{0}=R_{0}(n,N,M,\ell_{M},L,q,\alpha)>0$ and $C=C(n,N,M,\ell_{M},L,q,\alpha)>0$ such that for any $0<R<R_{0}$,~\eqref{eq:smallinsky1} implies~\eqref{eq:smallinsky2}. We define $\varepsilon_{M}:=\varepsilon^{2}/2^{n+6}$ and suppose that~\eqref{eq:MainSmallnessA} and~\eqref{eq:MainSmallnessB} are in action. Since $V(\cdot)\leq|\cdot|$, this implies validity of~$\eqref{eq:smallinsky1}_{1}$ and, since $M<\frac{M_{0}}{2}$,~ $\eqref{eq:smallinsky1}_{2}$ holds as well. We now claim that~\eqref{eq:smallinsky1} is satisfied uniformly on $\ball_{R/2}(x_{0})$. To this end, note that with $R'=\frac{1}{2}R$, we have for all $x\in\ball_{R'}(x_{0})$ that $\ball_{R'}(x)\subset\ball_{R}(x_{0})$, and so
\begin{align*}
|(Du)_{\ball_{R'}(x)}| & \leq \left\vert\dashint_{\ball_{R'}(x)}Du-(Du)_{\ball_{R}(x_{0})}\right\vert + |(Du)_{\ball_{R}(x_{0})}| \\ 
& \leq \Big(\frac{R}{R'} \Big)^{n}\dashint_{\ball_{R}(x_{0})}|Du-(Du)_{\ball_{R}(x_{0})}| + |(Du)_{\ball_{R}(x_{0})}| \\ 
& \!\!\!\!\!\!\!\!\!\stackrel{\eqref{eq:MainSmallnessA},~\eqref{eq:MainSmallnessB}}{\leq} 2^{n} \varepsilon_{M}  + M \stackrel{0<\varepsilon<1}{\leq} M+1 < \frac{M_{0}}{2}. 
\end{align*}
On the other hand, successively applying Lemma~\ref{lem:Efunction}~\ref{item:EfctC}, we obtain by $\ball_{R'}(x)\subset\ball_{R}(x_{0})$ 
\begin{align*}
\mathbf{E}[u;x,R'] & = \dashint_{\ball_{R'}(x)}V(\nabla u(y)-(Du)_{\ball_{R'}(x)})\dif y + \frac{|D^{s}u|(\ball_{R'}(x))}{\omega_{n}(R')^{n}} \\ 
& \!\!\!\!\!\!\!\!\!\stackrel{\text{Lem.~\ref{lem:Efunction}~\ref{item:EfctC}}}{\leq} 2\dashint_{\ball_{R'}(x)}V(\nabla u(y)-(Du)_{\ball_{R}(x_{0})})\dif y +2V((Du)_{\ball_{R}(x_{0})}-(Du)_{\ball_{R'}(x)}) \\ & + \Big(\frac{R}{R'}\Big)^{n}\frac{|D^{s}u|(\ball_{R}(x_{0}))}{\omega_{n}R^{n}} \\ 
& \leq 2\Big(\frac{R}{R'}\Big)^{n}\dashint_{\ball_{R}(x_{0})}V(\nabla u(y)-(Du)_{\ball_{R}(x_{0})})\dif y \\ 
& + 2V\Big(\dashint_{\ball_{R'}(x)}|\nabla u(y) - (Du)_{\ball_{R}(x_{0})}|\dif y + \frac{|D^{s}u|(\ball_{R'}(x))}{\omega_{n}(R')^{n}}\Big) \\ & + \Big(\frac{R}{R'}\Big)^{n}\frac{|D^{s}u|(\ball_{R}(x_{0}))}{\omega_{n}R^{n}} \\ 
& \!\!\!\!\!\!\!\!\!\!\!\!\!\!\!\!\!\!\!\!\!\!\!\!\!\!\!\!\!\!\!\!\stackrel{\text{Lem.~\ref{lem:Efunction}~\ref{item:EfctC},~$V(\cdot)\leq|\cdot|$, $R'=R/2$}}{\leq} 2^{n+5} \mathbf{E}[u;x_{0},R] < \varepsilon^{2}.
\end{align*}
We then conclude by Corollary~\ref{cor:iterate} that we have $\mathbf{E}[u;x,r]\leq c(\frac{r}{R'})^{2\alpha}\mathbf{E}[u;x,R']$ for all $x\in\ball_{R'}(x_{0})$ and all $0<r<R'$, with $c>0$ being independent of $u,x,r$, and $R'$. By the definition of $\mathbf{E}[u;x_{0},r]$, this particularly implies that $|D^{s}u|(\ball_{r}(x))$ is bounded independently of $x\in\ball_{R'}(x_{0})$ and $0<r<R'$. A covering argument then entails that $D^{s}u\equiv 0$ in $\ball_{R'}(x_{0})$. By Lemma~\ref{lem:Efunction}~\ref{item:EfctA}, we have with $\Psi(t):=(\sqrt{2}-1)\min\{t,t^{2}\}$
\begin{align*}
\Psi\Big(\dashint_{\ball_{r}(x)}|\nabla u(y)-(Du)_{\ball_{r}(x)}|\dif y \Big) \leq \mathbf{E}[u;x,r] \leq c\Big(\frac{r}{R'}\Big)^{2\alpha}\mathbf{E}[u;x,R'], 
\end{align*}
the constant $c>0$ still being independent of $u,x,r$ and $R'$; as a consequence, 
\begin{align*}
\frac{1}{r^{\alpha}}\dashint_{\ball_{r}(x)}|\nabla u(y)-(Du)_{\ball_{r}(x)}|\dif y \leq  \frac{c}{(R')^{\alpha}}
\end{align*}
with a constant $c=c(n,N,M,\ell_{M},L,q,\alpha)>0$. Now the classical \textsc{Campanato-Meyers} embedding $\mathscr{L}^{p,\lambda}\simeq \hold^{0,(\lambda-n)/p}$ for $n<\lambda \leq n+p$ then implies that $\nabla u\in\hold^{0,\alpha}(\ball_{R'}(x);\R^{N\times n})$. Once this is established, we may pass to the corresponding Euler-Lagrange system to conclude that $u$ is of class $\hold^{2,\alpha}$ by the usual \textsc{Schauder} theory; cf.~\cite{Gm20,GK1}. This is amenable to bootstrapping, and implies that $u\in\hold^{\infty}(\ball_{R/2}(x_{0});\R^{N})$. This yields the $\varepsilon$-regularity result from Theorem~\ref{thm:main1}. For the partial regularity result it then suffices to realise that, by the Lebesgue differentiation theorem for Radon measures, for $\mathscr{L}^{n}$-a.e. $x_{0}\in\omega$ there exists $a\in\R^{N\times n}$ with
\begin{align*}
\limsup_{r\searrow 0} \Big(\dashint_{\ball_{r}(x_{0})}|\nabla u - a|\dif y + \frac{|D^{s}u|(\ball_{r}(x_{0}))}{\omega_{n}r^{n}}\Big)=0, 
\end{align*}
so that~\eqref{eq:MainSmallnessA} and~\eqref{eq:MainSmallnessB} are satisfied for any such $x_{0}$ with a suitable $M_{0}$ depending on $u$ and $x_{0}$. By the definition of local $\bv$-minimality for compactly supported variations, this implies that any such local minimizer is $\hold^{\infty}$-partially regular in $\Omega$. The proof is complete. 
\begin{remark}[Uniformity of threshold radii]
Usually, $\varepsilon$-regularity results underlying the $\hold^{\infty}$-partial regularity of local ($\bv$-)minimizers are stated in analogy with the $\varepsilon$-regularity result for \emph{minimizers} as in Theorem~\ref{thm:main1}. This particularly concerns the existence of some $R_{0}>$ such that if the smallness assumptions~\eqref{eq:MainSmallnessA} and~\eqref{eq:MainSmallnessB} are fulfilled at some $x_{0}$ for some $0<R<R_{0}$, then $u$ is of class $\hold^{\infty}(\ball_{R/2}(x_{0});\R^{N})$. In the situation of local ($\bv$-)minimizers for compactly supported variations, it is not clear to us whether such a uniform $R_{0}$ exists. In fact, the above argument reduces the partial regularity of local $\bv$-minimizers for compactly supported variations to the minimality of $\overline{\mathscr{F}}_{u}^{*}[-;\omega]$ for compactly supported variations on \emph{some} $\omega\Subset\Omega$ with $x_{0}\in\omega$. However, the $\varepsilon$-regularity result then only applies to balls $\ball_{R}(x_{0})\Subset\omega$ but not all balls $\ball_{R}(x_{0})\Subset\Omega$ for $0<R<R_{0}$ with some fixed $R_{0}>0$. This, in turn, is so because we do not know whether $u$ is a minimizer for compactly supported variations on all balls $\ball_{R}(x_{0})\Subset\Omega$ as is the case in non-relaxed problems; still, this does not affect the $\hold^{\infty}$-partial regularity for local minimizers for compactly supported variations.  
\end{remark}
\subsection{Consequences and extensions of Theorems~\ref{thm:main1} and~\ref{thm:main2}}
\label{sec:extensions} 
We conclude the paper by discussing several extensions that were already alluded to in Section~\ref{sec:keynov}. On the one hand, this comprises integrands of critical Orlicz growth and, on the other hand, variational integrals depending on differential expressions. 
\subsubsection{Quasiconvex functionals with Orlicz growth}\label{sec:Orlicz}
We start by addressing the case of integrands of a certain Orlicz growth that falls outside the realm commonly studied in the literature but appear as a special case of Theorem~\ref{thm:main1}. To this end, we briefly pause and introduce the requisite terminology first; see e.g.~\cite[Chpt.~8]{AdamsFournier} for more background information. 

We say that $\Phi\colon\R_{\geq 0}\to\R_{\geq 0}$ is an \emph{N-function} provided it can be written as 
\begin{align*}
\Phi(t)=\int_{0}^{t}a(s)\dif s,\qquad t\geq 0, 
\end{align*}
where $a\colon [0,\infty)\to[0,\infty)$ is a non-decreasing, right-continuous function such that $a(0)=0$, $a(t)>0$ for $t>0$ and $\lim_{t\to\infty}a(t)=\infty$. An $N$-function $\Phi\colon[0,\infty)\to [0,\infty)$ then is said to be of class $\Delta_{2}$ provided there exists $c>0$ such that $\Phi(2t)\leq c\Phi(t)$ for all $t\geq 0$, and we denote $\Delta_{2}(\Phi)$ the smallest such constant $c$. Moreover, we say that $\Phi$ is of class $\nabla_{2}$ provided the Fenchel conjugate $\Phi^{*}(t):=\sup_{s>0}st-\Phi(s)$ is of class $\Delta_{2}$. If $\Phi$ is both of class $\Delta_{2}$ and $\nabla_{2}$, we write $\Phi\in\Delta_{2}\cap\nabla_{2}$. For future reference, we record the fundamental inequality (cf.~\cite[\S 8.3~(3)]{AdamsFournier})
\begin{align}\label{eq:FenchelInequality}
\Phi^{*}\Big(\frac{\Phi(t)}{t}\Big) \leq \Phi(t)\qquad\text{for all}\;t>0. 
\end{align}
For an $N$-function  function $\Phi$, a measurable set $\Omega\subset\R^{n}$ and a measurable map $v\colon\Omega\to\R^{m}$, we define the associated \emph{Luxemburg norm} by
\begin{align*}
\|v\|_{\lebe^{\Phi}(\Omega)}:=\inf\left\{\lambda>0\colon\;\int_{\Omega}\Phi\left(\frac{|v|}{\lambda}\right)\dif x \leq 1\right\}
\end{align*}
and introduce Lebesgue- or Sobolev-type spaces by $\lebe^{\Phi}(\Omega;\R^{m}):=\{v\colon\Omega\to\R^{m}\colon\;\|v\|_{\lebe^{\Phi}(\Omega)}<\infty\}$ and $\sobo^{1,\Phi}(\Omega;\R^{m}):=\{v\colon\Omega\to\R^{m}\colon\;\|v\|_{\sobo^{1,\Phi}(\Omega)}:=\|v\|_{\lebe^{\Phi}(\Omega)}+\|\nabla v\|_{\lebe^{\Phi}(\Omega)}<\infty\}$. As usual, $\sobo_{0}^{1,\Phi}(\Omega;\R^{m})$ then is defined as the $\|\cdot\|_{\sobo^{1,\Phi}(\Omega)}$-closure of $\hold_{c}^{\infty}(\Omega;\R^{m})$. Finally, we record the following variant of H\"{o}lder's inequality in Orlicz spaces (cf.~\cite[\S 8.11]{AdamsFournier}): 
\begin{align}\label{eq:HoelderOrlicz}
\int_{\Omega}uv\dif x \leq 2 \|u\|_{\lebe^{\Phi^{*}}(\Omega)}\|v\|_{\lebe^{\Phi}(\Omega)}\qquad\text{for all}\;u\in\lebe^{\Phi^{*}}(\Omega;\R^{m}),\;v\in\lebe^{\Phi}(\Omega;\R^{m}). 
\end{align}

The partial regularity results available so far for quasiconvex multiple integrals of Orlicz growth (cf. e.g.~\cite{DLSV,ILV,Irving}) are confined to $N$-functions which are of class $\Delta_{2}\cap\nabla_{2}$. This particularly rules out integrands which have almost linear growth, yet fail to be of class $\nabla_{2}$. The typical example of such integrands that we then have in mind is as follows: 
\begin{example}[Generalised $\mathrm{L}\log\mathrm{L}$-growth]\label{ex:Orlicz}
Let $s_{1},...,s_{m}>0$ and let $\Phi\colon\R_{\geq 0}\to\R_{\geq 0}$ be a twice differentiable convex function that equals 
\begin{align*}
\Phi(t)=t\log(t)^{s_{1}}(\log(\log(t)))^{s_{2}}...(\log(...(\log(t))))^{s_{m}}\qquad\text{for all sufficiently large $t>0$}.
\end{align*}
Then $\Phi$ is a Young function which is of class $\Delta_{2}$ but not of class $\nabla_{2}$. Specifically, $\Phi^{*}$ has exponential growth and thus $\Phi$ fails to qualify as a $\nabla_{2}$-function. 
\end{example}
To state the corresponding partial regularity result in analogy with Theorems~\ref{thm:main1} and~\ref{thm:main2}, we suppose similarly to ~\ref{item:H2} or~\ref{item:H2p} that 
\begin{enumerate}[label={(B\arabic{*})},start=1]
\item\label{item:B1} $\Phi\colon\R_{\geq 0}\to\R_{\geq 0}$ is an $N$-function of class $\Delta_{2}$.   
\item\label{item:B2} $\Phi(t)\leq c(1+|t|^{q})$ holds for all $t\geq 0$ for some fixed $c>0$ and $1<q<\frac{n}{n-1}$. 
\end{enumerate}
Properties~\ref{item:B1} and~\ref{item:B2} are fulfilled for the integrands from Example~\ref{ex:Orlicz}. Next, let $\Phi\colon\R_{\geq 0}\to\R_{\geq 0}$ be an $N$-function of class $\Delta_{2}$. We then define
\begin{align}\label{eq:generalisedVfunction}
V_{\Phi}(z):=\Phi((1+|z|^{2})^{\frac{1}{2}})-\Phi(1). 
\end{align}
and consider an integrand $F\colon\R^{N\times n}\to\R$ which satisfies the following set of hypotheses: 
\begin{enumerate}[label={(A\arabic{*})},start=1]
\item\label{item:A1} There exists $L>0$ such that $|F(z)|\leq L(1+\Phi(|z|))$ holds for all $z\in\R^{N\times n}$. 
\item\label{item:A2} For any $m>0$ there exists $\ell_{m}>0$ such that for all $z_{0}\in\R^{N\times n}$ with $|z_{0}|\leq m$, $F-\ell_{m} V_{\Phi}$ is quasiconvex at $z_{0}$.  
\item\label{item:A3} $F\in\hold^{\infty}(\R^{N\times n})$. 
\end{enumerate}
In analogy with hypotheses~\ref{item:H2} and~\ref{item:H2p} as discussed in Section~\ref{sec:props},~\ref{item:A2} is related to coerciveness of the associated variational integral $\int_{\Omega}F(\nabla u)\dif x$ in Dirichlet subclasses of $\sobo^{1,\Phi}(\Omega;\R^{N})$. Adopting these conventions, the main theorem of the present section then is as follows:
\begin{theorem}\label{thm:Orlicz}
Let $\Omega\subset\R^{n}$ be open and bounded with Lipschitz boundary and let $F\colon\R^{N\times n}\to\R$ satisfy~\emph{\ref{item:A1}--\ref{item:A3}} subject to~ \emph{\ref{item:B1}} and~\emph{\ref{item:B2}}. Moreover, suppose that $u\in\sobo_{\locc}^{1,\Phi}(\Omega;\R^{N})$ is a \emph{local minimizer} of the functional $v\mapsto \int F(Dv)\dif x$ in the sense that we have for all $\ball_{r}(x_{0})\Subset\Omega$
\begin{align}\label{eq:locminOrlicz}
\int_{\ball_{r}(x_{0})}F(\nabla u)\dif x\leq \int_{\ball_{r}(x_{0})}F(\nabla u +\nabla\varphi)\dif x\qquad\text{for all}\;\varphi\in\sobo_{0}^{1,\Phi}(\ball_{r}(x_{0});\R^{N}). 
\end{align}
Then $u$ is $\hold^{\infty}$-partially regular in $\Omega$. 
\end{theorem} 
The preceding theorem is a consequence of Theorem~\ref{thm:main1} and the following three results: 
\begin{lemma}\label{lem:Orlicz1}
Let $\Phi\in\hold([0,\infty))\cap\hold^{1}((0,\infty))$ be a convex function that satisfies  $\Phi'(t_{0})>0$ for some $t_{0}\in(0,1]$. Then there exists a constant $c>0$ such that 
\begin{align}\label{eq:VOrlicz}
\int_{\ball_{1}(0)}V(z+\nabla\varphi)-V(z)\dif x \leq c \int_{\ball_{1}(0)}V_{\Phi}(z+\nabla\varphi)-V_{\Phi}(z)\dif x
\end{align}
holds for all $z\in\R^{N\times n}$ and all $\varphi\in\sobo_{0}^{1,\infty}(\ball_{1}(0);\R^{N})$. 
\end{lemma} 
\begin{proof} 
For~\eqref{eq:VOrlicz}, it suffices to establish that $G\colon z\mapsto cV_{\Phi}(z)-V(z)$ is convex for some $c>0$ since then~\eqref{eq:VOrlicz} is nothing but the quasiconvexity of $G$. Write $G=g\circ\langle\cdot\rangle\circ|\cdot| + 1 -c\Phi(1)$ with $g(t)=c\Phi(t)-t$ and $\langle t\rangle:=\sqrt{1+t^{2}}$. Both $g\colon [1,\infty)\to\R$ and $\langle\cdot\rangle\colon \R\to [1,\infty),|\cdot|\colon\R^{N\times n}\to\R$ are convex. Recalling that $\Phi'(t_{0})>0$ for some $t_{0}\in (0,1]$ and $\Phi'$ is non-decreasing by convexity of $\Phi$, we may choose $c>1/\Phi'(t_{0})$ to conclude that $g$ is non-decreasing on $[1,\infty)$. For this choice of $c$, $g\circ\langle\cdot\rangle$ is non-decreasing and convex, and this suffices to conclude that $G$ is convex. This completes the proof. 
\end{proof} 
The second ingredient is a variant of Lemma~\ref{lem:LSC}. Its proof is a variation of the proof of Lemma~\ref{lem:LSC} and is provided for the reader's convenience in the appendix, Section~\ref{sec:OrliczSec}:
\begin{proposition}[Lower semicontinuity]\label{lem:Orlicz2}
Let $F\in\hold(\R^{N\times n})$ be a quasiconvex integrand that satisfies the growth bound~ \emph{\ref{item:A1}} for some $N$-function $\Phi$ of class $\Delta_{2}$. Given an open and bounded set $\Omega\subset\R^{n}$ with Lipschitz boundary and letting $u_{0}\in\sobo^{1,\Phi}(\Omega;\R^{N})$, suppose that $u,u_{1},u_{2},...\in\sobo_{u_{0}}^{1,\Phi}(\Omega;\R^{N}):=u_{0}+\sobo_{0}^{1,\Phi}(\Omega;\R^{N})$ are such that $u_{j}\to u$ strongly in $\lebe^{1}(\Omega;\R^{N})$ and $\sup_{j\in\mathbb{N}}\|\Phi(\nabla u_{j})\|_{\lebe^{1}(\Omega)}<\infty$. Then we have 
\begin{align*}
\int_{\Omega}F(\nabla u)\dif x \leq \liminf_{j\to\infty}\int_{\Omega}F(\nabla u_{j})\dif x.
\end{align*}
\end{proposition}
Finally, we require a slight generalisation of the Lipschitz-type bound~\eqref{eq:lipschitz}: 
\begin{lemma}\label{lem:OrliczBound}
Let $\Phi$ be an $N$-function of class $\Delta_{2}$ and let $F\in\hold(\R^{N\times n})$ be rank-one convex with~\emph{\ref{item:A1}}. Then there exists a constant $C=C(\Delta_{2}(\Phi),N,n)>0$ such that we have 
\begin{align}\label{eq:lipschitzOrlicz}
|F(z)-F(w)|\leq C \frac{\Phi(1+|z|+|w|)}{1+|z|+|w|}|z-w|\qquad\text{for all}\;z,w\in\R^{N\times n}. 
\end{align}
\end{lemma} 
The previous result follows by direct inspection of the proof of the corresponding result for power functions (e.g. cf.~\cite[Thm.~5.2]{Giusti}). We now may pass to the 
\begin{proof}[Proof of Theorem~\ref{thm:Orlicz}]
We start by proving that any local minimizer of the functional $\mathscr{F}$ in the sense of~\eqref{eq:locminOrlicz} is a local $\bv$-minimizer of $\overline{\mathscr{F}}^{*}$ for compactly supported variations in the sense of Definition~\ref{def:locmin}. Hence let $u\in\sobo_{\locc}^{1,\Phi}(\Omega;\R^{N})$ be a local minimizer in the sense of~\eqref{eq:locminOrlicz}. Then, in particular, $u\in\sobo_{\locc}^{1,1}(\Omega;\R^{N})$ and so for any $x_{0}\in\Omega$ we find a radius $r>0$ such that $\trace_{\partial\!\ball_{r}(x_{0})}(u)\in\sobo^{1-1/q,q}(\partial\!\ball_{r}(x_{0});\R^{N})$ for $1<q<\frac{n}{n-1}$ as in~\ref{item:B2}. In consequence, we find $u_{0}\in\sobo^{1,q}(\Omega;\R^{N})$ such that $\trace_{\partial\!\ball_{r}(x_{0})}(u_{0})=\trace_{\partial\!\ball_{r}(x_{0})}(u)$ $\mathscr{H}^{n-1}$-a.e. on $\partial\!\ball_{r}(x_{0})$. We then have $u-u_{0}\in\sobo_{0}^{1,\Phi}(\ball_{r}(x_{0});\R^{N})$ (cf.~\ref{item:A1}) and hence there exists $(\psi_{j})\subset\hold_{c}^{\infty}(\Omega;\R^{N})$ with $\spt(\psi_{j})\Subset\ball_{r}(x_{0})$ for all $j\in\mathbb{N}$ such that $\|(u-u_{0})-\psi_{j}\|_{\sobo^{1,\Phi}(\ball_{r}(x_{0}))}\to 0$. Passing to a non-relabeled subsequence, we then also have  $u_{0}+\psi_{j}\stackrel{*}{\rightharpoonup} \mathbf{u}$ in $\bv(\Omega;\R^{N})$, where $\mathbf{u}$ is the extension of $u|_{\ball_{r}(x_{0})}$ to $\Omega$ by $u_{0}$. On the one hand, we then have 
\begin{align}\label{eq:Orlizzo1}
\begin{split}
\int_{\ball_{r}(x_{0})}F(\nabla (u_{0}+\psi_{j}))\dif x & \leq \int_{\ball_{r}(x_{0})}|F(\nabla u_{0}+\nabla\psi_{j})-F(\nabla u)|\dif x \\ & + \int_{\ball_{r}(x_{0})}F(\nabla u)\dif x =: \mathrm{I}_{j}+\mathrm{II}. 
\end{split}
\end{align}
Write $\vartheta_{j}:= 1+|\nabla (u_{0}+\psi_{j})|+|\nabla u|$ and $\theta_{j}:=|\nabla u-(\nabla (u_{0}+\psi_{j}))|$, whereby we have $\|\theta_{j}\|_{\lebe^{\Phi}(\ball_{r}(x_{0}))}\to 0$ as $j\to\infty$. Passing to another subsequence, we may assume that either (i) $\|\Phi(\vartheta_{j})\|_{\lebe^{1}(\ball_{r}(x_{0}))}< 1$ or (ii) $\|\Phi(\vartheta_{j})\|_{\lebe^{1}(\ball_{r}(x_{0}))}\geq 1$ for all $j\in\mathbb{N}$. In the first case, we note that for all $j\in\mathbb{N}$ 
\begin{align*}
\left\vert\!\left\vert \Phi^{*}\left(\frac{\Phi(\vartheta_{j})}{\vartheta_{j}}\right)\right\vert\!\right\vert_{\lebe^{1}(\ball_{r}(x_{0}))} \stackrel{\eqref{eq:FenchelInequality}}{\leq} \|\Phi(\vartheta_{j})\|_{\lebe^{1}(\ball_{r}(x_{0}))}\leq 1, 
\end{align*}
so that 
\begin{align}\label{eq:TabeaTscherpelSuperBate}
\left\vert\!\left\vert \frac{\Phi(\vartheta_{j})}{\vartheta_{j}}\right\vert\!\right\vert_{\lebe^{\Phi^{*}}(\ball_{r}(x_{0}))}\leq 1. 
\end{align}
Since $\Phi$ satisfies~\ref{item:B1}, we may use Lemma~\ref{lem:OrliczBound}. Then by H\"{o}lder's inequality for Orlicz spaces~\eqref{eq:HoelderOrlicz} and using that $\|\theta_{j}\|_{\lebe^{\Phi}(\ball_{r}(x_{0}))}\to 0$,  
\begin{align*}
\mathrm{I}_{j} \stackrel{\text{Lem.~\ref{lem:OrliczBound}}}{\leq} C\int_{\ball_{r}(x_{0})} \frac{\Phi(\vartheta_{j})}{\vartheta_{j}}\theta_{j}\dif x \stackrel{\eqref{eq:HoelderOrlicz}}{\leq} 2C\left\vert\!\left\vert \frac{\Phi(\vartheta_{j})}{\vartheta_{j}}\right\vert\!\right\vert_{\lebe^{\Phi^{*}}(\ball_{r}(x_{0}))} \|\theta_{j}\|_{\lebe^{\Phi}(\ball_{r}(x_{0}))}\stackrel{\eqref{eq:TabeaTscherpelSuperBate}}{\longrightarrow} 0. 
\end{align*}
In the second case, we define $a_{j}:=\|\Phi(\vartheta_{j})\|_{\lebe^{1}(\ball_{r}(x_{0}))}^{-1}$ and set $\Phi_{j}:=a_{j}\Phi$. Then we have 
\begin{align}\label{eq:truecrime}
\begin{split}
\mathrm{I}_{j} & \stackrel{\text{Lem.~\ref{lem:OrliczBound}}}{\leq} C\int_{\ball_{r}(x_{0})} \frac{\Phi(\vartheta_{j})}{\vartheta_{j}}\theta_{j}\dif x = Ca_{j}^{-1}\int_{\ball_{r}(x_{0})} \frac{\Phi_{j}(\vartheta_{j})}{\vartheta_{j}}\theta_{j}\dif x \\ 
& \;\;\stackrel{\eqref{eq:HoelderOrlicz}}{\leq} 2Ca_{j}^{-1}\left\vert\!\left\vert \frac{\Phi_{j}(\vartheta_{j})}{\vartheta_{j}}\right\vert\!\right\vert_{\lebe^{\Phi_{j}^{*}}(\ball_{r}(x_{0}))}\|\theta_{j}\|_{\lebe^{\Phi_{j}}(\ball_{r}(x_{0}))} 
\end{split}
\end{align} 
again by H\"{o}lder's inequality for Orlicz spaces. Now note that 
\begin{align*}
\int_{\ball_{r}(x_{0})} \Phi_{j}^{*}\left(\frac{\Phi_{j}(\vartheta_{j})}{\vartheta_{j}} \right)\dif x \stackrel{\eqref{eq:FenchelInequality}}{\leq} \int_{\ball_{r}(x_{0})}\Phi_{j}(\vartheta_{j})\dif x =1, 
\end{align*}
whereby 
\begin{align}\label{eq:tababen}
\left\vert\!\left\vert \frac{\Phi_{j}(\vartheta_{j})}{\vartheta_{j}}\right\vert\!\right\vert_{\lebe^{\Phi_{j}^{*}}(\ball_{r}(x_{0}))} \leq 1. 
\end{align}
Since we are dealing with the second case, we have $a_{j}\leq 1$ for all $j\in\mathbb{N}$. Thus, for all $\lambda>0$, 
\begin{align*}
\int_{\ball_{r}(x_{0})}\Phi_{j}\left(\frac{\theta_{j}}{\lambda}\right)\dif x = a_{j}\int_{\ball_{r}(x_{0})}\Phi\left(\frac{\theta_{j}}{\lambda}\right)\dif x \leq \int_{\ball_{r}(x_{0})}\Phi\left(\frac{\theta_{j}}{\lambda}\right)\dif x
\end{align*}
and therefore
\begin{align}\label{eq:BabenT}
\|\theta_{j}\|_{\lebe^{\Phi_{j}}(\ball_{r}(x_{0}))}\leq \|\theta_{j}\|_{\lebe^{\Phi}(\ball_{r}(x_{0}))}\to 0. 
\end{align}
In view of~\eqref{eq:truecrime},~\eqref{eq:tababen} and~\eqref{eq:BabenT}, it remains to establish that $(a_{j}^{-1})$ is bounded. By~\ref{item:B1}, we have $\Phi(2t)\leq C\Phi(t)$ for all $t\geq 0$. We choose $k\in\mathbb{N}$ with $\|\vartheta_{j}\|_{\lebe^{\Phi}(\ball_{r}(x_{0}))}\leq 2^{k}$ for all $j$, so 
\begin{align*}
a_{j}^{-1} & = \int_{\ball_{r}(x_{0})}\Phi\Big(\|\vartheta_{j}\|_{\lebe^{\Phi}(\ball_{r}(x_{0}))}\frac{\vartheta_{j}}{\|\vartheta_{j}\|_{\lebe^{\Phi}(\ball_{r}(x_{0}))}}\Big)\dif x \\ 
& \leq \Delta_{2}(\Phi)^{k}\int_{\ball_{r}(x_{0})}\Phi\Big(\frac{\vartheta_{j}}{\|\vartheta_{j}\|_{\lebe^{\Phi}(\ball_{r}(x_{0}))}}\Big)\dif x \leq \Delta_{2}(\Phi)^{k}.
\end{align*}
Summarising, we have $\mathrm{I}_{j}\to 0$ as $j\to\infty$, and going back to~\eqref{eq:Orlizzo1} we conclude that 
\begin{align}\label{eq:OrliczMin1}
\overline{\mathscr{F}}_{u}^{*}[u;\ball_{r}(x_{0})]\leq \liminf_{j\to\infty}\int_{\ball_{r}(x_{0})}F(\nabla (u_{0}+\psi_{j}))\dif x \leq \int_{\ball_{r}(x_{0})}F(\nabla u)\dif x.
\end{align}
Next let $\varphi\in\bv_{c}(\ball_{r}(x_{0});\R^{N})$ and let $(v_{j})\subset u_{0}+\sobo_{0}^{1,q}(\ball_{r}(x_{0});\R^{N})$ be a generating sequence for $\overline{\mathscr{F}}_{u_{0}}^{*}[u+\varphi;\ball_{r}(x_{0})]$. Then $v_{j}-u \in\sobo_{0}^{1,\Phi}(\ball_{r}(x_{0});\R^{N})$, whereby the local minimality~\eqref{eq:locminOrlicz} yields 
\begin{align*}
\int_{\ball_{r}(x_{0})}F(\nabla u)\dif x \stackrel{\eqref{eq:locminOrlicz}}{\leq}  \int_{\ball_{r}(x_{0})}F(\nabla u + \nabla(v_{j}-u))\dif x = \int_{\ball_{r}(x_{0})}F(\nabla v_{j})\dif x, 
\end{align*}
and passing to the limit on the right-hand side yields 
\begin{align}\label{eq:OrliczMin2}
\int_{\ball_{r}(x_{0})}F(\nabla u)\dif x \leq \limsup_{j\to\infty}\int_{\ball_{r}(x_{0})}F(\nabla v_{j})\dif x = \overline{\mathscr{F}}_{u}^{*}[u+\varphi;\ball_{r}(x_{0})]. 
\end{align}
Combining~\eqref{eq:OrliczMin1} and~\eqref{eq:OrliczMin2} yields that $u$ is a local $\bv$-minimizer for compactly supported variations. 

Now it suffices to check that $F$ satisfies~\ref{item:H1}--\ref{item:H3}: Condition~\ref{item:H1} is satisfied by~\ref{item:A1} and~\ref{item:B2}, condition~\ref{item:H2} is fulfilled by~\ref{item:A2} and Lemma~\ref{lem:Orlicz1}, and~\ref{item:H3} holds by assumption~\ref{item:A3}. Hence Theorem~\ref{thm:Orlicz} follows from Theorem~\ref{thm:main1}. 
\end{proof}
\begin{remark}
In the setting of the proof of Theorem~\ref{thm:Orlicz} from above, we have 
\begin{align}\label{eq:sandwichdiscuss}
\int_{\ball_{r}(x_{0})}F(\nabla u)\dif x =: \mathscr{F}[u;\ball_{r}(x_{0})] \stackrel{\text{Prop.~\ref{lem:Orlicz2}}}{\leq} \overline{\mathscr{F}}_{u_{0}}^{*}[u;\ball_{r}(x_{0})] \stackrel{\eqref{eq:OrliczMin1}}{\leq} \mathscr{F}[u;\ball_{r}(x_{0})],
\end{align}
so $\mathscr{F}[u;\ball_{r}(x_{0})]  =  \overline{\mathscr{F}}_{u_{0}}^{*}[u;\ball_{r}(x_{0})]$. Note that the second inequality of~\eqref{eq:sandwichdiscuss} necessitates the existence of a sequence $(v_{j})\subset \sobo_{u_{0}}^{1,q}(\Omega;\R^{n})$ and a convergence $v_{j}\rightsquigarrow u$, at least as strong as weak*-convergence on $\bv$, for which the functional $\mathscr{F}[-;\ball_{r}(x_{0})]$ proves upper semicontinuous at $u$: $\limsup_{j\to\infty}\mathscr{F}[v_{j};\ball_{r}(x_{0})]\leq \mathscr{F}[u;\ball_{r}(x_{0})]$. The proof of Theorem~\ref{thm:Orlicz} then shows that one might take the strong convergence in $\sobo^{1,\Phi}(\ball_{r}(x_{0});\R^{N})$ as a potential choice for '$\rightsquigarrow$'. 

As established by \textsc{Schmidt} \cite[Cor.~4.4]{SchmidtPR1}, in the setting where $1<p<q<\frac{np}{n-1}$ and $F\in\hold^{\infty}(\R^{N\times n})$ is quasiconvex with $c|z|^{p}\leq F(z) \leq C(1+|z|^{q})$, for any $u_{0}\in\sobo^{1,p}(\Omega;\R^{N})$ there exists $u\in \sobo_{u_{0}}^{1,p}(\Omega;\R^{N})$ such that 
\begin{align}\label{eq:TooWeakMinimality}
\int_{\Omega}F(\nabla u)\dif x \leq \int_{\Omega}F(\nabla v)\dif x\qquad\text{for all}\;v\in (\sobo_{u_{0}}^{1,p}\cap\sobo_{\locc}^{1,q})(\Omega;\R^{N}). 
\end{align}
This is a consequence of semicontinuity results by \textsc{Fonseca \& Mal\'{y}}~\cite{FM} or the second named author~\cite{Kristensen97}. In view of the above proof of Theorem~\ref{thm:Orlicz}, at a first glance it seems that the notion of minimality embodied in~\eqref{eq:TooWeakMinimality} implies that $u$ necessarily is a local minimizer of $\overline{\mathscr{F}}$ and thus the partial regularity for local minimizers follows by reduction to Theorem~\ref{thm:main1} or~\ref{thm:main2} as above. However, this is not so: Whereas the corresponding variant of~\eqref{eq:OrliczMin2} persists by virtue of~\eqref{eq:TooWeakMinimality} and $\sobo_{u_{0}}^{1,q}(\ball_{r}(x_{0});\R^{N})\subset (\sobo_{u}^{1,p}\cap\sobo_{\locc}^{1,q})(\ball_{r}(x_{0});\R^{N})$, it is inequality~\eqref{eq:OrliczMin1} for which there is no reason to hold in this new context. In fact, incorporating the obvious modifications, the term corresponding to $\mathrm{I}_{j}$ as in~\eqref{eq:Orlizzo1} does not need to vanish as $j\to\infty$. Namely, the essentially only available tool to arrive at this implication is the Lipschitz bound~\eqref{eq:lipschitz} in combination with \emph{strong} convergence in $\sobo^{1,p}(\ball_{r}(x_{0});\R^{N})$ as a choice for a potential convergence '$\rightsquigarrow$'. Still, we then require $q\leq p$ to conclude $\limsup_{j\to\infty}\mathrm{I}_{j}=0$, but then the problem trivialises as we enter the realm of standard growth conditions. 
\end{remark}
\subsubsection{Differential conditions}\label{sec:diffcond}
Theorem~\ref{thm:main2} directly admits a generalisation to scenarios where the gradient is replaced by differential operators such as the (trace-free) symmetric gradient. Such operators play an important role in the context of nonlinear elasticity, cf.~e.g.~\cite{FonMul1,FuchsSeregin} and the references therein. This procedure works by a reduction to the full gradient case as introduced by the first author in~\cite{Gm20} for the symmetric gradient and in~\cite{CG} for general elliptic operators. Henceforth, let $\mathbb{A}$ be a first order, constant coefficient, linear and homogeneous differential operator on $\R^{n}$ from $\R^{m_{1}}$ to $\R^{m_{2}}$; especially, $\A$ has a representation 
\begin{align}\label{eq:diffopform}
\A u= \sum_{j=1}^{n}\A_{j}\partial_{j}u\;\;\;\text{for}\;\;\;u\colon\R^{n}\to\R^{m_{1}}
\end{align}  
with linear maps $\A_{j}\colon\R^{m_{1}}\to\R^{m_{2}}$. Following \textsc{H\"{o}rmander}~\cite{Hoermander} or \textsc{Spencer}~\cite{Spencer}, we then say that $\A$ is \emph{elliptic} provided its Fourier symbol $\A[\xi]:=\sum_{j=1}^{n}\A_{j}\xi_{j}\colon\R^{m_{1}}\to\R^{m_{2}}$ is injective for any $\xi=(\xi_{1},...,\xi_{n})\in\R^{n}\setminus\{0\}$. Let $1< p \leq q <\frac{np}{n-1}$. For open and bounded Lipschitz domains $\Omega\Subset\Omega'\subset\R^{n}$ and $u_{0}\in\sobo_{0}^{1,q}(\Omega';\R^{m_{1}})$, we then put for $G\colon\R^{m_{2}}\to\R$ 
\begin{align*}
\overline{\mathscr{G}}_{u_{0}}[\mathbf{u};\Omega,\Omega'] = \inf\left\{ \liminf_{j\to\infty}\int_{\Omega'}G(\mathbb{A}u)\dif x\colon\;\begin{array}{c} (u_{j})\subset\mathscr{A}_{u_{0}}^{q}(\Omega,\Omega'),\\ u_{j}\rightharpoonup \mathbf{u}\;\text{in}\;\sobo^{1,p}(\Omega';\R^{m_{1}})\end{array} \right\}
\end{align*}
and introduce $\overline{\mathscr{G}}_{u_{0}}[u;\Omega]$, $\overline{\mathscr{G}}_{u}[u;\ball_{r}(x_{0})]$ as in Section~\ref{sec:props}; the notion of local minimality for compactly supported variations then is analogous to the one from Definition~\ref{def:locmin}. In this context, the corresponding variant of the usual quasiconvexity is the $\A$-\emph{quasiconvexity} of $G$ at $z_{0}\in\R^{m_{2}}$: 
\begin{align}\label{eq:strongQCdiffcond}
G(z_{0})\leq \dashint_{\ball_{1}(0)}G(z_{0}+\A\varphi)\dif x\qquad\text{for all}\;\varphi\in\hold_{c}^{\infty}(\ball_{1}(0);\R^{m_{1}}),
\end{align} 
and we refer the reader to~\cite{Dac82a,FonMul1} for more on the underlying terminology. The counterpart of Theorem~\ref{thm:main2} then is as follows: 
\begin{theorem}\label{thm:diffops}
Let $\Omega\subset\R^{n}$ be open and bounded with Lipschitz boundary, $1<p\leq q<\min\{\frac{np}{n-1},p+1\}$, $\A$ be an elliptic operator of the form~\eqref{eq:diffopform} and let $G\in\hold^{\infty}(\R^{m_{2}})$ satisfy the growth bound $|G(z)|\leq c(1+|z|^{q})$ for all $z\in\R^{m_{2}}$ for some $c>0$. If for any $m>0$ there exists $\ell_{m}>0$ such that $z\mapsto G-\ell_{m}V_{p}$ is $\A$-quasiconvex at every $z_{0}\in\R^{m_{2}}$ with $|z_{0}|\leq m$, then every local minimizer of $\overline{\mathscr{G}}$ for compactly supported variations is $\hold^{\infty}$-partially regular. 
\end{theorem}
\begin{proof}
We follow~\cite{CG} and pick $\Pi_{\A}\colon \R^{m_{1}\times n}\to \R^{m_{2}}$ such that $\mathbb{A}u=\Pi_{\A}[\nabla u]$ for any smooth $u\colon\R^{n}\to\R^{m_{1}}$. Then $F:=G\circ\Pi_{\A}$ is of class $\hold^{\infty}(\R^{m_{1}\times n})$, satisfies the growth bound $|F(z)|\leq C(1+|z|^{q})$ for all $z\in\R^{m_{1}\times n}$ and satisfies~\ref{item:H2p}. In fact, by \cite[Proof of Thm.~1.1,~Eq.~(4.9)--(4.10)]{CG}, we use $p>1$ to conclude the existence of a constant $c=c(\A,p)>0$ such that 
\begin{align}\label{eq:BAUteam}
\int_{\ball_{1}(0)}V_{p}(z+\nabla\varphi)-V_{p}(z)\dif x \leq c\int_{\ball_{1}(0)}V_{p}(\Pi_{\A}(z)+\Pi_{\A}(\nabla\varphi))-V_{p}(\Pi_{\A}(z))\dif x
\end{align}
holds for all $\varphi\in\hold_{c}^{\infty}(\ball_{1}(0);\R^{m_{1}})$. From here, we see that the assumptions of the present theorem imply~\ref{item:H2p} for $F$. Since $\overline{\mathscr{F}}$ and $\overline{\mathscr{G}}$ have the same local minimizers for compactly supported variations, the claim now follows from Theorem~\ref{thm:main2}. The proof is complete. 
\end{proof} 
\begin{remark}\label{rem:diffcondsp1}
The reader will notice that~\eqref{eq:BAUteam}, and hence Theorem~\ref{thm:diffops}, does not directly generalise to $p=1$. This is because of \textsc{Ornstein}'s Non-Inequality, cf.~\cite{Ornstein,KirchheimKristensen}. Still, as established in the purely linear growth case $p=q=1$, partial regularity persists in the $1$-strongly symmetric quasiconvex context~\cite{Gm20} (i.e., the $1$-strong $\mathbb{A}$-quasiconvexity variant of~\eqref{eq:strongQCdiffcond} with $\mathbb{A}u:=\frac{1}{2}(Du+Du^{\top})$ being the symmetric gradient). For the class of so-called $\mathbb{C}$\emph{-elliptic} operators $\mathbb{A}$, this result has been generalised in~\cite{BaerlinKessler}. Being equivalent to having finite dimensional nullspaces, $\mathbb{C}$-ellipticity is strongly intertwined with the existence of $\lebe^{1}$-trace operators for the underlying function spaces \cite{BDG}. As the method of the proof of Theorem~\ref{thm:main1} is strongly based on the specific trace-preserving operators from Section~\ref{sec:Fubini}, it would be interesting to know whether the method of the proof of Theorem~\ref{thm:main1} also generalises to this case. 
\end{remark}

\section{Appendix}\label{sec:appendix}
In this section we provide the proofs of some auxiliary results that appeared in the main part of the paper. 
\subsection{Proof of the Hardy-Littlewood-type selection lemma}
We begin with the results on the selection of good points as gathered in Section~\ref{sec:HLW}:
\begin{proof}[Proof of Lemma~\ref{lem:HLW}]
Since $f\colon [r,s]\to\R$ is monotonously increasing, right-continuous and bounded, $f$ is of bounded variation on $[r,s]$. In particular, there exists a finite, positive Borel measure (the Lebesgue-Stieltjes measure) $f'$ on $[r,s]$ such that $f'((t_{1},t_{2}])=f(t_{2})-f(t_{1})$ for all $r<t_{1}<t_{2}\leq s$; as usual, $f'$ is obtained as the Carath\'{e}odory extension of the set function $\eta$ defined by $\eta((t_{1},t_{2}]):=f(t_{2})-f(t_{1})$ for $r\leq t_{1}< t_{2}\leq s$ and $\eta(\{r\})=0$. In particular, we have $f'([r,s])\leq f(s)-f(r)$. 

We define upper and lower maximal functions for $t\in (r,s)$ by 
\begin{align*}
(f')_{\mathrm{u}}^{*}(t):=\sup_{\tau \in (t,s]}\frac{1}{\tau-t}\int_{[t,\tau]}f'\;\;\text{and}\;\;(f')_{\mathrm{l}}^{*}(t):=\sup_{\tau \in [r,t)}\frac{1}{t-\tau}\int_{[\tau,t]}f'. 
\end{align*}
We now briefly derive the requisite weak-$(1,1)$-type estimate for $(f')_{\mathrm{u}}^{*}$. Given $\lambda>0$, let $\mathcal{O}_{\lambda}:=\{t\in(r,s)\colon\;(f')_{\mathrm{u}}^{*}(t)>\lambda\}$. For each $t\in\mathcal{O}_{\lambda}$, we find $\tau_{t}\in (t,s]$ such that 
\begin{align*}
\int_{[t,\tau_{t}]}f' \geq \lambda (\tau_{t}-t). 
\end{align*}
By use of the Vitali covering lemma, we then find a sequence $(t_{j})\subset[r,s]$ and a corresponding sequence $(\tau_{t_{j}})\subset [r,s]$ such that for each $j\in\mathbb{N}$ there holds $t_{j}<\tau_{t_{j}}\leq s$, the intervals $[t_{j},\tau_{t_{j}}]$ are mutually disjoint and we have
\begin{align*}
\mathcal{O}_{\lambda} \subset\bigcup_{j}5[t_{j},\tau_{t_{j}}]\;\;\text{and}\;\;\int_{[t_{j},\tau_{t_{j}}]}f' \geq \lambda (\tau_{t_{j}}-t_{j})\;\;\text{for all}\;j\in\mathbb{N}. 
\end{align*}
Here, given an interval $[x,y]\subset[r,s]$, we write $[x,y]=[z-\vartheta,z+\vartheta]$ and then define $5[x,y]:=[z-5\vartheta,z+5\vartheta]\cap[r,s]$. 
Therefore, 
\begin{align*}
\mathscr{L}^{1}(\mathcal{O}_{\lambda})& \leq 5\sum_{j}(\tau_{t_{j}}-t_{j}) \leq \frac{5}{\lambda}\sum_{j}\int_{[t_{j},\tau_{t_{j}}]}f' \\ 
& \leq \frac{5}{\lambda}\sum_{j}\int_{(t_{j},\tau_{t_{j}}]}f' + \frac{5}{\lambda}\sum_{j}f'(\{t_{j}\}) \\ & \leq \frac{10}{\lambda} \int_{[r,s]}f' \leq  \frac{10}{\lambda}(f(s)-f(r)). 
\end{align*}
We may argue analogously for $(f')_{\mathrm{l}}^{*}$. We then obtain  
\begin{align*}
\mathscr{L}^{1}(\mathcal{O}_{\lambda}^{\mathrm{u}}):=\mathscr{L}^{1}(\{t\in (r,s)\colon\;(f')_{\mathrm{u}}^{*}(t)>\lambda\})\leq \frac{10}{\lambda}(f(s)-f(r)), \\ 
\mathscr{L}^{1}(\mathcal{O}_{\lambda}^{\mathrm{l}}):=\mathscr{L}^{1}(\{t\in (r,s)\colon\;(f')_{\mathrm{l}}^{*}(t)>\lambda\})\leq \frac{10}{\lambda}(f(s)-f(r)).
\end{align*}
We may assume without loss of generality that $f(r)<f(s)$ as otherwise there is nothing to prove. For $c>0$ to be fixed later, put $\lambda^{*}:=c\frac{f(s)-f(r)}{s-r}$. Then 
\begin{align*}
\mathscr{L}^{1}(\mathcal{O}_{\lambda^{*}}^{\mathrm{u}}):=\mathscr{L}^{1}(\{t\in (r,s)\colon\;(f')_{\mathrm{u}}^{*}(t)>c\frac{f(s)-f(r)}{s-r}\})\leq \frac{10}{c}(s-r), \\ 
\mathscr{L}^{1}(\mathcal{O}_{\lambda^{*}}^{\mathrm{l}}):=\mathscr{L}^{1}(\{t\in (r,s)\colon\;(f')_{\mathrm{l}}^{*}(t)>c\frac{f(s)-f(r)}{s-r}\})\leq \frac{10}{c}(s-r). 
\end{align*}
In consequence, $\mathscr{L}^{1}((\mathcal{O}_{\lambda^{*}}^{\mathrm{u}})^{\complement}),\mathscr{L}^{1}((\mathcal{O}_{\lambda^{*}}^{\mathrm{l}})^{\complement})>(1-\frac{10}{c})(s-r)$ and therefore 
\begin{align}\label{eq:Korn}
\begin{split}
2\Big(1-\frac{10}{c}\Big)(s-r) & \leq \mathscr{L}^{1}((\mathcal{O}_{\lambda^{*}}^{\mathrm{u}})^{\complement})+\mathscr{L}^{1}((\mathcal{O}_{\lambda^{*}}^{\mathrm{l}})^{\complement}) \\ 
& \leq \mathscr{L}^{1}((\mathcal{O}_{\lambda^{*}}^{\mathrm{u}})^{\complement}\cup (\mathcal{O}_{\lambda^{*}}^{\mathrm{l}})^{\complement}) + \mathscr{L}^{1}((\mathcal{O}_{\lambda^{*}}^{\mathrm{u}})^{\complement}\cap (\mathcal{O}_{\lambda^{*}}^{\mathrm{l}})^{\complement}) \\ &  \leq (s-r) + \mathscr{L}^{1}((\mathcal{O}_{\lambda^{*}}^{\mathrm{u}})^{\complement}\cap ((\mathcal{O}_{\lambda^{*}}^{\mathrm{l}})^{\complement})
\end{split}
\end{align}
yields 
\begin{align}\label{eq:lowerboundmeasure}
\Big(1-\frac{20}{c}\Big)(s-r) \leq \mathscr{L}^{1}((\mathcal{O}_{\lambda^{*}}^{\mathrm{u}})^{\complement}\cap (\mathcal{O}_{\lambda^{*}}^{\mathrm{l}})^{\complement}).
\end{align}
Hence \eqref{eq:constructssr} follows; indeed, if $t\in 
(\mathcal{O}_{\lambda^{*}}^{\mathrm{u}})^{\complement}\cap ((\mathcal{O}_{\lambda^{*}}^{\mathrm{l}})^{\complement}$, then 
\begin{align*}
&c\frac{f(s)-f(r)}{s-r}\geq (f')_{\mathrm{u}}^{*}(t) \geq \frac{f(\tau)-f(t)}{\tau-t}\qquad \text{for all}\;\tau\in (t,s]\;\;\text{and}\\
& c\frac{f(s)-f(r)}{s-r}\geq (f')_{\mathrm{l}}^{*}(t) \geq \frac{f(t)-f(\tau)}{t-\tau}\qquad\text{for all}\; \tau \in [r,t). 
\end{align*}
In fact, since $f'((a,b])=f(b)-f(a)$ for any semi-open interval, we have e.g. for all $\tau\in(t,s]$:
\begin{align*}
\frac{f(\tau)-f(t)}{\tau-t} = \frac{f'((t,\tau])}{\tau-t}\leq \frac{f'([t,\tau])}{\tau-t}\leq (f')_{\mathrm{u}}^{*}(t) \leq c\frac{f(s)-f(r)}{s-r}. 
\end{align*}
We now argue that this can be achieved for $\widetilde{r}<\widetilde{s}<\widetilde{t}$ such that \eqref{eq:comparessr}  holds.
Recalling that $\mathscr{L}^{1}(E)<\theta(s-r)$, we deduce similarly as in~\eqref{eq:Korn} but now using~\eqref{eq:lowerboundmeasure}  and $\mathscr{L}^{1}(E^{\complement})\geq (1-\theta)(s-r)$
\begin{align}\label{eq:semiHLW}
\Big(1-\frac{20}{c}-\theta\Big)(s-r) \leq \mathscr{L}^{1}((\mathcal{O}_{\lambda^{*}}^{\mathrm{u}})^{\complement}\cap(\mathcal{O}_{\lambda^{*}}^{\mathrm{l}})^{\complement}\cap E^{\complement}). 
\end{align}
We then conclude from \eqref{eq:semiHLW} that 
\begin{align}\label{eq:cappchose}
c:=\frac{800}{1-8\theta}\Longrightarrow \frac{7}{8}(s-r)<\mathscr{L}^{1}((\mathcal{O}_{\lambda^{*}}^{\mathrm{u}})^{\complement}\cap(\mathcal{O}_{\lambda^{*}}^{\mathrm{l}})^{\complement}\cap E^{\complement}). 
\end{align}
Namely, note that we have 
\begin{align*}
\frac{7}{8}(s-r) < \Big(1-\frac{20}{c}-\theta\Big)(s-r) \Leftrightarrow \frac{20}{c} < \frac{1-8\theta}{8}\Leftrightarrow \frac{160}{1-8\theta}<c, 
\end{align*}
and the last condition is certainly fulfilled for our choice of $c$. Therefore, whenever a measurable subset $\mathfrak{A}\subset(r,s)$ is such that $\mathscr{L}^{1}(\mathfrak{A})\geq \tfrac{1}{8}(s-r)$, the intersection $(\mathcal{O}_{\lambda^{*}}^{\mathrm{u}})^{\complement}\cap(\mathcal{O}_{\lambda^{*}}^{\mathrm{l}})^{\complement}\cap E^{\complement}\cap\mathfrak{A}$ must be non-empty. Now consider the intervals 
\begin{align*}
&I_{1}=(\tfrac{3}{4}r+\tfrac{s}{4}-\tfrac{1}{16}(s-r),\tfrac{3}{4}r+\tfrac{s}{4}+\tfrac{1}{16}(s-r)), \\ & I_{2}=(\tfrac{3}{4}s+\tfrac{r}{4}-\tfrac{1}{16}(s-r),\tfrac{3}{4}s+\tfrac{r}{4}+\tfrac{1}{16}(s-r)), 
\end{align*}
their intersections with $(\mathcal{O}_{\lambda^{*}}^{\mathrm{u}})^{\complement}\cap(\mathcal{O}_{\lambda^{*}}^{\mathrm{l}})^{\complement}\cap E^{\complement}$ must be non-empty each, and so we may choose the requisite $\widetilde{r}$ to belong to $I_{1}\cap(\mathcal{O}_{\lambda^{*}}^{\mathrm{u}})^{\complement}\cap(\mathcal{O}_{\lambda^{*}}^{\mathrm{l}})^{\complement}\cap E^{\complement}$ and $\widetilde{s}$ to belong to $I_{2}\cap(\mathcal{O}_{\lambda^{*}}^{\mathrm{u}})^{\complement}\cap(\mathcal{O}_{\lambda^{*}}^{\mathrm{l}})^{\complement}\cap E^{\complement}$. Then, by definition of $I_{1},I_{2}$, we especially have 
\begin{align*}
&\frac{1}{8}(s-r)\leq (\widetilde{s}-\widetilde{r})\leq s-r,
\end{align*}
and so \eqref{eq:comparessr} follows\footnote{Clearly, for different choices of $\widetilde{r},\widetilde{s}$ the constant $\frac{1}{8}$ can be improved to $\frac{3}{4}$, but for us it is only important that there exists some  constant.}. The proof is complete. 
\end{proof} 
\subsection{$\lebe^{p}$-approach to relaxations}\label{sec:LpApproach} 
In the main part of the paper we directly extended $\mathscr{F}$ from a subclass of $\sobo^{1,q}$ to subclasses of $\sobo^{1,p}$ by semicontinuity. It is equally natural to  first extend $\mathscr{F}$ from subclasses of $\sobo^{1,q}$ to larger spaces $X$ by infinity and then pass to the lower semicontinuous envelopes. Here we execute this task for $X=\lebe^{p}$ endowed with the strong $\lebe^{p}$-topology and compare the resulting functionals, leading to their equality and $\Gamma$-approximability by standard growth functionals under natural assumptions. For simplicity, we here focus on the analogues of~\eqref{eq:bdryrelaxed1}; results for functionals in the spirit of~\eqref{eq:bdryrelaxed2} can be obtained with the obvious modifications. 

To this end, fix exponents $1 \leq p < q<\infty$ and assume that $F \colon \R^{N\times n} \to \R$ is a continuous integrand satisfying the natural $q$-growth condition~\ref{item:H1}, so $
|F(z)| \leq L ( 1 + |z|^{q})$ for all $ z \in \R^{N\times n}$, where $L > 0$ is a constant.  Let $\Omega$ be a non-empty, bounded and open subset of $\R^n$
and $g \in \sobo^{1,q}( \R^{n} ; \R^{N} )$. Fix $\Omega'$ satisfying $\Omega \Subset \Omega^{\prime} \Subset \R^n$
and define the functional
\begin{align}
\mathscr{F}_{g}^{\lebe^{p}}[u;\Omega,\Omega'] := \left\{
\begin{array}{ll}
  \displaystyle{\int_{\Omega'} \! F( \nabla u) \, \dif x} & \mbox{ if } u \in \mathscr{A}_{g}^{q}(\Omega,\Omega'),\\
  +\infty & \mbox{ if } u \in \lebe^{p}( \Omega';\R^N ) \setminus \mathscr{A}_{g}^{q}(\Omega,\Omega'),
\end{array}
\right.
\end{align}
where the class of admissible maps $\mathscr{A}_{g}^{q}(\Omega,\Omega')$ is defined as in~\eqref{eq:SolidClass}. The lower semicontinuous envelope of $\mathscr{F}_{g}^{\lebe^{p}}[-;\Omega,\Omega']$ in the strong topology of $\lebe^{p}( \Omega';\R^N )$ is denoted by
$\overline{\mathscr{F}}_{g}^{\lebe^{p}}[u;\Omega,\Omega']$. Thus we have for each $u \in \lebe^{p}( \Omega';\R^N )$ the formula
\begin{align*}
\overline{\mathscr{F}}_{g}^{\lebe^{p}}[u;\Omega,\Omega'] = \inf \biggl\{ \liminf_{j \to \infty} \mathscr{F}_{g}^{\lebe^{p}}[u_{j};\Omega,\Omega']\colon \begin{array}{c} (u_{j})\subset \lebe^{p}(\Omega';\R^{N}),\\ u_{j} \to u\; \text{strongly in } \lebe^{p}( \Omega^{\prime} , \R^N )\end{array} \biggr\} .
\end{align*}
We begin with the following dichotomy on the boundedness of $\overline{\mathscr{F}}_{g}^{\lebe^{p}}$ from below, being reflected by a $q$-growth condition on the quasiconvexification of $F$:
\begin{lemma}\label{lem:AppendixLemma1} Under the above assumptions we have either $\overline{\mathscr{F}}_{g}^{\lebe^{p}}[-;\Omega,\Omega'] \equiv -\infty$ on $\lebe^{p}( \Omega';\R^N )$ or
$$
\inf_{u \in \lebe^{p}( \Omega';\R^N )} \overline{\mathscr{F}}_{g}^{\lebe^{p}}[u;\Omega,\Omega'] > -\infty .
$$
The latter is in turn equivalent to $F$ having \emph{real-valued quasiconvex envelope $F^{\qc}$ satisfying a natural $q$-growth condition}, so $|F^{\qc}(z)|\leq L'(1+|z|^{q})$ for all $z\in\R^{N\times n}$ and some $L'>0$.
\end{lemma}

\begin{proof}
Assume that there exists $u \in \lebe^{p}( \Omega';\R^N )$ such that $\overline{\mathscr{F}}_{g}^{\lebe^{p}}[u;\Omega,\Omega'] > -\infty$. Then we have in particular
for each sequence $( u_{j})\subset\mathscr{A}_{g}^{q}(\Omega,\Omega')$ with $u_{j} \to u$ in $\lebe^{p}( \Omega';\R^N )$ that
\begin{align*}
\liminf_{j \to \infty} \int_{\Omega^{\prime}} \! F( \nabla u_{j}) \, \dd x \geq \overline{\mathscr{F}}_{g}^{\lebe^{p}}[u;\Omega,\Omega'].
\end{align*}
Our first aim is to establish 
\begin{align}\label{eq:DacClaim} 
\inf\left\{\int_{\ball_{1}(0)}F(\nabla\varphi)\dif x\colon\; \varphi\in\sobo_{0}^{1,\infty}(\ball_{1}(0);\R^{N}),\;\|\varphi\|_{\lebe^{\infty}(\ball_{1}(0))}\leq 1\right\}>-\infty, 
\end{align}
from where we then deduce $F^{\qc}(0)>-\infty$ by~\eqref{eq:QCenvelope}\emph{ff.}. 

Towards~\eqref{eq:DacClaim}, take an ascending sequence of open sets $\Omega_{j} \Subset \Omega_{j+1} \Subset \Omega$ and corresponding Lipschitz cut-offs $\rho_{j}\colon \Omega \to [0,1]$,
$\mathbbm{1}_{\Omega_{j}} \leq \rho_{j} \leq \mathbbm{1}_{\Omega_{j+1}}$ and 
\begin{align*}
| \nabla \rho_{j} | \leq d_{j} := \frac{4}{\mathrm{dist}( \Omega_{j},\partial \Omega_{j+1})}.\end{align*}
Let $( \Phi_{\varepsilon} )_{\varepsilon > 0}$ be a family of $\varepsilon$-rescaled standard mollifiers on $\R^n$. For a descending null sequence $\varepsilon_{j} \searrow 0$ put
$v_{j} := ( 1-\rho_{j})g + \rho_{j} \Phi_{\varepsilon_{j}}*u$. Then $v_j \in \mathscr{A}_{g}^{q}(\Omega,\Omega')$ and $v_j \to u$ in $\lebe^{p}( \Omega'; \R^N )$.
Fix $\ball_{r_0}(x_{0}) \Subset \Omega$ and put $r_{j} =2^{-j}r_0$. Take Lipschitz cut-offs $\mu_{j} \colon \R^n \to [0,1]$, $\mathbbm{1}_{\ball_{r_j}(x_{0})} \leq \mu_{j}
\leq \mathbbm{1}_{\ball_{r_{j-1}}(x_{0})}$ and $| \nabla \mu_{j}| \leq \frac{4}{r_{j-1}-r_{j}}=\frac{2^{j+2}}{r_{0}}$ for $j \in \N$. For $\varphi \in \sobo^{1,\infty}_{0}(\ball_{1}(0);\R^N )$
with $\| \varphi \|_{\lebe^{\infty}(\ball_{1}(0))} \leq 1$, we put
$$
u_{j}(x) = \bigl( 1-\mu_{j}(x) \bigr)v_{j}(x) + r_{j}\varphi \left( \frac{x-x_{0}}{r_{j}} \right),\qquad x \in \Omega'.
$$
Then $u_{j} \in \mathscr{A}_{g}^{q}(\Omega,\Omega')$, $u_{j} \to u$ in $\lebe^{p}( \Omega';\R^N )$ uniformly in $\varphi$ as above. Thus we can find $j_{0} \in \N$ such that
$$
\mathscr{F}_{g}^{\lebe^{p}}[u_{j};\Omega,\Omega'] = \int_{\Omega'} \! F( \nabla u_{j}) \, \dif x > \overline{\mathscr{F}}_{g}^{\lebe^{p}}[u;\Omega,\Omega'] -1
$$
for $j \geq j_0$. In particular,
$$
\int_{\Omega^{\prime}\setminus \ball_{r_{j_0}}(x_{0})} \! F\bigl( \nabla ((1-\mu_{j_0})v_{j_0}) \bigr) \, \dd x + r_{j_0}^{n}\int_{\ball_{1}(0)} \! F( \nabla \varphi ) \, \dd x
> \overline{\mathscr{F}}_{g}^{\lebe^{p}}[u;\Omega,\Omega'] -1
$$
holds for all $\varphi \in \sobo^{1,\infty}_{0}(\ball_{1}(0);\R^N )$ with $\| \varphi \|_{\lebe^{\infty}(\ball_{1}(0))} \leq 1$. 

Hence we have $F^{\qc}(0)>-\infty$ and therefore that $F^{\qc} > -\infty$ everywhere (see~\eqref{eq:defQCenvelope}\emph{ff.}). Since in particular $F^{\qc} \leq F \leq L\bigl(1+ | \cdot |^{q}\bigr)$ on $\R^{N\times n}$ we deduce from~\eqref{eq:KristImplication} that
$| F^{\qc}(z) | \leq cL \bigl( |z|^{q}+1 \bigr)$ for all $z \in \R^{N\times n}$ for some constant $c \geq 1$; the envelope $F^{\qc}$ is then also
locally Lipschitz. Now take a Lipschitz cut-off between $\Omega$ and $\Omega'$, $\rho \colon \R^n \to [0,1]$, such that $\mathbbm{1}_{\Omega} \leq \rho \leq \mathbbm{1}_{\Omega^\prime}$ and $| \nabla \rho | \leq \frac{4}{d}$ with $d:= \mathrm{dist }( \Omega , \partial \Omega^{\prime})$. Since
$$
\inf_{u \in \lebe^{p}( \Omega^{\prime};\R^N )} \overline{\mathscr{F}}_{g}^{\lebe^{p}}[u;\Omega,\Omega'] = \inf_{u \in \mathscr{A}_{g}^{q}(\Omega,\Omega')} \int_{\Omega^{\prime}}\! F( \nabla u) \, \dd x
$$
and for $u \in \mathscr{A}_{g}^{q}(\Omega,\Omega')$ we estimate (recall that $|F^{\qc}(z)|\leq cL(1+|z|^{q})$ for all $z\in\R^{N\times n}$)
\begin{eqnarray*}
  \int_{\Omega^{\prime}} \! F( \nabla u) \, \dd x &=& \int_{\Omega^\prime} \! F\bigl( \nabla (\rho u) \bigr) \, \dd x
  + \int_{\Omega^{\prime}\setminus \Omega} \! \bigl( F(\nabla g) - F ( \nabla (\rho g)) \bigr) \, \dd x\\
  &\geq& \mathscr{L}^{n}( \Omega^{\prime})F^{\qc}(0) + \int_{\Omega^{\prime}\setminus \Omega} \! \bigl( F(\nabla g) - F ( \nabla (\rho g)) \bigr) \, \dd x
\end{eqnarray*}
we conclude the proof of the first part of the lemma. Because the second part of the lemma is obvious from the above proof we
are done.
\end{proof}
\begin{proposition}\label{prop:AppendixMain} Fix an exponent $r \in [p,q]$. Then under the assumptions of Lemma~\ref{lem:AppendixLemma1}, the following are
equivalent:
\begin{enumerate}
\item\label{item:AppendixA1} For each $t \in \R$ the sublevel set $\bigl\{ u \in \lebe^{p}( \Omega^{\prime}; \R^N )\colon\; \overline{\mathscr{F}}_{g}^{\lebe^{p}}[u;\Omega,\Omega'] \leq t \bigr\}$ is 
contained and bounded in $\sobo^{1,r}( \Omega^{\prime};\R^{N})$ if $r>1$ and in $\bv( \Omega^{\prime};\R^N )$ if $r=1$.
\item\label{item:AppendixA3} There exist $c>0$, $z_{0} \in \R^{N\times n}$ such that $F-c \langle \cdot \rangle^{r}$ is quasiconvex at $z_0$.
\item\label{item:AppendixA2} There exist constants $c_{1}>0$, $c_{2} \in \R$ such that
$$
  \overline{\mathscr{F}}_{g}^{\lebe^{p}}[u;\Omega,\Omega'] \geq \left\{
\begin{array}{ll}
  c_{1}\| u \|_{\sobo^{1,r}(\Omega')}^{r} + c_{2} & \mbox{ if } r>1,\\
  c_{1}\| u \|_{\bv(\Omega')} + c_{2} & \mbox{ if } r=1
\end{array}
\right.
$$
holds for all $u \in \lebe^{p}( \Omega^{\prime};\R^N )$.
\end{enumerate}
\end{proposition}
In the display of~\ref{item:AppendixA2} we use the convention that the norm of $u$ is defined to be $+\infty$ when $u$ does not belong to the corresponding space.

\begin{proof}
Note that when $u \in \mathscr{A}_{g}^{q}(\Omega,\Omega')$, then $\overline{\mathscr{F}}_{g}^{\lebe^{p}}[u;\Omega,\Omega'] \leq \int_{\Omega'} \! F( \nabla u) \, \dd x$. Consequently, for all $t\in\R$ we have the inclusion
$$  
\bigl\{ u \in \mathscr{A}_{g}^{q}(\Omega,\Omega')\colon \, \int_{\Omega^{\prime}} \! F( \nabla u) \, \dd x \leq t \bigr\}\subset \bigl\{ u \in \mathscr{A}_{g}^{q}(\Omega,\Omega') \colon \, \overline{\mathscr{F}}_{g}^{\lebe^{p}}[u;\Omega,\Omega'] \leq t \bigr\}
$$
and so condition~\ref{item:AppendixA1} implies in particular that the sublevel sets on the left-hand side are bounded in $\sobo^{1,r}( \Omega';\R^N )$ if $r>1$ and in $\bv(\Omega';\R^{N})$ if $r=1$. We hereby conclude~\ref{item:AppendixA3} by Proposition~\ref{qc-coerciv}. In turn,~\ref{item:AppendixA3} implies that
$$
\int_{\Omega'} \! F( \nabla u) \, \dd x \geq c_{1}\int_{\Omega^\prime} \! | \nabla u|^{r} \, \dd x +c_2\qquad\text{for all}\;u\in\mathscr{A}_{g}^{q}(\Omega,\Omega') 
$$
by Proposition~\ref{qc-coerciv}. By lower semicontinuity of the $\sobo^{1,r}$- or $\bv$-norms for strong convergence of $\sobo^{1,q}$-maps in $\lebe^{r}(\Omega';\R^{N})$ and our above convention of e.g. $\|u\|_{\sobo^{1,r}(\Omega')}=\infty$ if $u\notin\sobo^{1,r}(\Omega';\R^{N})$, we conclude~\ref{item:AppendixA2}. Since~\ref{item:AppendixA2} obviously implies~\ref{item:AppendixA1}, the proof is complete. 
\end{proof}
\begin{corollary} 
Under the mutually equivalent conditions of the preceding proposition, we have 
\begin{align}\label{eq:equalityenvelopes}
\begin{split}
&\overline{\mathscr{F}}_{g}^{\lebe^{p}}[u;\Omega,\Omega']=\overline{\mathscr{F}}_{g}[u;\Omega,\Omega']\qquad\text{for all}\;u\in\sobo^{1,p}(\Omega';\R^{N}),\\
&\overline{\mathscr{F}}_{g}^{\lebe^{1}}[u;\Omega,\Omega']=\overline{\mathscr{F}}_{g}^{*}[u;\Omega,\Omega']\qquad\text{for all}\;u\in\bv(\Omega';\R^{N}).
\end{split}
\end{align}
\end{corollary}
\begin{proof} 
We focus on the case $p>1$, the case $p=1$ being analogous. We only need to establish '$\geq$' in~\eqref{eq:equalityenvelopes}, and then may assume that $\overline{\mathscr{F}}_{g}^{\lebe^{p}}[u;\Omega,\Omega']<\infty$ as otherwise there is nothing to prove. By Proposition~\ref{prop:AppendixMain}~\ref{item:AppendixA2}, we then have $\overline{\mathscr{F}}_{g}^{\lebe^{p}}[u;\Omega,\Omega']>-\infty$. Let $(u_{j})\subset\mathscr{A}_{g}^{q}(\Omega,\Omega')$ be such that $u_{j}\to u$ strongly in $\lebe^{p}(\Omega';\R^{N})$ and $\overline{\mathscr{F}}_{g}^{\lebe^{p}}[u;\Omega,\Omega']=\lim_{j\to\infty}\mathscr{F}_{g}^{\lebe^{p}}[u_{j};\Omega,\Omega']$. Again by Proposition~\ref{prop:AppendixMain}, $(u_{j})$ is bounded in $\sobo^{1,p}(\Omega';\R^{N})$. Hence, $(u_{j})$ possesses a subsequence $(u_{j_{k}})$ such that $u_{j_{k}}\rightharpoonup v$ weakly in $\sobo^{1,p}(\Omega';\R^{N})$ for some $v\in\sobo^{1,p}(\Omega';\R^{N})$. Since $u_{j}\to u$ strongly in $\lebe^{p}(\Omega';\R^{N})$, we have $u=v$ and thus 
\begin{align*}
\overline{\mathscr{F}}_{g}^{\lebe^{p}}[u;\Omega,\Omega'] = \liminf_{k\to\infty}\mathscr{F}_{g}^{\lebe^{p}}[u_{j_{k}};\Omega,\Omega'] \geq \overline{\mathscr{F}}_{g}[u;\Omega,\Omega']. 
\end{align*}
The proof is complete. 
\end{proof} 
We finally address a $\Gamma$-approximability result. Note that, if we choose $s>q$ in the following, this allows to $\Gamma$-approximate the relaxed functionals by multiple integrals with integrands of pointwise standard growth (i.e.,~\eqref{eq:prowth} with the exponent $s$), whereas if $F$ is quasiconvex and $s=q$, the integrand $F+\varepsilon\langle\cdot\rangle^{q}$ still has $q$-growth and is $q$-strongly quasiconvex.  
\begin{remark}
Fix an exponent $s \geq q$, assume $g \in \sobo^{1,s}( \R^n; \R^N )$ and define for $\varepsilon > 0$ the functional
$$
\mathscr{F}_{s,\varepsilon,g}^{\lebe^{p}}[u;\Omega,\Omega'] := \left\{
\begin{array}{ll}
  \displaystyle{\int_{\Omega^{\prime}} \! \biggl( F( \nabla u) + \varepsilon \langle \nabla u \rangle^{s} \biggr) \, \dd x} & \mbox{ if } u \in \mathscr{A}^{s}_{g}(\Omega,\Omega')\\
  +\infty & \mbox{ if } u \in \lebe^{p}( \Omega^{\prime};\R^N ) \setminus \mathscr{A}^{s}_{g}(\Omega,\Omega'),
\end{array}
\right.
$$
where we recall $\langle z \rangle = \sqrt{|z|^{2}+1}$. We then have that $\mathscr{F}_{s,\varepsilon,g}^{\lebe^{p}}[-;\Omega,\Omega']\to \overline{\mathscr{F}}_{g}^{\lebe^{p}}[-;\Omega,\Omega']$ in the sense of $\Gamma$-convergence
on $\lebe^{p}( \Omega^{\prime};\R^N )$ with the strong $\lebe^{p}$-topology as $\varepsilon \searrow 0$. 

Indeed, let $u_{\varepsilon} \to u$ in $\lebe^{p}( \Omega^{\prime};\R^N )$. We aim to show that $\liminf_{\varepsilon \searrow 0} \mathscr{F}_{s,\varepsilon,g}^{\lebe^{p}}[u_{\varepsilon};\Omega,\Omega'] \geq \overline{\mathscr{F}}_{g}^{\lebe^{p}}[u;\Omega,\Omega']$ 
and may without loss of generality assume that the left-hand side is smaller than $+\infty$. We then extract a subsequence (say $\varepsilon_{j} \searrow 0$) such that
$$
\mathscr{F}_{s,\varepsilon_{j},g}^{\lebe^{p}}[u_{\varepsilon_{j}};\Omega,\Omega'] \to \liminf_{\varepsilon \searrow 0} \mathscr{F}_{s,\varepsilon,g}^{\lebe^{p}}[u_{\varepsilon};\Omega,\Omega'] < +\infty
$$
From the definitions we then have $u_{\varepsilon_j} \in \mathscr{A}^{s}_{g}(\Omega,\Omega')$ and since $u_{\varepsilon_j} \to u$ in $\lebe^{p}( \Omega^{\prime};\R^N )$ that
$$
\lim_{j \to \infty} \int_{\Omega^{\prime}} \! \biggl( F( \nabla u_{\varepsilon_j}) + \varepsilon_{j} \langle \nabla u_{\varepsilon_j}\rangle^{s} \biggr) \, \dd x
\geq \liminf_{j\to\infty}\mathscr{F}_{g}^{\lebe^{p}}[u_{\varepsilon_{j}};\Omega,\Omega'] \geq \overline{\mathscr{F}}_{g}^{\lebe^{p}}[u;\Omega,\Omega']. 
$$
Next, we must show the limsup inequality, namely given $u \in \lebe^{p}( \Omega^{\prime}; \R^N )$ that we can find $( u_{\varepsilon} )_{\varepsilon > 0}$ in
$\lebe^{p}( \Omega^{\prime};\R^N )$ such that $u_{\varepsilon} \to u$ strongly in $\lebe^{p}( \Omega^{\prime};\R^N )$ together with
$$\limsup_{\varepsilon \searrow 0} \mathscr{F}_{s,\varepsilon,g}^{\lebe^{p}}[u_{\varepsilon};\Omega,\Omega'] \leq \overline{\mathscr{F}}_{g}^{\lebe^{p}}[u;\Omega,\Omega'].$$ Clearly we may assume that $\overline{\mathscr{F}}_{g}^{\lebe^{p}}[u;\Omega,\Omega']< +\infty$
and can then by definition find a sequence $( u_{j})$ in $\mathscr{A}_{g}^{s}(\Omega,\Omega')$ such that $u_{j} \to u$ in $\lebe^{p}( \Omega^{\prime};\R^N )$ and
$\int_{\Omega^{\prime}} \! F( \nabla u_{j}) \, \dd x \to \overline{\mathscr{F}}_{g}^{\lebe^{p}}[u;\Omega,\Omega']$. Take a descending null sequence $( \varepsilon_j )$ such that
$\varepsilon_{j}\int_{\Omega^{\prime}} \! \langle \nabla u_{j} \rangle^{s} \, \dd x \to 0$. Define $u_{\varepsilon} = u_{1}$ for $\varepsilon > \varepsilon_{1}$
and for $j > 1$, $u_{\varepsilon} = u_{j}$ for $\varepsilon \in ( \varepsilon_{j},\varepsilon_{j-1}]$. Then $u_{\varepsilon} \to u$ in $\lebe^{p}( \Omega^{\prime}; \R^N )$
and $ \mathscr{F}_{s,\varepsilon,g}^{\lebe^{p}}[u_{\varepsilon};\Omega,\Omega'] \to \overline{\mathscr{F}}_{g}^{\lebe^{p}}[u;\Omega,\Omega']$, and the statement follows. 
\end{remark}
\subsection{Proof of Lemma~\ref{lem:shiftconnect}}\label{sec:proofshiftconnect}
In this section, we provide the elementary 
\begin{proof}[Proof of Lemma~\ref{lem:shiftconnect}]
For any open and bounded Lipschitz domain $\omega'\subset\R^{n}$ and any map $w\in\sobo^{1,q}(\omega';\R^{N})$, we have with $\widetilde{w}:=w-a$
\begin{align}\label{eq:GaussAppendix}
\begin{split}
\int_{\omega'}F_{\nabla a}(\nabla\widetilde{w})\dif x  & = \int_{\omega'}F(\nabla w)-F(\nabla a)-\langle F'(\nabla a),\nabla (w-a)\rangle\dif x \\ 
&  =  \int_{\omega'}F(\nabla w)-F(\nabla a)\dif x \\ & + \int_{\partial\omega'}\langle F'(\nabla a),(a-\trace_{\partial\omega'}(w))\otimes\nu_{\partial\omega'}\rangle\dif\mathscr{H}^{n-1},
\end{split}
\end{align}
where $\nu_{\partial\omega'}$ is the outer unit normal to $\partial\omega'$. Identity~\ref{eq:GaussAppendix} implies that $(u_{j})\subset\mathscr{A}_{u_{0}}^{q}(\Omega,\Omega')$ is generating for $\overline{\mathscr{F}}_{u_{0}}^{*}[\mathbf{u};\Omega,\Omega']$ if and only if $(\widetilde{u}_{j})\subset\mathscr{A}_{\widetilde{u}_{0}}^{q}(\Omega,\Omega')$ is generating for $\overline{\mathscr{F}}_{\nabla a,\widetilde{u}_{0}}^{*}[\widetilde{\mathbf{u}};\Omega,\Omega']$ with $\widetilde{\mathbf{u}}:=\mathbbm{1}_{\Omega}\widetilde{u}+\mathbbm{1}_{\Omega'\setminus\overline{\Omega}}\widetilde{u}_{0}$. Specifically, $\overline{\mathscr{F}}_{v}^{*}[u;\Omega]<\infty$ implies that $\overline{\mathscr{F}}_{\nabla a,\widetilde{v}}^{*}[\widetilde{u};\Omega]<\infty$, and by Remark~\ref{rem:Columbo} we then infer Lemma~\ref{lem:shiftconnect}~\ref{item:shiftconnect1}: Let $x_{0}\in\Omega$ and $r>0$ are such that $\ball_{r}(x_{0})\Subset\Omega$ and $\mathcal{M}Du(x_{0},r)+\mathcal{M}\lambda(x_{0},r)<\infty$, where $\lambda\in\mathrm{RM}_{\mathrm{fin}}(\Omega')$ is the weak*-limit of a suitable subsequence $(|Du_{j_{k}}|)$. From above we deduce that $(\widetilde{u}_{j_{k}})$ is generating for $\overline{\mathscr{F}}_{\nabla a,\widetilde{u}_{0}}^{*}[\widetilde{\mathbf{u}};\Omega,\Omega']$, and $\mathcal{M}D\widetilde{u}(x_{0},r)<\infty$ holds trivially. Let $\mu\in\mathrm{RM}_{\mathrm{fin}}(\Omega')$ be a weak*-limit of a subsequence $(|D\widetilde{u}_{j_{k}(l)}|)$ of $(|D\widetilde{u}_{j_{k}}|)$. If $\varepsilon>0$ is sufficiently small such that $\ball_{r+2\varepsilon}(x_{0})\Subset\Omega$, we then have 
\begin{align*}
\int_{\ball_{r+\varepsilon}(x_{0})\setminus\overline{\ball}_{r-\varepsilon}(x_{0})}\dif\mu & \leq \liminf_{l\to\infty}\int_{\ball_{r+\varepsilon}(x_{0})\setminus\overline{\ball}_{r}(x_{0})}\dif|D\widetilde{u}_{j_{k}(l)}| \\ 
& \leq \limsup_{l\to\infty}\int_{\overline{\ball}_{r+2\varepsilon}(x_{0})\setminus\ball_{r-2\varepsilon}(x_{0})}\dif\,(|Du_{j_{k}(l)}|+|\nabla a|\mathscr{L}^{n})\\
& \leq |\lambda|(\overline{\ball}_{r+2\varepsilon}\setminus\ball_{r-2\varepsilon})+ \mathscr{L}^{n}(\overline{\ball}_{r+2\varepsilon}(x_{0})\setminus\ball_{r-2\varepsilon}(x_{0}))|\nabla a|\\ 
& \leq 4\varepsilon\mathcal{M}\lambda(x_{0},r) + \mathscr{L}^{n}(\overline{\ball}_{r+2\varepsilon}(x_{0})\setminus\ball_{r-2\varepsilon}(x_{0}))|\nabla a|
\end{align*}
for $\overline{\ball}_{r+2\varepsilon}(x_{0})\setminus\ball_{r-2\varepsilon}(x_{0}))$ is compact. Dividing the previous inequality by $2\varepsilon$, we obtain $\mathcal{M}D\widetilde{u}(x_{0},r)+\mathcal{M}\mu(x_{0},r)<\infty$, and from here Remark~\ref{rem:Columbo} implies~\ref{item:shiftconnect1}. Based on~\eqref{eq:GaussAppendix}, Lemma~\ref{lem:shiftconnect}~\ref{item:shiftconnect2} then follows by similar means. To see~Lemma~\ref{lem:shiftconnect}~\ref{item:shiftconnect3}, $\ell_{m}>0$ be the constant from~\ref{item:H2}. Using that $|\nabla a|\leq m$,~\ref{item:H2} then yields with $\ell^{(m)}:=\ell_{m}/(16(1+m^{2})^{3/2})$ for all test maps $\varphi\in\hold_{c}^{\infty}(\ball_{s}(x_{0});\R^{N})$
\begin{align*}
\ell^{(m)}\int_{\ball_{s}(x_{0})}V(\nabla\varphi)\dif x & \stackrel{\text{Lem.~\ref{lem:Efunction}~\ref{item:VpCompa1}}}{\leq} \ell_{m}\int_{\ball_{s}(x_{0})}V(\nabla a + \nabla\varphi)-V(\nabla a)-\langle V'(\nabla a),\nabla\varphi\rangle\dif x  \\ 
& \;\;\;\;\;= \ell_{m}\int_{\ball_{s}(x_{0})}V(\nabla a + \nabla\varphi)-V(\nabla a)\dif x \\ & \;\;\;\; \stackrel{\text{\ref{item:H2}}}{\leq} \int_{\ball_{s}(x_{0})}F(\nabla a +\nabla\varphi)-F(\nabla a)\dif x = \int_{\ball_{s}(x_{0})}F_{\nabla a}(\nabla\varphi)\dif x. 
\end{align*}
The case of general maps $\varphi\in\sobo_{0}^{1,q}(\ball_{s}(x_{0});\R^{N})$ then follows from quasiconvexity and \ref{item:H1}, hereafter~\eqref{eq:lipschitz}, and smooth approximation. The proof is complete. 
\end{proof} 
\begin{figure}
\begin{center}
\begin{tikzpicture}[scale = 0.5,rotate=30]
\node at (13.5,0.8) {{\Huge $\Omega$}};
\path [fill=black!15!white, fill opacity = .5] (2,0) to [out=270,in=150] (3,-2.5) to [out=-10,in=15] (5,-2) to [out=10, in =160]  (10,-2.5) to [out=-20,in=280] (12.5,2) to [out=120, in=0] (8,1) to [out=180, in =0] (4,3) to [out=170, in = 90] (2,0);
\draw[-, blue!40!white, fill =black!40!white , thick] (4,-0.02) -- (5,1) -- (6,-0.02) -- (4,-0.02);
\draw [-] (13,1) -- (11.75,1.4);
\draw[-,blue!30!white, dotted, ultra thick] (2,4) -- (2.5,3.5);
\draw[-,blue!30!white, ultra thick] (3,3) -- (2.5,3.5);
\draw[-,blue!30!white,ultra thick] (3,3) -- (4,4);
\draw[-,blue!30!white, dotted, ultra thick] (4,4) -- (4.5,4.5);
\draw[-,blue!30!white, dotted, ultra thick] (4,4) -- (3.5,4.5);
\draw[-,blue!30!white, dotted, ultra thick] (6,4) -- (5.5,4.5);
\draw[-,blue!30!white, dotted, ultra thick] (6,4) -- (6.5,4.5);
\draw[-,blue!30!white, dotted, ultra thick] (8,4) -- (7.5,4.5);
\draw[-,blue!30!white, dotted, ultra thick] (8,4) -- (8.5,4.5);
\draw[-,blue!30!white, dotted, ultra thick] (10,4) -- (9.5,4.5);
\draw[-,blue!30!white, dotted, ultra thick] (10,4) -- (10.5,4.5);
\draw[-,blue!30!white,ultra thick] (4,4) -- (5,3);
\draw[-,blue!30!white,ultra thick] (5,3) -- (6,4);
\draw[-,blue!30!white,ultra thick] (6,4) -- (7,3);
\draw[-,blue!30!white,ultra thick] (7,3) -- (8,4);
\draw[-,blue!30!white,ultra thick] (8,4) -- (9,3);
\draw[-,blue!30!white,ultra thick] (9,3) -- (10,4);
\draw[-,blue!30!white,ultra thick] (10,4) -- (11,3);
\draw[-,blue!30!white,ultra thick] (11,3) -- (11.5,3.5);
\draw[-,blue!30!white, dotted, ultra thick] (11.5,3.5) -- (12,4);
\draw[-,blue!30!white,ultra thick] (3,3) -- (11,3);
\draw[-,blue!30!white,dotted, ultra thick] (3,3) -- (2,3);
\draw[-,blue!30!white, dotted, ultra thick] (11,3) -- (12,3);
\draw[-,blue!30!white,dotted, ultra thick] (1.5,2.5) -- (2,2);
\draw[-,blue!30!white,dotted, ultra thick] (1.5,1.5) -- (2,2);
\draw[-,blue!30!white,dotted, ultra thick] (1.5,0.5) -- (2,0);
\draw[-,blue!30!white,dotted, ultra thick] (1.5,-0.5) -- (2,0);
\draw[-,blue!30!white,dotted, ultra thick] (1.5,-1.5) -- (2,-2);
\draw[-,blue!30!white,dotted, ultra thick] (1.5,-2.5) -- (2,-2);
\draw[-,blue!30!white,ultra thick] (3,-3) -- (2,-2);
\draw[-,blue!30!white,dotted, ultra thick] (3,-3) -- (3.5,-3.5);
\draw[-,blue!30!white,dotted, ultra thick] (3,-3) -- (2.5,-3.5);
\draw[-,blue!30!white,dotted, ultra thick] (5,-3) -- (5.5,-3.5);
\draw[-,blue!30!white,dotted, ultra thick] (5,-3) -- (4.5,-3.5);
\draw[-,blue!30!white,dotted, ultra thick] (7,-3) -- (7.5,-3.5);
\draw[-,blue!30!white,dotted, ultra thick] (7,-3) -- (6.5,-3.5);
\draw[-,blue!30!white,dotted, ultra thick] (9,-3) -- (9.5,-3.5);
\draw[-,blue!30!white,dotted, ultra thick] (9,-3) -- (8.5,-3.5);
\draw[-,blue!30!white,dotted, ultra thick] (11,-3) -- (11.5,-3.5);
\draw[-,blue!30!white,dotted, ultra thick] (11,-3) -- (10.5,-3.5);
\draw[-,blue!30!white, ultra thick] (3,-3) -- (11,-3);
\draw[-,blue!30!white, ultra thick, dotted] (3,-3) -- (2,-3);
\draw[-,blue!30!white, ultra thick, dotted] (11,-3) -- (12,-3);
\draw[-,blue!30!white, ultra thick, dotted] (12,-2) -- (12.5,-1.5);
\draw[-,blue!30!white, ultra thick, dotted] (12,-2) -- (12.5,-2.5);
\draw[-,blue!30!white, ultra thick, dotted] (12,0) -- (12.5,0.5);
\draw[-,blue!30!white, ultra thick, dotted] (12,0) -- (12.5,-0.5);
\draw[-,blue!30!white, ultra thick, dotted] (12,2) -- (12.5,2.5);
\draw[-,blue!30!white, ultra thick, dotted] (12,2) -- (12.5,1.5);
\draw[-,blue!30!white, ultra thick, dotted] (12,1) -- (12.5,1);
\draw[-,blue!30!white, ultra thick, dotted] (2,1) -- (1.5,1);
\draw[-,blue!30!white, ultra thick, dotted] (12,-1) -- (12.5,-1);
\draw[-,blue!30!white, ultra thick, dotted] (2,-1) -- (1.5,-1);
\draw[-,blue!30!white, ultra thick] (3,-1) -- (2,0);
\draw[-,blue!30!white,ultra thick] (2,2) -- (3,1);
\draw[-,blue!30!white,ultra thick] (2,2) -- (3,3);
\draw[-,blue!30!white,ultra thick] (3,3) -- (4,2);
\draw[-,blue!30!white,ultra thick] (4,2) -- (5,3);
\draw[-,blue!30!white,ultra thick] (5,3) -- (6,2);
\draw[-,blue!30!white,ultra thick] (6,2) -- (7,3);
\draw[-,blue!30!white,ultra thick] (7,3) -- (8,2);
\draw[-,blue!30!white,ultra thick] (8,2) -- (9,3);
\draw[-,blue!30!white,ultra thick] (9,3) -- (10,2);
\draw[-,blue!30!white,ultra thick] (10,2) -- (11,3);
\draw[-,blue!30!white,ultra thick] (11,3) -- (12,2);
\draw[-,blue!30!white,ultra thick] (2,2) -- (12,2);
\draw[-,blue!30!white,ultra thick] (3,1) -- (4,2);
\draw[-,blue!30!white,ultra thick] (4,2) -- (5,1);
\draw[-,blue!30!white,ultra thick] (5,1) -- (6,2);
\draw[-,blue!30!white,ultra thick] (6,2) -- (7,1);
\draw[-,blue!30!white,ultra thick] (7,1) -- (8,2);
\draw[-,blue!30!white,ultra thick] (8,2) -- (9,1);
\draw[-,blue!30!white,ultra thick] (9,1) -- (10,2);
\draw[-,blue!30!white,ultra thick] (10,2) -- (11,1);
\draw[-,blue!30!white,ultra thick] (11,1) -- (12,2);
\draw[-,blue!30!white,ultra thick] (2,1) -- (12,1);
\draw[-,blue!30!white,ultra thick] (2,0) -- (3,1);
\draw[-,blue!30!white,ultra thick] (3,1) -- (4,0);
\draw[-,blue!30!white,ultra thick] (4,0) -- (5,1);
\draw[-,blue!30!white,ultra thick] (5,1) -- (6,0);
\draw[-,blue!30!white,ultra thick] (6,0) -- (7,1);
\draw[-,blue!30!white,ultra thick] (7,1) -- (8,0);
\draw[-,blue!30!white,ultra thick] (8,0) -- (9,1);
\draw[-,blue!30!white,ultra thick] (9,1) -- (10,0);
\draw[-,blue!30!white,ultra thick] (10,0) -- (11,1);
\draw[-,blue!30!white,ultra thick] (11,1) -- (12,0);
\draw[-,blue!30!white,ultra thick] (2,0) -- (12,0);
\draw[-,blue!30!white,ultra thick] (3,-1) -- (4,0);
\draw[-,blue!30!white,ultra thick] (4,0) -- (5,-1);
\draw[-,blue!30!white,ultra thick] (5,-1) -- (6,0);
\draw[-,blue!30!white,ultra thick] (6,0) -- (7,-1);
\draw[-,blue!30!white,ultra thick] (7,-1) -- (8,0);
\draw[-,blue!30!white,ultra thick] (8,0) -- (9,-1);
\draw[-,blue!30!white,ultra thick] (9,-1) -- (10,0);
\draw[-,blue!30!white,ultra thick] (10,0) -- (11,-1);
\draw[-,blue!30!white,ultra thick] (11,-1) -- (12,0);
\draw[-,blue!30!white,ultra thick] (2,-1) -- (12,-1);
\draw[-,blue!30!white,ultra thick] (2,-2) -- (3,-1);
\draw[-,blue!30!white,ultra thick] (3,-1) -- (4,-2);
\draw[-,blue!30!white,ultra thick] (4,-2) -- (5,-1);
\draw[-,blue!30!white,ultra thick] (5,-1) -- (6,-2);
\draw[-,blue!30!white,ultra thick] (6,-2) -- (7,-1);
\draw[-,blue!30!white,ultra thick] (7,-1) -- (8,-2);
\draw[-,blue!30!white,ultra thick] (8,-2) -- (9,-1);
\draw[-,blue!30!white,ultra thick] (9,-1) -- (10,-2);
\draw[-,blue!30!white,ultra thick] (10,-2) -- (11,-1);
\draw[-,blue!30!white,ultra thick] (11,-1) -- (12,-2);
\draw[-,blue!30!white,ultra thick] (2,-2) -- (12,-2);
\draw[-,blue!30!white,ultra thick] (3,-3) -- (4,-2);
\draw[-,blue!30!white,ultra thick] (4,-2) -- (5,-3);
\draw[-,blue!30!white,ultra thick] (5,-3) -- (6,-2);
\draw[-,blue!30!white,ultra thick] (6,-2) -- (7,-3);
\draw[-,blue!30!white,ultra thick] (7,-3) -- (8,-2);
\draw[-,blue!30!white,ultra thick] (8,-2) -- (9,-3);
\draw[-,blue!30!white,ultra thick] (9,-3) -- (10,-2);
\draw[-,blue!30!white,ultra thick] (10,-2) -- (11,-3);
\draw[-,blue!30!white,ultra thick] (11,-3) -- (12,-2);
\node at (8,-6) {\text{(a)}};
\node at (5.5,-4.3) {\text{$\triangle$}};
\draw[-] (5.5,-3.8) -- (5.1,0.4);
\end{tikzpicture}
\hspace{0.7cm}
\begin{tikzpicture}[scale = 1.5, rotate=35]
\draw[-, black,dotted, fill=black!10!white] (3.5,-0.21) -- (5,1.25) -- (6.5,-0.21) -- (3.5,-0.21);
\node at (4.9,-0.75) {$\triangle$};
\draw[-] (5,-0.6) -- (5.5,0.35);
\draw[->] (5,1.25) -- (5,2);
\node at (5.4,2) {$e\in\mathbb{S}^{n-1}$};
\draw[rounded corners=1.8mm,orange!60!white,  fill=orange!60!white, opacity=0.2] (3.4,-0.21) -- (5,1.35) -- (5.1,1.25) -- (3.5,-0.31) -- cycle;
\draw[rounded corners=1.8mm,orange!60!white, fill=orange!60!white, opacity=0.2] (3.4,-0.10) -- (5,1.46) -- (5.1,1.36) -- (3.5,-0.20) -- cycle;
\draw[rounded corners=1.8mm,orange!60!white, fill=orange!60!white, opacity=0.2] (3.4,0.01) -- (5,1.57) -- (5.1,1.47) -- (3.5,-0.09) -- cycle;
\draw[rounded corners=1.8mm,orange!60!white, fill=orange!60!white, opacity=0.2] (3.4,0.12) -- (5,1.68) -- (5.1,1.58) -- (3.5,0.02) -- cycle;
\draw[rounded corners=1.8mm,orange!60!white, fill=orange!60!white, opacity=0.2] (3.4,0.23) -- (5,1.79) -- (5.1,1.69) -- (3.5,0.13) -- cycle;
\draw[rounded corners=1.8mm,red!60!white,  fill=red!60!white, opacity=0.2] (3.45,-0.28) -- (6.55,-0.28) -- (6.55,-0.14) -- (3.45,-0.14) -- cycle;
\draw[rounded corners=1.8mm,red!60!white,  fill=red!60!white, opacity=0.2] (3.45,-0.17) -- (6.55,-0.17) -- (6.55,-0.03) -- (3.45,-0.03) -- cycle;
\draw[rounded corners=1.8mm,red!60!white,  fill=red!60!white, opacity=0.2] (3.45,-0.06) -- (6.55,-0.06) -- (6.55,0.08) -- (3.45,0.08) -- cycle;
\draw[rounded corners=1.8mm,red!60!white,  fill=red!60!white, opacity=0.2] (3.45,0.05) -- (6.55,0.05) -- (6.55,0.19) -- (3.45,0.19) -- cycle;
\draw[rounded corners=1.8mm,red!60!white,  fill=red!60!white, opacity=0.2] (3.45,0.16) -- (6.55,0.16) -- (6.55,0.3) -- (3.45,0.3) -- cycle;
\draw[rounded corners=1.8mm,yellow!60!white,  fill=yellow!60!white, opacity=0.2] (6.47,-0.28) -- (6.58,-0.20) -- (5.025,1.33) -- (4.92,1.24) -- cycle;
\draw[rounded corners=1.8mm,yellow!60!white,  fill=yellow!60!white, opacity=0.2] (6.47,-0.17) -- (6.58,-0.09) -- (5.025,1.44) -- (4.92,1.35) -- cycle;
\draw[rounded corners=1.8mm,yellow!60!white,  fill=yellow!60!white, opacity=0.2] (6.47,-0.06) -- (6.58,0.02) -- (5.025,1.55) -- (4.92,1.46) -- cycle;
\draw[rounded corners=1.8mm,yellow!60!white,  fill=yellow!60!white, opacity=0.2] (6.47,0.05) -- (6.58,0.13) -- (5.025,1.66) -- (4.92,1.57) -- cycle;
\draw[rounded corners=1.8mm,yellow!60!white,  fill=yellow!60!white, opacity=0.2] (6.47,0.16) -- (6.58,0.24) -- (5.025,1.77) -- (4.92,1.68) -- cycle;
\draw[<->] (6.6,-0.28) -- (6.6,0.27);
\node at (6.8,0) {$\delta$};
\node at (4,-0.8) {\text{(b)}};
\end{tikzpicture}
\caption{The grid-covering argument underlying the proof of Proposition~\ref{lem:Orlicz2}. For a triangulation as depicted in (a), the triangles $\triangle$ are shifted by a certain multiple of a vector not tangent to the faces of $\triangle$, see (b). Since $e$ is non-tangent to the faces of any $\triangle$, one may run the usual argument in the proof of the Hardy-Littlewood maximal theorem to conclude that $\mathcal{M}_{\triangle}(\mu)<\infty$ $\mathscr{L}^{1}$-a.e.; note that, if $e$ were tangent to a face of $\triangle$, the usual disjointness provided by Vitali's covering lemma would become vacuous and a strictly positive integral could potentially be counted infinitely often. In this case, the requisite finiteness would not follow anymore.}
\label{fig:grid}
\end{center}
\end{figure}
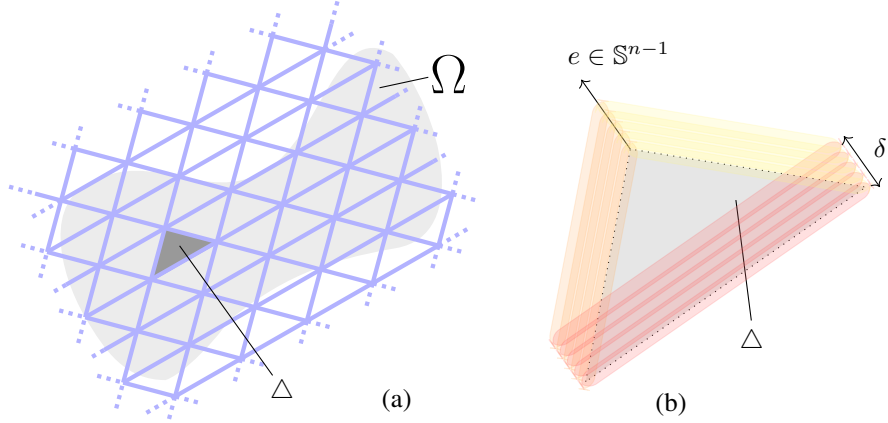
\subsection{Proof of Proposition~\ref{lem:Orlicz2}}\label{sec:OrliczSec}
In this concluding section, we provide the proof of the lower semicontinuity result from Proposition~\ref{lem:Orlicz2} for which we adapt the approach by \textsc{Chen} and the second author, cf.~\cite[\S 5]{CK}. Related results are for instance due to \textsc{Focardi}~\cite{Focardi}, however, work subject to different hypotheses than ours.  
\begin{proof}[Proof of Proposition~\ref{lem:Orlicz2}]
Let $u,u_{1},u_{2},...\in\sobo_{u_{0}}^{1,\Phi}(\Omega;\R^{N})$ as in Proposition~\ref{lem:Orlicz2}, where it is no loss of generality to assume  
\begin{align}\label{eq:subsequenceOrlicz}
\liminf_{j\to\infty}\int_{\Omega}F(\nabla u_{j})\dif x = \lim_{j\to\infty}\int_{\Omega}F(\nabla u_{j})\dif x
\end{align}
by passing to a subsequence if necessary. Moreover, we may suppose that the extended Dirichlet datum satisfies $u_{0}\in(\sobo^{1,1}\cap\sobo^{1,\Phi})(\R^{n};\R^{N})$ and so, tacitly thinking of $u,u_{j}$ to be extended by $u_{0}$ to $\R^{n}$, we may assume that that $u,u_{j}\in\sobo^{1,\Phi}(\R^{n};\R^{N})$. Given $\varepsilon>0$, we then find $\widetilde{u}\in(\hold^{\infty}\cap\sobo^{1,\Phi})(\R^{n};\R^{N})$ such that $\|u-\widetilde{u}\|_{\sobo^{1,\Phi}(\R^{n})}<\frac{\varepsilon}{2}$. Let $\mathcal{T}$ be a triangulation of $\R^{n}$ which is uniform and regular, and denote $\mathcal{T}_{\Omega}$ the subcollection of triangles $\triangle\in\mathcal{T}$ with $\triangle\cap\Omega\neq\emptyset$. Then there exists $0<\delta<1$ such that, if $\mathrm{diam}(\triangle)<\delta$ for all $\triangle\in\mathcal{T}$ and $a_{\mathcal{T}}\colon\R^{n}\to\R^{N}$ is the piecewise affine-linear map arising as the Lagrange interpolation of $\widetilde{u}$ on the nodes of the triangle $\triangle\in\mathcal{T}$, we have $\|\widetilde{u}-a_{\mathcal{T}}\|_{\sobo^{1,\Phi}(\R^{n})}<\frac{\varepsilon}{2}$. We then set $a_{\triangle}:=a_{\mathcal{T}}|_{\triangle}$. As $\mathcal{T}$ is assumed uniform and regular, we find a unit vector $e\in\mathbb{S}^{n-1}$ not being tangent to any of the faces of any triangle $\triangle\in\mathcal{T}$; this is so because $\mathcal{T}$ being uniform and regular forces all the triangles $\triangle\in\mathcal{T}$ to be congruent to a finite number of simplices. For non-negative Radon measures $\mu$ on $\R^{n}$, we define a maximal-type function 
\begin{align*}
\mathcal{M}_{\triangle}(\mu)(t):=\sup_{\substack{0<\varepsilon<t\\ t+\varepsilon\leq\delta}}\frac{1}{2\varepsilon}\int_{te+\partial\triangle+\ball_{\varepsilon}(0)}\dif\mu,\qquad t\in [0,\delta].
\end{align*}
Since $e$ is not tangent to the faces of $\triangle$, we conclude that for $\mathscr{L}^{1}$-a.e. $t\in(0,\delta)$ we have that $\mathcal{M}_{\triangle}(\mu)(t)<\infty$; see Figure~\ref{fig:grid}. 

In the following, we write $\Phi=\Phi(|\cdot|)$ with slight abuse of notation. Passing to another subsequence if necessary, we may assume that $\Phi(u_{j})\mathscr{L}^{n}+\Phi(\nabla u_{j})\mathscr{L}^{n}\stackrel{*}{\rightharpoonup} \lambda\in\mathrm{RM}_{\mathrm{fin}}(\R^{n})$. Since $\mathcal{T}_{\Omega}$ consists of finitely many simplices, the set of all 
\begin{align}\label{eq:spick}
s\in \bigcap_{\triangle\in\mathcal{T}_{\Omega}}\{t\in (0,\delta)\colon\;\mathcal{M}_{\triangle}(\lambda)(t)<\infty\;\text{and}\;\mathcal{M}_{\triangle}(\Phi(u)\mathscr{L}^{n}+\Phi(\nabla u)\mathscr{L}^{n})(t)<\infty\}
\end{align}
still has full $\mathscr{L}^{1}$-measure in $(0,\delta)$. Especially, we may pick an $s$ such that~\eqref{eq:spick} holds and, with $\triangle':=se+\triangle$, $\mathcal{T}'_{\Omega}:=\{\triangle'\colon\;\triangle\in\mathcal{T}_{\Omega}\}$ is still a uniform and regular triangulation of a neighbourhood of $\Omega$.

Now let $\tau>1$. Denoting the center of $\triangle'$ by $x_{\triangle'}$, we set $\tau\triangle':=x_{\triangle'}+\tau(\triangle'-x_{\triangle'})$.  We then let $\mathscr{J}_{\triangle'}\colon\sobo^{1,\Phi}(\triangle';\R^{N})\to\sobo^{1,\Phi}(2\triangle';\R^{N})$ be a bounded linear extension operator satisfying 
\begin{align}\label{eq:OrliczExtension}
\int_{2\triangle'}\Phi(\mathscr{J}_{\triangle'}v)+\Phi(\nabla\mathscr{J}_{\triangle'}v)\dif x \leq c \int_{\triangle'}\Phi(v)+\Phi(\nabla v)\dif x,\; v\in\sobo^{1,\Phi}(\triangle';\R^{N})
\end{align}
with a constant $c=c(\triangle',n,N,\Phi)>0$. Since $\Phi$ is of class $\Delta_{2}$ and $\mathcal{T}'_{\Omega}$ is still uniform and regular, such an operator can be obtained by routine means for the unit simplex $\triangle^{1}:=\mathrm{co}(\{0,e_{1},...,e_{n}\})$ and then employing an affine change of variables. We set, with the trace-preserving operator $\widetilde{\mathbb{E}}$ from Section~\ref{sec:Fubini}, 
\begin{align*}
w:=\mathscr{J}_{\triangle'}\widetilde{\mathbb{E}}_{\triangle'}(u)\;\;\;\text{and}\;\;\;w_{j}:=
\mathscr{J}_{\triangle'}\widetilde{\mathbb{E}}_{\triangle'}(u_{j})
\end{align*}
and claim that 
\begin{align}\label{eq:extensionOrlicz}
\int_{2\triangle'}\Phi(w-w_{j})+\Phi(\nabla w-\nabla w_{j})\dif x\to 0,\;\;\;j\to\infty.
\end{align}
For future reference, we then remark that an argument analogous to that underlying proof of Lemma~\ref{lem:extensionoperator}~\ref{item:extension2} (in fact, a combination of Lemma~\ref{lem:extensionoperator}~\ref{item:extension1} with Jensen's inequality) yields 
\begin{align}\label{eq:Orliczgradstab}
\int_{\triangle'}\Phi(\nabla^{l}\widetilde{\mathbb{E}}_{\triangle'}v)\dif x \leq c\int_{\triangle'}\Phi(\nabla^{l}v)\dif x\qquad\text{for all}\;v\in\sobo^{1,\Phi}(\triangle';\R^{N}),\;l\in\{0,1\}.
\end{align}
In view of~\eqref{eq:OrliczExtension}, it then suffices to establish~\eqref{eq:extensionOrlicz} with the domain of integration changed to $\triangle'$. To this aim, let $\varepsilon'>0$ be given. Denote $(\ball^{i})$ the Whitney cover of $\triangle'$ as displayed in Section~\ref{sec:Fubini}, cf.~\ref{item:Whitney1}--\ref{item:Whitney3} and pick a partition of unity $(\rho_{i})$ subject to $(\ball^{i})$ satisfying~\ref{item:Whitney4}--\ref{item:Whitney6}. For each $i\in\mathbb{N}$, the definition of the averaged Taylor polynomials (cf.~\eqref{eq:ATP0}) implies
\begin{align}\label{eq:stottle}
\nabla \Pi_{\ball^{i}}^{1}u_{j}\to \nabla\Pi_{\ball^{i}}^{1}u\;\;\;\text{in}\;\R^{N\times n}\;\text{as}\;j\to\infty
\end{align}
because of $u_{j}\to u$ strongly in $\lebe^{1}(\R^{n};\R^{N})$. From Section~\ref{sec:Fubini}, we recall the notation $U_{\kappa}^{\complement}:=\{x\in U\colon\;\mathrm{dist}(x,\partial U)\leq\kappa\}$. Using a similar argument as for Lemma~\ref{lem:extensionoperator}~\ref{item:extension1}--\ref{item:extension3a} in conjunction with the convexity and the $\Delta_{2}$-property of $\Phi$, we conclude that there exist constants $C_{1},C_{2}>1$ (solely depending on $\Phi$ and $n$) such that, with $\vartheta>1$ fixed and close to $1$ and $\kappa>0$ sufficiently small
\begin{align*}
\limsup_{j\to\infty}\mathrm{I}_{j}^{\triangle',\kappa} & := \limsup_{j\to\infty}\int_{{\triangle'}{_{\kappa}^{\complement}}} \Phi(\widetilde{\mathbb{E}}_{\triangle'}u_{j})+\Phi(\nabla\widetilde{\mathbb{E}}_{\triangle'} u_{j})+\Phi(\widetilde{\mathbb{E}}_{\triangle'}u)+\Phi(\nabla\widetilde{\mathbb{E}}_{\triangle'} u)\dif x  \\ 
& \leq C_{1} \limsup_{j\to\infty}\int_{{\triangle'}{_{C_{2}\times\kappa}^{\complement}}}\Phi(u_{j})+\Phi(\nabla u_{j})+\Phi(u)+\Phi(\nabla u)\dif x \\ 
& \leq C_{1}\int_{\partial\triangle'+\ball_{C_{2}\times\kappa\times\vartheta}(0)}\dif\lambda + C_{1}\int_{\partial\triangle'+\ball_{C_{2}\times\kappa\times\vartheta}(0)}\Phi(u)+\Phi(\nabla u)\dif x \\ 
& \leq 2C_{1}C_{2}\times\kappa\times\vartheta(\mathcal{M}_{\triangle}(\lambda)(s)+\mathcal{M}_{\triangle}(\Phi(u)\mathscr{L}^{n}+\Phi(\nabla u)\mathscr{L}^{n})(s)), 
\end{align*}
and because of~\eqref{eq:spick}, there exist $\kappa=\kappa(\triangle)>0$ and $j_{0}=j_{0}(\triangle)$ such that $j\geq j_{0}$ implies 
\begin{align}\label{eq:monkfix}
\mathrm{I}_{j}^{\triangle',\kappa}<\frac{\varepsilon'}{\Delta_{2}(\Phi)}.
\end{align}
Fix such $\kappa$ and $j_{0}$. We then pick the minimal index set $\mathcal{I}\subset\mathbb{N}_{0}$ such that $\{x\in\triangle'\colon\sum_{i\in\mathcal{I}}\rho_{i}(x)=1\}\cup{\triangle'}{_{\kappa}^{\complement}}=\triangle'$; in particular, $\mathcal{I}$ is finite. On the one hand, we have by the local uniformly finite overlap of the balls $\ball^{i}$ and the $\Delta_{2}$-property of $\Phi$
\begin{align}\label{eq:weaktostrong1}
\begin{split}
\int_{\triangle'}\Phi\Big(\sum_{i\in\mathcal{I}}\rho_{i}(\Pi_{\ball^{i}}^{1}u-\Pi_{\ball^{i}}^{1}u_{j})\Big)\dif x & \leq c\sum_{i\in\mathcal{I}}\int_{\ball^{i}}\Phi(\Pi_{\ball^{i}}^{1}u-\Pi_{\ball^{i}}^{1}u_{j})\dif x  \\ &  \leq c\sum_{i\in\mathcal{I}}\int_{\ball^{i}}\Phi(\|\Pi_{\ball^{i}}^{1}u-\Pi_{\ball^{i}}^{1}u_{j}\|_{\lebe^{\infty}(\ball^{i})})\dif x \\ 
& \!\!\!\!\stackrel{\eqref{eq:ATP}_{1}}{\leq} c\sum_{i\in\mathcal{I}}r(\ball^{i})^{n}\Phi\Big(\dashint_{\ball^{i}}|u-u_{j}|\dif x\Big)\to 0
\end{split}
\end{align}
as $j\to\infty$, since $u_{j}\to u$ strongly in $\lebe^{1}(\Omega;\R^{N})$ by assumption and $\mathcal{I}$ is finite. On the other hand, by virtue of $\Phi$ being of class $\Delta_{2}$, 
\begin{align*}
\int_{\triangle'}\Phi\Big(\nabla\sum_{i\in\mathcal{I}}\rho_{i}(\Pi_{\ball^{i}}^{1}u-\Pi_{\ball^{i}}^{1}u_{j})\Big)\dif x & \leq c \int_{\triangle'}\Phi\Big(\sum_{i\in\mathcal{I}}\rho_{i}(\nabla (\Pi_{\ball^{i}}^{1}u-\Pi_{\ball^{i}}^{1}u_{j}))\Big)\dif x \\ 
& \!\!\!\!\!\!\!\! + c \int_{\triangle'}\Phi\Big(\sum_{i\in\mathcal{I}}(\Pi_{\ball^{i}}^{1}u-\Pi_{\ball^{i}}^{1}u_{j})\otimes\nabla\rho_{i}\Big)\dif x  =: \mathrm{II}_{j} + \mathrm{III}_{j}. 
\end{align*}
Since $\mathcal{I}$ is finite, we have $\lim_{j\to\infty}\mathrm{II}_{j}=0$ because of~\eqref{eq:stottle}, and $\lim_{j\to\infty}\mathrm{III}_{j}=0$ by an argument similar to that underlying~\eqref{eq:weaktostrong1}. Specifically, there exists $j_{1}\in\mathbb{N}$ such that
\begin{align}\label{eq:NTeeger}
\sum_{l\in\{0,1\}}\int_{\triangle'}\Phi\Big(\nabla^{l}\sum_{i\in\mathcal{I}}\rho_{i}(\Pi_{\ball^{i}}^{1}u-\Pi_{\ball^{i}}^{1}u_{j})\Big)\dif x < \frac{\varepsilon'}{2}\qquad\text{for all}\;j\geq j_{1}.
\end{align}
Using the convexity and $\Phi$ being of class $\Delta_{2}(\Phi)$, we obtain for all $j\geq \max\{j_{0},j_{1}\}$
\begin{align*}
\int_{\triangle'}\Phi(w-w_{j})+\Phi(\nabla w-\nabla w_{j})\dif x & \leq \frac{\Delta_{2}(\Phi)}{2}\int_{{\triangle'}{_{\kappa}^{\complement}}}\Phi(w)+\Phi(w_{j})+\Phi(\nabla w)+\Phi(\nabla w_{j})\dif x \\ 
& + \int_{\bigcup_{i\in\mathcal{I}}\ball^{i}}\Phi(w-w_{j})+\Phi(\nabla w-\nabla w_{j})\dif x <\varepsilon'
\end{align*}
by~\eqref{eq:monkfix} and~\eqref{eq:NTeeger}. As indicated above, this implies~\eqref{eq:extensionOrlicz}.

We then pick $b_{j,\triangle'}\in\R^{N}$ such that $w_{j}-\mathscr{J}_{\triangle'}a_{\triangle'}+b_{j,\triangle'}$ has zero mean on $\tau\triangle'$ and put, with $\rho_{\triangle'}\in\hold_{c}^{\infty}(\tau\triangle';[0,1])$ satisfying $\mathbbm{1}_{\triangle'}\leq\rho_{\triangle'}\leq\mathbbm{1}_{\tau\triangle'}$ and $|\nabla\rho_{\triangle'}|\leq \frac{c}{(\tau-1)\mathrm{diam}(\triangle')}$, 
\begin{align*}
v_{j}:=\begin{cases} 
u_{j}+b_{j,\triangle'}&\text{in}\;\triangle', \\ 
a_{\triangle'} + \rho_{\triangle'}(w_{j}-\mathscr{J}_{\triangle'}a_{\triangle'}+b_{j,\triangle'})&\;\text{in}\;\tau\triangle'\setminus\triangle'. 
\end{cases} 
\end{align*}
Since $\widetilde{\mathbb{E}}$ is trace-preserving and the traces of $a_{\triangle'}$ and $\mathscr{J}_{\triangle'}a_{\triangle'}$ coincide along $\partial\triangle'$, we conclude that $v_{j}\in\sobo^{1,\Phi}(\tau\triangle';\R^{N})$. On the other hand, by the growth condition on $F$, $v_{j}-a_{\triangle'}\in\sobo_{0}^{1,\Phi}(\tau\triangle';\R^{N})$ is admissible in the definition of quasiconvexity and therefore 
\begin{align}\label{eq:neutralest}
\tau^{n}\mathscr{L}^{n}(\triangle')F(\nabla a_{\triangle'}) & \leq \int_{\tau\triangle'\setminus\triangle'}F(\nabla v_{j})\dif x + \int_{\triangle'}F(\nabla u_{j})\dif x =: \mathrm{IV}_{j}+\mathrm{V}_{j}. 
\end{align}
Next note that by the convexity and the $\Delta_{2}$-assumption on $\Phi$, we have 
\begin{align}\label{eq:wordsup}
\begin{split}
\int_{\tau\triangle'}\Phi(\nabla (w_{j}-\mathscr{J}_{\triangle'}a_{\triangle'}))\dif x & \leq c\int_{\tau\triangle'}\Phi(\nabla (w_{j}-w))\dif x  \\ & \!\!\!\!\!\!\!\!\!\!\!\!\!\!\!\!\!\!\!\!+ c\int_{\tau\triangle'}\Phi(\nabla (w-\mathscr{J}_{\triangle'}a_{\triangle'}))\dif x \\
& \!\!\!\!\!\!\!\!\!\!\!\!\!\!\!\!\!\!\!\!\!\!\!\stackrel{\eqref{eq:extensionOrlicz}}{\to} c\int_{\tau\triangle'}\Phi(\nabla (w-\mathscr{J}_{\triangle'}a_{\triangle'}))\dif x \\ 
& \!\!\!\!\!\!\!\!\!\!\!\!\!\!\!\!\!\!\!\!\!\!\!\stackrel{\eqref{eq:OrliczExtension}}{\leq} c\int_{\triangle'}\Phi(w-a_{\triangle'})\dif x + c\int_{\triangle'}\Phi(\nabla (w-a_{\triangle'}))\dif x\\
& \!\!\!\!\!\!\!\!\!\!\!\!\!\!\!\!\!\!\!\!\!\!\!\!\!\!\!\!\!\!\!\!\!\!\!\!\stackrel{\text{Lem.~\ref{lem:extensionoperator}~\ref{item:extension3a},\;\eqref{eq:Orliczgradstab}}}{\leq}c\int_{\triangle'}\Phi(u-a_{\triangle'})\dif x + \int_{\triangle'}\Phi(\nabla (u-a_{\triangle'}))\dif x
\end{split}
\end{align}
as $j\to\infty$. By the growth condition on $F$ and the $\Delta_{2}$-condition of $\Phi$, we have with constants $\mathtt{C}=\mathtt{C}(\tau-1,\Phi)>0$, $\mathtt{C}'=\mathtt{C}'(\tau-1,\Phi)>0$ (that satisfy $\limsup_{\tau\searrow 1} \mathtt{C}(\tau-1,\Phi),\mathtt{C}'(\tau-1,\Phi)=\infty$) by Poincar\'{e}'s inequality in $\sobo^{1,\Phi}(\tau\triangle';\R^{N})$
\begin{align*}
\mathrm{IV}_{j} & \leq C\int_{\tau\triangle'\setminus\triangle'}1+\Phi(\nabla a_{\triangle'})+\Phi\Big(\frac{w_{j}-\mathscr{J}_{\triangle'}a_{\triangle'}+b_{j,\triangle'}}{\tau-1}\Big)+\Phi\big(\nabla (w_{j}-\mathscr{J}_{\triangle'}a_{\triangle'}) \big)\dif x \\
& \leq C(\tau^{n}-1)\mathscr{L}^{n}(\triangle')(1+\Phi(\nabla a_{\triangle'})) + \mathtt{C}\int_{\tau\triangle'}\Phi(\nabla (w_{j}-\mathscr{J}_{\triangle'}a_{\triangle'}))\dif x \\ 
& \!\!\stackrel{\eqref{eq:wordsup}}{\to :} \mathrm{IV} \leq C(\tau^{n}-1)\mathscr{L}^{n}(\triangle')(1+\Phi(\nabla a_{\triangle'})) \\ & + \mathtt{C}'\Big(\int_{\triangle'}\Phi(u-a_{\triangle'})\dif x + c\int_{\triangle'}\Phi(\nabla (u-a_{\triangle'}))\dif x\Big)
\end{align*}
as $j\to\infty$, which in conjunction with~\eqref{eq:neutralest} yields 
\begin{align}\label{eq:OrliczCombine}
\begin{split}
\tau^{n}\mathscr{L}^{n}(\triangle')F(\nabla a_{\triangle'}) & \leq ( \liminf_{j\to\infty}\int_{\triangle'}F(\nabla u_{j})\dif x\Big)  \\ & + C(\tau^{n}-1)\mathscr{L}^{n}(\triangle')(1+\Phi(\nabla a_{\triangle'}))\\ 
&  + \mathtt{C}'\Big(\int_{\triangle'}\Phi(u-a_{\triangle'})\dif x + \int_{\triangle'}\Phi(\nabla (u-a_{\triangle'}))\dif x\Big). 
\end{split}
\end{align}
Summing inequality~\eqref{eq:OrliczCombine} over the finite set of all simplices $\triangle'\in\mathcal{T}'_{\Omega}$, we find
\begin{align*}
\mathrm{VI} & := \tau^{n} \int_{\bigcup_{\triangle'\in\mathcal{T}'_{\Omega}}\triangle'}F(\nabla a_{\mathcal{T}'})\dif x \leq  \liminf_{j\to\infty} \int_{\bigcup_{\triangle'\in\mathcal{T}'_{\Omega}}\triangle'}F(\nabla u_{j})\dif x + C(\tau^{n}-1)\times \\ & \times\int_{\bigcup_{\triangle'\in\mathcal{T}'_{\Omega}}\triangle'}(1+\Phi(\nabla a_{\mathcal{T}'}))\dif x + \mathtt{C}'\int_{\bigcup_{\triangle'\in\mathcal{T}'_{\Omega}}\triangle'}\Phi((u-a_{\triangle'}))+\Phi(\nabla (u-a_{\triangle'}))\dif x \\ & =:\mathrm{VII}+\mathrm{VIII}+\mathrm{IX}. 
\end{align*}
We then have 
\begin{align*}
\mathrm{VI} & = \tau^{n} \int_{\bigcup_{\triangle'\in\mathcal{T}'_{\Omega}}\triangle'}F(\nabla a_{\mathcal{T}'})-F(\nabla u)\dif x + \tau^{n} \int_{\bigcup_{\triangle'\in\mathcal{T}'_{\Omega}}\triangle'}F(\nabla u)\dif x \\ 
& \geq - \tau^{n} \int_{\bigcup_{\triangle'\in\mathcal{T}'_{\Omega}}\triangle'}|F(\nabla a_{\mathcal{T}'})-F(\nabla u)|\dif x + \tau^{n} \int_{\bigcup_{\triangle'\in\mathcal{T}'_{\Omega}}\triangle'}F(\nabla u)\dif x \\ 
& \!\!\!\!\!\!\stackrel{\text{Lem.~\ref{lem:OrliczBound}}}{\geq} - C\tau^{n} \int_{\bigcup_{\triangle'\in\mathcal{T}'_{\Omega}}\triangle'}\frac{\Phi(1+|\nabla u|+|\nabla a_{\mathcal{T}'}|)}{1+|\nabla u|+|\nabla a_{\mathcal{T}'}|}|\nabla u-\nabla a_{\mathcal{T}'}|\dif x \\ & + \tau^{n} \int_{\bigcup_{\triangle'\in\mathcal{T}'_{\Omega}}\triangle'\setminus\Omega}F(\nabla u)\dif x + \tau^{n} \int_{\Omega}F(\nabla u)\dif x =: \mathrm{VI}_{1}+\mathrm{VI}_{2}+\mathrm{VI}_{3}. 
\end{align*}
We then first send $\varepsilon\searrow 0$ and hereafter $\delta\searrow 0$. Recalling how $a_{\mathcal{T}'}$ is constructed and so $\lim_{\varepsilon\searrow 0}\|\Phi(\nabla (u-a_{\mathcal{T}'}))\|_{\lebe^{1}(\R^{n})}=0$, using a similar argument as in~\eqref{eq:Orlizzo1}\emph{ff}. we find that $\lim_{\varepsilon\searrow 0}\mathrm{VI}_{1}=0$. On the other hand, as $F(\nabla u)\in\lebe^{1}(\R^{n})$, we have $\lim_{\varepsilon\searrow 0}\mathrm{VI}_{2}=0$. Similarly, we obtain that $\lim_{\varepsilon\searrow 0}\mathrm{IX}=0$. Since $u_{j}=u_{0}$ outside $\Omega$, we equally have 
\begin{align*}
\mathrm{VII} = \liminf_{j\to\infty} \int_{\Omega}F(\nabla u_{j})\dif x + \int_{(\bigcup_{\triangle'\in\mathcal{T}'_{\Omega}}\triangle')\setminus\Omega}F(\nabla u_{0})\dif x \stackrel{\varepsilon\searrow 0}{\longrightarrow} \liminf_{j\to\infty} \int_{\Omega}F(\nabla u_{j})\dif x. 
\end{align*}
Finally, since by construction $\|1+\Phi(\nabla a_{\mathcal{T}'})\|_{\lebe^{1}(\R^{n})}$ is bounded as $\varepsilon\searrow 0$, we may send $\tau\searrow 1$ to obtain 
\begin{align*}
\int_{\Omega}F(\nabla u)\dif x \leq \liminf_{j\to\infty}F(\nabla u_{j})\dif x,
\end{align*}
and the proof is complete. 
\end{proof} 

\begin{center}
\vspace{0.5cm}
\textbf{Declaration}
\end{center}
The authors hereby declare that there are no conflicts of interest. 
\end{document}